\documentclass[fleqn]{zzz}
\usepackage{mathrsfs}
\usepackage{xcolor}
\usepackage{tikz}
\usepackage{amsmath,amsthm,amsbsy,amssymb}
\usepackage{graphicx}
\usepackage{rotating}
\usepackage{fixmath}
\usepackage[pdftex]{hyperref}
\usepackage{soul}
\usepackage{stmaryrd}
\usepackage{relsize}
\usepackage{amssymb}
\usepackage{bm}
\usepackage{comment}
\usepackage{soul}
\usetikzlibrary{decorations}
\usepackage{comment}

\pgfdeclaredecoration{dashsoliddouble}{initial}{
 \state{initial}[width=\pgfdecoratedinputsegmentlength]{
 \pgfmathsetlengthmacro\lw{.3pt+.5\pgflinewidth}
 \begin{pgfscope}
 \pgfpathmoveto{\pgfpoint{0pt}{\lw}}%
 \pgfpathlineto{\pgfpoint{\pgfdecoratedinputsegmentlength}{\lw}}%
 \pgfmathtruncatemacro\dashnum{%
 round((\pgfdecoratedinputsegmentlength-3pt)/6pt)
 }
 \pgfmathsetmacro\dashscale{%
 \pgfdecoratedinputsegmentlength/(\dashnum*6pt + 3pt)
 }
 \pgfmathsetlengthmacro\dashunit{3pt*\dashscale}
 \pgfsetdash{{\dashunit}{\dashunit}}{0pt}
 \pgfusepath{stroke}
 \pgfsetdash{}{0pt}
 \pgfpathmoveto{\pgfpoint{0pt}{-\lw}}%
 \pgfpathlineto{\pgfpoint{\pgfdecoratedinputsegmentlength}{-\lw}}%
 \pgfusepath{stroke}
 \end{pgfscope}
 }
}

\hypersetup{
colorlinks = true,
urlcolor = blue,
linkcolor = black,
}

\definecolor{Mycolor2}{HTML}{e85d04}
\sethlcolor{black}
 
\newcommand{\Phyp}[5]{\,\mbox{}_{#1}P_{#2}\!\left({#3};{#4};{#5}\right)}
\newcommand{\Ohyp}[5]{\,\mbox{}_{#1}{\bm{F}}_{#2}\!\left(
\genfrac{}{}{0pt}{}{#3}{#4};#5\right)}
\newcommand{\Whyp}[5]{\,\mbox{}_{#1}W_{#2}\!\left({#3};{#4};{#5}\right)}
\newcommand{\qhyp}[5]{\,\mbox{}_{#1}\phi_{#2}\!\left(\!\!\begin{array}{c}{#3}\\[0.10cm]{#4}\end{array};{#5}\right)}
\newcommand{\qphyp}[6]{\,{}_{#1}\phi_{{#2}}^{{#3}}\!\left(\!\!
\begin{array}{c}{#4}\\[0.10cm] {#5}\end{array};#6\!\right)}

\newcommand{\hyp}[5]{\,\mbox{}_{#1}F_{#2}\!\left(\!\!
\begin{array}{c}{#3}\\ {#4}\end{array};#5\right)}
\newcommand{\qpWhyp}[6]{\,\sideset{_{#1}^{\phantom{\mid}}}{_{#2}^{#3}}
{\mathop{W}}\!\left({#4};{#5};#6\right)}

\newtheorem{thm}{Theorem}[section]
\newtheorem{cor}[thm]{Corollary}

\newtheorem{rem}[thm]{Remark}
\newtheorem{lem}[thm]{Lemma}
\newtheorem{defn}[thm]{Definition}
\newtheorem{prop}[thm]{Proposition}

\makeatletter
\def\eqnarray{\stepcounter{equation}\let\@currentlabel=\theequation
\global\@eqnswtrue
\tabskip\@centering\let\\=\@eqncr
$$\halign to\displaywidth\bgroup\hfil\global\@eqcnt\z@
$\displaystyle\tabskip\z@{##}$&\global\@eqcnt\@ne
\hfil$\displaystyle{{}##{}}$\hfil
&\global\@eqcnt\tw@ $\displaystyle{##}$\hfil
\tabskip\@centering&\llap{##}\tabskip\z@\cr}

\def\endeqnarray{\@@eqncr\egroup
\global\advance\c@equation\m@ne$$\global\@ignoretrue}

\def\@yeqncr{\@ifnextchar [{\@xeqncr}{\@xeqncr[5pt]}}
\makeatother

\newcommand{\Z}{\mathbb{Z}} 
\newcommand{\C}{\mathbb{C}} 
\newcommand{\N}{\mathbb{N}} 

\newcommand{\RR}{{{\mathbb R}}}
\newcommand{\CC}{{{\mathbb C}}}
\newcommand{\CCast}{{{\mathbb C}^\ast}}
\newcommand{\CCdag}{{{\mathbb C}^\dag}}
\newcommand{\CCddag}{{{\mathbb C}^\ddag}}

\newcommand{\expe}{{\mathrm e}}

\newcommand{\dd}{{\mathrm d}}

\usepackage{scalerel,url}
\let\svus_
\catcode`_=\active %
\def_{\ifmmode\expandafter\svus\expandafter\bgroup\expandafter\lowerit
\else\svus\fi}
\def\lowerit#1{\ThisStyle{\raisebox{-2\LMpt}{$\SavedStyle#1$}}\egroup}
\makeatletter
\AtBeginDocument{
\check@mathfonts
\fontdimen13\textfont2=3.5pt
\fontdimen14\textfont2=3.5pt
\fontdimen16\textfont2=2.5pt
\fontdimen17\textfont2=2.5pt
 }
\makeatother

\let\emptyset\varnothing
\begin{document}
\renewcommand{\PaperNumber}{***}

\FirstPageHeading

\ShortArticleName{$q$ and $q^{-1}$-symmetric and dual orthogonal polynomials and functions and $q$-Chaundy representations}

\ArticleName{The $q$ and $q^{-1}$-symmetric orthogonal polynomials in the $q$-Askey scheme, their dual polynomials and functions, orthogonality, generating functions and relations and nonterminating  $q$-Chaundy double product representations}
\Author{Howard S. Cohl$\,^{\ast}\orcidB{}$ 
and 
Roberto S. Costas-Santos$\,^{\dag}\orcidA{}$
}
\AuthorNameForHeading{H.~S.~Cohl and 
R.~S.~Costas-Santos
}
\Address{$^\ast$ Applied and Computational 
Mathematics Division, National Institute 
of Standards 
and Tech\-no\-lo\-gy, Mission Viejo, CA 92694, USA
\URLaddressD{
\href{http://www.nist.gov/itl/math/msg/howard-s-cohl.cfm}
{http://www.nist.gov/itl/math/msg/howard-s-cohl.cfm}
}
} 
\EmailD{howard.cohl@nist.gov} 

\Address{$^\dag$ Department of Quantitative 
Methods, Universidad Loyola Andaluc\'ia, 
E-41704 Seville, Spain
} 
\URLaddressD{
\href{http://www.rscosan.com}
{http://www.rscosan.com}
}
\EmailD{rscosa@gmail.com} 


\ArticleDates{Received~\today~in final form ????; 
Published online ????}
\Abstract{
We derive double-product representations of 
nonterminating basic hypergeometric series
using diagonalization, a method introduced by 
Theo William Chaundy in 1943. 
We refer to this result as the $q$-Chaundy theorem 
and several limiting $q\to 1^{-}$ cases are considered.
Using the $q$-Chaundy theorem, we explore properties 
of the symmetric and $q^{-1}$-symmetric basic 
hypergeometric orthogonal polynomials in the $q$-Askey 
scheme. 
These are the continuous dual $q$ and $q^{-1}$-Hahn 
polynomials, the $q$ and $q^{-1}$-Al-Salam--Chihara 
polynomials, the continuous big $q$ and $q^{-1}$-Hermite 
polynomials and the continuous $q$ and $q^{-1}$-Hermite 
polynomials. 
For instance, we show how many known (and unknown) 
generating functions can be easily derived for these 
polynomials. 
We also explore other methods to find generating 
functions for these polynomials.
By applying the $q$-Chaundy theorem to the 
Ismail--Masson $q$-exponential generating function for 
continuous $q$ and $q^{-1}$-Hermite polynomials, we 
are able to derive alternative expansions of these 
generating functions, and from these, new terminating 
basic hypergeometric representations for the continuous 
$q$ and $q^{-1}$-Hermite polynomials. 
New quadratic transformations for the terminating basic 
hypergeometric series involved connect these representations. 
For the $q$ and $q^{-1}$-symmetric subfamilies of the 
Askey--Wilson polynomials and as well their dual polynomials, 
which include the big and little $q$-Jacobi polynomials and 
the $q^{-1}$-Bessel polynomials, we discuss, and show how 
to exploit special orthogonality relations (integral and 
infinite series), connection formulas, and duality relations 
for these infinite families to derive new generating relations, 
and as well summation and integration formulas.
}.


\Classification{33D45, 05A15, 42C05, 33D15}


\tableofcontents
\addtocontents{toc}{\protect\color{black}}
\section{Introduction}\label{sec:1}

In this paper, we exploit a method introduced by 
Theo William Chaundy in 1943 (see 
 \cite[just below (74)]{Chaundy43}) for re-expressing
double summation of nonterminating expressions in terms 
of an infinite sum of terminating expressions. 
This method is sometimes referred 
to as diagonal summation or simply diagonalization.
Chaundy applied this method to re-write products of 
generalized hypergeometric series.
Two nice examples which appeared in Chaundy's original paper are the following formulas
cf.~\cite[(72), (74)]{Chaundy43}
\begin{eqnarray}
&&\hyp11{a}{b}{z}
\hyp11{c}{d}{w}
=\sum_{n=0}^\infty
\frac{(a)_n}{(b)_n}
\frac{z^n}{n!}\hyp32{-n,1-b-n,c}{1-a-n,d}{-\frac{w}{z}},\\
&&\hyp21{a,b}{c}{z}
\hyp21{d,e}{f}{w}
=\sum_{n=0}^\infty
\frac{(a)_n(b)_n}{(c)_n}
\frac{z^n}{n!}\hyp43{-n,1-c-n,d,e}{1-a-n,1-b-n,f}{\frac{w}{z}},
\end{eqnarray}
which express the product of two nonterminating generalized hypergeometric functions as an infinite sum over terminating generalized basic hypergeometric functions.
Sometimes the 
formulas which result from this method for a product 
of two generalized hypergeometric functions
lead to very beautiful representations
in terms of a single generalized hypergeometric series. 
For several nice examples, see for instance Clausen's 
formula \cite{Clausen1828} (see also \cite[(8.8.1)]{GaspRah})
\begin{equation}
\left\{
\hyp21{a,b}{a+b+\frac12}{z}\right\}^2
=\hyp32{2a,2b,a+b}{a+b+\frac12,2a+2b}{z},
\end{equation}
and Bailey's formula \cite[(2.11)]{Bailey1928}
\begin{equation}
\hyp11{a}{2a}{z}\hyp11{b}{2b}{-z}=\hyp23{\frac12(a+b), 
\frac12(a+b+1)}{a+\frac12,b+\frac12,a+b}{\frac14z^2}.
\end{equation}
Other nice examples can be found
in \cite[\S4.3]{ErdelyiHTF}.
The goal of this paper is to extend the general results 
which were conceived of by Chaundy to the $q$-realm and 
to investigate some of their applications. One particular 
focus of this paper is to utilize our derived $q$-Chaundy 
result to provide some easy proofs for some of the most 
important product generating functions for basic hypergeometric
orthogonal polynomials in the $q$-Askey scheme and also 
for their $q^{-1}$ analogues. See for instance 
\cite{MR4421370,BergIsmail1996,GasperRahman2005,Koekoeketal,KoelinkStokman2003}, 
\cite[\href{http://dlmf.nist.gov/18.28}{\S18.28}]{NIST:DLMF} 
for some detailed descriptions of these polynomials and 
their properties.

\medskip
\noindent This manuscript is organized as follows. In \S\ref{sec:2}, \S\ref{sec:3}, we introduce the mathematical preliminaries and orthogonal polynomials, functions and their important properties which we will rely upon respectively; in \S\ref{sec:4O} we present a summary of many known orthogonality relations for the above described family of polynomials/functions; in \S\ref{sec:4} we give some general theorems for the expansion of a product of two nonterminating basic hypergeometric series which we refer to as $q$-Chaundy expansions and also derive several limits and special cases including terminating cases; in \S\ref{sec:6} we summarize and derive important generating functions for the above described polynomials/functions;
\section{Mathematical preliminaries}\label{sec:2}

Recall the notion of a {\it multiset} 
which extends the definition of a set where the multiplicity of elements is allowed. This notion becomes important for 
hypergeometric functions, where numerator and denominator parameter entries may be identical.
\begin{defn}\label{def:1.1}
We adopt the following conventions for succinctly 
writing elements of multisets. To indicate 
sequential positive and negative 
elements, we write
\[
\pm a:=\{a,-a\}.
\]
\noindent We also adopt an analogous notation
\[
z^{\pm}:=\{z,z^{-1}\}.
\]
\end{defn}
\noindent We adopt the following set 
notations: $\mathbb N_0:=\{0\}\cup
\mathbb N=\{0, 1, 2,\ldots\}$, and we 
use the sets $\mathbb Z$, $\mathbb R$, 
$\mathbb C$ which represent 
the integers, real numbers, and 
complex numbers respectively, 
$\CCast:=\CC\setminus\{0\}$, 
$\CCdag:=\{z\in\CCast: |z|<1\}$,\
$\CCddag:=\CC\setminus(\{0\}\cup\{z\in\CCast:|z|=1\})$.
\noindent Consider $q\in\CCdag$, $n\in\mathbb N_0$.
Define the sets 
\begin{eqnarray}
&&\hspace{-7.5cm}\Omega_q^n:=\{q^{-k}: 
k\in\mathbb N_0,~0\le k\le n-1\},\\
&&\hspace{-7.5cm}\Omega_q:=\Omega_q^\infty
=\{q^{-k}:k\in\mathbb N_0\},
\\ &&\hspace{-7.5cm}\Upsilon_q:=\{q^k:k\in\Z\}.
\end{eqnarray}

\subsection{Binomial coefficients, shifted and {\em q}-shifted factorials}\label{sec:2.1}

Let $n,k\in\N_0$ such that $n\ge k$. Then the binomial coefficient is defined by
\begin{equation}
\binom{n}{k}:=\frac{n!}{k!(n-k)!}. 
\end{equation}
Note that $\binom{z}{2}=\frac12z(z-1)$ for $z\in\CC$.
We will use \cite[(1.2.58-59)]{GaspRah}
\begin{eqnarray}
\label{binomic}
&&\hspace{-7.9cm}\binom{n+k}{2}=
\binom{n}{2}+\binom{k}{2}+kn,\\
&&\hspace{-7.9cm}\binom{n-k}{2}
=\binom{n}{2}+\binom{k}{2}+k(1-n).
\label{binomid}
\end{eqnarray}
We will need the
shifted factorial \begin{equation}
(a)_n:=(a)(a+1)\cdots(a+n-1),
\end{equation}
and one also has
\begin{equation} (a)_n=\frac{\Gamma(a+n)}{\Gamma(a)},
\end{equation}
where $a+n\not\in-\N_0$ and $\Gamma$ is the gamma function \cite[\href{http://dlmf.nist.gov/5}{Chapter 5}]{NIST:DLMF}.
One useful limit is
\begin{equation}
\lim_{t\to \infty} \dfrac{(\lambda\pm \mu i t)_m}
{t^{2m}}=\mu^{2m}.
\label{limPochtmu}
\end{equation}

\medskip
\noindent We also need 
the 
$q$-shifted factorial 
$(a;q)_n=(1-a)(1-qa)\cdots(1-q^{n-1}a)$, 
$n\in\mathbb N_0$ and 
one may define
\begin{eqnarray}
&&\hspace{-11.3cm}(a;q)_\infty:=\prod_{n=0}^\infty 
(1-aq^{n}),\label{poch.id:2}
\end{eqnarray}
where $|q|<1$. 
Note that infinite $q$-shifted factorials $(a;q)_\infty$ are not defined when they vanish in a denominator, that is if $a\in\Omega_q$. There are many places in the sequel where these cases must be avoided. In most cases below, we will not mention this.
Furthermore, define 
\[
(a;q)_b:=\frac{(a;q)_\infty}
{(a q^b;q)_\infty},
\]
where $a q^b\not \in \Omega_q$.
We will also use the common 
notational product conventions
\begin{eqnarray}
&&\hspace{-8cm}(a_1, \ldots, a_k)_b:=
(a_1)_b\cdots(a_k)_b,\nonumber\\
&&\hspace{-8cm}(a_1, \ldots, a_k;q)_b:=
(a_1;q)_b\cdots(a_k;q)_b,\nonumber
\end{eqnarray}
where $b\in{\mathbb C}\cup\{\infty\}$.

The $q$-shifted factorial 
also has the 
following useful properties
\cite[(1.8.10), (1.8.10), 
(1.8.11), (1.8.17), (1.8.19), (1.8.22)]
{Koekoeketal}:
\begin{eqnarray}
\label{poch.id:3} 
&&\hspace{-3.1cm}
(a;q^{-1})_n=q^{-\binom{n}{2}}(-a)^n(a^{-1};q)_n,\\
&&\hspace{-3.1cm}(a;q)_{n+k}=(a;q)_k(aq^k;q)_n 
= (a;q)_n(aq^n;q)_k,
\label{qPoch1}
\\
\label{poch.id:5}&&\hspace{-3.1cm} (a;q)_n
=(q^{1-n}/a;q)_n(-a)^nq^{\binom{n}{2}},\\
&&\hspace{-3.1cm}\frac{(a;q)_{n-k}}{(b;q)_{n-k}}=
\left(\frac{b}{a}\right)^k
\frac{(a;q)_n(\frac{q^{1-n}}{b};q)_k}
{(b;q)_n(\frac{q^{1-n}}{a};q)_k},\quad 
a,b\ne 0,\ k=0,1,2,\ldots,n, 
\label{qPoch2}\\
\label{qPochq3}
&&\hspace{-3.1cm}
(q^{-n-k};q)_k=q^{-\binom{k}{2}}
(-q)^{-k}q^{-nk}
\frac{(q;q)_k(q^{1+k};q)_n}{(q;q)_n}
\\
\label{qPochq2}&&\hspace{-3.1cm}
(\pm a;q)_n=(a^2;q^2)_n.
\end{eqnarray}
From \eqref{qPochq3}, one also has \cite[(1.8.18)]{Koekoeketal}
\begin{equation}
(q^{-m};q)_n=(-1)^{n}q^{\binom{n}{2}}q^{-nm}\frac{(q;q)_m}{(q;q)_{m-n}}.
\label{qPochiden2}
\end{equation}
Note the equivalent representation of \eqref{poch.id:3} 
which is very useful for obtaining limits which we
often need is
\begin{equation*}
a^n\left(\frac{x}{a};q\right)_n=
q^{\binom{n}{2}}(-x)^n\left(\frac{a}{x};q^{-1}\right)_n,
\end{equation*}
therefore
\begin{equation}
\lim_{a\to0}\,a^n\left(\frac{x}{a};q\right)_n=
\lim_{b\to\infty}\,\frac{1}{b^n}\left(xb;q\right)_n=
q^{\binom{n}{2}}(-x)^n.
\label{critlim}
\end{equation}

\noindent 
Other useful limit representations include
\begin{equation}
\lim_{\lambda\to\infty}\frac{(a\lambda;q)_n}{(b\lambda;q)_n}
=\left(\frac{a}{b}\right)^n,
\label{lambdamu}
\end{equation}
which follows from \eqref{poch.id:5}, 
and
\begin{equation}
\lim_{\alpha\to0}\frac{(\frac{a}{\alpha};q^2)_k}{(\frac{b}{\alpha};q)_k}=q^{\binom{k}{2}}\left(\frac{a}{b}\right)^k.
\label{limq2}
\end{equation}

The {\it theta function} $\vartheta(z;q)$ (sometimes referred to as a modified theta function 
\cite[(11.2.1)]{GaspRah})
is defined
by 
Jacobi's triple product identity and is
given by {\cite[(1.6.1)]{GaspRah}} (see also \cite[(2.3)]{Koornwinder2014})
\begin{equation}
\vartheta(z;q):=
(z,q/z;q)_\infty=\frac{1}{(q;q)_\infty}\sum_{n=-\infty}^\infty (-1)^nq^{\binom{n}{2}}z^n,
\label{tfdef}
\end{equation}
where $z\ne 0$, $|q|<1$. Note that $\vartheta(q^n;q)=0$ if
$n\in\Z$.
We will adopt the product convention
for theta functions for $a_k\in\C$ for $k\in\N$, namely
\begin{eqnarray}
&&\hspace{-8cm}\vartheta(a_1,\ldots,a_k;q):=\vartheta(a_1;q)\cdots\vartheta(a_k;q).\nonumber
\end{eqnarray}

\noindent Furthermore, one has the following
identities
 \cite[p.~8]{Ismailetal2022},
 \cite[(6)]{CohlCostasSantos23},
\begin{eqnarray}
&&\hspace{-9.4cm}\label{sq}
(a^2;q)_\infty=(\pm a,\pm q^\frac12 a;q)_\infty,\\
&&\hspace{-9.4cm}\label{sqroot}
(a;q^\frac12)_\infty=(a,q^\frac12 a;q)_\infty,\\
&&\hspace{-9.5cm}\label{aiden}
\frac{(a,\frac{q}{a};q)_\infty}{(qa,\frac{1}{a};q)_\infty}=\frac{\vartheta(a;q)}{\vartheta(qa;q)}=-a.
\end{eqnarray}

\noindent 
Define the Jackson $q$-integral as in \cite[(1.11.2)]{GaspRah}
\begin{eqnarray}
&&\hspace{-2.3cm}\int_a^b f(u;q)\,{\mathrm d}_qu=(1-q)b\sum_{n=0}^\infty q^nf(q^nb;q)-(1-q)a\sum_{n=0}^\infty q^nf(q^na;q).
\label{qint}
\end{eqnarray}

\subsection{Generalized hypergeometric series}
Generalized hypergeometric functions provide an important generalization of the {\it Gauss hypergeometric function} defined
by \cite[\href{http://dlmf.nist.gov/15.2.E1}{(15.2.1)}]{NIST:DLMF}
\begin{equation}
\hyp21{a,b}{c}{z}:=\sum_{k=0}^\infty
\frac{(a)_k(b)_k}{(c)_k}\frac{z^k}{k!},
\end{equation}
where $c\not\in-\N_0$ and $|z|<1$.
Let ${\bf a}:=\{a_1,\ldots,a_r\}$,
${\bf b}:=\{b_1,\ldots,b_s\}$, $r,s\in\mathbb N_0$ be multisets of cardinality $r,s\in\mathbb N_0$ respectively.
The {\it generalized hypergeometric
function} \cite[\href{http://dlmf.nist.gov/16}{Chapter 16}]{NIST:DLMF} is defined by the 
infinite series \cite[\href{http://dlmf.nist.gov/16.2.E1}{(16.2.1)}]{NIST:DLMF}
\begin{equation} \label{genhyp}
\hyp{r}{s}{{\bf a}}
{{\bf b}}{z}
:=
\sum_{k=0}^\infty
\frac{({\bf a})_k}
{({\bf b})_k}
\frac{z^k}{k!},
\end{equation}
where $b_j\not \in -\mathbb N_0$, for 
$j\in\{1, \dots, s\}$;
and elsewhere by analytic continuation.
Further, define Olver's (scaled or regularized) 
generalized hypergeometric series
\begin{equation}\label{genhyp-b}
\Ohyp{r}{s}{{\bf a}}{{\bf b}}{z}:=
\frac{1}{\Gamma({\bf b})}\hyp{r}{s}{{\bf a}}{{\bf b}}{z}=
\sum_{k=0}^\infty \frac{({\bf a})_k}{\Gamma({\bf b}+k 
)}
\frac{z^k}{k!},
\end{equation}
which is entire for all $a_l,b_j\in\mathbb C$, 
$l\in\{1,\ldots,r\}$, $j\in\{1,\ldots,s\}$.
Both the generalized and Olver's generalized 
hypergeometric series, if nonterminating, 
are entire if $r\le s$, convergent for 
$|z|<1$ if $r=s+1$ and divergent if $r\ge s+1$.

\subsubsection{Generalized hypergeometric orthogonal polynomials in the Askey scheme}
\label{Askeys}

The Askey scheme is a scheme of hypergeometric orthogonal polynomials. We will consider some of the polynomials in this scheme which satisfy continuous orthogonality relations. These are considered classic hypergeometric orthogonal polynomials. 
We will now introduce some of these classical hypergeometric orthogonal polynomials which can be defined in terms of terminating generalized hypergeometric series.
Let $n\in\N_0$.
The Wilson polynomials can be defined as
\cite[(9.1.1)]{Koekoeketal}
\begin{equation}
W_n(x^2;a,b,c,d):=(a+b,a+c,a+d;q)_n\hyp43{-n,n+a+b+c+d-1,a\pm ix}{a+b,a+c,a+d}{1}.
\label{Wildef}
\end{equation}
The continuous dual Hahn polynomials can be
defined as \cite[(9.3.1)]{Koekoeketal}
\begin{equation}
S_n(x^2;a,b,c):=(a+b,a+c;q)_n\hyp32{-n,a\pm ix}{a+b,a+c}{1}.
\label{cdHdef}
\end{equation}
The Jacobi polynomials can be defined as
\cite[(9.8.1)]{Koekoeketal}
\begin{equation}
P_n^{(\alpha,\beta)}(x)=\dfrac{(\alpha+1)_n}{n!}
\hyp21{-n,\alpha+\beta+n+1}{\alpha+1}{\frac{1-x}2}.
\label{Jacdef}
\end{equation}
Jacobi polynomials can be obtained as a limit from the Wilson polynomials 
\cite[p.~219]{Koekoeketal}
\begin{equation}
\label{limW2J}
\lim_{t\to \infty} 
\dfrac{W_n\big(\frac12(1-x)t^2;\frac12(\alpha+1),\frac12(\alpha+1),
\frac12(\beta+1)\pm it\big)}{n!\, t^{2n}}=
P_n^{(\alpha,\beta)}(x).
\end{equation}
The ultraspherical polynomials can be defined as \cite[(9.8.19)]{Koekoeketal}
\begin{equation}
C_n^\lambda(x)=\frac{(2\lambda)_n}{n!}\hyp21{-n,n+2\lambda}{\lambda+\frac12}{\frac{1-x}{2}}.
\end{equation}
The Legendre polynomials can be defined as \cite[(9.8.62)]{Koekoeketal}
\begin{equation}
P_n(x)=\hyp21{-n,n+1}{1}{\frac{1-x}{2}}.
\end{equation}
The Chebyshev polynomials of the first kind can be defined as \cite[(9.8.35)]{Koekoeketal}
\begin{equation}
T_n(x)=\hyp21{-n,n}{\frac12}{\frac{1-x}{2}}.
\end{equation}
The Chebyshev polynomials of the second kind can be defined as \cite[(9.8.36)]{Koekoeketal}
\begin{equation}
U_n(x)=(n+1)\hyp21{-n,n+2}{\frac32}{\frac{1-x}{2}}.
\end{equation}
The Hermite polynomials can be defined as \cite[(9.15.1)]{Koekoeketal}
\begin{equation}
H_n(x)=(2x)^n\hyp20{-\frac{n}{2},-\frac12(n-1)}{-}{-\frac{1}{x^2}}.
\label{Hermdef}
\end{equation}

\subsection{Basic hypergeometric series}\label{sec:2.1b}

Define ${\bf a}:=\{a_1,\ldots,a_r\}$,
${\bf b}:=\{b_1,\ldots,b_s\}$. 
The basic hypergeometric series, which we 
will often use, is defined for
$q,z\in\CCast$ such that $|q|<1$, $s,r\in\mathbb N_0$, 
$b_j\not\in\Omega_q$, $j=1, \ldots, s$, as
 \cite[(1.10.1)]{Koekoeketal}
\begin{equation}
\qhyp{r}{s}{\bf a}
{\bf b}
{q,z}:=
{}_r\phi_s({\bf a};{\bf b};q,z)
:=\sum_{k=0}^\infty
\frac{({\bf a};q)_k}
{(q,{\bf b};q)_k}
\left((-1)^kq^{\binom k2}\right)^{1+s-r}
z^k.
\label{2.11}
\end{equation}
\noindent {For $s+1>r$, ${}_{r}\phi_s$ is an entire
function of $z$, for $s+1=r$ then 
${}_{r}\phi_s$ is convergent for $|z|<1$, and 
for $s+1<r$ the series
is divergent unless it is terminating.}
Note that when we refer to a basic hypergeometric
function with {\it arbitrary argument} $z$, we 
mean that the argument does not necessarily 
depend on the other parameters, namely the $a_j$'s, 
$b_j$'s nor $q$. However, for the arbitrary 
argument $z$, it very-well may be that the domain 
of the argument is restricted, such as for $|z|<1$.

Redefine the multiset ${\bf a}:=
\{a_1,\ldots,a_{r-1}\}$.
If one has a nonterminating basic hypergeometric series which has a numerator parameter given by some $q^{-n}$ for $n\in\N_0$, then the basic hypergeometric series {\it terminates} because $(q^{-n};q)_k=0$ for $k\ge n+1$. This leads to the following definition.
A {\it terminating} basic hypergeometric series 
which appear in basic hypergeometric orthogonal
polynomials are defined as
\begin{equation}
\qhyp{r}{s}{q^{-n},{\bf a}}
{\bf b}{q,z}:=\sum_{k=0}^n
\frac{(q^{-n},{\bf a};q)_k}{(q,{\bf b};q)_k}
\left((-1)^kq^{\binom k2}\right)^{1+s-r}z^k,
\label{2.12}
\end{equation}
where $b_j\not\in\Omega_q^n$, $j=1, \ldots, s$.
Whereas for $s+1<r$, a nonterminating basic hypergeometric series is divergent, for a terminating basic hypergeometric series, this is no longer the case.

Note that we refer to a basic hypergeometric
series as {\it $\ell$-balanced} if
$q^\ell a_1\cdots a_r=b_1\cdots b_s$, 
and {\it balanced} if $\ell=1$.
A basic hypergeometric series ${}_{r+1}\phi_r$ is 
{\it well-poised} if 
the parameters satisfy the relations
\[
qa_1=b_1a_2=b_2a_3=\cdots=b_ra_{r+1}.
\]
It is {\it very-well-poised} if in addition, 
$\{a_2,a_3\}=\pm qa_1^\frac12$.

Define the 
well-poised 
basic hypergeometric series
${}_{r+1}P_r$ \cite[(2.1.11)]{GaspRah}
\begin{equation}
\label{rpPr}
{}_{r+1}P_r(a;a_2,\ldots,a_{r+1};q,z)
:=\qhyp{r+1}{r}{a,a_2,\ldots,a_{r+1}}
{\frac{qa}{a_2},\ldots,\frac{qa}{a_{r+1}}}{q,z},
\end{equation} 
where $\frac{qa}{a_2},\ldots,\frac{qa}{a_{r+1}}\not\in\Omega_q$ 
and the very-well-poised basic hypergeometric series
${}_{r+1}W_r$ \cite[(2.1.11)]{GaspRah}
\begin{eqnarray}
\hspace{-2.4cm}
{}_{r+1}W_r(a;a_4,\ldots,a_{r+1};q,z)&:=&{}_{r+1}P_{r}
(a;\pm q\sqrt{a},a_4,\ldots,a_{r+1};q,z)\nonumber\\
&=&\qhyp{r+1}{r}{a,\pm q\sqrt{a},a_4,\ldots,a_{r+1}}
{\pm \sqrt{a},\frac{qa}{a_4},\ldots,\frac{qa}{a_{r+1}}}
{q,z},\label{rpWr}
\end{eqnarray} 
where $\sqrt{b},\frac{qb}{a_4},\ldots,\frac{qb}
{a_{r+1}}\not\in\Omega_q$. 
\begin{prop}
For any $a\in \mathbb C$, $a\ne 1$, 
the following identity holds:
\begin{eqnarray}
&&\hspace{-1.0cm}
{}_{r+1}W_r(a;a_4,\ldots,a_{r+1};q,z)
\nonumber\\ &&\hspace{0.1cm}
=\frac 1{1-a}\biggl(
{}_{r-1}P_{r-2}
(a;a_4,\ldots,a_{r+1};q,z)-
a\, {}_{r-1}P_{r-2}
(a;a_4,\ldots,a_{r+1};q,q^2z)
\biggr)
\label{Whyp1}\\ 
&&\hspace{0.1cm}
=\frac{(qa;q)_\infty}{(a;q)_\infty}
\Phyp{r-1}{r-2}{a}{a_4,\ldots,a_{r+1}}{q,z}
+\frac{(\frac{q}{a};q)_\infty}{(\frac{1}{a};q)_\infty}
\Phyp{r-1}{r-2}{a}{a_4,\ldots,a_{r+1}}{q,q^2z}.
\label{Whyp2}
\end{eqnarray} 
\end{prop}
\begin{proof}
Taking advantage of the relation
\[
\dfrac{(\pm q\sqrt a;q)_k}
{(\pm \sqrt a;q)_k}=\dfrac{1-q^{2k}a}
{1-a},
\]
and inserting this in the infinite series representation of ${}_{r+1}W_r$ and writing as two separate terms provides \eqref{Whyp1}.
The representation \eqref{Whyp2} follows from \eqref{Whyp1} by
inserting \eqref{aiden}. This
completes the proof.
\end{proof}

In the sequel, we will use the following notation 
${}_{r+1}\phi_s^m$, $m\in\mathbb Z$
(originally due to van de Bult \& Rains
\cite[p.~4]{vandeBultRains09}), 
for basic hypergeometric series with
zero parameter entries.
Consider $p\in\mathbb N_0$. 
Then define
\begin{equation}\label{topzero} 
{}_{r+1}\phi_s^{-p}\left(\begin{array}{c}{\bf a}\\
{\bf b}\end{array};q,z
\right)
:=
\qhyp{r+p+1}{s}
{{\bf a},\overbrace{0,\ldots,0}^{p}}
{\bf b}{q, z},
\end{equation}
\begin{equation}\label{botzero}
{}_{r+1}\phi_s^{\,p}\left(\begin{array}{c}{\bf a}\\
{\bf b}\end{array};q,z\right)
:=
\qhyp{r+1}{s+p}{\bf a}
{{\bf b},
\overbrace{0,\ldots,0}^{p}}{q,z},
\end{equation}
where $b_1,\ldots,b_s\not
\in\Omega_q\cup\{0\}$, and
${}_{r+1}\phi_s^0:={}_{r+1}\phi_{s}$.
The nonterminating basic hypergeometric series 
${}_{r+1}\phi_s^m({\bf a};{\bf b};q,z)$, 
${\bf a}:=\{a_1,\ldots,a_{r+1}\}$,
${\bf b}:=\{b_1,\ldots,b_s\}$, is well-defined for 
$s-r+m\ge 0$. 
In particular ${}_{r+1}\phi_s^m$ is an entire function 
of $z$ for $s-r+m>0$, convergent for $|z|<1$ for $s-r+m=0$ 
and divergent if $s-r+m<0$.
Note that we will move interchangeably between the
van de Bult \& Rains notation and the alternative
notation with vanishing numerator and denominator parameters
which are used on the right-hand sides of \eqref{topzero} 
and \eqref{botzero}.

We will often use (frequently without mentioning) the 
following limit transition 
formulas, which can be found in 
\cite[(1.10.3-5)]{Koekoeketal}
\begin{eqnarray}
&&\hspace{-2.4cm}\label{limit1}
\lim_{\lambda\to\infty}
\qhyp{r}{s}{a_1, \ldots, a_{r-1},\lambda a_r}{b_1, \ldots, b_s}
{q,\frac{z}{\lambda}}
=\qhyp{r-1}{s}{a_1, \ldots, a_{r-1}}{b_1, \ldots, b_s}{q,a_rz},\\
&&\hspace{-2.4cm}\label{limit2}\lim_{\lambda\to\infty}
\qhyp{r}{s}{a_1, \ldots, a_{r}}{b_1, \ldots, b_{s-1},
\lambda b_s}{q,\lambda z}=\qhyp{r}{s-1}{a_1, \ldots, a_{r}}
{b_1, \ldots, b_{s-1}}{q,\frac{z}{b_s}},\\
&&\hspace{-2.4cm}\label{limit3}\lim_{\lambda\to\infty}
\qhyp{r}{s}{a_1, \ldots, a_{r-1},\lambda a_r}{b_1, \ldots, b_{s-1},
\lambda b_s}{q,z}
=\qhyp{r-1}{s-1}{a_1, \ldots, a_{r-1}}{b_1, \ldots, b_{s-1}}
{q,\frac{a_r}{b_s}z}.
\end{eqnarray}
\subsubsection{Generalized hypergeometric series:~the $q\to1^{-}$ limit of basic hypergeometric series}
Let $r, s\in\mathbb N_0\cup\{-1\}$, 
for any ${\bf a}\in\CCast^{r+1}$, ${\bf b}\in\CCast^{s}$,
$q\in \CCdag$. 
In order to obtain such limits, we need the following
result 
\cite[\S1.10]{Koekoeketal}
\begin{equation}\label{eq:limbhs2hs}
\lim_{q\to1^{-}} \qhyp{r+1}{s}{q^{\bf a}}{q^{\bf b}}
{q,(q-1)^{s-r}z}=\hyp{r+1}{s}{\bf a}{\bf b}{z}.
\end{equation} 
We extend this result using the 
van de Bult--Rains notation for basic hypergeometric 
series as follows.
\begin{lem}\label{lem:limbhs2hsp}
Let $r,s\in\mathbb N_0\cup\{-1\}$, $p\in\mathbb Z$, 
such that $p\ge r-s$,
for any ${\bf a}\in\CCast^{r+1}$, ${\bf b}\in\CCast^{s}$,
$q\in \CCdag$. Then one has the following limit 
relation for basic hypergeometric series
\begin{eqnarray}
&&\lim_{q\to1^{-}}\qphyp{r+1}{s}{p}{q^{\bf a}}
{q^{\bf b}}{q,(q-1)^{s-r}z}=\hyp{r+1}{s}{\bf a}{\bf b}
{(-1)^pz}.
\label{eq:limbhs2hsp}
\end{eqnarray}
\end{lem}
\begin{proof}
By definition, if $p\ge 0$ then we have
\[
\lim_{q\to1^{-}}\qphyp{r+1}{s}{p}
{q^{\bf a}}{q^{\bf b}}{q,(q-1)^{s-r}z}=
\lim_{q\to1^{-1}}
\sum_{k=0}^{\infty}\dfrac{(q^{\bf a};q)_k}
{(q,q^{\bf b},{\bf e};q)_k}\left((q-1)^{s-r}z\right)^k 
\left((-1)^k q^{{\binom k2}}\right)^{s-r+p},
\]
where ${\bf e}$ is the dimension $p$ null vector. 
Observe that when $q\to1^{-}$, since $({\bf e};q)_k\to 1$
then we obtain the same limit, i.e., 
\[\begin{split}
\lim_{q\to1^{-}}\qphyp{r+1}{s}{p}
{q^{\bf a}}{q^{\bf b}}{q,(q-1)^{s-r}z}=&
\lim_{q\to1^{-}}\sum_{k=0}^{\infty}\dfrac{(q^{\bf a};q)_k}
{(q^{\bf b};q)_k}\dfrac{\left((-1)^p z\right)^k 
(q-1)^{(s-r)k}}{(q;q)_k}
\left((-1)^k q^{\binom k2}\right)^{s-r}\\[3mm]
=&\lim_{q\to1^{-}}
\qhyp{r+1}{s}{q^{\bf a}}{q^{\bf b}}{q,(q-1)^{s-r}(-1)^pz}.
\end{split}\]
If we apply the limiting identity \eqref{eq:limbhs2hs} 
the desired result for $p\ge 0$ holds. 
If $p<0$, then 
\[
\lim_{q\to1^{-}}\qphyp{r+1}{s}{p}
{q^{\bf a}}{q^{\bf b}}{q,(q-1)^{s-r}z}= 
\sum_{k=0}^{\infty}\dfrac{(q^{\bf a},{\bf e};q)_k}
{(q^{\bf b};q)_k}\dfrac{z^k (q-1)^{(s-r)k}}{(q;q)_k}
\left((-1)^k q^{\binom k2}\right)^{s+p-r},
\]
and the remaining part of the proof is analogous 
to the previous case. Hence the 
result holds.
\end{proof}
\subsubsection{Important nonterminating summations}
One has the following important nonterminating summations. 
The $q$-binomial theorem states \cite[(1.11.1)]{Koekoeketal}
\begin{equation}
\label{qbinom}
\qhyp10{a}{-}{q,z}=\frac{(az;q)_\infty}{(z;q)_\infty}, 
\quad q\in\CCdag,\quad |z|<1.
\end{equation}

\noindent Also, one has two $q$-analogues of the exponential 
function which are due to Euler \cite[(1.14.1)]{Koekoeketal} 
(see also \cite[Theorem 12.2.6]{Ismail:2009:CQO}).
\begin{thm}[Euler]
\label{Euler}
Let $q\in\CCdag$, $z\in\CC$. Then
\begin{eqnarray}
\label{qexp}
&&\hspace{-6cm}{\mathrm e}_q(z):=\qphyp00{-1}{-}{-}{q,z}
=\frac{1}{(z;q)_\infty}, \quad |z|<1,\\
&&\hspace{-6cm}{\mathrm E}_q(z)=\qhyp00{-}{-}{q,-z}=(-z;q)_\infty.
\label{qexp2}
\end{eqnarray}
\end{thm}
\begin{proof}
See proof of \cite[Theorem 12.2.6]{Ismail:2009:CQO}.
\end{proof}

\noindent Note that \eqref{qexp} is the $a=0$ special case of the 
$q$-binomial theorem \eqref{qbinom}.
One also has the Cauchy sum
\cite[\href{http://dlmf.nist.gov/17.5.E5}{(17.5.5)}]{NIST:DLMF}
\begin{equation}
\qhyp11{a}{b}{q,\frac{b}{a}}
=\frac{(\frac{b}{a};q)_\infty}{(b;q)_\infty}.
\label{qC}
\end{equation}
And there is also the $q$-Gauss summation 
\cite[\href{http://dlmf.nist.gov/17.6.E1}{(17.6.1)}]{NIST:DLMF}
\begin{equation}
\qhyp21{a,b}{c}{q,\frac{c}{ab}}=
\frac{(\frac{c}{a},\frac{c}{b};q)_\infty}
{(c,\frac{c}{ab};q)_\infty},
\label{qGs}
\end{equation}
where $|c/(ab)|<1$.
\subsubsection{Important terminating summations}
One important terminating summation is the
terminating $q$-binomial sum, namely \cite[(1.11.2)]{Koekoeketal} or
\begin{equation}
\qhyp10{q^{-n}}{-}{q,z}=(q^{-n}z;q)_n=
q^{-\binom{n}{2}}\left(-\dfrac{z}{q}\right)^n\left(\dfrac{q}{z};q\right)_n, 
\label{termqbinom}
\end{equation}
which converges for all values of $z$, but the second equality is not valid for $z=0$.
Other important terminating summations include the 
$q$-Chu--Vandermonde sum 
\cite[\href{http://dlmf.nist.gov/17.6.E2}{(17.6.2)}]{NIST:DLMF}
\begin{equation}
\label{qChuVander}
\qhyp21{q^{-n},a}{b}{q,q}=a^n\frac{\left(\dfrac{b}{a};q\right)_n}{(b;q)_n},
\end{equation}
the reversed version of the $q$-Chu--Vandermonde sum \cite[(II.7)]{GaspRah}
\begin{equation}
\qhyp21{q^{-n},a}{b}{q,\frac{q^nb}{a}}=\frac{\left(\dfrac{b}{a};q\right)_n}{(b;q)_n}.
\label{qChuVanderR}
\end{equation}
One useful limit of \eqref{qChuVander} is
\begin{equation}
\qhyp21{q^{-n},0}{a}{q,q}=\frac{q^{\binom{n}{2}}(-a)^n}{(a;q)_n}.
\label{limqChu}
\end{equation}
One also has the $q$-Pfaff--Saalsch\"utz sum
\cite[\href{http://dlmf.nist.gov/17.7.E4}{(17.7.4)}]{NIST:DLMF}
\begin{equation}
\label{qPS}
\qhyp32{q^{-n},a,b}{c,q^{1-n}\frac{ab}{c}}{q,q}
=\frac{\left(\dfrac{c}{a},\dfrac{c}{b};q\right)_n}{\left(c,\dfrac{c}{ab};q\right)_n}.
\end{equation}
\subsubsection{Important nonterminating transformations}
One has the following nonterminating transformation between a ${}_2\phi_2$ 
and a ${}_2\phi_1$ cf.~ \cite[(III.4)]{GaspRah}
\begin{equation}
\qhyp22{a,b}{c,\frac{abz}{c}}{q,z}=
\frac{(\frac{bz}{c};q)_\infty}{(\frac{abz}{c};q)_\infty}
\qhyp21{a,\frac{c}{b}}{c}{q,\frac{bz}{c}}.
\label{rel2122}
\end{equation}
One also has the following
nonterminating transformations
\cite[(1.13.8-9)]{Koekoeketal}
\begin{eqnarray}
&&\hspace{-1.9cm}\qphyp{1}{0}{1}{a}{-}{q,z}
=(a,z;q)_\infty\qphyp{0}{1}{-2}{-}{z}{q,a}
=(z;q)_\infty\qhyp01{-}{z}{q,az}.
\label{trans10101}
\end{eqnarray}
Another useful nonterminating transformation is
\begin{equation}
\qhyp11{a}{b}{q,z}=(z;q)_\infty\qhyp12{\frac{b}{a}}{b,z}{q,az}.
\label{nt1112}
\end{equation}
Another important nonterminating 
transformation formula that we will use is
Bailey's transformation of a nonterminating very-well-poised ${}_8W_7$ to a sum of two nonterminating balanced ${}_4\phi_3$'s
\cite[\href{http://dlmf.nist.gov/17.9.E16}{(17.9.16)}]{NIST:DLMF}
\vspace{12pt}
\begin{eqnarray}
&&\hspace{-0.9cm}{}_8W_7\left(a;b,c,d,e,f;q,\frac{q^2a^2}{bcdef}\right)
=\frac{(qa,\frac{qa}{de},\frac{qa}{df},\frac{qa}{ef};q)_\infty}
{(\frac{qa}{def},\frac{qa}{d},\frac{qa}{e},\frac{qa}{f};q)_\infty}
\qhyp43{\frac{qa}{bc},d,e,f}{\frac{qa}{b},\frac{qa}{c},\frac{def}{a}}
{q,q}\nonumber\\&&
\hspace{3cm}+\frac{(qa,\frac{q^2a^2}{bdef},\frac{q^2a^2}{cdef},\frac{qa}{bc},d,e,f
;q)_\infty}
{(\frac{q^2a^2}{bcdef},\frac{def}{qa},\frac{qa}{b},\frac{qa}{c},\frac{qa}{d},\frac{qa}{e},\frac{qa}{f};q)_\infty}
\qhyp43{\frac{q^2a^2}{bcdef},\frac{qa}{de},\frac{qa}{df},\frac{qa}{ef}}
{\frac{q^2a^2}{bdef},\frac{q^2a^2}{cdef},\frac{q^2a}{def}}{q,q},
\label{Baileytran}
\end{eqnarray}
where $\frac{qa}{b},\frac{qa}{c},\frac{qa}{d},\frac{qa}{e},\frac{qa}{f},\frac{qa}{def},\frac{def}{qa},\frac{q^2a^2}{bcdef}\not\in\Upsilon_q$. Another important nonterminating transformation that we will use is 
\cite[\href{http://dlmf.nist.gov/17.9.E13}{(17.9.13)}]{NIST:DLMF}
\begin{equation}
\qhyp32{a,b,c}{d,e}{q,\frac{de}{abc}}=\frac{(\frac{e}{b},\frac{e}{c};q)_\infty}{(e,\frac{e}{bc};q)_\infty}\qhyp32{\frac{d}{a},b,c}{d,\frac{qbc}{e}}{q,q}+\frac{(\frac{d}{a},b,c,\frac{de}{bc};q)_\infty}
{(d,e,\frac{bc}{e},\frac{de}{abc};q)_\infty}
\qhyp32{\frac{e}{b},\frac{e}{c},\frac{de}{abc}}{\frac{de}{bc},\frac{qe}{bc}}{q,q}.
\label{3phi2nbtran}
\end{equation}

By taking advantage of Bailey's transformation of a very-well-poised
nonterminating ${}_8W_7$ \eqref{Baileytran}, we can obtain some interesting and insightful single basic hypergeometric series transformation formulas. 
{The following sequence of
transformation formulas is implied by \cite[Figure 2]{vandeBultRains09}.}
\begin{rem}
\label{vd}
In order to simplify the constraints for 
the nonterminating infinite $q$-shifted 
factorials, modified theta functions, and 
nonterminating basic hypergeometric series 
expressions which we will present below, 
we will mostly avoid adding the constraints 
which must occur to prevent vanishing 
denominator factors which are not defined.
\end{rem}

\begin{thm}
\label{thm21}
Let $q\in\CCdag$, $z,a,b,c,d,e\in\CCast$. 
Then

\begin{eqnarray}
\label{W761tran}
&&\hspace{-0.6cm}\qpWhyp{7}{6}{1}{a}{b,c,d,e}{q,\frac{q^2a^2}{bcde}}\!=\!
\frac{(qa,\frac{qa}{de};q)_\infty}{(\frac{qa}{d},\frac{qa}{e};q)_\infty}\qhyp32{\frac{qa}{bc},d,e}{\frac{qa}{b},\frac{qa}{c}}{q,\frac{qa}{de}},
\\
\label{W652tran}
&&\hspace{-0.6cm}\qpWhyp{6}{5}{2}{a}{b,c,d}{q,\frac{q^2a^2}{bcd}}\!=\!
\frac{(qa;q)_\infty}{(\frac{qa}{d};q)_\infty}\qhyp22{\frac{qa}{bc},d}{\frac{qa}{b},\frac{qa}{c}}{q,\frac{qa}{d}}\!=\!
\frac{(qa,\frac{qa}{bd};q)_\infty}{(\frac{qa}{b},\frac{qa}{d};q)_\infty}\qhyp21{b,d}{\frac{qa}{c}}{q,\frac{qa}{bd}},\\
\label{W543tran}
&&\hspace{-0.6cm}\qpWhyp{5}{4}{3}{a}{b,c}{q,\frac{q^2a^2}{bc}}\!=\!
(qa;q)_\infty\qhyp12{\frac{qa}{bc}}{\frac{qa}{b},\frac{qa}{c}}{q,qa}
\!=\!
\frac{(qa;q)_\infty}{(\frac{qa}{b};q)_\infty}\qhyp11{b}{\frac{qa}{c}}{q,\frac{qa}{b}}\nonumber\\
&&\hspace{3.15cm}\!=\!\frac{(qa,\frac{qa}{bc};q)_\infty}{(\frac{qa}{b},\frac{qa}{c};q)_\infty}\qphyp201{b,c}{-}{q,\frac{qa}{bc}}\!=\!\frac{(qa,b;q)_\infty}{(\frac{qa}{c};q)_\infty}
\qphyp11{-1}{\frac{qa}{bc}}{\frac{qa}{b}}{q,b},
\\
\label{W434tran}
&&\hspace{-0.6cm}\qpWhyp{4}{3}{4}{a}{b}{q,\frac{q^2a^2}{b}}\!=\!
(qa;q)_\infty\qhyp01{-}{\frac{qa}{b}}{q,qa}
\!=\!
\frac{(qa;q)_\infty}{(\frac{qa}{b};q)_\infty}\qphyp101{b}{-}{q,\frac{qa}{b}}\nonumber\\
&&\hspace{2.85cm}\!=\!(qa,b;q)_\infty
\qphyp01{-2}{-}{\frac{qa}{b}}{q,b},
\\
\label{W325tran}
&&\hspace{-0.6cm}\qpWhyp{3}{2}{5}{a}{-}{q,q^2a^2}\!=\!
(qa;q)_\infty\qphyp001{-}{-}{q,qa}
\!=\!(qa,\pm iq\sqrt{a};q)_\infty\!\qphyp03{-2}{-}{-q,\pm iq\sqrt{a}}{q,qa},
\end{eqnarray}
and we assume that there are no vanishing denominator factors (see Remark \ref{vd}). Furthermore, even though the left-hand sides of \eqref{W761tran}--\eqref{W325tran} are entire functions of their arguments, one assumes that the right-hand sides will have the modulus of their arguments less than unity in case they are not entire functions of their arguments.
\end{thm}

\begin{proof}
For \eqref{W761tran}, start with Bailey's transformation of a nonterminating very-well-poised ${}_8W_7$ to a sum of two nonterminating balanced ${}_4\phi_3$'s
\eqref{Baileytran}
and take the limit $b\to\infty$ (or $c\to\infty$), replace $f\mapsto b$ (or $f\mapsto c$),
and then apply the three-term nonterminating transformation for a ${}_3\phi_2$
\eqref{3phi2nbtran}.
For \eqref{W652tran}, start with \eqref{W761tran},
take the limit as $e\to\infty$ (or $d\to\infty$) 
produces the ${}_2\phi_2$ representation and taking 
the limit as $c\to\infty$ (or $b\to\infty$) and then 
replacing variables accordingly so that $b,c,d$ remains, produces
the result.
For \eqref{W543tran}, starting with \eqref{W652tran}:~(i) 
taking the limit in the ${}_2\phi_2$ representation as 
$d\to\infty$ produces the ${}_1\phi_2$ representation; 
(ii) taking the limit as $b\to\infty$ (or $c\to\infty$) 
in the ${}_2\phi_2$ representation or taking the limit 
as $b\to\infty$ (or $d\to\infty$) in the ${}_2\phi_1$ 
representation produces the ${}_1\phi_1$ representation; 
(iii) taking the limit as $c\to\infty$ in the ${}_2\phi_1$ 
representation produces the ${}_2\phi_0^1$ representation.
For \eqref{W434tran}, starting with \eqref{W543tran}:~(i) 
taking the limit as $c\to\infty$ (or $b\to\infty$) in 
the ${}_1\phi_2$ representation, or taking the limit 
$b\to\infty$ in the ${}_1\phi_1$ representation 
produces the ${}_0\phi_1$ representation;
(ii) taking the limit $c\to\infty$ in the ${}_1\phi_1$ 
representation produce, or taking the limit as $
c\to\infty$ (or $b\to\infty$) in the ${}_2\phi_0^1$ 
representation 
produces 
the ${}_1\phi_0^1$ representation.
For \eqref{W325tran} starting with \eqref{W434tran}, 
taking the limit as $b\to\infty$ in either the 
${}_0\phi_1$ or ${}_1\phi_0^1$ representations produces the result. 
This
completes the proof.
\end{proof}

\subsubsection{Important terminating transformations}

\noindent An important terminating transformation which we will
use is \cite[\href{http://dlmf.nist.gov/17.9.E8}{(17.9.8)}]{NIST:DLMF}
\begin{equation}
\qhyp32{q^{-n},a,b}{c,d}{q,q}=\left(\frac{ab}{c}\right)^n
\frac{(\frac{cd}{ab};q)_n}{(d;q)_n}
\qhyp32{q^{-n},\frac{c}{a},\frac{c}{b}}{c,\frac{cd}{ab}}{q,q}.
\label{3phi2term}
\end{equation}
Another terminating transformation which is very useful is if you start with Bailey's transformation for a nonterminating ${}_8W_7$ to a sum of two balanced ${}_4\phi_3$'s
\eqref{Baileytran} and set either $d,e,f$ to some $q^{-n}$, $n\in\N_0$. Then one obtains the following useful terminating transformation \cite[\href{http://dlmf.nist.gov/17.9.E15}{(17.9.15)}]{NIST:DLMF}
\begin{equation}
\Whyp87{a}{q^{-n},b,c,d,e}{q,\frac{q^{n+2}a^2}{bcde}}=
\frac{(\frac{qa}{d},\frac{qa}{e};q)_n}{(qa,\frac{qa}{de};q)_n}
\qhyp43{q^{-n},\frac{qa}{bc},d,e}{\frac{qa}{b},\frac{qa}{c},q^{-n}\frac{de}{a}}{q,q}.
\end{equation}

\noindent In \cite[Exercise 1.4ii]{GaspRah}, one 
finds the inversion formula for
terminating basic hypergeometric series.

\begin{thm}[Gasper \& Rahman's (2004) Inversion Theorem]
\label{thm:1.2}
Let $m, n, k, r, s\in\mathbb N_0$, $0\le k\le r$, 
$0\le m\le s$, 
$a_k, b_m\not\in
\Omega^n_q\cup\{0\}$,
$q\in\mathbb C^\ast$ such that $|q|\ne 1$.
Then,
\begin{eqnarray}
\qhyp{r+1}{s}{q^{-n},a_1, \ldots, a_r}
{b_1, \ldots, b_s}{q,z}&=&\frac{(a_1, \ldots, a_r;q)_n}
{(b_1, \ldots, b_s;q)_n}\left(\frac{z}
{q}\right)^n\left((-1)^nq^{\binom{n}{2}}
\right)^{s-r-1}\nonumber\\
&&
\label{inversion}\times\sum_{k=0}^n
\frac{\left(q^{-n},\frac{q^{1-n}}{b_1}, \ldots, \frac{q^{1-n}}{
b_s};q\right)_k}
{\left(q,\frac{q^{1-n}}{a_1}, \ldots, \frac{q^{1-n}}
{a_r};q\right)_k}\left(\frac{b_1\cdots b_s}{a_1\cdots a_r}
\frac{q^{n+1}}{z}\right)^k.
\end{eqnarray}
\end{thm}
\noindent From the above inversion formula \eqref{inversion},
one may derive the following useful terminating 
basic hypergeometric transformation lemma.
\begin{lem}
\label{lem:1.3}
Let {$p\in\mathbb Z$}, $n,r,s\in\mathbb N_0$, 
$a_k, b_m\not\in\Omega^n_q\cup\{0\}$,
$z, q\in\mathbb C^\ast$ such that $|q|\ne 1$.
Then
\begin{eqnarray}
\qphyp{r+1}{{s}}{p}{q^{-n},a_1, \ldots, a_r}
{b_1, \ldots, b_{{s}}}{q,z}&=&\frac{(a_1, \ldots, a_r; q)_n}
{(b_1, \ldots, b_{{s}};q)_n}\left(\frac{z}{q}\right)^n
\left((-1)^nq^{\binom{n}{2}}\right)^{s-r{+p}-1}\nonumber\\ 
&&\times\qphyp{{s+1}}{r}{{s-r+p}}
{q^{-n},\frac{q^{1-n}}{b_1}, \ldots, \frac{q^{1-n}}{b_{{s}}}
}{\frac{q^{1-n}}{a_1}, \ldots, \frac{q^{1-n}}{a_r}}{q,\frac{b_1\cdots b_{{s}}}
{a_1\cdots a_r}\frac{q^{(1-p)n+p+1}}{z}}.
\label{ivg3}
\end{eqnarray}
\end{lem}
\begin{proof}
In a straightforward calculation, if we write 
 \eqref{inversion} and we apply \eqref{poch.id:3} 
assuming all the parameters are nonzero, and then 
we apply identities \eqref{limit1} and \eqref{limit2} 
one obtains \eqref{ivg3}.
This completes the proof.
\end{proof}


\begin{cor}\label{cor:1.5}
Let $n,r\in\mathbb N_0$, $q\in\mathbb C^\ast$ such 
that $|q|\ne 1$, and for $0\le k\le r$, let $a_k, 
b_k\not\in\Omega^n_q\cup\{0\}$. Then,
\begin{eqnarray}
&&\hspace{-0.8cm}\qhyp{r+1}{r}{q^{-n},a_1,\ldots,a_r}
{b_1,\ldots,b_r}{q,z}\nonumber\\
&&\hspace{0.2cm}=
\label{cor:1.5:r1}
q^{-\binom{n}{2}}
(-1)^n
\frac{(a_1,\ldots,a_r;q)_n}
{(b_1,\ldots,b_r;q)_n}
\left(\frac{z}{q}\right)^n\!\!\!
\qhyp{r+1}{r}{q^{-n},
\frac{q^{1-n}}{b_1},\ldots,
\frac{q^{1-n}}{b_r}}
{\frac{q^{1-n}}{a_1},\ldots,
\frac{q^{1-n}}{a_r}}{q,
\frac{q^{n+1}}{z}\frac{b_1\cdots b_r}
{a_1\cdots a_r}}.
\end{eqnarray}
\end{cor}
\begin{proof}
Take $r=s$, $p=0$ in \eqref{ivg3}, which completes
the proof.
\end{proof}
\noindent Note that in Corollary \ref{cor:1.5}
if the terminating basic hypergeometric
series on the left-hand side is balanced
then the argument of the terminating basic 
hypergeometric series on the right-hand side 
is $q^2/z$.

Another equality we can use is the following 
connecting relation between terminating
basic hypergeometric series with base $q$, and 
with base $q^{-1}$:
\begin{eqnarray} 
&&\hspace{-0.3cm}\qhyp{r+1}{r}{q^{-n},a_1, \ldots, a_r}
{b_1, \ldots, b_r}{q,z}=
\qhyp{r+1}{r}{q^{n}, a^{-1}_1, \ldots, a^{-1}_r}
{b^{-1}_1, ..., b^{-1}_r}{q^{-1}, 
\dfrac{a_1 a_2\cdots a_r}
{b_1 b_2\cdots b_r}\dfrac{z}{q^{n+1}}}\nonumber\\
&&\hspace{0.5cm}=
q^{-\binom{n}{2}}
\left(-\frac zq\right)^n
\frac{(a_1,\ldots,a_r;q)_n}
{(b_1,\ldots,b_r;q)_n}
\qhyp{r+1}{r}{q^{-n},
\frac{q^{1-n}}{b_1}, \ldots, 
\frac{q^{1-n}}{b_r}}
{\frac{q^{1-n}}{a_1}, \ldots, 
\frac{q^{1-n}}{a_r}}
{q,\frac{b_1\cdots b_r}{a_1\cdots a_r}\frac{q^{n+1}}{z}}.
\label{qtopiden} 
\end{eqnarray}
In order to understand the procedure for obtaining 
the $q^{-1}$ analogs of the basic 
hypergeo\-metric orthogonal polynomials studied 
in this manuscript, let us consider a special 
case in detail. 
Let $n\in\mathbb N_0$,
\begin{equation}
\label{freqs}
f_{n,r}(q):=f_{n,r}(q;z(q);{\bf a}(q),{\bf b}(q)):=g_r(q)
\qhyp{r+1}{r}{q^{-n},{\bf a}(q)}{{\bf b}(q)}{q,z(q)},
\end{equation}
where 
\[
\hspace{-0.2cm}\left.
\begin{array}{c}
{\bf a}(q):=\left\{{a_1(q)},\ldots,{a_r(q)}\right\}\\[0.3cm]
{\bf b}(q):=\left\{{b_1(q)},\ldots,{b_r(q)}\right\}
\end{array}
\right\},
\]
\noindent which will suffice, for instance, for 
the study of the 
terminating basic hypergeometric 
representations for the 
Askey--Wilson polynomials. 

In order to obtain the corresponding
$q^{-1}$ hypergeometric representations of $f_{n,r}(q)$, one 
only needs to consider the corresponding $q^{-1}$-function:
\begin{equation}
f_{n,r}(q^{-1})=g_r(q^{-1})
\qhyp{r+1}{r}{q^n,{\bf a}(q^{-1})}{{\bf b}(q^{-1})}{q^{-1},z(q^{-1})}.
\label{invertedrep}
\end{equation}
\begin{prop}\label{thm:2.6}
Let $r,k\in\mathbb N_0$, $0\le k\le r$, $a_k(q)\in\mathbb 
C$, $b_k(q)\in\Omega_q$, $q\in\mathbb C^\ast$ such that 
$|q|\ne 1$, $z(q)\in\mathbb C$.
Define ${\bf a}(q):=(a_1(q),\ldots,a_r(q))$, 
${\bf b}(q):=(b_1(q),\ldots,b_r(q))$ and a multiplier 
function $g_r(q):=g_r(q;z(q);{\bf a}(q);{\bf b}(q))$ 
which is not of basic hypergeometric type (some 
multiplicative combination of powers and $q$-Pochhammer symbols), 
and $z(q):=z(q;{\bf a}(q);{\bf b}(q))$. Then defining 
$f_{n,r}(q)$ as in \eqref{freqs}, one has
\begin{equation}
\label{termreprr}
f_{n,r}(q^{-1})
=g_r(q^{-1})\qhyp{r+1}{r}{q^{-n},{\bf a}^{-1}(q^{-1})}
{{\bf b}^{-1}(q^{-1})}{q,\frac{q^{n+1}a_1(q^{-1})
\cdots a_r(q^{-1}) z(q^{-1})}{b_1(q^{-1})\cdots b_r(q^{-1})}}.
\end{equation}
\end{prop}
\begin{proof}
By using \eqref{poch.id:3} repeatedly with 
the definition \eqref{2.11} in \eqref{invertedrep}, 
one obtains the $q^{-1}$ terminating 
representation \eqref{termreprr}, which 
corresponds to the original terminating basic 
hypergeo\-me\-tric representation \eqref{freqs}. 
This completes the proof.
\end{proof}

\medskip
Now consider the more general case. 
Let $r,s\in\mathbb N_0$, $0\le t\le r$, $0\le u\le 
s$, and let
\begin{equation}
\hspace{-0.2cm}\left.
\begin{array}{c}
{\bf a}(q):=\{a_1(q),\ldots,a_{r-t}(q),\overbrace{0,\ldots,0}^t\}
\\[0.3cm]
{\bf b}(q):=\{b_1(q),\ldots,b_{s-u}(q),\underbrace{0,\ldots,0}_u\}
\end{array}
\right\},
\end{equation}
where either $t>0$, $u=0$, or $u>0$, $t=0$, or $t=u=0$. Moreover, 
 as above, 
 a multiplier function 
$g_{r,s,t,u}(q):=g_{r,s,t,u}(q;z(q);{\bf a}(q);{\bf b}(q))$ 
and $z(q):=z(q;{\bf a}(q);{\bf b}(q))$.
Define 
\begin{equation}
\label{frstueq}
f_{r,s,t,u}(q):=
g_{r,s,t,u}(q)
\qhyp{r+1}{s}{q^{-n},{\bf a}(q)}
{{\bf b}(q)}{q,z(q)}.
\end{equation}
In order to obtain the $q^{-1}$-representation of 
$f_{r,s,t,u}$, one must again compute
\begin{equation}
f_{r,s,t,u}(q^{-1})=g_{r,s,t,u}(q^{-1})
\qhyp{r+1}{s}{q^n,{\bf a}(q^{-1})}{{\bf b}(q^{-1})}
{q^{-1},z(q^{-1})}.
\label{invertedgenrep}
\end{equation}
This can be obtained by repeated use of 
\eqref{poch.id:3} using the definition \eqref{2.11} 
and various combinations of \eqref{limit1}--
\eqref{limit3}.

\section{Basic hypergeometric orthogonal po\-ly\-no\-mials and functions}\label{sec:3}
We will study a subset of basic hypergeometric orthogonal 
polynomials in the $q$-Askey scheme, which we refer to as 
 basic hypergeometric orthogonal polynomials. 
Basic hypergeometric orthogonal polynomials satisfy an 
orthogonality relation which will be given either as 
an integral or a sum.

\medskip
\noindent 
Let $n,n'\in\N_0$, $x\in\CC$, ${\bf a}\in\CC$, be a set of parameters. 
Consider a sequence of orthogonal polynomials $\{p_n\}$, orthogonal with respect to either a continuous measure $\dd\mu(x)$ or a discrete measure $\dd\mu_n$, with support (either bounded or unbounded) given respectively by $A$ and $B$. Further consider situations 
where the measure can be written respectively in terms of a continuous weight function 
$\dd\mu(x)={\sf w}(x;{\bf a})\,\dd x$ or a discrete
weight function 
$\dd\mu_n={\sf w}_n({\bf a})$.
A sequence of orthogonal polynomial are orthogonal if and only if it satisfies a three-term recurrence relation with coefficients $A_n, B_n, C_n\in\RR$ given by \cite[\S2]{AskeyIsmail84}
\begin{equation}
p_{n+1}(x;{\bf a})=(A_nx+B_n)\,p_n(x;{\bf a})-C_n\,p_{n-1}(x;{\bf a}),
\end{equation}
with initial conditions $p_{-1}(x;{\bf a})=0$, $p_0(x;{\bf a})=1$.
Define ${\sf W}({\bf a})$, often referred to as the {\it total mass}, as the integral or sum of the weight function over its support, given respectively by
\begin{eqnarray}
&&\hspace{-11.8cm}{\sf W}({\bf a})=\int_A {\sf w}(x;{\bf a})\,\dd x,\\
&&\hspace{-11.8cm}{\sf W}({\bf a})=\sum_{n\in B} {\sf w}_{n}({\bf a}).
\end{eqnarray}
In this case, one has the following continuous or discrete orthogonality relations
respectively, 
\begin{eqnarray}
&&\hspace{-8.5cm}\int_A p_{n}(x;{\bf a}) p_{n'}(x;{\bf a}) \,{\sf w}(x;{\bf a})\,\dd x=h_n({\bf a}) \delta_{n,n'},\\
&&\hspace{-8.5cm}\sum_{n\in B} p_{n}({\bf a}) p_{n'}({\bf a}) \,{\sf w}_n({\bf a})=h_n({\bf a}) \delta_{n,n'},
\end{eqnarray}
where the $L^2$ (or $\ell^2$)-norm of orthogonality is given by \cite[(2.5)]{AskeyIsmail84}
\begin{equation}
h_n({\bf a})={\sf W}({\bf a})\frac{A_0}{A_n}\prod_{k=1}^n C_k.
\end{equation}

\medskip
\begin{rem}
\label{rem:2.7}
In the remainder of the paper, we will examine orthogonal 
polynomials in $x=\frac12(z+z^{-1})\in \CCast$.
Note that in this case, $x=x(z)$ is invariant under 
the map $z\mapsto z^{-1}$, so all functions (including 
po\-ly\-no\-mials) in $x$ will also satisfy this invariance.
\end{rem}

The Askey--Wilson polynomials are at the top of 
the {\it symmetric} 
family of basic hypergeometric orthogonal polynomials. 
The continuous dual $q$-Hahn $p_n(x;a,b,c|q)$, 
Al-Salam--Chihara $Q_n(x;a,b|q)$, 
continuous big $q$-Hermite $H_n(x;a|q)$ and 
continuous $q$-Hermite $H_n(x|q)$ polynomials 
are the $d\to c\to b\to a\to 0$ limit cases 
of the Askey--Wilson polynomials, namely
\begin{eqnarray}
&&\hspace{-7.8cm}
p_n(x;a,b,c|q)=\lim_{d\to 0}p_n(x;a,b,c,d|q),\label{cdqHl}\\
&&\hspace{-7.8cm}
Q_n(x;a,b|q)=\lim_{c\to 0}p_n(x;a,b,c|q),\label{ASCl}\\
&&\hspace{-7.8cm}
H_n(x;a|q)=\lim_{b\to 0}Q_n(x;a,b|q),\label{cbqHl}\\
&&\hspace{-7.8cm}
H_n(x|q)=\lim_{a\to 0}H_n(x;a|q).\label{cqHl}
\end{eqnarray}
The continuous dual $q$-Hahn and Al-Salam--Chihara polynomials 
are symmetric in the variables $a,b,c$, and $a,b$ respectively.
By starting with representations of the Askey--Wilson 
polynomials \eqref{aw:def3}, we can obtain terminating basic 
hypergeometric series representations of the
symmetric family.

Furthermore, the $q^{-1}$-symmetric family is also a 
set of symmetric polynomials in their parameters $a, b, c$. 
These polynomials can also be obtained as 
$c\to b\to a\to 0$ limit cases
\begin{eqnarray}
&&\hspace{-7.7cm}
Q_n(x;a,b|q^{-1})=\lim_{c\to 0}p_n(x;a,b,c|q^{-1}),\\
&&\hspace{-7.7cm}
H_n(x;a|q^{-1})=\lim_{b\to 0}Q_n(x;a,b|q^{-1}),
\\
&&\hspace{-7.7cm}
H_n(x|q^{-1})=\lim_{a\to 0}H_n(x;a|q^{-1}).
\end{eqnarray}
These $q^{-1}$-polynomials have been studied previously 
in \cite{BergIsmail1996,KoelinkStokman2003}.

\subsection{The $q$ and $q^{-1}$-symmetric and dual families}
\label{sec:2.2.1}

\noindent In the $q$-Askey-scheme of basic hypergeometric orthogonal polynomials, at the very top is the Askey--Wilson polynomials. These polynomials are symmetric in four parameters which we refer as $a,b,c,d$. There are exists a sequence of subfamilies of the Askey--Wilson polynomials which are symmetric in three and two parameters, these are the continuous dual $q$-Hahn and Al-Salam--Chihara polynomials. These families are obtained by setting one of the parameters equal to zero in the Askey--Wilson polynomials. By continuing setting parameters equal to zero, one obtains the continuous big $q$-Hermite polynomial and the continuous $q$-Hermite polynomials. By replacing $q\mapsto q^{-1}$ we obtain new subfamilies of the Askey--Wilson polynomials which are also symmetric in their parameters, namely the continuous dual $q^{-1}$-Hahn, $q^{-1}$-Al-Salam--Chihara, and as well the continuous big $q^{-1}$-Hermite and continuous $q^{-1}$-Hermite polynomials.
We refer to these families as the $q$ and $q^{-1}$-symmetric subfamilies of the Askey--Wilson polynomials. If one considers duality then one finds that there exists duality
relations between the $q$ and $q^{-1}$-symmetric subfamilies of the Askey--Wilson polynomials and the other well-studied families of basic hypergeometric orthogonal polynomials in the $q$-Askey scheme. These are the big $q$-Jacobi polynomials and functions, little $q$-Jacobi polynomials and functions and the $q$-Bessel polynomials and the $q^{-1}$-Bessel functions. In the remainder of this section, we will study some important properties of these polynomials and functions.
\subsubsection{The Askey--Wilson polynomials}
\label{sec:2.2.1b}
The Askey--Wilson polynomials have the following
terminating ${}_4\phi_3$ basic hypergeometric series
representations
 \cite[(15)]{CohlCostasSantos20b}.

\begin{thm}
Let $n\in\mathbb N_0$, $q\in\CCddag$,
$x=\frac12(z+z^{-1})\in \CCast$, $z\in\CCast$, 
$a,b,c,d\in\CCast$.
Then, the Askey--Wilson polynomials have the following terminating basic hypergeometric series representations
\begin{eqnarray}
\hspace{0.30cm}
\label{aw:def1} 
&&\hspace{-1.2cm}p_n(x;a,b,c,d|q) := a^{-n} (ab,ac,ad;q)_n 
\qhyp43{q^{-n},q^{n-1}abcd, az^{\pm}}{ab,ac,ad}{q,q}\\
\label{aw:def2} &&\hspace{1.6cm}=q^{-\binom{n}{2}} (-a)^{-n} 
\frac{(\frac{abcd}{q};q)_{2n}
(a z^{\pm};q)_n}
{(\frac{abcd}{q};q)_n}
\qhyp43{q^{-n},
\frac{q^{1-n}}{ab},
\frac{q^{1-n}}{ac},
\frac{q^{1-n}}{ad}
}
{\frac{q^{2-2n}}{abcd},\frac{q^{1-n}}{a}z^{\pm}}{q,q}\\
\label{aw:def3} &&\hspace{1.6cm}=z^n(ab,cz^{-1},dz^{-1};q)_n
\qhyp43{q^{-n},az,bz,\frac{q^{1-n}}{cd}}
{ab,\frac{q^{1-n}}{c}z,\frac{q^{1-n}}{d}z}{q,q}\\
&&\hspace{1.6cm}=z^n
\dfrac{\left(\frac{abcd}{q};q\right)_{2n}\left(\frac{b}{z},\frac{c}{z},\frac{d}{z}
,\frac{bcd}{qz};q\right)_{n}}
{\left(\frac{abcd}{q};q\right)_{n}
\left(\frac{bcd}{qz};q\right)_{2n}}
\label{aw:def6}
{}_8W_7\left(
\frac{q^{1-2n}z}{bcd};q^{-n},
\frac{q^{1-n}}{bc},
\frac{q^{1-n}}{bd},
\frac{q^{1-n}}{cd},
az;q,
\frac{qz}{a}\right)\\
&&\hspace{1.6cm}\label{aw:def7}
=z^n
\dfrac{\left(\frac{a}{z},bc,bd,cd;q\right)_n}
{\left(bcdz;q\right)_{n}}
{}_8W_7\left(
\frac{bcdz}{q};q^{-n},bz,cz,dz,q^{n-1}abcd;q,
\frac{q}{az}\right)\\
&&\hspace{1.6cm}\label{aw:def5}=a^{-n}\dfrac{\left(ac,ad,bz^{\pm};q\right)_n}{\left(\frac{b}{a};q\right)_n} 
\,{}_8W_7\left(
\frac{q^{-n}a}{b};q^{-n},\!
\frac{q^{1-n}}{bc},\frac{q^{1-n}}{bd}, az^{\pm};q,q^n cd\right)\\
&&\hspace{1.6cm}=z^n \frac{(\frac{a}{z},\frac{b}{z},\frac{c}{z},\frac{d}{z};q)_n}{\left(z^{-2};q\right)_n} 
\,{}_8W_7\left(
q^{-n}z^2;q^{-n},az,bz,cz,dz
;
q,\frac{q^{2-n}}{abcd}
\right).
\label{aw:def4}
\end{eqnarray}
\label{AWthm}
\end{thm}
\begin{proof}
See proof of \cite[Theorem 7]{CohlCostasSantos20b}.
\end{proof}

\noindent The $q^{-1}$-Askey--Wilson polynomials
$p_n(x;a,b,c,d|q^{-1})$ which are simply renormalized Askey--Wilson po\-ly\-no\-mials with
parameters given by their reciprocals,
are
given by
\begin{eqnarray}
&&\hspace{-5cm}p_n(x;a,b,c,d|q^{-1})
=q^{-3\binom{n}{2}}(-abcd)^np_n(x;\tfrac{1}{a},\tfrac{1}{b},
\tfrac{1}{c},\tfrac{1}{d}|q).
\end{eqnarray}
This easily follows from Theorem \ref{AWthm}, 
Proposition \ref{thm:2.6}, and Remark \ref{rem:2.7}.

\medskip
\noindent The Askey--Wilson polynomials have the following connection relation due to Ismail
\cite[Theorem 16.4.2]{Ismail}
\begin{thm}Let $n\in\N_0$, $q\in\CCddag$, $a,b,c,d,e,f,g,h\in\CCast$. Then
\begin{eqnarray}
&&\hspace{-0.8cm}p_n(x;a,b,c,d|q)=(q,ab,bd,cd;q)_n\sum_{k=0}^n
\frac{p_k(x;e,f,g,h|q)(q^kd)^{k-n}(q^{n-1}abcd;q)_k}
{(q;q)_{n-k}(q,ad,bd,cd,q^{k-1}efgh;q)_k}\nonumber\\
&&\hspace{-0.8cm}\times
\sum_{l=0}^{n-k}\left(\frac{qd}{h}\right)^l\frac{(q^{k-n},q^{n+k-1}abcd,q^keh,q^kfh,q^kgh;q)_l}{(q,q^{2k}efgh,q^kad,q^kbd,q^kcd;q)_l}
\!\qhyp43{q^{-(n-k-l)},q^{n+k+l-1}abcd,q^{k+l}dh,\frac{d}{h}}{q^{k+l}ad,q^{k+l}bd,q^{k+l}cd}{q,q}\!.
\label{conAW}
\end{eqnarray}
\label{conAWthm}
\end{thm}
\begin{proof}
See proof of \cite[Theorem 16.4.2]{Ismail}.
\end{proof}
This connection coefficient simplifies to a ${}_5\phi_4$ in the case $d=h$, and this result was first given in the Askey--Wilson memoir \cite[\S6]{AskeyWilson}.
Unfortunately, in terms of computing connection relations for the $q$ and $q^{-1}$-symmetric subfamilies of the Askey--Wilson polynomials, taking the limit as $d=h\to0$ or $d=h\to\infty$ in order to produce connection relations for the continuous dual $q$-Hahn polynomials and the continuous dual $q^{-1}$-Hahn polynomials does not 
initially produce anything helpful. If one would like to compute useful connection relations for these polynomials, then one needs to start with 
\eqref{conAW} and take the limits with respect to the other parameters.

\medskip
\noindent 
The limit transition between the 
Askey--Wilson polynomials and the Wilson polynomials is
\cite[14.3.17]{Koekoeketal}:
\begin{equation}\label{Limaw:def1}\begin{split}
W_n(x^2;a,b,c)=&\lim_{q\to 1} 
\dfrac{p_n(
\frac12(q^{ix}+q^{-ix});q^a,q^b,q^c,q^d|q)}
{(1-q)^{3n}}.
\end{split}\end{equation}


\subsubsection{The continuous dual {\it q} and $q^{-1}$-Hahn 
polynomials}\label{sec:3.3}
The continuous dual $q$-Hahn polynomials are symmetric 
in three parameters $a,b,c$.
One has the following basic hypergeometric representations 
of the continuous dual $q$-Hahn polynomials
\begin{cor}\label{cor:4.1} 
Let $n\in\mathbb N_0$, 
$x=\frac12(z+z^{-1})\in \CCast$, $z\in\CCast$, 
$q\in\CCddag$, $a,b,c\in\CCast$.
Then, the continuous dual $q$-Hahn polynomials have the following terminating basic hypergeometric series representations
\begin{eqnarray}
\label{cdqH:def1} &&\hspace{-1.45cm}p_n(x;a,b,c|q) 
:= a^{-n} (ab,ac;q)_n 
\qhyp{3}{2}{q^{-n}, az^{\pm}}{ab,ac}{q,q}\\
\label{cdqH:def2} &&\hspace{1.05cm}=q^{-\binom{n}{2}}(-a)^{-n} 
(az^{\pm};q)_n\qhyp{3}{2}{q^{-n}, 
\frac{q^{1-n}}{ab},
\frac{q^{1-n}}{ac}
}{\frac{q^{1-n}}{a}z^{\pm}}
{q,q^n bc}\\\label{cdqH:def3} &&\hspace{1.05cm}=z^n 
(ab,cz^{-1};q)_n\qhyp{3}{2}{q^{-n}, az, bz}
{ab,\frac{q^{1-n}}{c}z}{q,\frac{q}{cz}}\\
\label{cdqH:def4} &&\hspace{1.05cm}=z^n (az^{-1},bz^{-1};q)_n 
\qhyp{3}{2}{q^{-n},cz,\frac{q^{1-n}}{ab}}
{\frac{q^{1-n}}{a}z,\frac{q^{1-n}}{b}z}{q,q}\\
\label{cdqH:def6}&&\hspace{1.05cm}=a^{-n}\dfrac{\left(ac,bz^\pm;q\right)_n}
{\left(\frac{b}{a};q\right)_n} 
\qpWhyp{7}{6}{1}{\frac{q^{-n}a}{b}}
{q^{-n},\frac{q^{1-n}}{bc}, az^\pm}
{q,\frac{q c}{b}}\\
&&
\hspace{1.05cm}=z^n
\dfrac{(\frac{a}{z},\frac{b}{z},\frac{c}{z},\frac{abc}{qz};q)_{n}}
{\left(\frac{abc}{qz};q\right)_{2n}}
\label{cdqH:def7}
\qpWhyp{7}{6}{-1}{\frac{q^{1-2n}z}{abc}}
{q^{-n},q^{1-n}ab,q^{1-n}ac,q^{1-n}bc}
{q,q^{2n-1}abcz}
\\
&&\label{cdqH:def8}
\hspace{1.05cm}=z^n\dfrac{(ab,ac,bc;q)_n}
{\left(abcz;q\right)_{n}}
\qpWhyp{7}{6}{-1}{\frac{abcz}{q}}
{q^{-n},az,bz,cz}
{q,\frac{q^{n}}{z^2}}
\\
\label{cdqH:def5} &&\hspace{1.05cm}=z^n \frac{\left(\frac{a}{z},\frac{b}{z},\frac{c}{z};q\right)_n}{(z^{-2};q)_n} 
\qpWhyp{7}{6}{-1}{q^{-n}z^2}{q^{-n},az,bz,cz}
{q,\frac{q}{abcz}}
.
\end{eqnarray}
\end{cor}
\begin{proof} 
The representation \eqref{cdqH:def1} is derived by starting 
with \eqref{aw:def1} and replacing $b$, $c$, or $d\to 0$
(see also \cite[(14.3.1)]{Koekoeketal});
\eqref{cdqH:def2} is derived using \eqref{aw:def2} 
and taking, for instance, $d\to 0$;
\eqref{cdqH:def3} is derived by using \eqref{aw:def3} 
and taking $d\to 0$; 
\eqref{cdqH:def4} is derived by using \eqref{aw:def3} and 
taking $b\to 0$ and replacing $d\to b$. 
\eqref{cdqH:def6} is derived by taking \eqref{aw:def5} and taking the limit $d \to 0$, 
\eqref{cdqH:def7}, \eqref{cdqH:def8} are derived by taking \eqref{aw:def6}, \eqref{aw:def7} 
respectively and taking the limit $a \to 0$,
and 
\eqref{cdqH:def5} is derived by taking \eqref{aw:def4} and taking the limit $d \to 0$.
The limit formulas \eqref{limit1}, \eqref{limit2}, \eqref{limit3} are 
used whenever applicable. Also, whenever necessary the parameters should be renamed such that they are in the multiset $\{a,b,c\}$.
This completes the proof.
\end{proof}

\noindent 
The symmetric sub-families of the Askey--Wilson polynomials,
the continuous dual $q$-Hahn, Al-Salam--Chihara, continuous big $q$-Hermite, and continuous $q$-Hermite polynomials, can be computed using the limit relations \eqref{cdqHl}--\eqref{cqHl}.
Since there does not exist a connection relation for the continuous $q$-Hermite polynomials (except for \eqref{concqHqiH}, \eqref{concqiHqH}), we do not treat these further here.
For the symmetric sub-families of the Askey--Wilson polynomials, one has the following connection relations. First for the continuous dual $q$-Hahn polynomials.

\begin{thm} 
Let $n\in\N_0$, $q\in\CCddag$, $a,b,c,d,e,f\in\CCast$. Then
\begin{eqnarray}
&&\hspace{-1.1cm}p_n(x;a,b,c|q)=(q,ac,bc;q)_n\sum_{k=0}^n
p_k(x;d,e,f|q)\frac{(q^kc)^{k-n}}{(q;q)_{n-k}(q,ac,bc;q)_k}\nonumber\\
&&\hspace{2cm}\times
\sum_{l=0}^{n-k}
\left(\frac{qc}{f}\right)^l
\frac{(q^{-n-k},q^kdf,q^kef;q)_l}{(q,q^kac,q^kbc;q)_l}
\qhyp32{q^{-(n-k-l)},q^{k+l}cf,\frac{c}{f}}
{q^{k+l}ac,q^{k+l}bc}{q,q}.
\end{eqnarray}
\end{thm}

\begin{proof}
Start with Theorem \ref{conAWthm} and take the limit as $c,g\to 0$ and then replace $d\mapsto c$, $h\mapsto g$. This completes the proof.
\end{proof}

\medskip
The continuous dual $q^{-1}$-Hahn polynomials can 
be obtained from the Askey--Wilson polynomials as follows
\begin{equation}
p_n(x;a,b,c|q^{-1})=q^{-3\binom{n}{2}}
\left(-abc\right)^n
\lim_{d\to0}d^n\,p_n(x;\tfrac{1}{a},\tfrac{1}{b},
\tfrac{1}{c},\tfrac{1}{d}|q).
\label{limtcdqiH}
\end{equation}
Furthermore, the other members of the $q^{-1}$-symmetric family (also a set of symmetric polynomials 
in their parameters $a,b,c$) can also be obtained as 
$c\to b\to a\to 0$ limit cases.
\begin{eqnarray}
&&\hspace{-7.7cm}
Q_n(x;a,b|q^{-1})=\lim_{c\to 0}
p_n(x;a,b,c|q^{-1}),\label{qiASCl}\\
&&\hspace{-7.7cm}
H_n(x;a|q^{-1})=\lim_{b\to 0}Q_n(x;a,b|q^{-1}),
\label{cbqiHl}\\
&&\hspace{-7.7cm}
H_n(x|q^{-1})=\lim_{a\to 0}H_n(x;a|q^{-1}).\label{cqiHl}
\end{eqnarray}

\noindent Now we give the basic hypergeometric 
representations 
of the continuous dual $q^{-1}$-Hahn polynomials.
\begin{cor}
\label{cor:3.4}
Let $p_n(x;a,b,c|q)$ and all the respective parameters 
be defined as previously. Then, the continuous dual 
$q^{-1}$-Hahn polynomials have the following terminating basic hypergeometric series representations:
\begin{eqnarray}
&&\hspace{-1.3cm}
\label{cdqiH:1} p_n(x;a,b,c|q^{-1}):=
q^{-2\binom{n}{2}}
(abc)^n \left(
\frac{1}{ab},\frac{1}{ac}
;q\right)_n\qhyp{3}{2}{q^{-n},\frac{z^{\pm}}{a}}
{
\frac{1}{ab},\frac{1}{ac}
}{q,\frac{q^n}{bc}}\\
\label{cdqiH:2} 
&&\hspace{1.3cm}\hspace{-1mm}
=q^{-\binom{n}{2}}(-a)^n\left(\frac{z^{\pm}}{a};q\right)_n 
\qhyp{3}{2}{q^{-n},
q^{1-n}ab,
q^{1-n}ac
}{q^{1-n}az^{\pm}}{q,q}\\
\label{cdqiH:3} 
&&\hspace{1.3cm}\hspace{-1mm}
=q^{-2\binom{n}{2}} (abc)^n\left(\frac{1}{ab},\frac{z}
{c};q\right)_n 
\qhyp{3}{2}{q^{-n}, 
\frac{1}{az}, 
\frac{1}{bz}
}{
\frac{q^{1-n}c}{z},
\frac{1}{ab}
}{q,q}\\
\label{cdqiH:4} 
&&\hspace{1.3cm}\hspace{-1mm}
=q^{-2\binom{n}{2}} 
\left(\frac{ab}{z}\right)^n
\left(\frac{z}{a},\frac{z}
{b};q\right)_n\qhyp{3}{2}{q^{-n},\frac{1}
{cz},q^{1-n}ab}
{\frac{q^{1-n}a}{z},\frac{q^{1-n}b}{z}}
{q,\frac{qc}{z}}\\
&&\hspace{1.2cm}=
q^{-2\binom{n}{2}} 
(abc)^n
\frac{\left(
\frac{1}{ac},\frac{z^{\pm}}{b};q\right)_n}
{(\frac{a}{b};q)_n}
\label{cdqiH:6} 
\qpWhyp{7}{6}{-1}{\frac{q^{-n}b}{a}}
{q^{-n}, q^{1-n}bc,\frac{z^{\pm}}{a}}
{q,\frac{q^n a}{c}}\\
&&\hspace{1.2cm}=z^{n}
\dfrac{\left(\frac{1}{az},\frac{1}{bz},\frac{1}{cz}
,\frac{1}{qabcz};q\right)_{n}}
{\left(\frac{1}{q abcz};q\right)_{2n}}
\qpWhyp{7}{6}{1}{q^{1-2n}abcz}
{q^{-n},q^{1-n}ab,q^{1-n}ac,q^{1-n}bc}
{q,qz^2}
\label{cdqiH:7}
\\
&&\label{cdqiH:8}\hspace{1.2cm}
=q^{-2{\binom n2}}
(abc)^n \frac{\left(
\frac{1}{ab},\frac{1}{ac},\frac{1}{bc};q\right)_n}
{\left(\frac{1}{abcz};q\right)_{n}}
\qpWhyp{7}{6}{1}{\frac{1}{qabcz}}
{q^{-n},\frac{1}{az},\frac{1}{bz},\frac{1}{cz}}
{q,\frac{q^nz}{abc}}\\
&&\hspace{1.2cm}=q^{-2\binom{n}{2}}(abc)^n
\frac{\left(\frac{z}{a},\frac{z}{b},\frac{z}{c};q\right)_n}
{(z^2;q)_n} 
\label{cdqiH:5} 
\qpWhyp{7}{6}{1}{\frac{q^{-n}}{z^2}}
{q^{-n},\frac{1}{az},\frac{1}{bz},\frac{1}{cz}}
{q,\frac{q^{2-n}abc}{z}}.
\end{eqnarray}
\end{cor}
\begin{proof}
Each inverse representation is derived from the 
corresponding representation by applying the 
map $q\mapsto 1/q$ and using \eqref{poch.id:3}.
\end{proof}

Now we consider the connection relations for the $q^{-1}$-symmetric polynomials.
For the continuous dual $q^{-1}$-Hahn polynomials, one can compute these from the Askey--Wilson polynomials using the limit formula 
\eqref{limtcdqiH}. 
However, for the connection coefficients given in \eqref{conAW}, it doesn't make sense to do the $d$-limit. Instead one should compute the limit of $a$, $b$, or $c$ (using the symmetry in $a,b,c,d$ of the Askey--Wilson polynomials), e.g.,
\begin{equation}
p_n(x;a,b,d|q^{-1})=q^{-3\binom{n}{2}}
\left(-abd\right)^n
\lim_{c\to0}c^n\,p_n(x;\tfrac{1}{a},\tfrac{1}{b},
\tfrac{1}{c},\tfrac{1}{d}|q),
\label{limtcdqiH2}
\end{equation}
and then after obtaining the result replace $d\mapsto c$. One may also proceed similarly with choice of limits for variables for the the $q^{-1}$-symmetric subfamilies.
Using this procedure, we compute connection formulas for the $q^{-1}$-symmetric subfamilies.

\medskip 
\noindent For the continuous dual $q^{-1}$-Hahn polynomials, a representation of the connection relation is given as follows.
\begin{thm}
Let $n\in\N_0$, $q\in\CCddag$, $a,b,c,d,e,f\in\CCast$. Then
\begin{eqnarray}
&&\hspace{-1cm}p_n(x;a,b,c|q^{-1})=q^{-2\binom{n}{2}}(abc)^n(q,\tfrac{1}{ac},\tfrac{1}{bc};q)_n
\sum_{k=0}^n p_k(x;d,e,f|q^{-1})\frac{q^{2\binom{k}{2}}(abc)^{-k}}{(q;q)_{n-k}(q,\frac{1}{ac},\frac{1}{bc};q)_k}\nonumber\\
&&
\hspace{1cm}\times\sum_{l=0}^{n-k}
\left(\frac{q^{{n-k}}def}{abc}\right)^l
\frac{(q^{-(n-k)},\frac{q^k}{df},\frac{q^k}{ef};q)_l}{(q,\frac{q^k}{ac},\frac{q^k}{bc};q)_l}
\qhyp32{q^{-(n-k-l)},\frac{q^{k+l}}{cf},\frac{f}{c}}{\frac{q^{k+l}}{ac},\frac{q^{k+l}}{bc}}{q,\frac{q^n}{ab}}.
\label{cdqiH}
\end{eqnarray}
\label{cdqiHthm}
\end{thm}

\begin{proof}
Start with the connection coefficient of the Askey--Wilson polynomials \eqref{conAW}, replace \begin{equation}
(a,b,c,d,e,f,g,h)\mapsto(a^{-1},b^{-1},c^{-1},d^{-1},e^{-1},f^{-1},g^{-1},h^{-1}).
\end{equation}
Then multiply both sides of the connection relation by $q^{-3\binom{n}{2}}(-abcd)^n$, followed by, in the sum over $k$, multiplying in the numerator and denominator by $q^{-3\binom{k}{2}}(-efgh)^k$, and take the limit as $c,g\to 0$ using \eqref{limtcdqiH2}. Simplification using \eqref{critlim} several times completes the proof.
\end{proof}

\noindent 
Notice that the connection coefficients for the continuous dual $q^{-1}$-Hahn polynomials are given by a double sum. This is just like for the most general connection coefficients for the Askey--Wilson polynomials \eqref{conAW}. There also exist connection coefficients for the $q^{-1}$-Al-Salam--Chihara polynomials and the
continuous big $q^{-1}$-Hermite polynomials which are a double sum.
However, for the $q^{-1}$-Al-Salam--Chihara polynomials and the continuous big $q^{-1}$-Hermite polynomials, just as the Al-Salam--Chihara and continuous $q$-Hermite polynomials, there is a simplification which occurs.
This can be seen for instance by starting with the connection coefficients for Al-Salam--Chihara polynomials, Theorem \ref{connectionASC}, and replacing $q\mapsto q^{-1}$ throughout.
This produces the following connection relation for the $q^{-1}$-Al-Salam--Chihara polynomials.

\medskip
\noindent 
The limit transition between the 
continuous dual $q$-Hahn polynomials and the continuous dual Hahn polynomials is
\cite[(14.3.18)]{Koekoeketal}:
\begin{equation}\label{Limcdqh:def1}
S_n(x^2;a,b,c)=\lim_{q\to 1} \dfrac{p_n(\frac12(q^{ix}+q^{-ix});
q^a,q^b,q^c|q)}{(1-q)^{2n}}.
\end{equation}
A limit transition from the Askey--Wilson polynomials to the Jacobi polynomials is
\cite[p.~467]{Koekoeketal} 
\[
\lim_{q\to 1^{-}}q^{n(\alpha/2+1/4)}
\dfrac{p_n(x;q^{\alpha/2+1/4},q^{\alpha/2+3/4},
-q^{\beta/2+1/4},-q^{\beta/2+3/4}|q)}
{(q,-q^{(\alpha+\beta+1)/2},-q^{(\alpha+\beta+2)/2};q)_n}
=P_n^{(\alpha,\beta)}(x).
\]

\subsubsection{The {\it q} and $q^{-1}$-Al-Salam--Chihara 
polynomials}\label{sec:3.4}

\medskip
Now we give basic hypergeometric series representations for the Al-Salam--Chihara polynomials.
\begin{cor}\label{cor:3.5}
Let $n\in\mathbb N_0$, 
$x=\frac12(z+z^{-1})\in \CCast$, $z\in\CCast$, 
$q\in\CCddag$, $a,b\in\CCast$.
Then, the Al-Salam-Chihara polynomials have the following terminating basic hypergeometric series representations:
\begin{eqnarray}
\label{ASC:def1} 
&&\hspace{-3.8cm}Q_n(x;a,b|q):=a^{-n}(ab;q)_n
\qphyp{3}{1}{1}{q^{-n}, az^{\pm}}{ab}{q,q}\\
\label{ASC:def2} &&\hspace{-1.7cm}= q^{-\binom{n}{2}} 
(-a)^{-n} (az^{\pm};q)_n\qhyp22{q^{-n},\frac{q^{1-n}}{ab}}
{\frac{q^{1-n}}{a}z^{\pm}}{q,\frac{qb}{a}}\\
\label{ASC:def5} &&\hspace{-1.7cm}= z^n (ab;q)_n\qhyp31{q^{-n}, 
az,bz}{ab}{q,\frac{q^n}{z^2}}\\
\label{ASC:def4} &&\hspace{-1.7cm}=z^{n}(a z^{-1};q)_n 
\qhyp{2}{1}{q^{-n},bz}
{\frac{q^{1-n}}{a}z}{q,\frac{q}{az}}\\
\label{ASC:def3} &&\hspace{-1.7cm}=z^{n}(
az^{-1},bz^{-1}
;q)_n
\qphyp{2}{2}{-1}{q^{-n},\frac{q^{1-n}}{ab}}
{
\frac{q^{1-n}}{a}z,
\frac{q^{1-n}}{b}z
}{q,q}\\
\label{ASC:def7} &&\hspace{-1.70cm} =a^{-n}\dfrac{\left(bz^\pm;q\right)_n}
{\left(\frac{b}{a};q\right)_n} 
\qpWhyp{6}{5}{-2}{q^{-n}z^2}
{q^{-n},az,bz}
{q,\frac{q^n}{abz^2}}\\
\label{ASC:def6} &&\hspace{-1.70cm} 
=a^{-n} 
\frac{(bz^\pm;q)_n}{(\frac{b}{a};q)_n}
\qpWhyp{6}{5}{2}{\frac{q^{-n}a}{b}}
{q^{-n}, a z^{\pm}}
{q,\frac{q^{2-n}}{a^2_{r}}}
.
\end{eqnarray}
\end{cor}
\begin{proof}
The representation \eqref{ASC:def1} is derived by taking \eqref{cdqH:def1} 
and replacing $c\mapsto 0$ (see also 
\cite[(14.8.1)]{Koekoeketal}); \eqref{ASC:def2} is derived 
by taking \eqref{cdqH:def2} 
and replacing $c\mapsto 0$; \eqref{ASC:def5} is derived by 
taking \eqref{cdqH:def3} and replacing $c\mapsto 0$;
\eqref{ASC:def4} is derived by taking \eqref{cdqH:def3} 
replacing $b\mapsto 0$ (see also \cite[(14.8.1)]{Koekoeketal})
and interchanging $c$ and $a$;
\eqref{ASC:def3} is derived by taking \eqref{cdqH:def4} 
and replacing $c\mapsto 0$, 
\eqref{ASC:def7} is derived by taking \eqref{cdqH:def6} and replacing $c \mapsto 0$
and 
\eqref{ASC:def6} is derived by taking \eqref{cdqH:def5} and replacing, for instance $c \mapsto 0$.
The limit formulas \eqref{limit1} and \eqref{limit2} are used whenever 
applicable. 
Also, whenever necessary the parameters should be re-named such that they are in the multiset $\{a,b\}$.
This completes the proof.
\end{proof}

\noindent Next for the Al-Salam--Chihara polynomials one has the following compact connection relation where the connection coefficients are given by a terminating ${}_2\phi_1$.

\begin{thm}
\label{singleASC}
Let $n\in\N_0$, $q\in\CCddag$,
$a,b,c,d\in\CCast$. Then
\begin{eqnarray}
&&Q_n(x;a,b|q)=c^n(\tfrac{a}{c};q)_n\sum_{k=0}^n Q_k(x;c,d|q)\left(\frac{q}{a}\right)^k\frac{(q^{-n};q)_k}{(q,\frac{q^{1-n}c}{a};q)_k}
\qhyp21{q^{k-n},\frac{b}{d}}{\frac{q^{k-n+1}c}{a}}{q,\frac{qd}{a}}.
\label{conASC1}
\end{eqnarray} 
\label{connectionASC}
\end{thm}

\begin{proof}
First start with the connection relation for the continuous dual $q$-Hahn polynomials and take the limit as $b,e\to 0$, then replace $c\mapsto b$, $f\mapsto e$.
This produces 
\begin{eqnarray}
&&\hspace{-1.1cm}Q_n(x;a,b|q)=(q,ab;q)_n\sum_{k=0}^n
Q_k(x;c,d|q)\frac{(q^kb)^{k-n}}{(q;q)_{n-k}(q,ab;q)_k}\nonumber\\
&&\hspace{2cm}\times
\sum_{l=0}^{n-k}
\left(\frac{qb}{d}\right)^l
\frac{(q^{-n-k},q^kcd;q)_l}{(q,q^kab;q)_l}
\qhyp32{q^{-(n-k-l)},q^{k+l}bd,\frac{b}{d}}
{q^{k+l}ab,0}{q,q}.
\label{origASC}
\end{eqnarray}
Then setting $b=d$, the terminating ${}_3\phi_2$ becomes unity and you are left with a ${}_2\phi_1$ which can be
summed using the $q$-Chu--Vandermonde sum
\eqref{qChuVander}. 
This results in the following connection relation with one free parameter for the Al-Salam--Chihara polynomials, namely
\begin{equation}
Q_n(x;a,b|q)=(q;q)_n\sum_{k=0}^n
Q_k(x;c,b|q)\frac{c^{n-k}(\tfrac{a}{c};q)_{n-k}}{(q;q)_{n-k}(q;q)_k}.
\label{ASCrewr}
\end{equation}
Then by applying Askey's method, namely using the above connection relation rewritten as
\begin{equation}
Q_k(x;c,b|q)=(q;q)_k\sum_{l=0}^k
Q_l(x;c,d|q)\frac{d^{k-l}(\tfrac{b}{d};q)_{k-l}}{(q;q)_{k-l}(q;q)_l},
\end{equation}
and inserting it into \eqref{ASCrewr}, reversing the sums and simplifying completes the proof.
\end{proof}

\noindent 
By comparing \eqref{origASC} with Theorem \ref{singleASC} one obtains the following interesting summation formula.
\begin{cor}
Let $n,k\in\N_0$, $n\ge k$, $q\in\CCddag$, $a,b,c,d\in\CCast$. Then
\begin{eqnarray}
&&\hspace{-2.0cm}\sum_{l=0}^{n-k}\left(\frac{qb}{d}\right)^l
\frac{(q^{k-n},q^kcd;q)_l}{(q,q^kab;q)_l}
\qhyp32{q^{k+l-n},q^{k+l}bd,\frac{b}{d}}{q^{k+l}ab,0}{q,q}
\nonumber\\
&&\hspace{4cm}
=\frac{q^{-\binom{k}{2}}(bc)^n(\frac{a}{c};q)_n(ab;q)_k}{(-ab)^k(ab;q)_n(\frac{q^{1-n}c}{a};q)_k}
\qhyp21{q^{k-n},\frac{b}{d}}{\frac{q^{k-n+1}c}{a}}{q,\frac{qd}{a}}.\end{eqnarray}
\end{cor}

\begin{proof}
Comparing the $k$th terms of the right-hand sides of \eqref{conASC1} and \eqref{origASC} completes the proof.
\end{proof}

\noindent Using the Al-Salam-Chihara polynomial 
representations, we can compute their $q^{-1}$ analogs.

\begin{cor}\label{cor:3.6}
Let $Q_n(x;a,b|q)$ and the respective parameters 
be defined as previously. 
Then, the $q^{-1}$-Al-Salam-Chihara polynomials 
have the following terminating basic hypergeometric series representations:
\begin{eqnarray}
\label{qiASC:1}&&\hspace{-1.8cm}Q_n(x;a,b|q^{-1}):=
q^{-\binom{n}{2}}
(-b)^n 
\left(\frac{1}{ab}
;q\right)_n\qhyp31{q^{-n}, 
\frac{z^{\pm}}{a}
}{
\frac{1}{ab}
}
{q,\frac{q^na}{b}}\\\label{qiASC:2} &&\hspace{0.27cm}
=q^{-\binom{n}{2}}(-a)^{n}\left(\frac{z^{\pm}}{a};q\right)_n 
\qphyp{2}{2}{-1}{q^{-n},q^{1-n}ab}
{q^{1-n}az^{\pm}}{q,q}\\
\label{qiASC:5} &&\hspace{0.27cm}=q^{-\binom{n}{2}}(-ab
z)^n\left(\frac{1}{ab};q\right)_n 
\qphyp{3}{1}{1}{q^{-n},
\frac{1}{az},
\frac{1}{bz}
}
{\frac{1}{ab}}{q,q}\\
\label{qiASC:3} &&\hspace{0.27cm}=
q^{-\genfrac{(}{)}{0pt}{}{n}{2}} (-a)^n 
\left(\frac{1}{az};q\right)_n\qhyp{2}{1}{q^{-n}, 
\frac{z}{b}}
{q^{1-n}az}{q, qbz}\\
\label{qiASC:4} &&\hspace{0.27cm}=
q^{-2\binom{n}{2}}(abz)^n
\left(
\frac{1}{az},
\frac{1}{bz}
;q\right)_n
\qhyp22{q^{-n}, q^{1-n}ab}{
q^{1-n}az,q^{1-n}bz
}
{q,qz^2}\\
&&\hspace{0.27cm}
=q^{-{\binom{n}{2}}}
(-b)^{n}
\frac{(\frac{z^{\pm}}{b};q)_n}
{(\frac{a}{b};q)_n}
\label{qiASC:7HSC} 
\qpWhyp{6}{5}{-2}{\frac{q^{-n}b}{a}}
{q^{-n},\frac{z^{\pm}}{a}}
{q,q^na^2}\\
&&\hspace{0.27cm}=q^{-\genfrac{(}{)}{0pt}{}{n}{2}}(-abz)^n 
\frac{(\frac{z}{a},\frac{z}{b};q)_n}{(z^2;q)_n} 
\label{qiASC:6} 
\qpWhyp{6}{5}{2}{\frac{q^{-n}}{z^2}}
{q^{-n},\frac{1}{az},\frac{1}{bz}}
{q,\frac{q^{2-n}ab}{z^2}}.
\end{eqnarray}
\end{cor}
\begin{proof}
Each inverse representation is derived from the 
corresponding representation by apply\-ing the 
map $q\mapsto 1/q$ and using \eqref{poch.id:3}.
\end{proof}

\begin{thm}
Let $n\in\N_0$, $x\in\CC$, $q\in\CCddag$, $a,b,c,d\in\CCast$. Then
\begin{eqnarray}
&&\hspace{-1.3cm}Q_n(x;a,b|q^{-1})=q^{-\binom{n}{2}}(-b)^n(\tfrac{d}{b};q)_n\nonumber\\
&&\hspace{1.7cm}\times\sum_{k=0}^nQ_k(x;c,d|q^{-1})
q^{\binom{k}{2}}\left(-\frac{q}{d}\right)^k
\frac{(q^{-n};q)_k}{(q,\frac{q^{1-n}b}{d};q)_k}\qhyp21{q^{-n+k},\frac{c}{a}}{\frac{q^{1+k-n}b}{d}}{q,\frac{qa}{d}}.
\label{compactASCqi}
\end{eqnarray}
\end{thm}

\begin{proof}
Start with Theorem \ref{connectionASC}, replace throughout $q\mapsto q^{-1}$ and simplifying completes the proof.
\end{proof}

\begin{rem}
One may start with the connection relation for the continuous dual $q^{-1}$-Hahn polynomials, Theorem 
\ref{cdqiHthm}, take the limit as $b,e\to 0$, followed by $c\mapsto b$, $f\mapsto e$ which produces the following double summation representation of the connection relation for the $q^{-1}$-Al-Salam--Chihara polynomials
\begin{eqnarray}
&&\hspace{-1cm}Q_n(x;a,b|q^{-1})=q^{-\binom{n}{2}}(-a)^n(q,\tfrac{1}{ab};q)_n
\sum_{k=0}^n Q_k(x;c,d|q^{-1})\frac{q^{\binom{k}{2}}(-a)^{-k}}{(q;q)_{n-k}(q,\frac{1}{ab};q)_k}\nonumber\\
&&
\hspace{1cm}\times\sum_{l=0}^{n-k}
\left(\frac{q^{{n-k}}c}{a}\right)^l
\frac{(q^{-(n-k)},\frac{q^k}{cd};q)_l}{(q,\frac{q^k}{ab};q)_l}
\qhyp31{q^{-(n-k-l)},\frac{q^{k+l}}{bd},\frac{d}{b}}{\frac{q^{k+l}}{ab}}{q,\frac{q^{n-k-l}b}{a}}.
\label{origqiASC}
\end{eqnarray}
\end{rem}

\noindent By comparing \eqref{origqiASC} with \eqref{compactASCqi}, one obtains the following interesting summation formula.
\begin{cor}
Let $n,k\in\N_0$, $n\ge k$, $q\in\CCddag$, $a,b,c,d\in\CCast$
\begin{eqnarray}
&&\hspace{-0.5cm}\sum_{l=0}^{n-k}\left(\frac{q^{n-k}c}{a}\right)^l
\frac{(q^{k-n},\frac{q^k}{cd};q)_l}{(q,\frac{q^k}{ab};q)_l}
\qhyp31{q^{-(n-k-l)},\frac{q^{k+l}}{bd},\frac{d}{b}}{\frac{q^{k+l}}{ab}} 
{q,\frac{q^{n-k-l}b}{a}}
\nonumber\\
&&\hspace{3cm}
=q^{\binom{k}{2}}q^{-nk}\left(\frac{b}{a}\right)^n\left(-\frac{qa}{d}\right)^k\frac{(\frac{d}{b};q)_n(\frac{1}{ab};q)_k}{(\frac{1}{ab};q)_n(\frac{q^{1-n}b}{d};q)_k}
\qhyp21{q^{k-n},\frac{c}{a}}{\frac{q^{k-n+1}b}{d}}{q,\frac{qa}{d}}.
\end{eqnarray}
\end{cor}

\begin{proof}
Comparing the $k$th term of \eqref{origqiASC} and \eqref{compactASCqi} and solving for the double sum completes the proof.
\end{proof}

\noindent 
This begs the question whether a similar compact connection relation exists for the continuous big $q^{-1}$-Hermite polynomials. Indeed it is true.
For the continuous big $q^{-1}$-Hermite polynomials the connection relation is given as follows. 

\subsubsection{The continuous big {\it q} and $q^{-1}$-Hermite 
polynomials}\label{sec:3.5}
\begin{cor} 
Let 
$n\in\mathbb N_0$, 
$q\in\CCddag$,
$a\in\CCast$,
$x=\frac12(z+z^{-1})\in \CCast$, $z\in\CCast$.
The continuous big $q$-Hermite polynomials have the following terminating basic hypergeometric series representations:
\begin{eqnarray}
\label{cbqH:def1}
&&\hspace{-4.0cm}H_n(x;a|q)
:=a^{-n}\qphyp{3}{0}{2}{q^{-n}, az^\pm }{-}{q,q}\\
\label{cbqH:def2} 
&&\hspace{-2.1cm}= 
q^{-\binom{n}{2}} (-a)^{-n}
(az^\pm;q)_n 
\qhyp12{q^{-n}}{\frac{q^{1-n}z^\pm}{a}}
{q,\frac{q^{2-n}}{a^2}}\\
\label{cbqH:def3}
&&\hspace{-2.1cm}=z^{n}(az^{-1};q)_n 
\qphyp{1}{1}{-1}{q^{-n}}{\frac{q^{1-n}}{a}z}{q,\frac{q}{az}}\\
\label{cbqH:def4}
&&\hspace{-2.1cm}=z^n\qhyp{2}{0}{q^{-n}, az }{-}{q,\frac{q^n}{z^2}}\\
\label{cbqH:def5}
&&\hspace{-2.1cm}
=
z^n
\frac{(\frac{a}{z};q)_n}{(z^{-2};q)_n}
\qpWhyp{5}{4}{-3}{q^{-n}z^2}
{q^{-n}, az}
{q,\frac{q^{2n-1}}{az^3}}.
\end{eqnarray}
\end{cor}
\begin{proof}
The representation \eqref{cbqH:def1} is derived by taking \eqref{ASC:def1} and replacing $a_2\mapsto 0$
(see also \cite[(14.18.1)]{Koekoeketal});
\eqref{cbqH:def2} is derived by taking \eqref{ASC:def2} and replacing $a_2\mapsto 0$;
\eqref{cbqH:def3} is derived by taking \eqref{ASC:def3} and replacing $a_2\mapsto 0$; 
\eqref{cbqH:def4} is derived by taking \eqref{ASC:def4} or \eqref{ASC:def5} and 
replacing $b\mapsto 0$
(see also \cite[(14.18.1)]{Koekoeketal});
\eqref{cbqH:def5} is derived by taking by taking \eqref{ASC:def6} and 
mapping $b \mapsto 0$. 
The limit formulas \eqref{limit1} and \eqref{limit2} are used whenever applicable and the parameters should be renamed as $a$. This completes the proof.
\end{proof}

\begin{thm}
Let $n\in\N_0$, $q\in\CCdag$, $a\in\CCast$, $x=\frac12(z+z^{-1})\in \CCast$, $z\in\CCast$. Then,
the continuous big $q$-Hermite polynomials have the following nonterminating basic hypergeometric series representations:
\begin{eqnarray}
\label{cbqH:def3b}
&&\hspace{-3.1cm}H_n(x;a|q)=z^n
\frac{(az^{-1};q)_n}
{(\frac{q}{az};q)_\infty}
\qhyp11{\frac{qz}{a}}{\frac{q^{1-n}}{a}z}{q,\frac{q^{1-n}}{az}}
\\
\label{cbqH:defnt2}
&&\hspace{-1.3cm}=\left(\frac{a}{q}\right)^n
\frac{(az^\pm;q)_n}{(\frac{q^2}{a^2};q)_\infty(\frac{a^2}{q};q)_n}
\qpWhyp543{\frac{q^{1-n}}{a^2}}{\frac{q}{a}z^\pm}{q,\frac{q^{2-2n}}{a^2}}.
\end{eqnarray}
\end{thm}

\begin{proof}
The nonterminating representation \eqref{cbqH:def3b} is derived from
\eqref{cbqH:def3} using 
\cite[(1.13.6)]{Koekoeketal}, and \eqref{cbqH:defnt2} can be derived from \eqref{cbqH:def3b} using \eqref{W543tran}.
\end{proof}

\noindent 
As noticed in the Askey--Wilson memoir \cite[\S6]{AskeyWilson}, for the Askey--Wilson polynomials if $d=h$ and $c=g$ in \eqref{conAW}, then the connection coefficients become a ${}_3\phi_2$ which is balanced and can be summed. In the resulting expression, one can then take the limit as $b,f\to 0$. This produces a compact expression for the connection relation for the continuous big $q$-Hermite polynomials.

\begin{thm}
Let $n\in\N_0$, $q\in\CCddag$, $a,b\in\CCast$. Then
\begin{eqnarray}
&&H_n(x;a|q)=b^n\sum_{k=0}^n H_k(x;b|q)q^{-\binom{k}{2}}\left(-\frac{q^n}{b}\right)^k
\frac{(q^{-n};q)_k(\frac{a}{b};q)_{n-k}}{(q;q)_{k}}.
\end{eqnarray}
\label{thm326}
\end{thm}

\begin{proof}
Here one can start with the connection relation for the Askey--Wilson polynomials \eqref{conAW} and set $d=h$, $c=g$. Then the resulting connection coefficients become a ${}_3\phi_2$ which is balanced and can be summed. In the resulting expression, one then takes the limit as $b,c,d,f\to 0$,
which completes the proof.
\end{proof}

\noindent One may use the above connection relation to obtain a connection relation between continuous $q$-Hermite polynomials and continuous big $q$-Hermite polynomials.

\begin{rem}
One may also start with the connection relation for Al-Salam--Chihara polynomials given by 
\eqref{origASC} and then take
the limit as $a,c\to 0$ followed by $b\mapsto a$, $d\mapsto b$ which
produces the following alternative representation of the connection relation for the continuous big $q$-Hermite polynomials
\begin{eqnarray}
&&\hspace{-1.1cm}H_n(x;a|q)=(q;q)_n\sum_{k=0}^n
H_k(x;b|q)\frac{(q^ka)^{k-n}}{(q;q)_{n-k}(q;q)_k}\nonumber\\
&&\hspace{2cm}\times
\sum_{l=0}^{n-k}
\left(\frac{qa}{b}\right)^l
\frac{(q^{-n-k};q)_l}{(q;q)_l}
\qhyp32{q^{-(n-k-l)},q^{k+l}ab,\frac{a}{b}}
{0,0}{q,q}.
\label{cqH}
\end{eqnarray}
\end{rem}

\noindent 
By comparing the $k$th term of the above connection relation with that of \eqref{cqH}, one obtains the following interesting double summation expression.

\begin{cor}Let $n, k\in\N_0$, with $n\ge k$, $q\in\CCddag$, $a,b\in\CCast$. Then
\begin{eqnarray}
&&\hspace{-3.0cm}
\sum_{l=0}^{n-k}
\left(\frac{qa}{b}\right)^l
\frac{(q^{-n-k};q)_l}{(q;q)_l}
\qphyp302{q^{-(n-k-l)},q^{k+l}ab,\frac{a}{b}}
{-}{q,q}
\nonumber\\
&&\hspace{1cm}
\!=\!(ab)^nq^{-3\binom{k}{2}}\left(\frac{q^{2n-1}}{ab}\right)^k
\frac{(q^{-n};q)_k(q,\frac{a}{b};q)_{n-k}}{(q;q)_n}.
\end{eqnarray}
\end{cor}

\noindent 
We also have the following nice limits of the connection relation for Al-Salam--Chihara polynomials which give connection relations between the Al-Salam--Chihara polynomials and the continuous big $q$-Hermite polynomials.
\begin{cor}
Let $n\in\N_0$, $q\in\CCddag$, $a,b,c\in\CCast$. Then
\begin{eqnarray}
&&Q_n(x;a,b|q)=c^n(\tfrac{a}{c};q)_n\sum_{k=0}^n H_k(x;c|q)
\left(\frac{q}{a}\right)^k
\frac{(q^{-n};q)_k}{(q,\frac{q^{1-n}c}{a};q)_k}\qhyp11{q^{k-n}}{\frac{q^{k-n+1}c}{a}}{q,\frac{qb}{a}},\\
&&H_n(x;a|q)=b^n\sum_{k=0}^n
Q_k(x;b,c|q)q^{-\binom{k}{2}}\left(-\frac{q^n}{b}\right)^k
\frac{(q^{-n};q)_k}{(q;q)_k}
\qhyp20{q^{k-n},\frac{a}{c}}{-}{q,q^{n-k}\frac{a}{c}}.
\end{eqnarray}
\end{cor}

\begin{proof}
These connection relations follow from Theorem 
\ref{connectionASC} by taking the $d\to 0$, $a\to 0$, limits respectively. 
\end{proof}

\noindent Using the continuous big $q$-Hermite polynomials, we can compute 
their $q^{-1}$ representations.

\begin{cor}
Let $H_n(x;a|q)$ and the respective parameters be defined as previously. 
Then, the continuous big $q^{-1}$-Hermite polynomials have the following terminating basic hypergeometric series representations:
\begin{eqnarray}
\label{cbqiH:1} &&\hspace{-3.7cm}H_n(x;a|q^{-1}) :=a^{-n}\qhyp{3}{0}{q^{-n},\frac{z^{\pm}}
{a}}{-}{q,q^na^2}\\\label{cbqiH:2}&&\hspace{-1.45cm}=q^{-\binom{n}{2}}(-a)^{n} 
\left(\frac{z^{\pm}}{a};q\right)_n\qphyp{1}{2}{-2}{q^{-n}}
{q^{1-n}az^{\pm}}{q,q}\\\label{cbqiH:3}&&\hspace{-1.45cm}=
q^{-\genfrac{(}{)}{0pt}{}{n}{2}}(-a)^n
\left(\frac{1}{az};q\right)_n\qhyp{1}{1}{q^{-n}}
{q^{1-n}az}{q,qz^2}\\
\label{cbqiH:4} &&
\hspace{-1.45cm}=z^n\qphyp{2}{0}{1}{q^{-n},\frac{1}{az}}{-}
{q,\frac{qa}{z}}\\
&&\hspace{-1.45cm}=(a z^2)^n \frac{(\frac{z}
{a};q)_n}{(z^2;q)_n}
\qpWhyp{5}{4}{3}{\frac{q^{-n}}{z^2}}
{q^{-n},{\frac{1}{az}}}
{q,\frac{q^{2-n}a}{z^3}}.
\label{cbqiH:5} 
\end{eqnarray}
\end{cor}
\begin{proof}
For \eqref{cbqiH:1}--\eqref{cbqiH:5}, each inverse representation is derived from the corresponding representation 
by applying the map $q\mapsto 1/q$ and using \eqref{poch.id:3}.
\end{proof}

\begin{cor}
Let $H_n(x;a|q)$ and the respective parameters be defined as previously. 
Then, the continuous big $q^{-1}$-Hermite polynomials have the following nonterminating basic hypergeometric series representations:
\begin{eqnarray}
&&\hspace{-1.7cm}H_n(x;a|q^{-1})
=q^{-\binom{n}{2}}(-a)^n
(qz^2;q)_\infty
(\tfrac{1}{az};q)_n\qhyp12{qaz}{qz^2,q^{1-n}az}{q,q^{1-n}z^2}.
\label{cbqiH:6}
\end{eqnarray}
\end{cor}
\begin{proof}
For \eqref{cbqiH:6}, one applies the nonterminating transformation 
\eqref{nt1112} to \eqref{cbqiH:3}.
\end{proof}


\begin{thm}
Let $n\in\N_0$, $q\in\CCddag$, $a,b\in\CCast$. Then
\begin{equation}
H_n(x;a|q^{-1})=q^{-\binom{n}{2}}(-a)^n(q;q)_n\sum_{k=0}^n
H_k(x;b|q^{-1})q^{\binom{k}{2}}\left(-\frac{1}{a}\right)^k
\frac{(\frac{b}{a};q)_{n-k}}{(q;q)_{n-k}(q;q)_k}.
\label{cbqiH2}
\end{equation}
\label{thm3.21}
\end{thm}

\begin{proof}
Starting with the connection relation for 
continuous dual $q^{-1}$-Hahn polynomials 
\eqref{cdqiH} and setting $f=c$ causes the 
${}_3\phi_2$ in the sum over $l$ to become unity 
with the resulting connection coefficients being 
proportional to the expression to a terminating 
${}_3\phi_2$.
Then setting $e=b$, the terminating ${}_3\phi_2$ 
reduces to a terminating ${}_2\phi_1$ which 
can be summed using 
the reversed version of the $q$-Chu--Vandermonde sum
\eqref{qChuVanderR}.
This completes the proof.
\end{proof}

\noindent One may use the above connection relation to obtain a connection relation between continuous $q^{-1}$-Hermite polynomials and continuous big $q^{-1}$-Hermite polynomials.

\begin{rem}
In a similar fashion as above a double summation connection relation for the continuous big $q^{-1}$-Hermite polynomials can be derived which is given by 
\begin{eqnarray}
&&\hspace{-1cm}H_n(x;a|q^{-1})=\dfrac{(q;q)_n}{a^n}
\sum_{k=0}^n H_k(x;b|q^{-1})\frac{a^{k}}{(q;q)_{n-k}(q;q)_k}\nonumber\\
&&
\hspace{1cm}\times\sum_{l=0}^{n-k}
\left(\frac{q^{{n-k}}a}{b}\right)^l
\frac{(q^{-(n-k)};q)_l}{(q;q)_l}
\qhyp30{q^{-(n-k-l)},\frac{q^{k+l}}{ab},\frac{b}{a}}{-}{q,q^{n-2k-2l}a^2}.
\label{cbqiH}
\end{eqnarray}
\end{rem}

\noindent 
By comparing \eqref{cbqiH} with \eqref{cbqiH2}, we obtain the following interesting double sum.
\begin{cor}
Let $q\in\CCddag$, $n,k\in\N_0$ such that $n\ge k$, $a,b\in\CCast$. Then
\begin{equation}
\sum_{l=0}^{n-k}
\left(\frac{q^{{n-k}}a}{b}\right)^l
\frac{(q^{k-n};q)_l}{(q;q)_l}
\qhyp30{q^{k+l-n},\frac{q^{k+l}}{ab},\frac{b}{a}}{-}{q,q^{n-2k-2l}a^2}=q^{-\binom{n}{2}+\binom{k}{2}}(-a^2)^{n-k}\left(\dfrac{b}{a};q\right)_{n-k}.
\end{equation}
\end{cor}
\begin{proof}
Equating \eqref{cbqiH} with \eqref{cbqiH2}, selecting out the $k$th term on both sides and solving for the double sum completes the proof.
\end{proof}

\subsubsection{The continuous {\it q} and $q^{-1}$-Hermite 
polynomials}\label{sec:3.6}

\medskip
Now we give the basic hypergeometric series representation for the continuous $q$-Hermite polynomials.

\begin{cor}
\label{cor32}
Let $n\in\mathbb N_0$, $q\in\CCddag$,
$x=\frac12(z+z^{-1})\in \CCast$, $z\in\CCast$. Then, one has the 
following terminating basic hypergeometric representation 
for the continuous $q$-Hermite polynomials:
\begin{eqnarray}
\label{cqHrep}
&&\hspace{-3.1cm}H_n(x|q):=
z^n
\qphyp10{-1}{q^{-n}}{-}{q,\frac{q^{n}}{z^2}}\\
\label{cqH:def2}
&&\hspace{-1.6cm}
=\frac{z^n}{(z^{-2};q)_n}
\qpWhyp{4}{3}{-4}{q^{-n}z^2}{q^{-n}}
{q,\frac{q^{3n-2}}{z^{4}}}.
\end{eqnarray}
\end{cor}
\begin{proof}
To obtain \eqref{cqHrep}, start with \eqref{aw:def3} and take the limit as 
$d\to c\to b\to a\to 0$ sequentially.
For \eqref{cqH:def2}, take the 
limit as $a\to 0$ in \eqref{cbqH:def5} and use \eqref{limit2}.
\end{proof}

\noindent 
In the analysis which is performed in Section \ref{sec:6.8} below which is associated with the $q$-Chaundy representations of a product of basic hypergeometric series for the Ismail--Masson $q$-exponential generating function for the continuous $q$-Hermite polynomials \cite{IsmailMasson1994}, we will arrive at the following result which we provide prematurely because of its relevance to the current material. 
\begin{thm}
Let $q\in\CCdag$, $x=\frac12(z+z^{-1})\in \CCast$, 
$z\in\CCast$. Then
\begin{eqnarray}
&&\hspace{-3.6cm}H_n(x|q)=q^{-\frac14n}
(-q^\frac12;q^\frac12)_n\qphyp311{q^{-\frac12n},q^\frac14z^\pm}{-q^\frac12}{q^\frac12,q^{\frac12}}\label{acqH1}\\
&&\hspace{-2.15cm}=(-1)^nq^{-\frac14n^2}
(q^\frac14z^\pm;q^\frac12)_n\qhyp22{\pm q^{-\frac12n}}{q^{\frac14-\frac12n}z^\pm}{q^\frac12,-q^{\frac12}}.\label{acqH2}
\end{eqnarray}
\label{thm542}
\end{thm}
\begin{proof}
Comparing the terms of the series of the 
alternate expressions for the generating function 
${\sf O}(x,t|q)$ given by Theorem \ref{thm634} using Theorem \ref{thm430} below completes 
the proof.
\end{proof}

\begin{cor}
Let $n\in\N_0$, $q\in\CCddag$, $a\in\CCast$. Then
\begin{eqnarray}
&&\hspace{-5cm}H_n(x|q)=a^n\sum_{k=0}^n
H_n(x;a|q)q^{-\binom{k}{2}}\left(-\frac{q^{n}}{a}\right)^k
\frac{(q^{-n};q)_k}{(q;q)_k}.
\end{eqnarray}
\label{cor327}
\end{cor}

\begin{proof}
Start with Theorem \ref{thm326} and take the limit as $a\to0$ and then replacing $b\mapsto a$ completes the proof.
\end{proof}

\noindent Similarly, we can compute the basic 
hypergeometric representation of the continuous 
$q^{-1}$-Hermite polynomials.

\begin{cor} 
Let $H_n(x|q)$ and the respective parameters be defined as previously.
The continuous $q^{-1}$-Hermite polynomials have the following terminating basic hypergeometric series representation
\begin{eqnarray}
\label{cqiH:def1}
&&\hspace{-3.7cm}H_n(x|q^{-1})
=z^n\qphyp{1}{0}{1}{q^{-n}}{-}{q,\frac{q}{z^2}}\\
\label{cqiH:def2}
&&\hspace{-1.85cm}=\frac{q^{\binom{n}{2}}(-1)^nz^{3n}}{(z^2;q)_n}
\qpWhyp{4}{3}{4}{\frac{q^{-n}}{z^2}}{q^{-n}}{q,\frac{q^{2-n}}{z^{4}}}.
\end{eqnarray}
\end{cor}

\begin{proof}
The inverse representations \eqref{cqiH:def1}, \eqref{cqiH:def2}
are derived from Corollary \ref{cor32} by applying the map $q\mapsto 
1/q$ and using \eqref{poch.id:3}.
\end{proof}

\begin{cor} 
Let $H_n(x|q)$ and the respective parameters be defined as previously.
The continuous $q^{-1}$-Hermite polynomials have the following nonterminating basic hypergeometric series representation
\begin{eqnarray}
&&\hspace{-6.7cm}H_n(x|q^{-1})
=z^n(\tfrac{q}{z^2};q)_\infty
\qhyp01{-}{\frac{q}{z^2}}{q,\frac{q^{1-n}}{z^2}}.
\label{cqiH:def3}
\end{eqnarray}
\end{cor}

\begin{proof}
The representation \eqref{cqiH:def3} follows from \eqref{cqiH:def1} by applying the
transformation \eqref{trans10101}.
\end{proof}

\noindent 
In the analysis which is performed in Section \ref{sec:6.8} below which is associated with the $q$-Chaundy representations of a product of basic hypergeometric series for the Ismail--Masson $q$-exponential generating function for the continuous $q^{-1}$-Hermite polynomials \cite{IsmailMasson1994}, we will arrive at the following result which we provide prematurely because of its relevance to the current material.

\begin{thm}
Let $n\in\mathbb N_0$, $q\in\CCddag$, $x=\frac12(z+z^{-1})\in \CCast$, $z\in\CCast$. Then
\begin{eqnarray}
&&\hspace{-3.3cm}H_n(x|q^{-1})=q^{-\frac14n^2}
\frac{(q;q)_n}{(q^\frac12;q^\frac12)_n}\qhyp31{q^{-
\frac12n},q^\frac14z^\pm}{-q^\frac12}
{q^\frac12,-q^{\frac12n}}\label{acqiH1}\\
&&\hspace{-1.45cm}=(-1)^nq^{-\frac14n^2}
(q^\frac14z^\pm;q^\frac12)_n\qphyp22{-1}
{\pm q^{-\frac12n}}{q^{\frac14-\frac12n}z^\pm}
{q^\frac12,q^{\frac12}}.\label{acqiH2}
\end{eqnarray}
\label{thm550}
\end{thm}
\begin{proof}
Comparing the terms of the series of the alternate 
expressions for the generating function ${\sf N}(x,t|q)$ 
in Theorem \ref{thm:6.38} by using Theorem \ref{thm:6.39} completes the proof.
\end{proof}

\begin{cor}
Let $n\in\N_0$, $q\in\CCddag$, $a\in\CCast$. Then
\begin{eqnarray}
&&\hspace{-4cm}H_n(x|q^{-1})=a^n(q;q)_n\sum_{k=0}^n
H_k(x;a|q^{-1})
\frac{q^{2\binom{k}{2}}\left(\frac{q^{1-n}}{a}\right)^k}{(q;q)_{n-k}(q;q)_k}.
\end{eqnarray}
\label{cor334}
\end{cor}

\begin{proof}
Start with Theorem \ref{thm3.21} and take the limit as $a\to0$ and then replacing $b\mapsto a$ completes the proof.
\end{proof}

\noindent Note that there exist the connection relations between the continuous $q$-Hermite polynomials and the continuous $q^{-1}$-Hermite polynomials \cite[cf.~Exercises 2.29-30]{CohlIsmail20}
\begin{eqnarray}
&&\hspace{-5.0cm}H_n(x|q)
=(q;q)_n\sum_{k=0}^{\lfloor\frac{n}{2}\rfloor}\frac{q^{3\binom{k}{2}}\left(-q^{1-n}\right)^k}{(q;q)_k(q;q)_{n-2k}}H_{n-2k}(x|q^{-1}),
\label{concqHqiH}\\
&&\hspace{-5.0cm}H_n(x|q^{-1})
=(q;q)_n\sum_{k=0}^{\lfloor\frac{n}{2}\rfloor}\frac{q^{2\binom{k}{2}}\left(q^{1-n}\right)^k}{(q;q)_k(q;q)_{n-2k}}H_{n-2k}(x|q).
\label{concqiHqH}
\end{eqnarray}

\subsubsection{The big $q$-Jacobi polynomials and functions}

We now introduce the big $q$-Jacobi polynomials which as we will learn in \S\ref{dbqJ} are dual to both the continuous dual $q$ and $q^{-1}$-Hahn polynomials. We will also learn in \S\ref{dbqJ} that for the continuous dual $q$-Hahn polynomials, there also exists a duality with big $q$-Jacobi function defined in \eqref{bqJf1}. We refer to this property as function duality.

\medskip
\noindent The big $q$-Jacobi polynomials $P_n(x;a,b,c;q)$ can be defined with the following terminating basic hypergeometric representations.

\begin{thm}
Let $n\in\N_0$, $a,b,c,x\in\CC$. Then, the big $q$-Jacobi polynomials are defined as
\begin{eqnarray}
&&\hspace{-4cm}P_n(x;a,b,c;q):=\qhyp32{q^{-n},q^{n+1}ab,x}{qa,qc}{q,q}
\label{defbqJ}\\
&&\hspace{-1.45cm}=q^{\binom{n}{2}}(-qa)^n\frac{(\frac{bx}{c};q)_n}{(qa;q)_n}\qhyp32{q^{-n},\frac{q^{-n}c}{ab},\frac{qc}{x}}{qc,\frac{q^{1-n}c}{bx}}{q,q}\label{defbqJ2}\\
&&\hspace{-1.45cm}=q^{\binom{n}{2}}(-qc)^n\frac{(\frac{bx}{c};q)_n}{(qc;q)_n}\qhyp32{q^{-n},\frac{q^{-n}}{b},\frac{qa}{x}}{qa,\frac{q^{1-n}c}{bx}}{q,q}.\label{defbqJ3}
\end{eqnarray}
\end{thm}

\begin{proof}
The definition of the big $q$-Jacobi polynomials can be found here \cite[(14.5.1)]{Koekoeketal}. The second two representations can be obtained by using \eqref{3phi2term}.
\end{proof}

\noindent 
The big $q$-Jacobi polynomials satisfy the following useful symmetry relation
\begin{equation}
P_n(x;a,b,c;q)=P_n\left(x;c,\frac{ab}{c},a;q\right),
\label{bqJsym}
\end{equation}
which easily follows by comparing the terminating basic hypergeometric representations \eqref{defbqJ2}, \eqref{defbqJ3} and is clearly evident in \eqref{defbqJ}.


\medskip
\noindent The big $q$-Jacobi
function is a generalization
of the big $q$-Jacobi
polynomials \eqref{defbqJ}, such that the degree is allowed
to be any complex number (as opposed to 
a non-negative integer). We define this as 
\begin{eqnarray}
&&\hspace{-7.75cm}P_\mu(x;a,b,c;q):=
\qhyp32{q^{-\mu},q^{\mu+1}ab,x}{qa,qc}{q,q}.
\label{bqJf1}
\end{eqnarray}
The big $q$-Jacobi function reduces to the big $q$-Jacobi polynomial when $\mu=n\in\mathbb N_0$.
Note that the big $q$-Jacobi function \eqref{bqJf1} also satisfies the symmetry relation \eqref{bqJsym}, namely
\begin{equation}
P_\mu(x;a,b,c;q)=P_\mu\left(x;c,\frac{ab}{c},a;q\right).
\label{bqJfsym}
\end{equation}

It is important to compare the big $q$-Jacobi function which appears 
in \eqref{bqJf1}, with 
that similar function which appears in the important
papers by Koelink, Stokman and Rosengren
\cite{KoelinkRosengren2002,KoelinkStokman2001,KoelinkStokman2003}. They define
the big $q$-Jacobi function as 
\cite[(3.2)]{KoelinkStokman2001} 
\begin{equation}
\Phi_\mu(x;a,b,c;q):=\qhyp32{a\mu^\pm,-\frac{1}{x}}
{ab,ac}{q,-bcx}
\end{equation}
\begin{rem}
The Koelink--Stokman big $q$-Jacobi function is
related to our big $q$-Jacobi function as follows
\begin{eqnarray}
&&\hspace{-3.3cm}\Phi_\mu(x;a,b,c;q)=\frac{(q^{\mu+1}\frac{ab}{c},\frac{q^{-\mu}}{c};q)_\infty}{(\frac{qab}{c},\frac{1}{c};q)_\infty}
P_\mu(x;a,b,c;q)\nonumber\\
&&\hspace{0.2cm}+\frac{(q^{-\mu},x,q^{\mu+1}ab,\frac{qa}{c};q)_\infty}{(qa,c,\frac{qab}{c},\frac{x}{c};q)_\infty}\qhyp32{\frac{q^{\mu+1}ab}{c},\frac{q^{-\mu}}{c},\frac{x}{c}}{\frac{q}{c},\frac{qa}{c}}{q,q}.
\end{eqnarray}
The Koelink--Stokman 
big $q$-Jacobi
function can be identified with \eqref{bqJf1},
by taking 
\begin{eqnarray}
&&\hspace{-7.0cm}(a,b,c,\mu,x)\mapsto\left(\sqrt{qab},\sqrt{\frac{qa}{b}},\frac{\sqrt{q}c}{\sqrt{ab}},q^{\mu+\frac12}\sqrt{ab},-\frac{1}{x}\right).
\end{eqnarray}
If $\mu=n\in\mathbb N_0$, then the big $q$-Jacobi polynomial is related to the Koelink--Stokman big $q$-Jacobi function by
\begin{equation}
\Phi_n(x;a,b,c;q)=q^{-\binom{n}{2}}(-qc)^{-n}\frac{(qc;q)_n}{(\frac{qab}{c};q)_n}P_n(x;a,b,c;q).
\end{equation}
\end{rem}

\subsubsection{The little $q$-Jacobi polynomials and functions}

\noindent 
We now introduce the little $q$-Jacobi polynomials which as we will learn in \S\ref{dlqJ} are dual to both the $q$ and $q^{-1}$ Al-Salam--Chihara polynomials. We will also learn in \S\ref{dlqJ} that for the Al-Salam--Chihara polynomials, there also exists a duality with little $q$-Jacobi function which we now introduce in \eqref{lqJf1} below.
The little $q$-Jacobi polynomials $p_n(x;a,b;q)$ can be defined with the following terminating basic hypergeometric representations.

\begin{thm}Let $n\in\N_0$, $x,a,b\in\CCast$. Then
\begin{eqnarray}
\label{deflqJ}
&&\hspace{-2.8cm}p_n(x;a,b;q):=\qhyp21{q^{-n},q^{n+1}ab}{qa}{q,qx}\\
&&\hspace{-0.6cm}=q^{-\binom{n}{2}}(-qb)^{-n}\frac{(qb;q)_n}{(qa;q)_n}
\qhyp32{q^{-n},q^{n+1}ab,qbx}{qb,0}{q,q}\label{deflqJ2}\\
&&\hspace{-0.6cm}=(qbx;q)_n
\qhyp32{q^{-n},\frac{q^{-n}}{b},0}{qa,\frac{q^{-n}}{bx}}{q,q}
\label{deflqJ3}\\
&&\hspace{-0.6cm}=q^{-\binom{n}{2}}(-x)^n(\tfrac{1}{x};q)_n
\qhyp22{q^{-n},\frac{q^{-n}}{b}}{q^{1-n}x,qa}{q,q^{n+2}abx}
\label{deflqJ4}\\
&&\hspace{-0.6cm}=q^{\binom{n}{2}}(-qa)^n\frac{(qb;q)_n}{(qa;q)_n}
\qhyp31{q^{-n},q^{n+1}ab,\frac{1}{x}}{qb}{q,\frac{x}{a}}.
\label{deflqJ5}
\end{eqnarray}
\end{thm}
\begin{proof}
For \eqref{deflqJ}--\eqref{deflqJ3}, see \cite[(14.12.1)]{Koekoeketal}, \cite[(66, 67)]{KoornwinderMazzocco2018}. For \eqref{deflqJ4}, \eqref{deflqJ5}, apply \cite[\href{https://dlmf.nist.gov/17.9.E1}{(17.9.1)}]{NIST:DLMF}, \cite[\href{https://dlmf.nist.gov/17.9.E2}{(17.9.2)}]{NIST:DLMF} to \eqref{deflqJ} respectively.
\end{proof}

The little $q$-Jacobi
functions is a generalization
of the little $q$-Jacobi
polynomial \eqref{deflqJ}, such that the degree is allowed
to be any complex number (as opposed to 
a non-negative integer). We define this as 
\begin{eqnarray}
&&\hspace{-0.7cm}p_\mu(x;a,b;q):=
\qhyp21{q^{-\mu},q^{\mu+1}ab}{qa}{q,qx}
\label{lqJf1}\\
&&\hspace{1.5cm}=\frac{(q^{\mu+1}a,\frac{q^{-\mu}}{b};q)_\infty}{(qa,\frac{1}{b};q)_\infty}\qhyp32{q^{-\mu},q^{\mu+1}ab,qbx}{qb,0}{q,q}\nonumber\\
&&\hspace{3cm}+\frac{(q^{-\mu},q^{\mu+1}ab,qbx;q)_\infty}{(b,qa,qx;q)_\infty}\qhyp32{q^{\mu+1}a,\frac{q^{-\mu}}{b},qx}{\frac{q}{b},0}{q,q}\\
&&\hspace{1.5cm}=\frac{(qbx;q)_\infty}{(q^{\mu+1}bx;q)_\infty}
\qhyp32{q^{-\mu},\frac{q^{-\mu}}{b},0}
{qa,\frac{q^{-\mu}}{bx}}{q,q}\nonumber\\
&&\hspace{3cm}+
\frac{(q^{-\mu},q^{\mu+2}abx,
\frac{q^{-\mu}}{b};q)_\infty}{(qa,qx,\frac{q^{-\mu-1}}{bx};q)_\infty}
\qhyp32{qx,qbx,0}{q^{\mu+2}bx,q^{\mu+2}abx}{q,q},
\end{eqnarray}
where $|qx|<1$ for \eqref{lqJf1}.
Both of these functions reduce to the big and little $q$-Jacobi polynomials when $\mu=n\in\mathbb N_0$.
The little $q$-Jacobi function can be obtained from the big $q$-Jacobi function using the following limit transition
\begin{equation}
p_\mu(x;a,b;q)=\lim_{c\to\infty}P_\mu(qcx;a,b,c;q).
\end{equation}

\begin{rem}
It is important to compare the little $q$-Jacobi function which appears above
in \eqref{lqJf1} with 
that similar function which appears in the important
papers by Koelink, Stokman and Rosengren
\cite{KoelinkRosengren2002,KoelinkStokman2001,KoelinkStokman2003}. 
In \cite[(4.2)]{KoelinkStokman2001}, the little $q$-Jacobi function is defined as 
\begin{equation}
\phi_\mu(x;a,b;q):=\qhyp21{a\mu^\pm}{ab}{q,-bx}.
\label{KSRlJf}
\end{equation}
The Koelink--Stokman 
little $q$-Jacobi
function can be identified with \eqref{lqJf1}
by taking 
\begin{eqnarray}
&&\hspace{-4cm}(a,b,\mu,x)\mapsto\left(\sqrt{qab},\sqrt{\frac{qa}{b}},q^{\mu+\frac12}\sqrt{ab},-qbx\right).
\end{eqnarray}
\end{rem}

\subsubsection{The $q$ and $q^{-1}$-Bessel polynomials and functions}

\noindent 
We now introduce the $q$-Bessel polynomials which as we will learn in \S\ref{dualBessel} are dual to both the continuous big $q$ and $q^{-1}$ Hermite polynomials. We will also learn in \S\ref{dualBessel} that for the continuous big $q$-Hermite  polynomials, there also exists a duality with $q^{-1}$-Bessel function which we now introduce in \eqref{qiBf} below.
The $q$-Bessel polynomials $y_n(x;a;q)$ can be defined with the following terminating basic hypergeometric representations, namely,
\begin{eqnarray}
&&\hspace{-0.0cm}y_n(x;a;q):=\qhyp21{q^{-n},-q^na}{0}{q,qx}\\
&&\hspace{1.85cm}=(-q^nax)^n\qhyp21{q^{-n},\frac{1}{x}}{0}{q,-\frac{q^{1-n}}{a}}\\
&&\hspace{1.85cm}=q^{-\binom{n}{2}}(-x)^n(-a;q)_\infty\frac{(\tfrac{1}{x};q)_n}{(-a;q)_n}
\qhyp21{qx,0}{q^{1-n}x}{q,-q^na}\\
&&\hspace{1.85cm}=q^{-\binom{n}{2}}(-x)^n (\tfrac{1}{x};q)_n\qhyp11{q^{-n}}{q^{1-n}x}{q,-q^{n+1}ax}\\
&&\hspace{1.85cm}=\frac{(-qax;q)_\infty}{(qx;q)_\infty(-qax;q)_n}\qhyp11{-q^na}{-q^{n+1}ax}{q,q^{1-n}x}\\
&&\hspace{1.85cm}=q^{2\binom{n}{2}}(-qa)^n\qhyp30{q^{-n},-q^na,\tfrac{1}{x}}{-}{q,-\frac{x}{a}}\\
&&\hspace{1.85cm}=q^{-\binom{n}{2}}(-x)^n\frac{(\frac{1}{x};q)_n(-a;q)_{2n}}{(-a;q)_n}\qhyp32{q^{-n},0,0}{q^{1-n}x,-\frac{q^{1-2n}}{a}}{q,q}
\label{qB:7}\\
&&\hspace{1.85cm}=q^{-\binom{n}{2}}(-x)^n\frac{(-qax;q)_\infty(\frac{1}{x};q)_n}{(-qax;q)_n}\qhyp12{qx}{q^{1-n}x,-q^{n+1}ax}{q,-qax}\\
&&\hspace{1.85cm}=q^{-\binom{n}{2}}(-x)^n\frac{(\frac{1}{x};q)_n}{(-qax;q)_n}
\qpWhyp{5}{4}{3}{-ax}{q^{-n},-q^na}{q,-q^2ax^2}.
\end{eqnarray}

\noindent One also has the following terminating basic hypergeometric representations for the 
$q^{-1}$-Bessel polynomials $y_n(x;a;q^{-1})$, namely,
\begin{eqnarray}
&&\hspace{-1cm}y_n(x;a;q^{-1}):=\qhyp20{q^{-n},-\frac{q^n}{a}}{-}{q,-ax}\\
&&\hspace{3cm}\hspace{-1.65cm}=q^{-2\binom{n}{2}}\left(-\frac{ax}{q}\right)^n\qhyp20{q^{-n},x}{-}{q,-\frac{q^{2n}}{ax}}\\
&&\hspace{3cm}\hspace{-1.65cm}=(x;q)_n
\qhyp21{q^{-n},0}{\frac{q^{1-n}}{x}}{q,-q^{1-n}a}\\
&&\hspace{3cm}\hspace{-1.65cm}=q^{-2\binom{n}{2}}\left(-\frac{a}{q}\right)^n\qhyp32{q^{-n},-\frac{q^n}{a},x}{0,0}{q,q}
\label{qiB:def4}\\
&&\hspace{3cm}\hspace{-1.65cm}=q^{-3\binom{n}{2}}\left(\frac{a}{q}\right)^n\frac{(x;q)_n(-\frac{1}{a};q)_{2n}}{(-\frac{1}{a};q)_n}\qhyp12{q^{-n}}{-q^{1-2n}a,\frac{q^{1-n}}{x}}{q,-\frac{q^{2-2n}a}{x}}\\
&&\hspace{3cm}\hspace{-1.65cm}=q^{\binom{n}{2}}\left(\frac{q}{ax}\right)^n
\frac{(x;q)_n}{(-\frac{q}{ax};q)_n}
\qpWhyp{5}{4}{-3}{-\frac{1}{ax}}{q^{-n},-\frac{q^n}{a}}{q,\frac{a^2x}{q}}.
\end{eqnarray}
\vspace{0.2cm}

\noindent 
There exists a limit transition from little $q$-Jacobi polynomials to the $q$-Bessel polynomials
\cite[p.~529]{Koekoeketal}
\begin{equation}
\lim_{b\to 0}p_n\left(x;b,-\frac{a}{qb};q\right)
=y_n(x;a;q).
\end{equation}
Similarly, it is not hard to demonstrate by comparing basic hypergeometric representations that there also exists a limit transition from the little $q$-Jacobi polynomials to the $q^{-1}$-Bessel polynomials 
\begin{equation}
\lim_{b\to 0}p_n\left(\frac{x}{qb};-\frac{1}{qab},b;q\right)=y_n(x;a;q^{-1}).
\end{equation}

\noindent Now define the {\it $q^{-1}$-Bessel function} by extending the definition of the $q^{-1}$-Bessel polynomial to arbitrary degree values $\mu\in\CC$, namely
\begin{equation}
y_\mu(x;a;q^{-1}):=q^{-2\binom{\mu}{2}}\left(-\frac{a}{q}\right)^\mu\qhyp32{q^{-\mu},-\frac{q^\mu}{a},x}{0,0}{q,q},
\label{qiBf}
\end{equation}
so that the basic hypergeometric series representation is convergent for all $q\in\CCdag$.

\subsubsection{The $q$ and $q^{-1}$-Stieltjes--Wigert polynomials}

The Stieltjes--Wigert polynomials can be defined as follows \cite[(14.27.1)]{Koekoeketal}
\begin{equation}
S_n(x;q):=\frac{1}{(q;q)_n}
\qhyp11{q^{-n}}{0}{q,-q^{n+1}x}.
\label{SW}
\end{equation}
By applying \eqref{poch.id:3}, we can obtain a terminating basic hypergeometric representation of the $q^{-1}$-Stieltjes--Wigert polynomials, namely
\begin{equation}
S_n(x;q^{-1}):=\frac{q^{\binom{n}{2}}(-q)^n}{(q;q)_n}
\qhyp20{q^{-n},0}{-}{q,-x}.
\label{qiSW}
\end{equation}
In \cite{IsmailZhang2007}, a symmetry relation for the Stieltjes--Wigert polynomials was derived. This symmetry relation has a clear analogue for the $q^{-1}$-Stieltjes--Wigert polynomials, and we give both now.
\begin{thm}
Let $n\in\N_0$, $q\in\CCddag$, $x\in\CCast$. Then
\begin{eqnarray}
&&\hspace{-6.3cm}S_n(x;q)=q^{2\binom{n}{2}}(-qx)^nS_n\left(\frac{q^{-2n}}{x};q\right),
\label{SWrel1}\\
&&\hspace{-6.3cm}S_n(x;q^{-1})=q^{-2\binom{n}{2}}\left(-\frac{x}{q}\right)^n
S_n\left(\frac{q^{2n}}{x};q^{-1}\right).
\label{SWrel2}
\end{eqnarray}
\end{thm}
\begin{proof}
Starting with \cite[Theorem 2.1]{IsmailZhang2007} gives \eqref{SWrel1}. Then applying the map $q\mapsto q^{-1}$ yields \eqref{SWrel2}.
\end{proof}
\noindent 
It is straightforward to show that the $q$ and $q^{-1}$-Stieltjes--Wigert polynomials are special cases of the continuous $q^{-1}$ and $q$-Hermite polynomials respectively.
\begin{thm}
Let $n\in\N_0$, $q\in\CCddag$, $z, y\in\CCast$, $x=\frac12(z+z^{-1})\in\CCast$. Then
\begin{eqnarray}
&&\hspace{-1cm}S_n(y;q)=q^{\binom{n}{2}}\expe^{\frac12 i\pi n}
\frac{(qy)^{\frac12 n}}{(q;q)_n}
H_n\left[\sqrt{-q^ny}|q^{-1}\right]
=q^{\binom{n}{2}}\expe^{-\frac12 i\pi n}
\frac{(qy)^{\frac12 n}}{(q;q)_n}
H_n\left[\sqrt{-q^{-n}y^{-1}}|q^{-1}\right]
,\label{eqrel1}\\
&&\hspace{-1cm}S_n(y;q^{-1})=\expe^{-\frac12 i\pi n}
\frac{(qy)^{\frac12 n}}{(q;q)_n}
H_n\left[\sqrt{-q^{-n}y}|q\right]
=\expe^{\frac12 i\pi n}
\frac{(qy)^{\frac12 n}}{(q;q)_n}
H_n\left[\sqrt{-q^{n}y^{-1}}|q\right],
\label{eqrel2}\\
&&\hspace{-1.0cm}H_n(x|q)=\frac{q^{-\binom{n}{2}}(q;q)_n}{(-qz)^n}
S_n(-q^nz^2;q^{-1})=
q^{-\binom{n}{2}}(q;q)_n\left(-\frac{z}{q}\right)^nS_n(-q^nz^{-2};q^{-1})\label{eqrel3},\\
&&\hspace{-1.0cm}H_n(x|q^{-1})=(q;q)_nz^n S_n(-q^{-n}z^{-2};q)=(q;q)_nz^{-n} S_n(-q^{-n}z^2;q).\label{eqrel4}
\end{eqnarray}
\end{thm}

\begin{proof}
Comparing the terminating basic hypergeometric representation \eqref{SW} with \eqref{cqiH:def1} yields \eqref{eqrel1}. Comparing the terminating basic hypergeometric representation \eqref{qiSW} with \eqref{cqHrep}
yields \eqref{eqrel2}.
Inverting these expressions produces
\eqref{eqrel3}, \eqref{eqrel4}.
\end{proof}

\noindent 
Using the relations \eqref{SWrel1}, \eqref{SWrel2}, we have the following connection formulas between the continuous $q$ and $q^{-1}$-Hermite polynomials with the $q^{-1}$ and $q$-Stieltjes--Wigert polynomials respectively.
\subsection{Duality relations for the $q$ and $q^{-1}$-symmetric subfamilies\label{sec:3.8}}
\noindent There exist duality relations for the $q$ and $q^{-1}$-symmetric subfamilies of the Askey--Wilson polynomials. Some of these duality relations are well-known and others are not as well-known. We now make an attempt to summarize these duality relations. In \S\ref{sec:3.8} below we will describe some orthogonality relations which are non-trivially satisfied when one adopts the duality relations which are described in this section. 
The hierarchy of duality relations are described in Figure 
\ref{Figdual}.

\begin{figure}[!htb]
\caption{This figure depicts the duality relations for the $q$ and $q^{-1}$-symmetric subfamilies of the Askey--Wilson polynomials. Arrows represent limit transitions between the subfamilies of the Askey--Wilson polynomials. Dashed lines represent polynomial duality and double dashed lines represent function duality.\label{Figdual}}
\begin{tikzpicture}[level distance=.05cm,sibling distance=.04cm,scale=0.62,every node/.style={scale=0.62},inner sep=11pt]

\node (qB) at (11.40,-17.40) {
{\setlength{\fboxrule}{.020cm}
\fbox{
$\begin{array}{c}
{\displaystyle \,}\\[-12pt]
\text{\!$q$-Bessel}\\[1pt] 
\text{\phantom{\!}polynomials}
\end{array}$
}
}
};

\node (cbqiH) at (19.20,-17.40) {
{\setlength{\fboxrule}{.020cm}
\fbox{
$\begin{array}{c}
{\displaystyle \,}\\[-12pt]
\text{\!continuous big $q^{-1}$-Hermite}\\[1pt] 
\text{\phantom{\!}polynomials}
\end{array}$
}
}
};

\node (cbqH) at (-0.80,-17.40) {
{\setlength{\fboxrule}{.020cm}\fbox{
$\begin{array}{c}
{\displaystyle \,}\\[-12pt]
\text{\!continuous big $q$-Hermite}\\[1pt] 
\text{\phantom{\!}polynomials}
\end{array}$
}
}
};

\node (qiB) at (6.64,-17.40) {
{\setlength{\fboxrule}{.020cm}\fbox{
$\begin{array}{c}
{\displaystyle \,}\\[-12pt]
\text{\!$q^{-1}$-Bessel}\\[1pt] 
\text{\phantom{\!}functions}
\end{array}$
}
}
};

\node (ASC) at (2.20,-14.20) {
{\setlength{\fboxrule}{.020cm}\fbox{
$\begin{array}{c}
{\displaystyle \,}\\[-12pt]
\text{\!Al-Salam--Chihara}\\[1pt] 
\text{\phantom{\!}polynomials}
\end{array}$
}
}
};

\node (lqJ) at (8.90,-14.20) {
{\setlength{\fboxrule}{.020cm}\fbox{
$\begin{array}{c}
{\displaystyle \,}\\[-12pt]
\text{\!little $q$-Jacobi}\\[1pt] 
\text{\phantom{\!}functions/polynomials}
\end{array}$
}
}
};

\node (qiASC) at (15.90,-14.20) {
{\setlength{\fboxrule}{.020cm}\fbox{
$\begin{array}{c}
{\displaystyle \,}\\[-12pt]
\text{\!$q^{-1}$-Al-Salam--Chihara}\\[1pt] 
\text{\phantom{\!}polynomials}
\end{array}$}}};

\node (cdqH) at (2.50,-11.20) {
{\setlength{\fboxrule}{.020cm}\fbox{
$\begin{array}{c}
{\displaystyle \,}\\[-12pt]
\text{\!continuous dual $q$-Hahn}\\[1pt] 
\text{\phantom{\!}polynomials}
\end{array}$}}};

\node (bqJ) at (8.90,-11.20) {
{\setlength{\fboxrule}{.020cm}\fbox{
$\begin{array}{c}
{\displaystyle \,}\\[-12pt]
\text{\!big $q$-Jacobi}\\[1pt] 
\text{\phantom{\!}functions/polynomials}
\end{array}$}}};

\node (cdqiH) at (15.80,-11.20) {
{\setlength{\fboxrule}{.020cm}\fbox{
$\begin{array}{c}
{\displaystyle \,}\\[-12pt]
\text{\!continuous dual $q^{-1}$-Hahn}\\[1pt] 
\text{\phantom{\!}polynomials}
\end{array}$
}
}
};

\node (AW) at (8.90,-8.50) {
{\setlength{\fboxrule}{.020cm}\fbox{
$\begin{array}{c}
{\displaystyle \,}\\[-12pt]
\text{\!Askey--Wilson}\\[1pt] 
\text{\phantom{\!}polynomials}
\end{array}$
}
}
};

\draw[->,line width=1.5pt] (AW)--(bqJ);
\draw[->,line width=1.5pt] (AW)--(cdqH);
\draw[->,line width=1.5pt] (AW)--(cdqiH);
\draw[->,line width=1.5pt] (bqJ)--(lqJ);
\draw[->,line width=1.5pt] (lqJ)--(qiB);
\draw[->,line width=1.5pt] (lqJ)--(qB);
\draw[->,line width=1.5pt] (cdqH)--(ASC);
\draw[->,line width=1.5pt] (cdqiH)--(qiASC);
\draw[->,line width=1.5pt] (ASC)--(cbqH);
\draw[->,line width=1.5pt] (qiASC)--(cbqiH);
\draw[decoration={dashsoliddouble}, decorate,line width=1.5pt] (cbqH)--(qiB);
\draw[dashed,line width=1.5pt] (qB)--(cbqiH);
\draw[decoration={dashsoliddouble}, decorate,line width=1.5pt] (ASC)--(lqJ);
\draw[dashed,line width=1.5pt] (lqJ)--(qiASC);
\draw[decoration={dashsoliddouble}, decorate,line width=1.5pt] (cdqH)--(bqJ);
\draw[dashed,line width=1.5pt] (bqJ)--(cdqiH);

\end{tikzpicture}
\label{figure2}
\end{figure}

\noindent First we describe the duality relations between the continuous dual $q$-Hahn and continuous dual $q^{-1}$-Hahn polynomials
with the big $q$-Jacobi polynomials. 
Let $\mathfrak a:=\{a,b,c\}$, ${\bf a}:=\{a_1,a_2,a_3\}$,
${\bf a}={\mathfrak a}$.
These duality relations are given in the following theorem. 
In this and the following sections related to these dualities, we will adopt the notation $p_n[z;a,b,c|q]=p_n(x;a,b,c|q)$,
$Q_n[z;a,b|q]=Q_n(x;a,b|q)$ where $x=\frac12(z+z^{-1})\in \CCast$.

\subsubsection{Duality relations for the continuous dual $q$ and $q^{-1}$-Hahn polynomials\label{dbqJ}}

\begin{thm}
\label{thm311}
Let $q\in\CCddag$, $m,n\in\mathbb N_0$,
${\bf a}:=\{a_1,a_2,a_3\}$, $a_k\in\CCast$, 
$k,p,r,t\in{\bf 3}:=\{1,2,3\}$, such that $r\in{\bf 3}\setminus\{p\}$,
$t\in{\bf 3}\setminus\{p,r\}$. Then, the duality relations between the continuous dual $q$-Hahn and the big $q$-Jacobi polynomials are given by:
\begin{eqnarray}
&&\hspace{-0.20cm}
p_n\left[q^{m}a_p;{\bf a}|q\right]=
p_n\left[\frac{q^{-m}}{a_p};{\bf a}|q\right]=
p_n\left(\tfrac12\left(q^{m}a_p+\frac{q^{-m}}{a_p}\right);{\bf a}|q\right)\nonumber\\
&&\hspace{5.10cm}
=\frac{(a_{pr},a_{pt};q)_n}{a_p^n}\,
P_m\left(q^{-n};\frac{a_{pr}}{q},\frac{a_p}{a_r},\frac{a_{pt}}{q};q\right),
\label{dcdqHbqJabthm}\\
\nonumber\\
&&\hspace{-0.20cm}
p_n\left[\frac{q^{m}}{a_p};{\bf a}|q^{-1}\right]=
p_n[q^{-m}a_p;{\bf a}|q^{-1}]=
p_n\left(\tfrac12\left(\frac{q^{m}}{a_p}+q^{-m}a_p\right);{\bf a}|q^{-1}\right)\nonumber\\
&&\hspace{1cm}=q^{-2\binom{n}{2}-\binom{m}{2}}(abc)^n\!\left(\frac{-a_p}{qa_t}\right)^{\!m}\!
\frac{\left(\frac{1}{a_{pr}},\frac{1}{a_{pt}};q\right)_n\!\left(\frac{qa_t}{a_p};q\right)_m}{\left(\frac{1}{a_{pt}};q\right)_{m}}
P_m\left(\frac{q^n}{a_{pr}};\frac{1}{qa_{pr}},\frac{a_r}{a_p},\frac{a_t}{a_p};q\right).
\label{dcdqiHbqJabthm}
\end{eqnarray}
\end{thm}
\begin{proof}
See cf.~\cite[(44)]{Koorwinder2018}, \cite[\S4.3]{AtakishiyevKlimyk2006}.
\end{proof}

\noindent As an example of these duality relations, consider $(p,r,t)=(1,2,3)$ in which the duality relations reveal themselves through
\begin{eqnarray}
&&\hspace{-0.25cm}
p_n\left[q^{m}a;a,b,c|q\right]=
p_n\left[\frac{q^{-m}}{a};a,b,c|q\right]=
p_n\left(\tfrac12\left(q^{m}a+\frac{q^{-m}}{a}\right);a,b,c|q\right)\nonumber\\
&&\hspace{6.07cm}
=\frac{\left(ab,ac;q\right)_n}{a^n}\,
P_m\left(q^{-n};\frac{ab}{q},\frac{a}{b},\frac{ac}{q};q\right),
\label{dcdqHbqJab}\\
&&\hspace{-0.25cm}
p_n\left[\frac{q^{m}}{a};a,b,c|q^{-1}\right]=
p_n[q^{-m}a;a,b,c|q^{-1}]=
p_n\left(\tfrac12\left(\frac{q^{m}}{a}+q^{-m}a\right);a,b,c|q^{-1}\right)\nonumber\\
&&\hspace{2.07cm}
=q^{-2\binom{n}{2}-\binom{m}{2}}(abc)^n\!\left(\frac{-a}{qc}\right)^{\!m}\!
\frac{\left(\frac{1}{ab},\frac{1}{ac};q\right)_n\!\left(\frac{qc}{a};q\right)_m}{\left(\frac{1}{ac};q\right)_{m}}
P_m\left(\frac{q^n}{ab};\frac{1}{qab},\frac{b}{a},\frac{c}{a};q\right).
\label{dcdqiHbqJab}
\end{eqnarray}

\noindent 
One can invert the above relation to compute the duality relation between the big $q$-Jacobi and the continuous dual $q^{-1}$-Hahn polynomials, namely,
\begin{eqnarray}
&&\hspace{-0.5cm}P_m(q^{n+1}a;a,b,c;q)\nonumber\\
&&\hspace{0.5cm}=q^{2\binom{n}{2}}
q^{\binom{m}{2}}\left(\frac{\sqrt{q^3a^3b}}{c}\right)^n
\frac{
(-qc)^m(\frac{qab}{c};q)_m}{(qa,\frac{qab}{c};q)_n(qc;q)_m}
p_n\left[q^{m+\frac12}\sqrt{ab};\frac{1}{\sqrt{qab}},\sqrt{\frac{b}{qa}},\frac{c}{\sqrt{qab}}\Bigg|q^{-1}\right].
\label{dualA}
\end{eqnarray}
Similarly, one can use the symmetry of the big $q$-Jacobi polynomials \eqref{bqJsym} to compute the following interesting duality relation
\begin{eqnarray}
&&\hspace{-0.5cm}P_m(q^{n+1}c;a,b,c;q)=P_m\left(q^{n+1}c;c,\frac{ab}{c},a;q\right)\nonumber\\
&&\hspace{0.5cm}=q^{2\binom{n}{2}}
q^{\binom{m}{2}}\left(\frac{c\sqrt{q^3b}}{\sqrt{a}}\right)^n
\frac{
(-qa)^m(qb;q)_m}{(qb,qc;q)_n(qa;q)_m}
p_n\left[q^{m+\frac12}\sqrt{ab};\frac{1}{\sqrt{qab}},\sqrt{\frac{a}{qb}},\frac{1}{c}\sqrt{\frac{ab}{q}}\Bigg|q^{-1}\right].
\label{dualC}
\end{eqnarray}

\noindent Now we present a theorem which gives the duality relation between
the continuous dual $q$-Hahn polynomials and the big $q$-Jacobi function.

\begin{thm}
\label{thm3.33}
Let $n\in\mathbb N_0$, $\mu\in \mathbb C$, $q\in\CCddag$, $a,b,c\in\CCast$. 
Then, one has the following duality relations 
for the continuous dual $q$-Hahn polynomials 
with the big $q$-Jacobi functions:
\begin{eqnarray}
&&\hspace{-3.8cm}p_n\left[\frac{q^{-\mu}}{a};a,b,c|q\right]
=
\frac{(ab,ac;q)_n}{a^n}
P_\mu\left(q^{-n};\frac{ab}{q},\frac{a}{b},\frac{ac}{q};q\right),
\label{cdqHbqJfdl}\\
&&\hspace{-3.8cm}P_\mu(q^{-n};a,b,c;q)=\frac{(qab)^{\frac12n}}{(qa,qc;q)_n}p_n\left[q^{\mu+\frac12}\sqrt{ab},\sqrt{qab},\sqrt{\frac{qa}{b}},\frac{q^\frac12 c}{\sqrt{ab}}\bigg|q\right].
\label{cdqHbqJfdl2}
\end{eqnarray}
\end{thm}
\begin{proof}
The duality relation \eqref{cdqHbqJfdl} follows by comparing
the terminating basic hypergeometric 
${}_3\phi_2$ representation of the continuous dual $q$-Hahn polynomials \eqref{cdqH:def1} with $z=q^{-\mu}/a$ against the nonterminating basic hypergeometric ${}_3\phi_2$ representation of the big $q$-Jacobi function \eqref{bqJf1} (which is terminating because of the choice $x=q^{-n}$). One obtains \eqref{cdqHbqJfdl2} by replacing 
\[
(a,b,c)\mapsto\left(\sqrt{qab},\sqrt{\frac{qa}{b}},\frac{q^\frac12c}{\sqrt{ab}}\right),
\]
in \eqref{cdqHbqJfdl} and solving
for the big $q$-Jacobi function.
This completes the proof.
\end{proof}

\subsubsection{Duality relations for the $q$ and $q^{-1}$-Al-Salam--Chihara polynomials\label{dlqJ}}

\noindent As well there exist duality relations between the Al-Salam--Chihara and $q^{-1}$-Al-Salam--Chihara polynomials
with the little $q$-Jacobi polynomials. 

\begin{thm}
\label{thm312}
Let $q\in\CCddag$, $m,n\in\mathbb N_0$, $a,b\in\CCast$. Then, the duality relations between the Al-Salam--Chihara and the little $q$-Jacobi polynomials follow:
\begin{eqnarray}
&&\hspace{-0.0cm}
Q_n\left[\frac{q^{-m}}{a};a,b|q\right]
=Q_n\left[q^{m}a;a,b|q\right]
=Q_n\left(\tfrac12\left(q^{m}a+\frac{q^{-m}}{a}\right);a,b|q\right)\nonumber\\
&&\hspace{2.9cm}=q^{\binom{m}{2}}\frac{(-ab)^m}{a^n}
\frac{(ab;q)_n\left(\frac{qa}{b};q\right)_m}{(ab;q)_m}\,p_m\left(\frac{q^{-n}}{ab};\frac{a}{b},\frac{ab}{q};q\right),
\label{dASClqJa}\\
&&\hspace{-0.0cm}
Q_n\left[q^{m}b;a,b|q\right]
=Q_n\left[\frac{q^{-m}}{b};a,b|q\right]
=Q_n\left(\tfrac12\left(q^{m}b+\frac{q^{-m}}{b}\right);a,b|q\right)\nonumber\\
&&\hspace{2.9cm}=q^{\binom{m}{2}}\frac{(-ab)^m}{b^n}
\frac{(ab;q)_n\left(\frac{qb}{a};q\right)_m}{(ab;q)_m}\,p_m\left(\frac{q^{-n}}{ab};\frac{b}{a},\frac{ab}{q};q\right),\\
&&\hspace{-0.0cm}
Q_n\left[\frac{q^{m}}{a};a,b|q^{-1}\right]
=Q_n[q^{-m}a;a,b|q^{-1}]
=Q_n\left(\tfrac12\left(\frac{q^{m}}{a}+q^{-m}a\right);a,b|q^{-1}\right)\nonumber\\
&&\hspace{3.05cm}
=q^{-\binom{n}{2}-\binom{m}{2}}(-b)^n\left(-\frac{a}{qb}\right)^m
\frac{\left(\frac{1}{ab};q\right)_n\left(\frac{qb}{a};q\right)_m}
{\left(\frac{1}{ab};q\right)_m}\,
p_m\left(q^n;\frac{b}{a},\frac{1}{qab};q\right),
\label{dqiASClqJa}\\
&&\hspace{-0.0cm}
Q_n\left[\frac{q^{m}}{b};a,b|q^{-1}\right]
=Q_n[q^{-m}b;a,b|q^{-1}]
=Q_n\left(\tfrac12\left(\frac{q^{m}}{b}+q^{-m}b\right);a,b|q^{-1}\right)\nonumber\\
&&\hspace{3.05cm}
=q^{-\binom{n}{2}-\binom{m}{2}}(-a)^n\left(-\frac{b}{qa}\right)^m
\frac{\left(\frac{1}{ab};q\right)_n\left(\frac{qa}{b};q\right)_m}
{\left(\frac{1}{ab};q\right)_m}\,
p_m\left(q^n;\frac{a}{b},\frac{1}{qab};q\right).
\end{eqnarray}
\end{thm}
\begin{proof}
See \cite[(75)]{KoornwinderMazzocco2018}, \cite[p.~8]{Groenevelt2021}.
\end{proof}

\noindent Now we present a theorem which gives the duality relation between
the Al-Salam--Chihara polynomials with the little $q$-Jacobi function.

\begin{thm}
\label{thm3.35}

Let $n\in\mathbb N_0$, $\mu\in \mathbb C$, $q\in\CCdag$, $a,b\in\CCast$. 
Then, one has the following duality relations for the
Al-Salam--Chihara polynomials with the little $q$-Jacobi function:
\begin{eqnarray}
&&\hspace{-1.9cm}Q_n\left[\frac{q^{-\mu}}{a};a,b|q\right]
=\frac{(ab;q)_n}{a^n}
\frac{(\frac{qa}{b},\frac{q}{ab};q)_\infty}{(q^{\mu+1}\frac{a}{b},q^{1-\mu}\frac{1}{ab};q)_\infty}
p_\mu\left(\frac{q^{-n}}{ab};\frac{a}{b},\frac{ab}{q};q\right),
\label{ASClqJfdl}\\
&&\hspace{-1.9cm}p_\mu\left(\frac{q^{-1-n}}{b};a,b;q\right)=\frac{(qab)^{\frac12n}}{(qb;q)_n}
\frac{(q^{\mu+1}a,\frac{q^{-\mu}}{b};q)_\infty}{(qa,\frac{1}{b};q)_\infty}
Q_n\left[q^{\mu+\frac12}\sqrt{ab};\sqrt{qab},\sqrt{\frac{qb}{a}}\bigg|q\right].
\label{ASClqJfdl2}
\end{eqnarray}
\end{thm}
\begin{proof}
To derive the duality relation \eqref{ASClqJfdl}, use 
the ${}_2\phi_1$ representation of the Al-Salam--Chihara 
polynomials 
\eqref{ASC:def4}
with $a$ and $b$ 
interchanged due to their symmetry.
Then, the following values for the ${}_2\phi_1(a,b;c;q,z)$ are considered: 
$(a,b,c,z)=(q^{-n},az,q^{1-n}z/b,q/(bz))$. Use 
\cite[(III.2)]{GaspRah} 
which produces a nonterminating representation of the
Al-Salam--Chihara polynomials.
Then replacing $z=q^{-\mu}/a$,
the duality relation then follows by comparing
the resulting expression 
against the nonterminating basic hypergeometric ${}_2\phi_1$ representation of the little $q$-Jacobi function \eqref{lqJf1}.
To obtain \eqref{ASClqJfdl2}, one must replace $(a,b)\mapsto(\sqrt{qab},\sqrt{\frac{qb}{a}})$ in
\eqref{ASClqJfdl} and solve for the little $q$-Jacobi function. 
This completes the proof.
\end{proof}

\subsubsection{Duality relations for the continuous big $q$ and big $q^{-1}$-Hermite polynomials\label{dualBessel}}

\noindent Now we present a theorem which gives the duality relation between the continuous big $q$-Hermite polynomials with the $q^{-1}$-Bessel function.

\begin{thm}
\label{thm3.36}
Let $n\in\mathbb N_0$, $\mu \in \mathbb C$, 
$q\in\CCddag$, $a,b,c\in\CCast$. 
Then, one has the following duality relations 
for the continuous big $q$-Hermite polynomials with the $q^{-1}$-Bessel functions:
\begin{eqnarray}
&&\hspace{-5cm}H_n\left[q^\mu a;a|q\right]=q^{2\binom{\mu}{2}}a^{-n}(qa^2)^\mu
\,y_\mu\left(q^{-n};-\frac{1}{a^2};q^{-1}\right),
\label{dHnymu}\\
&&\hspace{-5cm}y_\mu(q^{-n};a;q^{-1})=q^{-2\binom{\mu}{2}}(-a)^{-\frac12 n}\left(-\frac{a}{q}\right)^\mu
H_n\left[\frac{q^{\mu}}{\sqrt{-a}};\frac{1}{\sqrt{-a}}|q\right],
\label{dHnymu2}
\end{eqnarray}
where in \eqref{dHnymu2}, the principal branch of the square root is taken.
\end{thm}
\begin{proof}
The duality relation between the continuous big $q$-Hermite polynomials and the 
$q^{-1}$-Bessel function can be found for instance by comparing the basic hypergeometric representations \eqref{cbqH:def1}, 
\eqref{qiBf}.
In order to obtain \eqref{dHnymu2}, one must invert \eqref{dHnymu}.
This completes the proof.
\end{proof}

\noindent 
The duality relation between the continuous $q^{-1}$-Hermite polynomials and the $q$-Bessel polynomials is given as follows.
\begin{thm}
Let $n,m\in\N_0$, $q\in\CCddag$, $a\in\CCast$
\begin{eqnarray}
&&\hspace{-6cm}H_n[q^{-m}a;a|q^{-1}]=q^{-\binom{m}{2}}a^{-n}\left(\frac{a^2}{q}\right)^m y_m(q^n;-\tfrac{1}{a^2};q),
\label{dHqinym}\\
&&\hspace{-6cm}y_m(q^n;a;q)=q^{2\binom{m}{2}}\frac{(-qa)^m}{(-a)^{\frac12n}}
H_n\left[\frac{q^{-m}}{\sqrt{-a}};\frac{1}{\sqrt{-a}}|q^{-1}\right],
\label{dHqinym2}
\end{eqnarray}
where in \eqref{dHqinym2}, the principal branch of the square root is taken.
\end{thm}

\begin{proof}
For instance, comparing the terminating basic hypergeometric series representations given by \eqref{cbqiH:2} and \eqref{qB:7} completes the proof. In order to obtain \eqref{dHqinym2}, one must invert \eqref{dHqinym}.
\end{proof}

\section{Orthogonality relations for the $q$ and $q^{-1}$-symmetric and dual families\label{sec:4O}}

There is an interesting history of the duality relations described in the previous subsection and the corresponding orthogonality relations which we present below. Some important literature corresponding to these include 
Rosengren (2000) \cite{RosengrenCONM2000} 
and also 
\cite{AtakishiyevKlimyk2006,Groenevelt2021,KoornwinderMazzocco2018} 
where dual orthogonality relations are discussed.
For instance, the dual orthogonality relation 
between little $q$-Jacobi polynomials and $q^{-1}$-Al-Salam-Chihara polynomials correspond to the (visibly 
dual) relations \cite[(4.5), (4.6)]{RosengrenCONM2000}. 
This dual orthogonality relation is also discussed by 
Groenevelt (2004), see 
\cite[see especially Remark 3.1]{Groenevelt2004}.
A discrete orthogonality of continuous 
dual $q^{-1}$-Hahn polynomials was found by Rosengren in 
\cite[(4.16)]{RosengrenCONM2000}. 
To see what it is explicitly, one needs to 
use \cite[(4.16), (4.10), Proposition 4.3]{RosengrenCONM2000}. 
Although it is not explained in \cite{RosengrenCONM2000}, 
this orthogonality is dual to the big $q$-Jacobi 
polynomials and can also be obtained from the 
orthogonality of $q$-Racah polynomials by letting 
$N\to\infty$. The same orthogonality relation was 
written down in a readable way by Atakishiyev and 
Klimyk (2004) \cite[Section 8]{AtakishiyevKlimyk2004}.
A related dual orthogonality relation between the 
big $q$-Jacobi functions and a system containing 
dual $q^{-1}$-Hahn polynomials and $q$-Bessel 
functions was discussed in Koelink and Stokman (2001) 
\cite{KoelinkStokman2001}. 
Note that the existence of a dual orthogonality 
relation between two families of orthogonal polynomials 
implies that duality relations exist between those two 
families. Examples of these duality relations are given 
in Theorems \ref{thm311} and \ref{thm312}.

\medskip
\noindent For a nice discussion of the existence 
of dual orthogonality in the general setting, one
might study in detail 
\cite[Chapter 2]{Ismail:2009:CQO}.
As pointed out by Ismail in \cite[Section 2.5]{Ismail:2009:CQO}, the orthogonality and the existence of dual orthogonality (duality) in \cite[(2.5.1), (2.5.3)]{Ismail:2009:CQO}, is guaranteed in the finite case.
This fact may have been discovered much earlier by Chebyshev \cite{Ismailpriv2023}.
The existence of dual orthogonality in the infinite case (the case where there are an infinite number of zeros) is demonstrated by Markov's theorem \cite[Theorem 2.6.1]{Ismail:2009:CQO}
which requires the uniqueness 
of the measure of orthogonality.
For the $q^{-1}$ polynomials,
the measure of orthogonality is 
not unique and one has an indeterminate moment problem \cite[Chapter 21]{Ismail:2009:CQO}. For the $q^{-1}$-Hermite polynomials, this
is completely described in \cite[Chapter 21]{Ismail:2009:CQO} (see also \cite{IsmailMasson1994}). For more advanced cases including continuous big $q^{-1}$-Hermite, $q^{-1}$-Al-Salam--Chihara and continuous dual $q^{-1}$-Hahn 
polynomials, 
the indeterminate moment problem is yet unsolved.


\medskip

Let $n,n'\in\N_0$, ${\bf a}$ be a set of parameters. 
Consider a sequence of eigenfunctions $\psi_n$ orthogonal with respect to a continuous or discrete weight function ${\sf w}$.
In the case of continuous orthogonality
\begin{eqnarray}
\int_a^b \psi_{n}(x;{\bf a}) \psi_{n'}(x;{\bf a}) \,{\sf w}(x;{\bf a})\,\dd x=h_n({\bf a}) \delta_{n,n'}.
\end{eqnarray}
If there is completeness of the eigenfunctions $\psi_n$ in some space of functions, then one often has a corresponding closure relation (see e.g., \cite{CohlCostasSantos22c}, \cite[Theorem 2.1]{IsmailZhangZhou2022}).
then the closure relation is given by
\begin{eqnarray}
\sum_{n=0}^\infty \frac{1}{h_n({\bf a})}\psi_n(x;{\bf a})\psi_{n}(y;{\bf a})=\frac{\delta(x-y)}{{\sf w}(x;{\bf a})},
\label{closC}
\end{eqnarray}
where $\delta$ is the Dirac delta distribution.
In the case of an infinite discrete orthogonality
\begin{eqnarray}
\sum_{n=0}^\infty \psi_{n,m}({\bf a}) \psi_{n,m'}({\bf a}) \,{\sf w}_n({\bf a})=h_m({\bf a}) \delta_{m,m'},
\end{eqnarray}
then the closure relation is given by
\begin{eqnarray}
\sum_{m=0}^\infty \frac{1}{h_m({\bf a})}\psi_{m,n}({\bf a})\psi_{m,n'}({\bf a})=\frac{\delta_{n,n'}}{{\sf w}_{n}({\bf a})}.
\label{closD}
\end{eqnarray}

\subsection{The Askey--Wilson polynomials}
\medskip
\noindent 
Now we describe the continuous orthogonality relation for the Askey--Wilson polynomials which were introduced in \S\ref{sec:2.2.1b}.
Let $q\in\CCdag$, ${\mathbf a}:=\{a,b,c,d\}$,
$a,b,c,d\in\CCast$ such that $|a|,|b|,|c|,|d|<1$. Then
the Askey--Wilson polynomials are a family of orthogonal polynomials symmetric in four parameters.
These polynomials have a continuous orthogonality relation and are orthogonal
on $x=\cos\theta\in(-1,1)$
with respect to the 
weight function
\begin{eqnarray}
&&\hspace{-5.1cm}
w_q(\cos\theta;{\bf a}):=
\frac{(\expe^{\pm 2i\theta};q)_\infty}
{({\bf a}\expe^{\pm i\theta};q)_\infty}=
\label{AWw}
\frac{(\pm\expe^{\pm i\theta},\pm q^\frac12 \expe^{\pm i\theta};q)_\infty}
{({\bf a}\expe^{\pm i\theta};q)_\infty},
\end{eqnarray}
where the second equality is due to
\eqref{sq}.
\begin{thm}
Let $m,n\in\N_0$, $q\in\CCdag$, $a,b,c,d$ are real, or occur in complex conjugate pairs if complex and $\max(|a|,|b|,|c|,|d|)<1$. Then 
the Askey--Wilson polynomials $p_n(x;{\bf a}|q)$ satisfy the following orthogonality relation 
\cite[(14.1.2)]{Koekoeketal}
\begin{equation}
\int_0^\pi p_m(x;{\bf a}|q)p_n(x;{\bf a}|q)w_q(x;{\bf a})\,
{\mathrm d}\theta=h_n({\bf a};q)\delta_{m,n},
\label{AWO}
\end{equation}
where
\begin{equation}
h_n({\bf a};q):=\frac{2\pi(q^{n-1}abcd;q)_n(q^{2n}abcd;q)_\infty}{(q^{n+1},q^nab,q^nac
,q^nad,q^nbc,q^nbd,q^ncd;q)_\infty}.
\label{AWn}
\end{equation}
\end{thm}

\begin{proof}
See proof of \cite[(14.1.2)]{Koekoeketal}.
\end{proof}

\noindent 
Orthogonality relations for the Askey--Wilson polynomials 
the complex plane are studied systematically in 
\cite[\S3]{MR2832754}.

\subsection{The continuous dual $q$ and $q^{-1}$-Hahn polynomials}

\medskip
\noindent 
Now we describe the continuous orthogonality relation for the continuous dual $q$-Hahn  polynomials which were introduced in \S\ref{sec:3.3}.
Let $q\in\CCdag$,
${\mathbf a}:=\{a,b,c\}$,
$a,b,c\in\CCast$ such that $|a|,|b|,|c|<1$.
Then, there exists a set of basic hypergeometric orthogonal polynomials which are symmetric in the three parameters, $a$, $b$, $c$. These polynomials are referred to as the continuous dual $q$-Hahn polynomials $p_n(x;{\bf a}|q)$. They satisfy a continuous orthogonality relation and are orthogonal
on $x=\cos\theta\in(-1,1)$
with respect to the 
weight function
\begin{eqnarray}
&&\hspace{-10.5cm}
w_q(\cos\theta;{\bf a}):=
\frac{(\expe^{\pm 2i\theta};q)_\infty}
{({\bf a}\expe^{\pm i\theta};q)_\infty},
\label{cdqHw}
\end{eqnarray}
whose orthogonality relation follows.
\begin{thm}
Let $m,n\in\N_0$, $q\in\CCdag$, and $a,b,c$ are real or one is real and the other two are complex conjugates, and
$\max(|a|,|b|,|c|)<1$. Then 
the continuous dual $q$-Hahn polynomials satisfy the following orthogonality relation 
\cite[(14.3.2)]{Koekoeketal}
\begin{equation}
\int_0^\pi p_m(x;{\bf a}|q)p_n(x;{\bf a}|q)w_q(x;{\bf a})\,
{\mathrm d}\theta=h_n({\bf a};q)\delta_{m,n},
\label{cdqHO}
\end{equation}
where
\begin{equation}
h_n({\bf a};q):=\frac{2\pi}{(q^{n+1},q^nab,q^nac
,q^nbc;q)_\infty}.
\label{cdqHn}
\end{equation}
\end{thm}

\medskip

\begin{proof}
See proof of \cite[(14.3.2)]{Koekoeketal}.
\end{proof}


\medskip
\noindent 
Now we describe known orthogonality relations for the infinite family given by the continuous dual $q^{-1}$-Hahn polynomials which were introduced in \S\ref{sec:3.3}.
Continuous orthogonality for the continuous dual $q^{-1}$-Hahn polynomials was worked out by Ismail and collaborators in \cite{IsmailZhang2022} where they showed that these polynomials form a symmetric (in three parameters) infinite family of orthogonal
polynomials with orthogonality relation given as follows.

\begin{thm}
Let $n,m\in\N_0$, $q\in\CCdag$, $x=\frac12(z+z^{-1})\in \CCast$, $a,b,c, z\in\CCast$. Then
\begin{eqnarray}
\int_{0}^{i\infty}p_n(\tfrac12(z+z^{-1});a,b,c|q^{-1})
p_{m}(\tfrac12(z+z^{-1});a,b,c|q^{-1})
w(z;a,b,c|q)\,\dd z=h_n(a,b,c|q)\delta_{m,n},
\label{cdqiHO}
\end{eqnarray}
where
\begin{equation}
w(z;a,b,c|q):=(z-z^{-1})\frac{(qaz^\pm,qbz^\pm,qcz^\pm;q)_\infty}{\vartheta(z^2;q)},
\label{cdqiHw}
\end{equation}
and
\begin{equation}
h_n(a,b,c|q):=q^{-4\binom{n}{2}}
(q,qab,qac,qbc;q)_\infty(q,\tfrac{1}{ab},\tfrac{1}{ac},\tfrac{1}{bc};q)_n
\left(\frac{a^2b^2c^2}{q}\right)^n\,\log q.
\label{cdqiHn}
\end{equation}
\label{cdqiHiOt}
\end{thm}

\begin{proof}
See \cite{IsmailZhang2022}, 
and we have rewritten their result 
by setting
\begin{equation}
(z,t_1,t_2,t_3)\mapsto(iz,iqa,iqb,iqc),
\end{equation}
and renormalizing their ${}_3\phi_2$ polynomials properly.
\end{proof}

\noindent 
There is an infinite discrete orthogonality relation for
the continuous dual $q^{-1}$-Hahn polynomials where
$p_n[z;a,b,c|q^{-1}]=p_n(\frac12(z+z^{-1});a,b,c|q^{-1})$, (see cf.~\cite[(4.30)]{AtakishiyevKlimyk2006} and Corollary \ref{cor:3.4}).
In order to investigate the region of convergence for the parameters involved in the following infinite discrete orthogonality relation, we will need the following lemma.

\begin{lem}
Let $n\in\N_0$, $q\in\CCdag$, $a,b,c\in\CCast$. Then one has 
as $m\to\infty$, that
\begin{eqnarray}
&&\hspace{-2cm}p_n[q^{-m}a;a,b,c|q^{-1}]\sim
q^{-nm}a^n.
\end{eqnarray}
\label{lem346}
\end{lem}

\begin{proof}
Start with \eqref{cdqiH:4}
and replace $z=q^{-m}a$, then 
one has as $m\to\infty$,
\[
p_n[q^{-m}a;a,b,c|q^{-1}]\sim q^{nm-\binom{n}{2}}(-a^2b)^n(\tfrac{1}{ab};q)_n\qhyp21{q^{-n},0}{\frac{1}{ab}}{q,q},
\]
which can be summed using \eqref{limqChu}.
\end{proof}

\begin{thm}
Let $n,n'\in\N_0$, $q\in\CCdag$, $a,b,c\in\CCast$. Then
\begin{eqnarray}
&&\hspace{-0.9cm}\sum_{m=0}^\infty
q^{\binom{m}{2}}(-qbc)^m
\frac{(\frac{q}{a^2};q)_{2m}(\frac{1}{ab},\frac{1}{ac},
\frac{1}{a^2};q)_m}{(\frac{1}{a^2};q)_{2m}(q,\frac{qb}{a},\frac{qc}{a}
;q)_m}
\,p_n\left[q^{-m}a;a,b,c|q^{-1}\right]
p_{n'}\left[q^{-m}a;a,b,c|q^{-1}\right]\nonumber\\
&&\hspace{4.8cm}=
q^{-4\binom{n}{2}}\left(\frac{a^2b^2c^2}{q}\right)^n \frac{(\frac{q}{a^2},qbc;q)_\infty(q,\frac{1}{ab},\frac{1}{ac},\frac{1}{bc};q)_n}{(\frac{qb}{a},\frac{qc}{a};q)_\infty}\delta_{n,n'}.
\label{AKporth}
\end{eqnarray}
\end{thm}

\begin{proof}
The orthogonality relation \eqref{AKporth}
is obtained from \cite[(4.30)]{AtakishiyevKlimyk2006}.
Consider $n=n'$. The left-hand side of \eqref{AKporth} as ${\sf R}_{n}(a,b,c;q)$, then 
using Lemma \ref{lem346}, we have 
as $m\to\infty$, the summand behaves like 
\begin{equation}
{\sf r}_{m,n}(a,b,c;q)\sim a^{2n}
\frac{(\frac{1}{ab},\frac{1}{ac},\frac{q}{a^2};q)_\infty}{(q,\frac{qb}{a},\frac{qc}{a};q)_\infty}
q^{\binom{m}{2}}
\left(-\frac{bc}{q^{2n-1}}\right)^m,
\end{equation}
where $\sum_m{\sf r}_{m,n}(a,b,c;q)=
{\sf R}_{n}(a,b,c;q)$. 
Hence, 
by using the direct comparison test 
${\sf R}_{n}(a,b,c;q)$ converges 
since the infinite series associated with \eqref{AKporth} is convergent. Therefore, it is convergent for all values of $a,b,c$ and $n\in\N_0$. This completes the proof.
\end{proof}

There is another orthogonality relation for the continuous dual $q^{-1}$-Hahn polynomials which is obtained through duality from the $q$-integral orthogonality of the big $q$-Jacobi polynomials \eqref{bqJO}. In order to study the convergence properties of these polynomials, we will need the following lemma.
\begin{lem}
Let $n\in\N_0$, $q\in\CCdag$, $a,b,c\in\CCast$. Then
as $n\to\infty$, one has 
\begin{equation}
p_n(x;a,b,c|q^{-1})\sim q^{-2\binom{n}{2}}
(abc)^n(\tfrac{1}{ab},\tfrac{1}{ac};q)_\infty
\qhyp22{\frac{z^\pm}{a}}{\frac{1}{ab},\frac{1}{ac}}{q,\frac{1}{bc}}.
\label{pn-large-n}
\end{equation}
\label{lem46}
\end{lem}

\begin{proof}
This proof is due to Xiang-Sheng Wang.
Start by using the generating function of continuous dual $q^{-1}$-Hahn polynomials ${\sf G}(t;a,b,c|q)$ given by 
\eqref{cdqinHgf-1} below.
Since
\begin{equation}
{\sf G}_1:= {\sf G}_1(a,b,c|q):= \lim_{t\to(abc)^{-1}}(1-abct){\sf G}(t;a,b,c|q)={1\over(q;q)_\infty}{}_2\phi_2\left(\begin{array}{c}
 \frac{z^{\pm}}{a}\\ \frac{1}{ab},\frac{1}{ac}
 \end{array};q,\frac{1}{bc}\right),
\end{equation}
we obtain from Darboux's method that
\begin{equation}
 \lim_{n\to\infty}{q^{2\binom{n}{2}}p_n(x;a,b,c|q^{-1})\over(abc)^n(\frac{1}{ab},\frac{1}{ac};q)_n}=(q;q)_\infty {\sf G}_1
 ={}_2\phi_2\left(\begin{array}{c}
 \frac{z^{\pm}}{a}\\ \frac{1}{ab},\frac{1}{ac}
 \end{array};q,\frac{1}{bc}\right).
\end{equation}
This proves \eqref{pn-large-n}.
\end{proof}

In order to determine the required constraints on the parameters for the following infinite discrete orthogonality relation for the continuous dual $q^{-1}$-Hahn polynomials, we will need the following asymptotic results as $n\to\infty$ which both follow from the above Lemma.

\begin{lem}
Let $n\in\N_0$, $q\in\CCdag$, $a,b,c\in\CCast$. Then
as $n\to\infty$, one has 
\begin{eqnarray}
&&\hspace{-2cm}p_n[q^{-m};a,b,c|q^{-1}]
\sim q^{-2\binom{n}{2}}(abc)^n(\tfrac{1}{ab},\tfrac{1}{ac};q)_\infty \qhyp22{q^{-m},\frac{q^m}{a^2}}{\frac{1}{ab},\frac{1}{ac}}{q,\frac{1}{bc}},
\label{pn-large-n1}\\
&&\hspace{-2cm}p_n[q^{-m};a,\tfrac{1}{qb},\tfrac{1}{qc}|q^{-1}]
\sim q^{-2\binom{n}{2}}\left(\frac{a}{q^2bc}\right)^n(\tfrac{qb}{a},\tfrac{qc}{a};q)_\infty \qhyp22{q^{-m},\frac{q^m}{a^2}}{\frac{qb}{a},\frac{qc}{a}}{q,q^2bc}.
\label{pn-large-n2}
\end{eqnarray}
\label{lem47}
\end{lem}

\begin{proof}
Start with Lemma \ref{lem46}, replace $z=q^{-m}a$ for \eqref{pn-large-n1} and then replace $(b,c)\mapsto(\tfrac{1}{qb},\frac{1}{qc})$ in order to obtain \eqref{pn-large-n2}.
\end{proof}
It is given in the following theorem. Note however that the orthogonality relation involves two different families of continuous dual $q^{-1}$-Hahn polynomials.

\begin{thm}
Let $m,m'\in\N_0$, $q\in\CCdag$, 
Let $a,b,c\in\CCast$. Then
\begin{eqnarray}
&&\hspace{1.0cm}
\frac{(\frac{1}{qbc};q)_\infty}{(\frac{1}{ab},\frac{1}{ac};q)_\infty}\left(\frac{b}{a}\right)^{2m}\frac{(\frac{1}{ab},\frac{1}{ab};q)_m}{(\frac{qb}{a},\frac{qb}{a};q)_m}
\sum_{n=0}^\infty
\frac{q^{4\binom{n}{2}}\left(\frac{q}{a^2b^2c^2}\right)^n}{(q,\frac{1}{ab},\frac{1}{ac},\frac{1}{bc};q)_n}
p_n\left[\frac{q^m}{a};a,b,c\bigg|q^{-1}\right]\!
p_n\left[\frac{q^{m'}}{a};a,b,c\bigg|q^{-1}\right]
\nonumber\\
&&\hspace{0.2cm}+
\frac{(qbc;q)_\infty}{(\frac{qb}{a},\frac{qc}{a};q)_\infty}\left(\frac{1}{qac}\right)^{2m}\frac{(\frac{qc}{a},\frac{qc}{a};q)_m}{(\frac{1}{ac},\frac{1}{ac};q)_m}
\sum_{n=0}^\infty
\frac{q^{4\binom{n}{2}}\left(\frac{q^5b^2c^2}{a^2}\right)^n}{(q,\frac{qb}{a},\frac{qc}{a},q^2bc;q)_n}
p_n\left[\frac{q^m}{a};a,\frac{1}{qb},\frac{1}{qc}\bigg|q^{-1}\right]\!
p_n\left[\frac{q^{m'}}{a};a,\frac{1}{qb},\frac{1}{qc}\bigg|q^{-1}\right]
\nonumber\\
&&\hspace{1.5cm}=q^{-\binom{m}{2}}\left(-\frac{b}{qa^2c}\right)^m
\frac{(\frac{q}{a^2},qbc,\frac{1}{qbc};q)_\infty}{(\frac{1}{ab},\frac{1}{ac},\frac{qb}{a},\frac{qc}{a};q)_\infty}
\frac{(\frac{1}{a^2};q)_{2m}(q,\frac{1}{ab},\frac{qc}{a};q)_m}{(\frac{q}{a^2};q)_{2m}(\frac{1}{a^2},\frac{1}{ac},\frac{qb}{a};q)_m}\delta_{m,m'}.
\label{DcdqiHO}
\end{eqnarray}
\end{thm}

\begin{proof}
Start with the $q$-integral orthogonality of the big $q$-Jacobi polynomials \eqref{bqJO} and express the $q$-integral as two infinite series using \eqref{qint}.
One may express the big $q$-Jacobi polynomial with argument 
$q^{n+1}a$ as a continuous dual $q^{-1}$-Hahn polynomial 
using \eqref{dualA}. For the big $q$-Jacobi polynomial with argument $q^{n+1}c$, one may express it as a continuous dual $q^{-1}$-Hahn polynomial using the duality relation \eqref{dualC} which takes advantage of the symmetry of big $q$-Jacobi polynomials given by \eqref{bqJsym}. Then one may insert these two duality relations and simplify. Finally after making the global replacement 
\begin{equation}
(a,b,c)\mapsto \left(\frac{b}{a},\frac{1}{qab},\frac{1}{qac}\right),
\end{equation}
and simplifying, this produces the orthogonality relation. In order to obtain the range of parameters for convergence, consider $m=m'$. Define the two terms on the left-hand side of \eqref{DcdqiHO} as ${\sf C}_{m}(a,b,c;q)$ and 
${\sf D}_{m}(a,b,c;q)$ then 
using Lemma \ref{lem47}, we have 
as $n\to\infty$, both summands behave like 
\begin{equation}
{\sf c}_{n,m}(a,b,c;q)\sim 
\frac{(\frac{1}{ab},\frac{1}{ac};q)_\infty}{(q,\frac{1}{bc};q)_\infty}
\left[\qhyp22{q^{-m},\frac{q^m}{a^2}}{\frac{1}{ab},\frac{1}{ac}}{q,\frac{1}{bc}}\right]^2
q^n,
\end{equation}
and
\begin{equation}
{\sf d}_{n,m}(a,b,c;q)\sim 
\frac{(\frac{qb}{a},\frac{qc}{a};q)_\infty}{(q,q^2bc;q)_\infty}
\left[\qhyp22{q^{-m},\frac{q^m}{a^2}}{\frac{qb}{a},\frac{qc}{a}}{q,q^2bc}\right]^2
q^n,
\end{equation}
where $\sum_n{\sf c}_{n,m}(a,b,c;q)=
{\sf C}_{m}(a,b;q)$, $\sum_n{\sf d}_{n,m}(a,b,c;q)=
{\sf D}_{m}(a,b,c;q)$. 
Hence, 
by using the direct comparison test 
${\sf C}_{m}(a,b,c;q)$,
${\sf D}_{m}(a,b,c;q)$,
converges if $|q|<1$
since both infinite series associated with \eqref{DcdqiHO} are convergent. 
This is because both asymptotic infinite series are nonterminating ${}_4\phi_3$'s with vanishing numerator parameters and argument $q$.
This completes the proof.
\end{proof}

\begin{rem}
Note that the second term on the left-hand side of \eqref{DcdqiHO} can be obtained from the first term by 
replacing
\begin{equation}
(b,c)\mapsto\left(\frac{1}{qc},\frac{1}{qb}\right),
\end{equation}
and using the symmetry of the continuous dual $q^{-1}$-Hahn polynomials in the $b,c$ parameters.
\end{rem}

\subsection{The $q$ and $q^{-1}$-Al-Salam--Chihara polynomials}
\noindent 
Now we describe the continuous orthogonality relation for the Al-Salam--Chihara  polynomials which were introduced in \S\ref{sec:3.4}.
Let ${\mathbf a}:=\{a,b\}$,
$a,b\in\CCast$ such that $|a|,|b|<1$.
Then Al-Salam--Chihara polynomials $p_n(x;{\bf a}|q)$ are orthogonal
on $x=\cos\theta\in(-1,1)$
with respect to the 
weight function
\begin{eqnarray}
&&\hspace{-10.2cm}
w_q(\cos\theta;{\bf a}):=
\frac{(\expe^{\pm 2i\theta};q)_\infty}
{({\bf a}\expe^{\pm i\theta};q)_\infty}.
\label{ASCw}
\end{eqnarray}
\begin{thm}
Let $m,n\in\N_0$, $q\in\CCdag$, if $a,b$ are real or complex conjugates and $\max(|a|,|b|)<1$. Then the Al-Salam--Chihara polynomials satisfy the following continuous orthogonality relation 
\cite[(14.8.2)]{Koekoeketal}
\begin{equation}
\int_0^\pi Q_m(x;{\bf a}|q)Q_n(x;{\bf a}|q)w_q(x;{\bf a})\,
{\mathrm d}\theta=h_n({\bf a};q)\delta_{m,n},
\label{ASCO}
\end{equation}
where
\begin{equation}
h_n({\bf a};q):=\frac{2\pi}{(q^{n+1},q^nab;q)_\infty}.
\label{ASCn}
\end{equation}
\end{thm}

\begin{proof}
See proof of \cite[(14.8.2)]{Koekoeketal}.
\end{proof}

\medskip


\medskip
\noindent  Now we describe known orthogonality relations for the infinite family of orthogonal polynomials given by the $q^{-1}$-Al-Salam--Chihara polynomials which were introduced in \S\ref{sec:3.3}.
For the $q^{-1}$-Al-Salam--Chihara polynomials, they satisfy the following orthogonality relation.

\begin{cor}
Let $n,m\in\N_0$, $q\in\CCdag$, $x=\frac12(z+z^{-1})\in\CCast$, $a,b\in\CCast$. Then
\begin{eqnarray}
\int_{0}^{i\infty}Q_n(\tfrac12(z+z^{-1});a,b|q^{-1})
Q_{m}(\tfrac12(z+z^{-1});a,b|q^{-1})
w(z;a,b|q)\,\dd z=h_n(a,b|q)\delta_{m,n},
\label{qiASCO}
\end{eqnarray}
where
\begin{equation}
w(z;a,b|q):=(z-z^{-1})\frac{(qaz^\pm,qbz^\pm;q)_\infty}{\vartheta(z^2;q)},
\label{qiASCw}
\end{equation}
and
\begin{equation}
h_n(a,b|q):=q^{-2\binom{n}{2}}
(q,qab;q)_\infty(q,\tfrac{1}{ab};q)_n
\left(\frac{ab}{q}\right)^n\,\log q.
\label{qiASCn}
\end{equation}
\label{cqiASCOc}
\end{cor}

\begin{proof}
This orthogonality relation can be 
found by taking the limit as $d\to0$ in
Theorem \ref{cdqiHiOt}.
\end{proof}

\noindent The following infinite series discrete orthogonality relation for the $q^{-1}$-Al-Salam--Chihara polynomials was first described 
See also Groenevelt (2021) \cite[(3.4)]{Groenevelt2021}.

\medskip
\noindent In order to investigate the region of convergence for the parameters involved in the following infinite discrete orthogonality relation, we will need the following lemma.

\begin{lem}
Let $n, m\in\N_0$, $q\in\CCdag$, $a,b\in\CCast$. Then one has 
as $m\to\infty$, that
\begin{equation}
Q_n[q^{-m}a;a,b|q^{-1}]
\sim q^{-\binom{n}{2}}(-a)^n(q^{-m};q)_n
\sim q^{-nm}a^n.
\end{equation}
\label{lem410}
\end{lem}

\begin{proof}
Start with the ${}_2\phi_1$ terminating basic hypergeometric representation of the $q^{-1}$-Al-Salam--Chihara polynomials \eqref{qiASC:3}, setting $z\mapsto q^{-m}a$ produces
\begin{equation}
Q_n[q^{-m}a;a,b|q^{-1}]=
q^{-\genfrac{(}{)}{0pt}{}{n}{2}} (-a)^n 
\left(\frac{q^{m}}{a^2};q\right)_n\qhyp{2}{1}{q^{-n}, 
\frac{q^{-m}a}{b}}
{q^{1-n-m}a^2}{q,q^{1-m}ab}.
\end{equation}
Now as $m\to\infty$ the argument of the infinite $q$-shifted factorial which appears as a multiplicative constant vanishes, so this term goes to unity,
\begin{equation}
Q_n[q^{-m}a;a,b|q^{-1}]\sim 
q^{-\genfrac{(}{)}{0pt}{}{n}{2}} (-a)^n 
\qhyp{2}{1}{q^{-n}, 
\frac{q^{-m}a}{b}}
{q^{1-n-m}a^2}{q,q^{1-m}ab}.
\end{equation}
Now in the limit as $m\to\infty$, using \cite[(1.10.5)]{Koekoeketal}, the terminating basic hypergeometric function becomes 
\begin{equation}
Q_n[q^{-m}a;a,b|q^{-1}]\sim 
q^{-\genfrac{(}{)}{0pt}{}{n}{2}} (-a)^n 
\qhyp{1}{0}{q^{-n}
}
{-}{q,q^{n-m}}.
\end{equation}
One can now use the terminating $q$-binomial
theorem
\eqref{termqbinom}
 and one has
\begin{equation}
Q_n[q^{-m}a;a,b|q^{-1}]\sim
q^{-\genfrac{(}{)}{0pt}{}{n}{2}} (-a)^n 
(q^{-m};q)_n.
\end{equation}
One can now use 
\eqref{qPochiden2}
which in the limit as $m\to\infty$ the fraction goes to $(q;q)_m/(q;q)_{m-n}\to 1$, which completes the proof.
\end{proof}

\begin{thm}
Let $n,n'\in\N_0$, $q\in\CCdag$, $a,b\in\CCast$. Then
\begin{eqnarray}
&&\hspace{-0.4cm}\sum_{m=0}^\infty
q^{2\binom{m}{2}}\left(\frac{qb}{a}\right)^m
\frac{(\frac{q}{a^2};q)_{2m}(\frac{1}{a^2},
\frac{1}{ab};q)_m}{(\frac{1}{a^2};q)_{2m}(q,\frac{qb}{a};q)_m}
\,Q_n\left[q^{-m}a;a,b;q^{-1}\right]
Q_{n'}\left[q^{-m}a;a,b;q^{-1}\right]\nonumber\\
&&\hspace{6.5cm}=
q^{-2\binom{n}{2}}\left(\frac{ab}{q}\right)^n 
\frac{(\frac{q}{a^2};q)_\infty(q,\frac{1}{ab};q)_n}
{(\frac{qb}{a};q)_\infty}\delta_{n,n'},
\label{ASCorthi}
\end{eqnarray}
\end{thm}

\begin{proof}
This is obtained using cf.~\cite[(3.4)]{Groenevelt2021} (corrected 
so that the degrees of the $q^{-1}$-Al-Salam--Chihara polynomials 
are $n,n'$ respectively and the Kronecker delta symbol 
is $\delta_{n,n'}$).
Consider $n=n'$. The left-hand side of \eqref{ASCorthi} as ${\sf E}_{n}(a,b;q)$, then 
using Lemma \ref{lem410}, we have 
as $m\to\infty$, the summand behaves like 
\begin{equation}
{\sf e}_{m,n}(a,b;q)\sim 
a^{2n}\frac{(\frac{1}{ab},\frac{q}{a^2};q)_\infty}{(q,\frac{qb}{a};q)_\infty}
q^{2\binom{m}{2}}
\left(\frac{b}{q^{2n-1}a}\right)^m,
\end{equation}
where $\sum_m{\sf e}_{m,n}(a,b;q)=
{\sf E}_{n}(a,b;q)$.
Hence, 
by using the direct comparison test 
${\sf E}_{n}(a,b;q)$ converges 
since the infinite series associated with \eqref{ASCorthi} is convergent. Therefore, it is convergent for all values of $a,b$, $a\ne 0$ and $n\in\N_0$. This completes the proof.
\end{proof}

\noindent 
There is also the orthogonality of $q^{-1}$-Al-Salam--Chihara polynomials which comes from the standard orthogonality relation of little $q$-Jacobi polynomials \cite[(14.12.2)]{Koekoeketal} by using the duality relation \eqref{dqiASClqJa}. In order to study the convergence properties of these polynomials, we will need the following lemma.

\begin{lem}
Let $n\in\N_0$, $q\in\CCdag$, $x=\frac12(z+z^{-1})\in \CCast$, $a, b, z\in\CCast$. Then
as $n\to\infty$, one has 
\begin{equation}
Q_n(x;a,b|q^{-1})\sim q^{-\binom{n}{2}}
(-b)^n\frac{(\frac{z^\pm}{b};q)_\infty}{(\frac{a}{b};q)_\infty}.
\label{Qn-large-n}
\end{equation}
\label{lem413}
\end{lem}

\begin{proof}
This proof is due to Xiang-Sheng Wang.
Start by using the generating function of $q^{-1}$-Al-Salam--Chihara polynomials ${\sf H}(t;a,b|q)$ given by 
\eqref{qiASCgfX} below.
Since
\begin{equation}
{\sf H}_1:= {\sf H}_1(a,b|q):= \lim_{t\to-\frac{1}{b}}(1+bt){\sf H}(t;a,b|q)={1\over(q;q)_\infty}
\qhyp21{ \frac{z^{\pm}}{a}}{\frac{1}{ab}}{q,\frac{a}{b}}=\frac{(\frac{z^\pm}{b};q)_\infty}{(q,\frac{a^\pm }{b};q)_\infty},
\end{equation}
where we have used the $q$-Gauss sum \eqref{qGs},
which requires that $|a|<|b|$, and we obtain from Darboux's method that
\begin{equation}
 \lim_{n\to\infty}\frac{q^{\binom{n}{2}}Q_n(x;a,b|q^{-1})}{(-b)^n}=(q,\tfrac{1}{ab};q)_\infty {\sf H}_1(t;a,b|q)
 =\frac{(\frac{z^\pm}{b};q)_\infty}{(\frac{a}{b};q)_\infty}.
\end{equation}
This proves \eqref{Qn-large-n}.
\end{proof}

\noindent If we let $m\in\N_0$, then for the special value $z=q^{-m}a$, we can obtain from Lemma \ref{lem413}, the following asymptotic expression.

\begin{lem}
Let $n, m\in\N_0$, $q\in\CCdag$,  $a, b\in\CCast$. Then
as $n\to\infty$, one has 
\begin{eqnarray}
&&\hspace{-3.5cm}Q_n[q^{-m}a;a,b|q^{-1}]\sim q^{-\binom{m}{2}}\left(-\frac{a}{qb}\right)^m(\tfrac{1}{ab};q)_\infty
\frac{(\frac{qb}{a};q)_m}{(\frac{1}{ab};q)_m}
q^{-\binom{n}{2}}(-b)^n.
\end{eqnarray} 
\label{lem414}
\end{lem}

\begin{proof}
Start with Lemma \ref{lem413} and replace $z=q^{-m}a$, then after simplification, the result is obtained.
\end{proof}

\begin{thm}
Let $m,m'\in\N_0$, $q\in\CCdag$, $a,b\in\CCast$, $|qb|<|a|$. Then 
\begin{eqnarray}
&&\hspace{-1.7cm}\sum_{n=0}^\infty \frac{q^{2\binom{n}{2}}\left(\frac{q}{ab}\right)^n}
{(q,\frac{1}{ab};q)_n}Q_n[q^{-m}a;a,b|q^{-1}]Q_n[q^{-m'}a;a,b|q^{-1}]\nonumber\\
&&\hspace{3cm}=q^{-2\binom{m}{2}}\left(\frac{a}{qb}\right)^m
\frac{(\frac{q}{a^2};q)_\infty}{(\frac{qb}{a};q)_\infty}\frac{(\frac{1}{a^2};q)_{2m}(q,\frac{qb}{a};q)_m}{(\frac{q}{a^2};q)_{2m}(\frac{1}{a^2},\frac{1}{ab};q)_m}\delta_{m,m'}.
\label{idqiASCO}
\end{eqnarray}
\end{thm}

\begin{proof}
Starting with the orthogonality relation for little $q$-Jacobi polynomials given below as \eqref{lqJO}, applying the duality relation between little $q$-Jacobi polynomials and $q^{-1}$-Al-Salam--Chihara polynomials \eqref{dqiASClqJa} and simplifying provides the orthogonality relation.
Now consider $m=m'$. Define the left-hand side of \eqref{idqiASCO} as ${\sf U}_{m}(a,b;q)$, then 
using Lemma \ref{lem414}, we have 
as $n\to\infty$, the summand behaves like 
\begin{equation}
{\sf u}_{n,m}(a,b;q)\sim 
\frac{(\frac{1}{ab};q)_\infty}{(q;q)_\infty}q^{-2\binom{m}{2}}
\left(\frac{a}{qb}\right)^{2m}
\frac{(\frac{qb}{a},\frac{qb}{a};q)_m}{(\frac{1}{ab},\frac{1}{ab};q)_m}
\left(\frac{qb}{a}\right)^n
,
\end{equation}
where $\sum_n{\sf u}_{n,m}(a,b;q)=
{\sf U}_{m}(a,b;q)$.
Hence, 
by using the direct comparison test 
${\sf U}_{m}(a,b;q)$ converges if $|qb|<|a|$
since the infinite series associated with \eqref{idqiASCO} is convergent. 
This is because the asymptotic infinite series is a nonterminating ${}_2\phi_1$'s with vanishing numerator parameters and argument $q$.
There will be a singularity of the asymptotic series when
$\frac{1}{ab}\in\Omega_q$. 
This completes the proof.
\end{proof}
\subsection{The continuous big $q$ and big $q^{-1}$-Hermite polynomials}
\noindent
Now we describe the continuous orthogonality relation for the Al-Salam--Chihara  polynomials which were introduced in 
\S\ref{sec:3.5}.
Let $a\in\CCast$ such that $|a|<1$.
Then continuous big $q$-Hermite polynomials 
$H_n(x;a|q)$ are orthogonal
on $x=\cos\theta\in(-1,1)$
with respect to the 
weight function
\begin{eqnarray}
&&\hspace{-10.5cm}
w_q(\cos\theta;a):=
\frac{(\expe^{\pm 2i\theta};q)_\infty}
{(a\expe^{\pm i\theta};q)_\infty}.
\label{cbqHw}
\end{eqnarray}

\begin{thm}Let $m,n\in\N_0$, $q\in\CCdag$, $a\in(-1,1)$. Then, the continuous big $q$-Hermite polynomials satisfy the following orthogonality relation
\cite[(14.18.2)]{Koekoeketal}:
\begin{equation}
\int_0^\pi H_m(x;{a}|q)H_n(x;{a}|q)w_q(x;{a})\,
{\mathrm d}\theta=h_n(q)\delta_{m,n},
\label{cbqHO}
\end{equation}
where
\begin{equation}
h_n(q):=\frac{2\pi}{(q^{n+1};q)_\infty}.
\label{cbqHn}
\end{equation}
\end{thm}

\begin{proof}
See proof of \cite[(14.18.2)]{Koekoeketal}.
\end{proof}

\medskip

\noindent Now we describe known orthogonality relations for the for the infinite family of orthogonal polynomials given by continuous big $q^{-1}$-Hermite polynomials which were introduced in \S\ref{sec:3.5}.
The continuous big $q^{-1}$-Hermite polynomials satisfy the following orthogonality relation.

\begin{cor}
Let $n,m\in\N_0$, $q\in\CCdag$, 
$x=\frac12(z+z^{-1})\in \CCast$, $a, z\in\CCast$. 
Then
\begin{eqnarray}
&&\hspace{-2.8cm}\int_{0}^{i\infty}H_n(\tfrac12(z+z^{-1});a|q^{-1})
H_{m}(\tfrac12(z+z^{-1});a|q^{-1})
w(z;a|q)\,\dd z=h_n(q)\delta_{m,n},
\label{cbqiHO}
\end{eqnarray}
where
\begin{equation}
w(z;a|q):=(z-z^{-1})\frac{(qaz^\pm;q)_\infty}{\vartheta(z^2;q)},
\label{cbqiHw}
\end{equation}
and
\begin{equation}
h_n(q):=q^{-\binom{n}{2}}
(q;q)_\infty(q;q)_n
\left(-\frac{1}{q}\right)^n\,\log q.
\label{cbqiHn}
\end{equation}
\label{cbqiHOt}
\end{cor}

\begin{proof}
This orthogonality relation can be 
found by taking the limit as $d\to0$ in
Corollary \ref{cqiASCOc}.
\end{proof}

\noindent 
There exists an infinite discrete orthogonality for the for the infinite family of orthogonal polynomials given by continuous big $q^{-1}$-Hermite polynomials
which one can obtain from the infinite discrete orthogonality of $q^{-1}$-Al-Salam--Chihara polynomials by taking the limit as $b\to 0$. 

\medskip
\noindent
In order to investigate the region of convergence for the parameters involved in the following infinite discrete orthogonality relation, we will need the following lemma.

\begin{lem}
Let $n\in\N_0$, $q\in\CCdag$, $a\in\CCast$. Then one has 
as $m\to\infty$, that
\begin{equation}
H_n[q^{-m}a;a|q^{-1}]\sim q^{-nm}a^n.
\end{equation}
\label{lem416}
\end{lem}

\begin{proof}
Start with \eqref{cbqiH:4}, replace $z=q^{-m}a$ and as $m\to\infty$ then
\[
H_m[q^{-m}a;a|q^{-1}]\sim q^{-mn}a^n\qhyp10{q^{-n}}{-}{q,q^{m+1}}\sim q^{-nm}a^n,
\]
which completes the proof.
\end{proof}

\noindent 
An infinite discrete orthogonality relation for the continuous big $q^{-1}$-Hermite polynomials is given 
in the following theorem.
\begin{thm}
Let $n,n'\in\N_0$, $q\in\CCdag$, $a\in\CCast$. Then the continuous big $q^{-1}$-Hermite polynomials satisfy the following infinite discrete orthogonality relation
\begin{eqnarray}
&&\hspace{-0.7cm}\sum_{m=0}^\infty
q^{3\binom{m}{2}}\left(-\frac{q}{a^2}\right)^m
\frac{(\frac{q}{a^2};q)_{2m}(\frac{1}{a^2}
;q)_{m}}{(\frac{1}{a^2};q)_{2m}(q
;q)_m}
H_n[q^{-m}a;a|q^{-1}]H_{n'}[q^{-m}a;a|q^{-1}]=\frac{q^{-\binom{n}{2}}}{(-q)^n}
\left(\frac{q}{a^2};q\right)_\infty\!\!\!(q;q)_n\delta_{n,n'}.
\label{cbqHorthi}
\end{eqnarray}
\label{thm358}
\end{thm}

\begin{proof}
Start with cf.~\cite[(3.4)]{Groenevelt2021} (corrected 
so that the degrees of the Al-Salam--Chihara polynomials 
are $n,n'$ respectively and the Kronecker delta symbol 
is $\delta_{n,n'}$).
Consider $n=n'$. The left-hand side of \eqref{cbqHorthi} as ${\sf S}_{n}(a;q)$, then 
using Lemma \ref{lem416}, we have 
as $m\to\infty$, the summand behaves like 
\begin{equation}
{\sf s}_{m,n}(a;q)\sim 
a^{2n}
\frac{(\frac{q}{a^2};q)_\infty}{(q;q)_\infty}
q^{3\binom{m}{2}}
\left(-\frac{1}{q^{2n-1}a^2}\right)^m,
\end{equation}
where $\sum_m{\sf s}_{m,n}(a;q)=
{\sf S}_{n}(a;q)$.  
Hence, 
by using the direct comparison test 
${\sf S}_{n}(a;q)$ converges 
since the infinite series associated with \eqref{cbqHorthi} is convergent. Therefore, it is convergent for all values of $a$, $a\ne 0$ and $n\in\N_0$. This completes the proof.
\end{proof}

\noindent 
One may also obtain an infinite discrete orthogonality relation for the continuous big $q^{-1}$-Hermite polynomials which comes from the orthogonality relation with $q$-Bessel polynomials.
In order to study the convergence properties of this orthogonality relation, we will need the asymptotics of the continuous big $q^{-1}$-Hermite polynomials as $n\to\infty$.

\begin{lem}
Let $n\in\N_0$, $q\in\CCdag$, $a\in\CCast$. Then
as $n\to\infty$, one has 
\begin{equation}
H_n(x;a|q^{-1})\sim q^{-2\binom{n}{2}}
(-a)^n(\tfrac{z^\pm}{a};q)_\infty.
\label{Hn-large-n}
\end{equation}
\label{lem418}
\end{lem}

\begin{proof}
This proof is due to Xiang-Sheng Wang. Start by using the generating function of the continuous big $q^{-1}$-Hermite polynomials  \eqref{cbqinHegf-2} below, ${\sf I}(t;a|q)$.
From ${\sf I}(t;a|q)$ 
we obtain the integral representation
\begin{equation}
 H_n(x;a|q^{-1})=q^{-\binom{n}{2}}{(q;q)_n\over2\pi i}\int_{C_R}{(-tz^\pm;q)_\infty\over(-ta;q)_\infty}{dt\over t^{n+1}},
\end{equation}
where $C_R$ is the circle centered at the origin with a radius $R=1/|a|$.
As $t\to -1/a$, we have
\begin{equation}
 {(-tz^\pm;q)_\infty\over(-ta;q)_\infty}\sim {(\frac{z^\pm}{a};q)_\infty\over(1+ta)(q;q)_\infty}.
\end{equation}
By the Darboux method, we obtain
\begin{align}
 {q^{n(n-1)/2}H_n(x;a|q^{-1})\over (q;q)_n}\sim{1\over2\pi i}\int_{C_R}{(\frac{z^\pm}{a};q)_\infty\over(1+ta)(q;q)_\infty}{dt\over t^{n+1}}
 ={(\frac{z^\pm}{a};q)_\infty(-a)^n\over(q;q)_\infty},
\end{align}
as $n\to\infty$. 
This proves \eqref{Hn-large-n}.
\end{proof}

\noindent 
Unfortunately, for the special argument $z=q^{-m}a$, Lemma \ref{lem418} doesn't give the correct asymptotic result as $n\to\infty$. Instead we will require a different result.

\begin{lem}
Let $m,n\in\N_0$, $q\in\CCdag$, $a\in\CCast$. Then 
as $n\to\infty$, one has
\begin{eqnarray}
&&\hspace{-9.0cm}H_n[q^{-m}a;a|q^{-1}]\sim q^{-2\binom{m}{2}}\left(\frac{a^2}{q}\right)^m a^{-n}.
\end{eqnarray}
\label{lem420}
\end{lem}

\begin{proof}
From \eqref{cbqinHegf-2} below, we have 
\begin{eqnarray}
&& \hspace{-1cm} 
\sum_{n=0}^\infty \frac{t^nq^{\binom{n}{2}} H_n(x;a|q^{-1})}{(q;q)_n}=(-tz^{-1};q)_\infty(-tz;q)_m
=\sum_{k=0}^m\left[{m\atop k}\right]_qq^{\binom{k}{2}}(tz)^k\sum_{j=0}^\infty{q^{\binom{j}{2}}\over(q,q)_j}(tz^{-1})^j. \nonumber
\end{eqnarray}
Assume $n\ge m$. By matching the coefficients of $t^n$ on both sides of the above identity, we have
\begin{align}
 H_n(x;a|q^{-1})=\sum_{k=0}^m\left[{m\atop k}\right]_q{(q;q)_nq^{k^2-kn}z^{2k-n}\over(q;q)_{n-k}}\sim q^{m^2-nm}z^{2m-n},
\end{align}
as $n\to\infty$, which completes the proof.
\end{proof}

\noindent Now we give the orthogonality relation for the continuous big $q^{-1}$-Hermite polynomial which comes from the orthogonality relation for the $q$-Bessel polynomials by applying the duality relation \eqref{dHqinym}.

\begin{thm}
Let $m,m'\in\N_0$, $q\in\CCdag$, $a\in\CCast$. Then the continuous big $q^{-1}$-Hermite polynomials satisfy the following infinite discrete orthogonality relation
\begin{eqnarray}
&&\hspace{-0.9cm}\sum_{n=0}^{\infty}\frac{q^{\binom{n}{2}}(-q)^n}{(q;q)_n}H_n[q^{-m}a;a|q^{-1}]H_n[q^{-m'}a;a|q^{-1}]
=q^{-3\binom{m}{2}}(\tfrac{q}{a^2};q)_\infty
\left(-\frac{a^2}{q}\right)^m
\frac{(\frac{1}{a^2};q)_{2m}(q;q)_m}{(\frac{q}{a^2};q)_{2m}(\frac{1}{a^2};q)_m}\delta_{m,m'}.
\label{qBesselcbqiHO}
\end{eqnarray} 
\label{thm359}
\end{thm}

\begin{proof}
One may obtain this orthogonality relation by starting with the orthogonality relation for the $q$-Bessel polynomials Theorem \ref{thm368} and then using the duality relation with continuous big $q^{-1}$-Hermite polynomials or by taking the limit as $b\to 0$ in \eqref{idqiASCO}.
Now consider $m=m'$. Define the left-hand side of \eqref{qBesselcbqiHO} as ${\sf T}_{m}(a;q)$, then 
using Lemma \ref{lem420}, we have 
as $n\to\infty$, the summand behaves like 
\begin{equation}
{\sf t}_{n,m}(a;q)\sim 
\frac{q^{-4\binom{m}{2}}\left(a^2/q\right)^{2m}}{(q;q)_\infty}
q^{\binom{n}{2}}
\left(-\frac{q}{a^2}\right)^n,
\end{equation}
where $\sum_n{\sf t}_{n,m}(a;q)=
{\sf T}_{m}(a;q)$.
Hence, 
by using the direct comparison test 
${\sf T}_{m}(a;q)$ converges 
since the infinite series associated with \eqref{qBesselcbqiHO} is convergent. Therefore, it is convergent for all values of $a$, $a\ne 0$ and $m\in\N_0$. This completes the proof.
\end{proof}

\begin{rem}
Note that one may also obtain Theorem \ref{thm359} by using the closure relation \eqref{closD} applied to Theorem \ref{thm358}.
\end{rem}

\subsection{The continuous $q$ and $q^{-1}$-Hermite polynomials}
\noindent
Now we describe the continuous orthogonality relation for the Al-Salam--Chihara  polynomials which were introduced in 
\S\ref{sec:3.6}.
The continuous $q$-Hermite polynomials $H_n(x|q)$ are orthogonal
on $x=\cos\theta\in(-1,1)$
with respect to the 
weight function
\begin{eqnarray}
&&\hspace{-10.5cm}
w_q(\cos\theta;a):=
(\expe^{\pm 2i\theta};q)_\infty.
\label{cqHw}
\end{eqnarray}
\begin{thm}
Let $m,n\in\N_0$, $q\in\CCdag$. Then 
the continuous $q$-Hermite polynomials satisfy the following continuous orthogonality relation 
\cite[(14.26.2)]{Koekoeketal}
\begin{equation}
\int_0^\pi H_m(x|q)H_n(x|q)w_q(x)\,
{\mathrm d}\theta=h_n(q)\delta_{m,n},
\label{cqHO}
\end{equation}
where $h_n$ is given by \eqref{cbqHn}.
\end{thm}

\begin{proof}
See proof of \cite[(14.26.2)]{Koekoeketal}.
\end{proof}

\medskip

\noindent 
Now we describe known orthogonality relations for the infinite family of orthogonal polynomials given by continuous $q^{-1}$-Hermite polynomials which were introduced in \S\ref{sec:3.6}.
It is well-known that the continuous $q^{-1}$-Hermite polynomials satisfy an indeterminate moment problem, so there exist an infinite number of orthogonality relations for these polynomials. This was studied systematically in Ismail \& Masson (1994) \cite{IsmailMasson1994}. We will give three of these orthogonality relations now; the weight functions for these orthogonality relations are presented in \cite[Theorem 21.6.4]{Ismail:2009:CQO}. 
The first orthogonality relation that we present follows from the corresponding orthogonality relation for the continuous big $q^{-1}$-Hermite polynomials, Corollary \ref{cbqiHOt}, after taking the limit $a\to 0$. This orthogonality relation was initially derived by Askey in 1989 \cite[Theorem 2]{Askey89cqiH}.

\begin{cor}
Let $n,m\in\N_0$, $q\in\CCdag$, $x=\frac12(z+z^{-1})\in\CCast$. Then
\begin{eqnarray}
&&\hspace{-2cm}\int_{0}^{i\infty}H_n(\tfrac12(z+z^{-1})|q^{-1})
H_{m}(\tfrac12(z+z^{-1})|q^{-1})
w(z|q)\,\dd z=h_n(q)\delta_{m,n},
\label{cqiHO}
\end{eqnarray}
where
\begin{equation}
w(z|q):=\frac{z-z^{-1}}{\vartheta(z^2;q)},
\label{cqiHw}
\end{equation}
and
$h_n(q)$ is defined by \eqref{cbqiHn}.
\end{cor}

\begin{proof}
This orthogonality relation can be 
found by taking the limit as $a\to0$ in
Corollary \ref{cbqiHOt}.
\end{proof}

\noindent
Now we give a second orthogonality relation which is an infinite discrete orthogonality relation with a free parameter $\alpha\in\CCast$.

\begin{thm}
Let $n,n'\in\N_0$, $q\in\CCdag$, $\alpha\in(q,1]$. Then
\begin{eqnarray}
&&\hspace{-1.0cm}\sum_{m=-\infty}^\infty q^{4\binom{m}{2}}(q\alpha^4)^m
\frac{(-q\alpha^2;q)_{2m}}{(-\alpha^2;q)_{2m}}
H_n[iq^m\alpha|q^{-1}]
H_{n'}[iq^m\alpha|q^{-1}]
=(q,-q\alpha^{\pm 2};q)_\infty \frac{q^{-\binom{n}{2}}}{(-q)^{n}}(q;q)_n\delta_{n,n'}.
\end{eqnarray}
\end{thm}

\begin{proof}
This orthogonality relation is a result due to Ismail and Masson \cite[\S6]{IsmailMasson1994} (see also \cite[(4.1)]{ChristiansenKoelink2008}).
\end{proof}

\noindent
The third orthogonality relation which corresponds to measure $w_2$ in \cite[Theorem 21.6.4]{Ismail:2009:CQO}, is given as follows.
\begin{thm}
Let $n,n'\in\N_0$, $q\in\CCdag$. Then
\begin{eqnarray}
&&\hspace{-0.7cm}\int_1^\infty H_n[iz|q^{-1}]H_{n'}[iz|q^{-1}](1+\tfrac{1}{z^{2}})\exp\left(\frac{2(\log z)^2}{\log q}\right)\dd z=(-q)^{-n} q^{-\frac18-\binom{n}{2}}(q;q)_n\sqrt{\frac{\pi\log q^{-1}}2}\delta_{n,n'}.
\end{eqnarray} 
\end{thm}

\begin{proof}
See the proof of \cite[(21.7.7)]{Ismail:2009:CQO} (see also \cite{AFW1994}).
\end{proof}

\subsection{The big $q$-Jacobi polynomials and functions\label{secbqJo}}
Now we present orthogonality relations for the big $q$-Jacobi polynomials and functions.

\noindent
\begin{thm}
Let $m,m'\in\N_0$, $q\in\CCdag$, $a,b,c\in\CCast$ such that $qa\in(0,1)$, $qb\in[0,1)$ and $c<0$. Then, 
the big $q$-Jacobi polynomials satisfy the following 
orthogonality relation \cite[(14.5.2)]{Koekoeketal}:
\begin{eqnarray}
&&\hspace{-0.5cm}\int_{qc}^{qa}
P_m(x;a,b,c;q)P_{m'}(x;a,b,c,q)
\frac{(\frac{x}{a},\frac{x}{c};q)_\infty}
{(x,\frac{bx}{c};q)_\infty}\,\dd_q x\nonumber\\
&&\hspace{0.9cm}=q^{\binom{m}{2}+1}a(1-q)(-q^2ac)^m\frac{(q,q^2ab,\frac{c}{a},\frac{qa}{c};q)_\infty}
{(qa,qb,qc,\frac{qab}{c};q)_\infty}
\frac{(qab;q)_{2m}(q,qb,\frac{qab}{c};q)_m}
{(q^2ab;q)_{2m}(qa,qc,qab;q)_m}\delta_{m,m'}.
\label{bqJO}
\end{eqnarray}
\end{thm}

\begin{proof}
See proof of \cite[(14.5.2)]{Koekoeketal}.
\end{proof}

\noindent One also has an infinite discrete orthogonality relation for big $q$-Jacobi polynomials which can be obtained through duality between the continuous dual $q^{-1}$-Hahn polynomials which was given in \cite{AtakishiyevKlimyk2006}.

\begin{thm}{Atakishiyev and Klimyk \cite[(4.30)]{AtakishiyevKlimyk2006}.}
\label{thm316}
Let $n,n'\in\N_0$, $q\in\CCdag$, $|ab|>1$, $|qb|<|a|$. Then, the big $q$-Jacobi polynomials satisfy the following orthogonality relation
\begin{eqnarray}
&&\hspace{-0.1cm}\sum_{m=0}^\infty
\frac{q^{-\binom{m}{2}}}{(-q^2ac)^m}
\frac{(qa,qc,qab,\pm\sqrt{q^3ab};q)_m}{(q,qb,\frac{qab}{c},\pm\sqrt{qab};q)_m}
\,P_m\left(q^{n+1}a;a,b,c;q\right)
P_m\left(q^{n'+1}a;a,b,c;q\right)\nonumber\\
&&\hspace{7.5cm}=q^{-n}
\frac{(q^2ab,\frac{c}{a};q)_\infty(q,\frac{qa}{c};q)_n}{(qb,qc;q)_\infty(qa,\frac{qab}{c};q)_n}\delta_{n,n'}.
\end{eqnarray}
\end{thm}
\begin{proof}
Starting with the discrete orthogonality relation for
the continuous dual $q^{-1}$-Hahn polynomials \eqref{AKporth}
The orthogonality relation \eqref{AKporth}
is obtained from \cite[(4.30)]{AtakishiyevKlimyk2006}
using \cite[(4.29)]{AtakishiyevKlimyk2006}
and comparing the result with
\eqref{cdqiH:1} twice.
Then, inserting the duality relation
\eqref{dcdqHbqJab}, simplifying and then making the replacement
$(a,b,c)\mapsto \left(\frac{1}{\sqrt{qab}},
\frac{\sqrt{b}}{\sqrt{qa}},\frac{c}{\sqrt{qab}}\right)$ 
completes the proof.
\end{proof}

\noindent 
Now we present continuous orthogonality of big and little $q$-Jacobi functions by starting
from continuous orthogonality of the
continuous dual $q$-Hahn polynomials and
the Al-Salam--Chihara polynomials.
The dual orthogonality for these functions is 
due to the following
duality relations for the continuous
dual $q$-Hahn polynomials
and the Al-Salam--Chihara polynomials
with the big and little $q$-Jacobi functions 
respectively, see Theorems \ref{thm3.33}, \ref{thm3.35}.

\medskip
\noindent One has the following
continuous orthogonality relation 
for the big $q$-Jacobi function. 
For fixed $q,a,b\in\CCast$, define the constant ${\sf A}$ as follows
\begin{equation}
{\sf A}:={\sf A}(a,b;q):=\frac{\log(qab)}{2\log q^{-1}}.
\label{Adef}
\end{equation}
In the sequel, we will use the constant ${\sf A}$ 
whenever we apply the following continuous orthogonality relations for big and little
$q$-Jacobi functions. Note that this orthogonality relation is an index transform (see for example \cite{Yakubovich1996}).
\begin{thm}
\label{cobqJ}
Let $n,n'\in\mathbb N_0$, $q\in\CCdag$, 
$a,b,c\in\CCast$. 
Then
\begin{eqnarray}
&&\hspace{-4.5cm}\int_0^{\frac{\pi}{\log q^{-1}}}
\frac{P_{{\sf A}+ix}(q^{-n};a,b,c;q)
P_{{\sf A}+ix}(q^{-n'};a,b,c;q)(q^{\pm 2ix};q)_\infty}{(q^{\frac12\pm ix}\sqrt{ab},q^{\frac12\pm ix}\sqrt{\frac{a}{b}},q^{\frac12\pm ix}\frac{c}{\sqrt{ab}};q)_\infty}\dd x
\nonumber\\
&&\hspace{-2.5cm}
=\frac{2\pi(qab)^n(q,\frac{qc}{b};q)_n}{(q,qa,qc,\frac{qc}{b};q)_\infty\log q^{-1}(qa,qc;q)_n}\delta_{n,n'}.
\end{eqnarray}
\end{thm}
\begin{proof}
This follows directly from the continuous orthogonality of 
continuous dual $q$-Hahn polynomials
\cite[(14.3.2)]{Koekoeketal}
and the continuous duality relation
between continuous dual $q$-Hahn polynomials and big $q$-Jacobi function in Theorem \ref{thm3.33}.
\end{proof}

\subsection{The little $q$-Jacobi polynomials and functions\label{seclqJo}}
\noindent We also have the following orthogonality relation for little $q$-Jacobi polynomials which can be obtained from the infinite discrete orthogonality of $q^{-1}$-Al-Salam--Chihara polynomials by using duality with little $q$-Jacobi polynomials \eqref{dqiASClqJa}.

\begin{thm}{Askey--Ismail \cite{AskeyIsmail84} (1984).}
\label{thm314}
Let $n,n'\in\N_0$, $q\in\CCdag$, $|ab|>1$, $|qb|<|a|$. Then, the little $q$-Jacobi polynomials satisfy the following orthogonality relation
\begin{eqnarray}
&&\hspace{-1.3cm}\sum_{m=0}^\infty
\left(\frac{1}{qa}\right)^m
\frac{(qa,qab,\pm\sqrt{q^3ab};q)_m}{(q,qb,\pm\sqrt{qab};q)_m}
\,p_m\left(q^n;a,b;q\right)
p_m\left(q^{n'};a,b;q\right)\nonumber\\
&&\hspace{6.5cm}=
\left(\frac{1}{qa}\right)^n \frac{(q^2ab;q)_\infty(q;q)_n}
{(qa;q)_\infty(qb;q)_n}\delta_{n,n'}.
\end{eqnarray}
\end{thm}
\begin{proof}
Starting with cf.~\cite[(3.4)]{Groenevelt2021} (corrected 
so that the degrees of the Al-Salam--Chihara polynomials 
are $n,n'$ respectively and the Kronecker delta symbol 
is $\delta_{n,n'}$), namely
\eqref{ASCorthi},
followed by using the duality relation
\eqref{dqiASClqJa}, simplifying and then making the 
replacement $(a,b)\mapsto \left(\frac{1}{\sqrt{qab}},
\frac{\sqrt{a}}{\sqrt{qb}}\right)$ completes the proof.
\end{proof}
Note that this orthogonality is clearly different 
from the standard orthogonality of the little 
$q$-Jacobi polynomials 
since in Theorem \ref{thm314}, the degrees of the little $q$-Jacobi polynomials are fixed and are given by the sum index.
We now give that orthogonality relation
\cite[(14.12.2)]{Koekoeketal}.

\begin{thm}
Let $n,n'\in\N_0$, $q\in\CCdag$, $a,b\in\CCast$. Then, one has the following infinite discrete orthogonality relation for little $q$-Jacobi polynomials
\cite[(14.12.2)]{Koekoeketal}:
\begin{eqnarray}
&&\hspace{-2.5cm}
\sum_{m=0}^\infty
(qa)^m\frac{(qb;q)_m}{(q;q)_m}
p_n(q^m;a,b;q)p_{n'}(q^{m};a,b;q)
\nonumber\\
&&\hspace{2cm}=(qa)^n\frac{(q^2ab;q)_\infty}{(qa;q)_\infty}\frac{(qab;q)_{2n}(q,qb;q)_n}{(q^2ab;q)_{2n}(qa,qab;q)_n}\delta_{n,n'}.
\label{lqJO}
\end{eqnarray}
\end{thm}
\begin{proof}
See proof of \cite[(14.12.2)]{Koekoeketal}.
\end{proof}

\noindent And for little $q$-Jacobi functions, one has the following continuous orthogonality relation which is an index transform.
\begin{thm}
\label{colqJ}
Let $n,n'\in\N_0$, $q\in\CCdag$, $a,b\in\CCast$ where ${\sf A}$ is defined in \eqref{Adef}.
Then
\begin{eqnarray}
&&\hspace{-4.5cm}\int_0^{\frac{\pi}{\log q^{-1}}}
\frac{p_{{\sf A}+ix}(\frac{q^{-1-n}}{b};a,b;q)
p_{{\sf A}+ix}(\frac{q^{-1-n'}}{b};a,b;q)(q^{\pm 2ix};q)_\infty}{(q^{\frac12\pm ix}\sqrt{ab},q^{\frac12\pm ix}\sqrt{\frac{a}{b}},q^{\frac12\pm ix}\sqrt{\frac{a}{b}},q^{\frac12\pm ix}\sqrt{\frac{b}{a}};q)_\infty}\dd x
\nonumber\\
&&\hspace{-2.5cm}
=\frac{2\pi(qab)^n(q;q)_n}{(q,qa,qa,qb,\frac{1}{b},\frac{1}{b};q)_\infty\log q^{-1}(qb;q)_n}\delta_{n,n'}.
\end{eqnarray}
\end{thm}
\begin{proof}
This follows directly from the continuous orthogonality of Al-Salam--Chihara polynomials \cite[(14.8.2)]{Koekoeketal}
and the continuous duality relation
between Al-Salam--Chihara polynomials and little $q$-Jacobi function in Theorem \ref{thm3.35}.
\end{proof}

\subsection{The $q$-Bessel polynomials and $q^{-1}$-Bessel functions\label{secqBo}}
\noindent 
The $q$-Bessel polynomials satisfy the following infinite discrete orthogonality relation
cf.~\cite[(14.22.2)]{Koekoeketal}.
\begin{thm}
Let $m,m'\in\N_0$, $q\in\CCdag$, 
$a\in(0,\infty)$. 
Then, the $q$-Bessel polynomials satisfy the following infinite discrete orthogonality relation:
\begin{eqnarray}
&&\hspace{-0.6cm}\sum_{n=0}^\infty \frac{q^{\binom{n}{2}}(qa)^n}{(q;q)_n}y_m(q^n;a;q)y_{m'}(q^n;a;q)=q^{\binom{m}{2}}(qa)^m(-qa;q)_\infty
\frac{(-a;q)_{2m}(q;q)_m}{(-qa;q)_{2m}(-a;q)_m}\delta_{m,m'}.
\end{eqnarray}
\label{thm368}
\end{thm}

\begin{proof}
See \cite[(14.22.2)]{Koekoeketal}.
\end{proof}

\noindent 
There is also an infinite discrete orthogonality for the $q$-Bessel polynomials which comes from it's duality with continuous big $q^{-1}$-Hermite polynomials \eqref{dHqinym}. It is now given in the following theorem.

\begin{thm}
\label{thm248}
Let $n,n'\in\N_0$, $q\in\CCdag$, $a\in\CCast$. Then, the $q$-Bessel polynomials satisfy the following infinite discrete orthogonality relation
\begin{eqnarray}
&&\hspace{-1.4cm}\sum_{m=0}^\infty
\frac{q^{-\binom{m}{2}}}{(qa)^m}
\frac{(-qa;q)_{2m}(-a;q)_m}
{(-a;q)_{2m}(q;q)_m}
\,y_m(q^{n};a;q)
y_m(q^{n'};a;q)
=\frac{q^{-\binom{n}{2}}}{(qa)^n}
(-qa;q)_\infty(q;q)_n
\delta_{n,n'}.
\label{orthogqB1}
\end{eqnarray}
\end{thm}

\begin{proof}
This orthogonality relation can be obtained by starting with the infinite discrete orthogonality relation for the continuous big $q^{-1}$-Hermite polynomials and using the duality relation with $q$-Bessel polynomials \eqref{dHqinym} and making straightforward replacements.
\end{proof}

\noindent 
From the continuous duality of continuous big $q$-Hermite polynomials with the $q^{-1}$-Bessel functions, one can obtain a continuous orthogonality for the $q^{-1}$-Bessel functions.
Define 
\begin{equation}
{\sf B}
:={\sf B}(a;q)
:=-\frac{\log(-a)}{2\log q^{-1}}.
\label{Bdef}
\end{equation}
For $q\in(0,1)$, and $a\in(-\infty,-1)$, one has that ${\sf B}<0$.
For $q\in(0,1)$, the continuous orthogonality relation for $q^{-1}$-Bessel functions given in Theorem \ref{thm437} below has a positive norm for $a\in(-\infty,-1)$. 
One has the following continuous orthogonality relation for the $q^{-1}$-Bessel function which
is an index transform.
\begin{thm}
\label{thm437}
Let $n,n'\in\N_0$, $q\in\CCdag$,  $a\in\CC\setminus(1,\infty)$, such that $|a|>1$. Then
\begin{eqnarray}
&&\hspace{-0.5cm}\int_0^{\frac{\pi}{\log q^{-1}}}y_{{\sf B}+ix}(q^{-n};a;q^{-1})
y_{{\sf B}+ix}(q^{-n'};a;q^{-1})
\frac{q^{-2x^2}(q^{\pm 2ix};q)_\infty}{(\frac{q^{\pm ix}}{\sqrt{-a}};q)_\infty}\dd x
=\frac{2\pi\,\expe^{-\frac{(\log(-a))^2}{2\log q^{-1}}}(q;q)_n}{(q;q)_\infty\log q^{-1}(-a)^n}\delta_{n,n'}.
\label{COqiBf}
\end{eqnarray}
\end{thm}
\begin{proof}
Start with the continuous orthogonality relation for the continuous big $q$-Hermite polynomials \eqref{cbqHO} and insert the duality relation between the continuous big $q$-Hermite polynomials with the $q^{-1}$-Bessel function \eqref{dHnymu}. The restriction on $a$ is due to the fact that continuous orthogonality for continuous big $q$-Hermite polynomials requires $|a|<1$, but the duality relation which connects to two families \eqref{dHnymu} involves the negative reciprocal squared of $a$. After simplification and straightforward replacements, the result follows.
\end{proof}

\section{Nonterminating and terminating {\it q-}Chaundy representations}\label{sec:4}
In this section, we derive two equivalent 
$q$-Chaundy infinite 
series representations for a product of two 
nonterminating basic hypergeometric series. 
These representations
are given by sums over terminating basic
hypergeometric series using the 
van de Bult \& Rains notation \eqref{topzero}, 
\eqref{botzero}. 
\subsection{Double product nonterminating representations}\label{sec:4.1}
It is 
interesting to note that with 
regard to Lemma \ref{lem:1.3},
the following $q$-Chaundy product
re\-pre\-sen\-tations \eqref{qCh1}, \eqref{qCh2} 
are clearly, term by term, inverses of each other.

\begin{thm}
\label{BT}
Let $r,s\in\mathbb N_0\cup\{-1\}$, $u,v\in\mathbb N_0$,
$p,\ell\in\mathbb Z$ such that $p\ge r-u$ and $\ell\ge s-v$,
${\bf a}\in\CCast^{r+1}$, ${\bf b}\in\CCast^{u}$, 
${\bf c}\in\CCast^{s+1}$, ${\bf d}\in\CCast^{v}$, 
$q\in\CCdag$. Then
\begin{eqnarray}
&&\hspace{-0.5cm}\qphyp{r+1}{u}{p}{{\bf a}}{{\bf b}}{q,{\sf X}}
\qphyp{s+1}{v}{\ell}{{\bf c}}{{\bf d}}{q,{\sf Y}}
=\sum_{n=0}^\infty\frac{({\bf a};q)_n {\sf {\sf X}}^n}
{(q,{\bf b};q)_n}
\left((-1)^nq^{\binom{n}{2}}\right)^{u-r+p}
\nonumber\\
&&\hspace{3.5cm}\times\qphyp{s+u+2}{r+v+1}{u-r+p+\ell}
{q^{-n},
{\bf c},
\frac{q^{1-n}}{\bf b}
}
{
{\bf d},
\frac{q^{1-n}}{{\bf a}}
}{q,
\frac{q^{1+p(1-n)}b_1\cdots b_u{\sf Y}}
{a_1\cdots a_{r+1}{\sf X}}}
\label{qCh1}\\
&&\hspace{0.0cm}=\sum_{n=0}^\infty
\frac{({\bf c};q)_n {\sf Y}^n}{(q,{\bf d};q)_n}
\left((-1)^nq^{\binom{n}{2}}\right)^{v-s+\ell}
\nonumber\\[-0.3cm]
&&\hspace{3.5cm}\times\qphyp{r+v+2}{s+u+1}{v-s+p+\ell}
{q^{-n},
{\bf a},
\frac{q^{1-n}}{\bf d}
}
{
{\bf b},
\frac{q^{1-n}}{\bf c}
}{q,
\frac{q^{1+\ell(1-n)}d_1\cdots d_v{\sf X}}
{c_1\cdots c_{s+1}{\sf Y}}},\label{qCh2}
\end{eqnarray}
where ${\sf X}$, ${\sf Y}$ are given 
such that the left-hand side is well-defined.
\end{thm}
\begin{proof}
First, consider the restriction $p,\ell\in\mathbb Z$ such 
that $p\ge r-u$ and $\ell\ge s-v$ so that both 
nonterminating basic hypergeometric series converge.
Then, starting with the left-hand side of 
\eqref{qCh1}, one writes out the double product 
of two nonterminating basic hypergeometric
series as two sums multiplied together using
 \eqref{topzero}, \eqref{botzero} with 
 ${\sf X}\mapsto g{\sf X}$, 
${\sf Y}\mapsto h{\sf X}$, for 
some $h{\sf X},g{\sf X}\in\CCdag$, namely
\begin{eqnarray}
&&\hspace{-0.5cm}\qphyp{r+1}{u}p{{\bf a}}{{\bf b}}{q,g{\sf X}}
\qphyp{s+1}{v}{\ell}{{\bf c}}{{\bf d}}{q,h{\sf X}}
\nonumber\\
\label{eq:83}&&=\sum_{n=0}^\infty\frac{({\bf a};q)_n}
{(q,{\bf b};q)_n}
\left((-1)^nq^{\binom{n}{2}}\right)^{u-r+p}
\sum_{k=0}^\infty\frac{({\bf c};q)_k}{(q,{\bf d};q)_k}
\left((-1)^kq^{\binom{k}{2}}\right)^{v-s+\ell}
(g{\sf X})^n(h{\sf X})^k.
\end{eqnarray}
Now make a double-index replacement $(n',k')=(n+k,k)$ 
or equivalently $(n,k)=(n'-k',k')$. This is referred 
to as diagonal summation (see 
 \cite[just below (74)]{Chaundy43}), and upon 
replacement $n'\mapsto n$, and $k'\mapsto k$, we 
have
\begin{eqnarray}
&&\hspace{-0.0cm}\qphyp{r+1}{u}p{{\bf a}}{{\bf b}}{q,g{\sf X}}
\qphyp{s+1}{v}{\ell}{{\bf c}}{{\bf d}}{q,h{\sf X}}
\nonumber\\
\label{eq:84}&&=\sum_{n=0}^\infty(g{\sf X})^n 
\sum_{k=0}^n\frac{({\bf c};q)_k}{(q,{\bf d};q)_k}
\left((-1)^kq^{\binom{k}{2}}\right)^{v-s+\ell}
\frac{({\bf a};q)_{n-k}}{(q,{\bf b};q)_{n-k}}
\left((-1)^{n-k}q^{\binom{n-k}{2}}\right)^{u-r+p}
\left(\frac{h}{g}\right)^k.
\end{eqnarray}
Now we use expressions 
 \eqref{binomid},
 \eqref{qPoch1},
 \eqref{qPoch2}, 
collecting terms using \eqref{topzero}, \eqref{botzero}, 
and replacing $g{\sf X}\mapsto {\sf X}$, $h{\sf X}\mapsto {\sf Y}$ produces 
 \eqref{qCh1}.
Without loss of generality interchanging the two 
basic hypergeometric series on the left-hand side 
of \eqref{qCh1} produces \eqref{qCh2}. 
The product of the two nonterminating basic hypergeometric series on the left-hand side of \eqref{qCh1}
is called their Cauchy product \cite[p.~147]{Knopp1990}. 
It is known \cite[p.~146-147]{Knopp1990} that if each 
series in the product is absolutely convergent, 
then their Cauchy product is absolutely convergent. 
Since a power series 
$\sum_{n=0}^\infty c_nz^n$ converges absolutely 
if $|z|<r$, where $r$ denotes the radius of 
convergence of the power series, one can justify 
the above calculations. This completes the proof.
\end{proof}

As an example of the above theorem, we now see that the square of a nonterminating 
very-well-poised basic hypergeometric series can be written 
as an infinite sum over a terminating well-poised basic 
hypergeometric series.

\begin{thm}
Let $q\in\CCdag$, 
$z,a,a_4,\ldots,a_{r+1}\in\CCast$.
Then 
\begin{eqnarray}
&&\hspace{-0.5cm}\Whyp{r+1}{r}{a}{a_4,\ldots ,a_{r+1}}
{q,z}\Whyp{r+1}{r}{a}{a_4,\ldots ,a_{r+1}}
{q,z}=\sum_{n=0}^\infty
\frac{(a,\pm q\sqrt{a},a_4,\ldots,a_{r+1};q)_nz^n}
{(q,\pm\sqrt{a},\frac{qa}{a_4},\ldots,\frac{qa}
{a_{r+1}};q)_n}\nonumber\\
&&\hspace{0.2cm}\times\Phyp{2r+2}{2r+1}{q^{-n}}
{a,\pm q\sqrt{a},a_4,\ldots,a_{r+1},
\pm\frac{q^{1-n}}{\sqrt{a}},
\frac{q^{-n}}{\sqrt{a}},
\frac{q^{1-n}}{a_4},
\ldots,
\frac{q^{1-n}}{a_{r+1}}}
{q,\frac{q^2a^2}{a_4^2\cdots a_{r+1}^2}},
\end{eqnarray}
where ${}_{2r+2}P_{2r+1}$ is a 
well-poised basic hypergeometric series.
\end{thm}
\begin{proof}
Straightforward substitution of the definition of
the nonterminating very-well-poised basic hypergeometric 
series \eqref{rpWr} in Theorem \ref{BT}
using the definition of a nonterminating well-poised 
basic hypergeometric series \eqref{rpPr} completes the proof. 
\end{proof}

\noindent Another specific nonterminating example is now given.

\begin{thm}
\label{thm:4.6}
Let $q\in\CCdag$, $a_k,c_k,b_1,b_2,d_1,d_2\in\CCast$, 
$k\in\{1,2,3,4\}$.
Then one has the following expansion of the product 
of two nonterminating balanced ${}_4\phi_3$'s with 
argument $q$ in terms of an infinite sum over a terminating 
3-balanced ${}_8\phi_7$ with argument $q^2$, namely
\begin{eqnarray}
&&\hspace{-0.4cm}\qhyp43{a_1,a_2,a_3,a_4}
{b_1,b_2,\frac{qa_1a_2a_3a_4}{b_1b_2}}{q,q}
\qhyp43{c_1,c_2,c_3,c_4}{d_1,d_2,\frac{qc_1c_2c_3c_4}
{d_1d_2}}{q,q}\nonumber\\
&&\hspace{0.5cm}
=\sum_{m=0}^\infty \frac{(a_1,a_2,a_3,a_4;q)_mq^m}
{(b_1,b_2,\frac{qa_1a_2a_3a_4}{b_1b_2};q)_m}
\qhyp87{q^{-m},c_1,c_2,c_3,c_4,
\frac{q^{1-m}}{b_1},\frac{q^{1-m}}{b_2},\frac{q^{-m}b_1b_2}
{a_1a_2a_3a_4}}{d_1,d_2,\frac{qc_1c_2c_3c_4}{d_1d_2},
\frac{q^{1-m}}{a_1},
\frac{q^{1-m}}{a_2},
\frac{q^{1-m}}{a_3},
\frac{q^{1-m}}{a_4}
}{q,q^2}.
\end{eqnarray}
\end{thm}
\begin{proof}
Taking $r=u=s=v=3$, $p=\ell=0$, ${\sf X}={\sf Y}=q$, 
${\bf a}=\{a_1,a_2,a_3,a_4\}$, ${\bf b}
=\{b_1,b_2,\frac{qa_1a_2a_3a_4}{b_1b_2}\}$,
${\bf c}=\{c_1,c_2,c_3,c_4\}$, ${\bf d}
=\{d_1,d_2,\frac{qc_1c_2c_3c_4}{d_1d_2}\}$ 
in Theorem \ref{BT} completes the proof.
\end{proof}

\subsection{Gasper's product expansion of two basic hypergeometric series}
\noindent 
\begin{thm}
Let $p,l\in\Z$, ${\bf a}$, ${\bf b}$, ${\bf c}$, ${\bf d}$, ${\bf e}$, ${\bf f}$ be multisets with cardinality given by $A,B,C,D,E,F\in\N_0$ respectively, $\gamma\in\CC$, $\sigma\in\CCast$, with the parameter values given such that the nonterminating basic hypergeometric series are convergent.
Then
\begin{eqnarray}
&&\hspace{-0.0cm}\qphyp{A+C}{B+D}{p+l}{{\bf a},{\bf c}}{{\bf b},{\bf d}}{q,xw}=
\sum_{n=0}^\infty 
\frac{(\frac{q\gamma}{\sigma};q)_{2n}(\gamma,\sigma,{\bf c},{\bf e};q)_n}{(\gamma;q)_{2n}(q,\frac{q\gamma}{\sigma},{\bf d},{\bf f};q)_n}
\left[(-1)^n q^{\binom{n}{2}}\right]^{D+F-C-E+l}\left(\frac{x}{\sigma}\right)^n\nonumber\\
&&\hspace{5.0cm}\times \qphyp{C+E+4}{D+F+3}{l}{q^{2n}\frac{\gamma}{\sigma},\pm q^{n+1}\sqrt{\frac{\gamma}{\sigma}},\frac{1}{\sigma},q^n{\bf c},q^n{\bf e}}
{q^{2n+1}\gamma,\pm q^n\sqrt{\frac{\gamma}{\sigma}},q^n{\bf d},q^n{\bf f}}
{q,q^{n(D+F-C-E+l)}x}\nonumber\\
&&\hspace{5.0cm}\times \qphyp{A+F+2}{B+E+2}{p}{q^{-n},q^n\gamma,{\bf a},{\bf f}}{\frac{q^{1-n}}{\sigma},q^{n+1}\frac{\gamma}{\sigma},{\bf b},{\bf e}}{q,qw}.
\label{Gaspexp}
\end{eqnarray}
\end{thm}
\begin{proof}
See proof of \cite[(4.7)]{Gasper1989a}. However, start with \cite[(4.6)]{Gasper1989a} with
\begin{equation}
A_n=\left[(-1)^nq^{\binom{n}{2}}\right]^{B-A+p+1}\frac{({\bf a};q)_n}{({\bf b};q)_n},\quad 
B_n=\left[(-1)^nq^{\binom{n}{2}}\right]^{D-C+l+1}\frac{({\bf c};q)_n}{({\bf d};q)_n}. 
\end{equation}
\end{proof}
\begin{rem}
It is not necessarily clear whether the expansion derived by Gasper (1989) \cite[(4.7)]{Gasper1989}
can somehow be connected to Theorem \ref{BT}.
\end{rem}

\noindent The $\sigma\to\infty$ limit of the Gasper expansion \eqref{Gaspexp} is interesting.

\begin{thm}
Let $p,l\in\Z$, ${\bf a}$, ${\bf b}$, ${\bf c}$, ${\bf d}$, ${\bf e}$, ${\bf f}$ be multisets with cardinality given by $A,B,C,D,E,F\in\N_0$ respectively, $\gamma\in\CC$, with the parameter values given such that the nonterminating basic hypergeometric series are convergent.
Then
\begin{eqnarray}
&&\hspace{-0.0cm}\qphyp{A+C}{B+D}{p+l}{{\bf a},{\bf c}}{{\bf b},{\bf d}}{q,xw}=
\sum_{n=0}^\infty 
\frac{(\gamma,{\bf c},{\bf e};q)_n}{(\gamma;q)_{2n}(q,{\bf d},{\bf f};q)_n}
\left[(-1)^n q^{\binom{n}{2}}\right]^{D+F-C-E+l+3}x^n\nonumber\\
&&\hspace{5.0cm}\times \qphyp{C+E}{D+F+1}{l}{q^n{\bf c},q^n{\bf e}}
{q^{2n+1}\gamma,q^n{\bf d},q^n{\bf f}}
{q,q^{n(D+F-C-E+l+2)}x}\nonumber\\
&&\hspace{5.0cm}\times \qphyp{A+F+2}{B+E}{p}{q^{-n},q^n\gamma,{\bf a},{\bf f}}{{\bf b},{\bf e}}{q,qw}.
\label{Gaspexp-2}
\end{eqnarray}
\end{thm}
\begin{proof}
Start with \eqref{Gaspexp} and take the limit as $\sigma\to\infty$ using \eqref{critlim}, followed by replacing 
$(p,l)\mapsto(p-2,l+2)$ completes the proof.
\end{proof}

\subsection{Single nonterminating basic hypergeometric representations}
By setting $s=v=\ell=0$, ${\bf c}=\{c\}$,
$c\in\CC$, ${\sf d}=\emptyset$ and utilizing the 
$q$-binomial theorem \eqref{qbinom} we produce the 
following interesting expansion of a basic hypergeometric 
function.

\begin{cor}\label{thm:2.4}
Let $q\in\CCdag$, $r\in\mathbb N_0\cup\{-1\}$, $u\in\mathbb N_0$, 
${\bf a}\in\CCast^{r+1}$, ${\bf b}\in\CCast^{u}$, 
$p\in\mathbb Z$, such that $p\ge r-u$, $c,{\sf X},{\sf Y}\in\CCast$ 
such that the left-hand side is well-defined.
Then
\begin{eqnarray}
&&\hspace{-0.5cm}\qphyp{r+1}{u}p{{\bf a}}{{\bf b}}{q,{\sf X}}
=\frac{({\sf Y};q)_\infty}{(c{\sf Y};q)_\infty}\sum_{m=0}^\infty
{\sf {\sf X}}^m
\frac{({\bf a};q)_m }{(q,{\bf b};q)_m}
\left((-1)^mq^{\binom{m}{2}}\right)^{u-r+p}
\nonumber\\
&&\hspace{3.5cm}\times\qphyp{u+2}{r+1}{u-r+p}
{q^{-m},
c,
\frac{q^{1-m}}{\bf b}
}
{
\frac{q^{1-m}}{{\bf a}}
}{q,
\frac{q^{1+p(1-m)}b_1\cdots b_u{\sf Y}}{a_1\cdots a_{r+1}
{\sf X}}}
\label{qCh1a}\\
&&\hspace{2.4cm}=
\frac{({\sf Y};q)_\infty}
{(c{\sf Y};q)_\infty}
\sum_{m=0}^\infty
\frac{(c;q)_m {\sf Y}^m}{(q;q)_m}
\qphyp{r+2}{u+1}{p}
{q^{-m},{\bf a}}
{{\bf b},\frac{q^{1-m}}{c}}
{q,\frac{q{\sf X}}
{c{\sf Y}}}.\label{qCh2a}
\end{eqnarray}
\end{cor}
\begin{proof}
Starting with Theorem \ref{BT} and setting $s=v=\ell=0$, 
${\bf c}=\{c\}$,
$c\in\CC$, ${\sf d}=\emptyset$ and 
then utilizing the $q$-binomial theorem 
\eqref{qbinom} completes the proof.
\end{proof}

There is an interesting special case if ${\sf Y}={\sf X}/c$. 
Then we obtain the following.

\begin{cor}\label{thm:2.5}
Let $q\in\CCdag$, $r\in\mathbb N_0\cup\{-1\}$, $u\in\mathbb N_0$, 
${\bf a}\in\CCast^{r+1}$, ${\bf b}\in\CCast^{u}$, $p\in\mathbb Z$, 
such that $p\ge r-u$, $c,{\sf X}\in\CCast$ such that the 
left-hand side is well-defined.
Then
\begin{eqnarray}
&&\hspace{-0.5cm}\qphyp{r+1}{u}p{{\bf a}}{{\bf b}}{q,{\sf X}}
=\frac{(\frac{1}{c}{\sf X};q)_\infty}
{({\sf X};q)_\infty}\sum_{m=0}^\infty
{\sf {\sf X}}^m
\frac{({\bf a};q)_m }{(q,{\bf b};q)_m}
\left((-1)^mq^{\binom{m}{2}}\right)^{u-r+p}
\nonumber\\
&&\hspace{3.5cm}\times\qphyp{u+2}{r+1}{u-r+p}
{q^{-m},
c,
\frac{q^{1-m}}{\bf b}
}
{\frac{q^{1-m}}{{\bf a}}}
{q,
\frac{q^{1+p(1-m)}b_1\cdots b_u}{c\,a_1\cdots a_{r+1}}}
\label{qCh1d}\\
&&\hspace{2.4cm}=
\frac{(\frac{1}{c}{\sf X};q)_\infty}
{({\sf X};q)_\infty}
\sum_{m=0}^\infty
\left(\frac{{\sf X}}{c}\right)^m\frac{(c;q)_m }{(q;q)_m}
\qphyp{r+2}{u+1}{p}
{q^{-m},
{\bf a}
}
{
{\bf b},
\frac{q^{1-m}}{c}
}{q,
q}.\label{qCh2d}\end{eqnarray}
\end{cor}
\begin{proof}
Replacing ${\sf Y}={\sf X}/c$ completes the proof.
\end{proof}

By setting $s=-1$, $v=\ell=0$ in Theorem \ref{BT}, one 
arrives at the interesting expansion of an arbitrary 
nonterminating basic hypergeometric series.

\begin{cor}
\label{BT1}
Let $q\in\CCdag$, $r\in\mathbb N_0\cup\{-1\}$, $u\in\mathbb N_0$, 
${\bf a}\in\CCast^{r+1}$, ${\bf b}\in\CCast^{u}$, 
$p\in\mathbb Z$, such that $p\ge r-u$, ${\sf X}$, 
${\sf Y}\in\CCast$ such that the left-hand side is 
well-defined. Then
\begin{eqnarray}
&&\hspace{-1.5cm}\qphyp{r+1}{u}p{{\bf a}}{{\bf b}}{q,{\sf X}}
=\frac{1}{({\sf Y};q)_\infty}\sum_{n=0}^\infty
\frac{({\bf a};q)_n {\sf {\sf X}}^n}{(q,{\bf b};q)_n}
\left((-1)^nq^{\binom{n}{2}}\right)^{u-r+p}
\nonumber\\
&&\hspace{3.5cm}\times\qphyp{u+1}{r+1}{u-r+p}
{q^{-n},
\frac{q^{1-n}}{\bf b}
}
{
\frac{q^{1-n}}{{\bf a}}
}{q,
\frac{q^{1+p(1-n)}b_1\cdots b_u{\sf Y}}{a_1\cdots a_{r+1}
{\sf X}}}
\label{qCh2-a}\\
&&\hspace{1.6cm}=\frac{1}{({\sf Y};q)_\infty}\sum_{n=0}^\infty
\frac{(-{\sf Y})^n\,q^{\binom{n}{2}}}{(q;q)_n}
\qphyp{r+2}{u}{p+1}
{q^{-n},
{\bf a}
}
{
{\bf b}
}{q,
\frac{q{\sf X}}
{{\sf Y}}}.\label{qCh3}
\end{eqnarray}
\end{cor}
\begin{proof}
Start with Theorem \ref{BT}, set $s=-1$, $v=\ell=0$ and 
using Euler's theorem (Euler's second sum) \eqref{qexp2} 
completes the proof.
\end{proof}

By setting ${\sf X}={\sf Y}$ in Theorem 
\ref{BT1}, one arrives at two further interesting 
expansions of a nonterminating basic hypergeometric series.

\begin{cor}
\label{BT2}
Let $q\in\CCdag$, $r\in\mathbb N_0\cup\{-1\}$, $u\in\mathbb N_0$, 
${\bf a}\in\CCast^{r+1}$, ${\bf b}\in\CCast^{u}$, $p\in\mathbb Z$, 
${\sf X}\in\CCast$ such that the left-hand side is 
well-defined.
Then
\begin{eqnarray}
&&\hspace{-1.5cm}\qphyp{r+1}{u}p{{\bf a}}{{\bf b}}{q,{\sf X}}
=\frac{1}{({\sf X};q)_\infty}\sum_{n=0}^\infty
\frac{({\bf a};q)_n {\sf {\sf X}}^n}{(q,{\bf b};q)_n}
\left((-1)^nq^{\binom{n}{2}}\right)^{u-r+p}
\nonumber\\
&&\hspace{3.5cm}\times\qphyp{u+1}{r+1}{u-r+p}
{q^{-n},
\frac{q^{1-n}}{\bf b}
}
{
\frac{q^{1-n}}{{\bf a}}
}{q,
\frac{q^{1+p(1-n)}b_1\cdots b_u
}{a_1\cdots a_{r+1}
}}
\label{qCh1e}\\
&&\hspace{1.6cm}=\frac{1}{({\sf X};q)_\infty}\sum_{n=0}^\infty
\frac{(-{\sf X})^n\,q^{\binom{n}{2}}}{(q;q)_n}
\qphyp{r+2}{u}{p+1}
{q^{-n},
{\bf a}
}
{
{\bf b}
}{q,
q}.\label{qCh2e}
\end{eqnarray}
\end{cor}
\begin{proof}
Setting ${\sf X}={\sf Y}$ in Theorem \ref{BT1} 
completes the proof.
\end{proof}

By setting ${\bf c}={\bf d}=\emptyset$ and $\ell=-1$ then 
one arrives at the following interesting alternative 
expansions of nonterminating basic hypergeometric series.

\begin{cor}
\label{BT3}
Let $q\in\CCdag$,
$r\in\mathbb N_0\cup\{-1\}$, $u\in\mathbb N_0$, 
${\bf a}\in\CCast^{r+1}$, ${\bf b}\in\CCast^{u}$, $p\in\mathbb Z$, 
such that functions are well-defined. Then
\begin{eqnarray}
&&\hspace{-1.0cm}\qphyp{r+1}{u}p{{\bf a}}{{\bf b}}{q,{\sf X}}
=({\sf Y};q)_\infty\sum_{n=0}^\infty\frac{({\bf a};q)_n 
{\sf {\sf X}}^n}{(q,{\bf b};q)_n}
\left((-1)^nq^{\binom{n}{2}}\right)^{u-r+p}
\nonumber\\
&&\hspace{3.5cm}\times\qphyp{u+1}{r+1}{u-r+p-1}
{q^{-n}
\frac{q^{1-n}}{\bf b}
}
{
\frac{q^{1-n}}{{\bf a}}
}{q,
\frac{q^{1+p(1-n)}b_1\cdots b_u{\sf Y}}
{a_1\cdots a_{r+1}{\sf X}}}
\label{qCh1b}\\
&&\hspace{2.0cm}=({\sf Y};q)_\infty\sum_{n=0}^\infty
\frac{{\sf Y}^n}{(q;q)_n}
\qphyp{r+2}{u}{p}
{q^{-n},
{\bf a}
}
{
{\bf b}
}{q,
\frac{q^{n}{\sf X}}
{{\sf Y}}}.\label{qCh2b}
\end{eqnarray}
\end{cor}
\begin{proof}
Starting with Theorem \ref{BT}, and
setting $s=\ell=-1$, $v=0$, using Euler's first 
sum (the $q$-binomial theorem \eqref{qbinom} 
with $a=0$) completes the proof.
\end{proof}

By setting ${\sf X}={\sf Y}$ in Theorem 
\ref{BT3}, one arrives at two further interesting 
expansions of a nonterminating basic hypergeometric series.

\begin{cor}
\label{BT4}
Let $q\in\CCdag$,
$r\in\mathbb N_0\cup\{-1\}$, $u\in\mathbb N_0$, 
${\bf a}\in\CCast^{r+1}$, ${\bf b}\in\CCast^{u}$, 
$p\in\mathbb Z$, such that functions are well-defined. Then
\begin{eqnarray}
&&\hspace{-0.5cm}\qphyp{r+1}{u}p{{\bf a}}{{\bf b}}{q,{\sf X}}
=({\sf X};q)_\infty\sum_{n=0}^\infty\frac{({\bf a};q)_n 
{\sf {\sf X}}^n}{(q,{\bf b};q)_n}
\left((-1)^nq^{\binom{n}{2}}\right)^{u-r+p}
\nonumber\\
&&\hspace{5.0cm}\times\qphyp{u+1}{r+1}{u-r+p-1}
{q^{-n}
\frac{q^{1-n}}{\bf b}
}
{
\frac{q^{1-n}}{{\bf a}}
}{q,
\frac{q^{1+p(1-n)}b_1\cdots b_u
}{a_1\cdots a_{r+1}
}}
\label{qCh1c}\\
&&\hspace{2.7cm}=({\sf X};q)_\infty\sum_{n=0}^\infty
\frac{{\sf X}^n}{(q;q)_n}
\qphyp{r+2}{u}{p}
{q^{-n},
{\bf a}
}
{
{\bf b}
}{q,q^{n}}.\label{qCh2c}
\end{eqnarray}
\end{cor}
\begin{proof}
Setting ${\sf Y}={\sf X}$ in Theorem \ref{BT3} completes 
the proof.
\end{proof}

By using the $q$-Gauss summation \eqref{qGs}, we arrive 
at an alternative representation for an arbitrary argument 
basic hypergeometric series.

\begin{cor}
\label{qG}
Let $q\in\CCdag$, $r\in\mathbb N_0\cup\{-1\}$, $u\in\mathbb N_0$, 
${\bf a}\in\CCast^{r+1}$, ${\bf b}\in\CCast^{u}$, 
$c,a,b\in\CCast$. Then
\begin{eqnarray}
&&\hspace{-0.5cm}\qphyp{r+1}{u}p{{\bf a}}{{\bf b}}{q,{\sf X}}
=\frac
{(c,\frac{c}{ab};q)_\infty}
{(\frac{c}{a},\frac{c}{b};q)_\infty}
\sum_{n=0}^\infty\frac{({\bf a};q)_n {\sf {\sf X}}^n}
{(q,{\bf b};q)_n}
\left((-1)^nq^{\binom{n}{2}}\right)^{u-r+p}
\nonumber\\
&&\hspace{3.5cm}\times\qphyp{u+3}{r+2}{u-r+p}
{q^{-n},
a,b,
\frac{q^{1-n}}{\bf b}
}
{
c,
\frac{q^{1-n}}{{\bf a}}
}{q,
\frac{q^{1+p(1-n)}c\,b_1\cdots b_u}{ab\,a_1\cdots a_{r+1}
{\sf X}}}
\label{qG1}\\
&&\hspace{2.0cm}
=
\frac{(c,\frac{c}{ab};q)_\infty}
{(\frac{c}{a},\frac{c}{b};q)_\infty}
\sum_{n=0}^\infty
\frac{(a,b;q)_n \left(\frac{c}{ab}\right)^n}{(q,c;q)_n}
\qhyp{r+v+2}{u+2}
{q^{-n},
{\bf a},
\frac{q^{1-n}ab}{c}
}
{
{\bf b},
\frac{q^{1-n}}{a}
\frac{q^{1-n}}{b}
}{q,
q{\sf X}},\label{qG2}
\end{eqnarray}
where ${\sf X}$ is given 
such that the left-hand side is well-defined.
\end{cor}
\begin{proof}
Start with the $q$-Chaundy Theorem \ref{BT}, setting
$s=v=1$, $\ell=0$, ${\bf c}=\{a,b\}$, ${\bf d}=\{c\}$, 
${\sf Y}=c/(ab)$ and using the $q$-Gauss sum 
\eqref{qGs} completes the proof.
\end{proof}

{
By using the Cauchy's summation \eqref{qC}, we arrive at
alternative representations for an arbitrary argument 
basic hypergeometric series.
\begin{cor}
\label{qCT}
Let $q\in\CCdag$, $r\in\mathbb N_0\cup\{-1\}$, $u\in\mathbb N_0$, 
${\bf a}\in\CCast^{r+1}$, ${\bf b}\in\CCast^{u}$, 
$c,d\in\CCast$. Then
\begin{eqnarray}
&&\hspace{-0.5cm}\qphyp{r+1}{u}p{{\bf a}}{{\bf b}}{q,{\sf X}}
=\frac
{(d;q)_\infty}
{(\frac{d}{c};q)_\infty}\sum_{n=0}^\infty\frac{({\bf a};q)_n 
{\sf X}^n}{(q,{\bf b};q)_n}
\left((-1)^nq^{\binom{n}{2}}\right)^{u-r+p}
\nonumber\\
&&\hspace{3.5cm}\times\qphyp{u+2}{r+2}{u-r+p}
{q^{-n},
c,
\frac{q^{1-n}}{\bf b}
}
{
d
\frac{q^{1-n}}{{\bf a}}
}{q,
\frac{q^{1+p(1-n)}c\,b_1\cdots b_u\,d}
{a_1\cdots a_{r+1}\,c\,{\sf X}}}
\label{qC1}\\
&&\hspace{2.0cm}
=
\frac{(d;q)_\infty}
{(\frac{d}{c};q)_\infty}
\sum_{n=0}^\infty
\frac{(c;q)_n q^{\binom{n}{2}}}{(q,d;q)_n}
\left(-\frac{d}{c}\right)^n
\qphyp{r+3}{u+1}{p+1}
{q^{-n},
{\bf a},
\frac{q^{1-n}}{d}
}
{
{\bf b},
\frac{q^{1-n}}{c}
}{q,
q{\sf X}},\label{qG2-a}
\end{eqnarray}
where ${\sf X}$ is given 
such that the left-hand side is well-defined.
\end{cor}
\begin{proof}
Start with the $q$-Chaundy Theorem \ref{BT}, setting
$v=1$, $s=\ell=0$, ${\bf c}=\{c\}$, ${\bf d}=\{d\}$, 
${\sf Y}=d/c$ and using the Cauchy sum \eqref{qC} 
completes the proof.
\end{proof}
}

There are potentially many applications of such theorems. 
One obvious implication is to use the 
$q$-Chu--Vandermonde summation \eqref{qChuVander} to obtain 
a representation of a nonterminating basic 
hypergeometric series.

\begin{cor}
Let $n,v\in\mathbb N_0$, $s\in\mathbb N_0\cup\{-1\}$, 
${\bf a}\in\CCast^{r}$, ${\bf b}\in\CCast^{u}$, 
${\bf c}\in\CCast^{s+1}$, ${\bf d}\in\CCast^{v}$, 
$q\in\CCdag$,
${\sf Y}\in\CCast$. Then
\begin{eqnarray}
&&\hspace{-0.1cm}
\qphyp{s+1}{v}{\ell}{{\bf c}}{{\bf d}}{q,{\sf Y}}
\!=\!\frac{(b;q)_n}{a^n(\frac{b}{a};q)_n}
\sum_{m=0}^nq^m\frac{(q^{-n},a;q)_m}{(q,b;q)_m}
\qphyp{s+3}{v+2}{\ell}
{q^{-m},{\bf c},\frac{q^{1-m}}b}
{q^{n-m+1},{\bf d},\frac{q^{1-m}}{a}}
{q,\frac{q^{n+1}b_1\cdots b_u{\sf Y}}
{a_1\cdots a_rq}}\nonumber\\[0.1cm]
&&\hspace{3.0cm}+\frac{(a;q)_n}{(b;q)_n}
\left((-1)^n q^{\binom{n}{2}}
\right)^{-1}
\sum_{m=1}^{\infty}{\sf Y}^m\frac{({\bf c};q)_m}
{(q,{\bf d};q)_m}\left((-1)^m q^{\binom{m}{2}}\right)^{v-s+\ell}
\nonumber\\
&&\hspace{5.5cm}\times\qphyp{s+3}{v+2}{\ell}
{q^{-n},
q^m{\bf c},
\frac{q^{1-n}}b}
{q^{m+1},q^m{\bf d},\frac{q^{1-n}}{a}}
{q,\frac{q^{n+m(v-s+\ell)}b{\sf Y}}{a}}
\label{tttqCh1a}
\end{eqnarray}
\begin{eqnarray}
&&\hspace{2.0cm}=
\frac{(b;q)_n}{a^{n}(\frac{b}{a};q)_n}
\sum_{m=0}^\infty{\sf Y}^m
\frac{({\bf c};q)_m}{(q,{\bf d};q)_m}
\left((-1)^mq^{\binom{m}{2}}\right)^{v-s+\ell}
\nonumber\\[-0.1cm]&&\hspace{5.5cm}\times
\qphyp{v+3}{s+2}{v-s+\ell}
{q^{-m},
q^{-n},a,
\frac{q^{1-m}}{\bf d}
}
{
b,
\frac{q^{1-m}}{{\bf c}}
}{q,
\frac{q^{2+\ell(1-m)}d_1\cdots d_v}
{c_1\cdots c_s{\sf Y}}}.\label{tttqCh2a}
\end{eqnarray}
\end{cor}
\begin{proof}
Start with Theorem \ref{thm:4.1} and for the terminating 
basic hypergeometric series substitute the 
$q$-Chu--Vandermonde summation \eqref{qChuVander}, i.e., 
setting $r=1$, $u=1$, $p=0$, ${\bf a}=\{a\}$, ${\bf b}=\{b\}$, 
${\sf X}=q$. Simplification completes the proof.
\end{proof}

Another application of Theorem \ref{thm:4.1} is the
generation of an expansion of a nonterminating
basic hypergeometric series using the 
$q$-Pfaff--Saalsch\"utz sum \eqref{qPS}.

\begin{cor}
\label{cor43}
Let $n,v\in\mathbb N_0$, $s\in\mathbb N_0\cup\{-1\}$, $\ell\in\mathbb Z$, 
such that $\ell\ge s-v$, 
${\bf c}\in\CCast^{s+1}$, ${\bf d}\in\CCast^{v}$, 
$q\in\CCdag$,
$a,b,c,{\sf Y}\in\CCast$. Then
\begin{eqnarray}
&&\hspace{-0.5cm}
\qphyp{s+1}{v}{\ell}{{\bf c}}{{\bf d}}{q,{\sf Y}}
=
\frac{(c,\frac{c}{ab};q)_n}{(\frac{c}{a},\frac{c}{b};q)_n}
\sum_{m=0}^nq^m\frac{(q^{-n},a,b;q)_m}{(q,c,q^{1-n}\frac{ab}{c};q)_m}
\nonumber\\[-0.0cm]&&\hspace{3.5cm}\times
\qphyp{s+4}{v+3}{\ell}
{q^{-m},
\frac{q^{1-m}}{c},
q^{n-m}\frac{c}{ab},
{\bf c}
}
{
q^{n-m+1},
\frac{q^{1-m}}{a},
\frac{q^{1-m}}{b},
{\bf d}
}{q,
\frac{q^{n+1+p(1-m)}b_1\cdots b_u{\sf Y}}{a_1\cdots a_rq}}\nonumber\\[0.1cm]
&&\hspace{0.0cm}+\frac{(a,b;q)_n}{(c,q^{1-n}\frac{ab}{c};q)_n}
\left((-1)^n q^{\binom{n}{2}}
\right)^{-1}
\sum_{m=1}^{\infty}{\sf Y}^m\frac{({\bf c};q)_m}{(q,{\bf d};q)_m}\left((-1)^m q^{\binom{m}{2}}\right)^{v-s+\ell}
\nonumber\\
&&\hspace{3.1cm}\times\qphyp{s+4}{v+3}{\ell}
{q^{-n},
\frac{q^{1-n}}{c},
\frac{c}{ab},
q^m{\bf c}
}
{
q^{m+1},
\frac{q^{1-n}}{a},
\frac{q^{1-n}}{b},
q^m{\bf d}
}{q,
\frac{q^{n+1+m(v-s+\ell)}b_1\cdots b_u{\sf Y}}{a_1\cdots a_rq}}
\label{tttqCh1b}\\
&&\hspace{0.0cm}=
\frac{(c,\frac{c}{ab};q)_n}{(\frac{c}{a},\frac{c}{b};q)_n}
\sum_{m=0}^\infty{\sf Y}^m
\frac{({\bf c};q)_m}{(q,{\bf d};q)_m}
\left((-1)^mq^{\binom{m}{2}}\right)^{v-s+\ell}
\nonumber\\[-0.1cm]&&\hspace{3.5cm}\times
\qphyp{v+4}{s+3}{v-s+\ell}
{q^{-m},
q^{-n},a,b,
\frac{q^{1-m}}{\bf d}
}
{
c,q^{1-n}\frac{ab}{c},
\frac{q^{1-m}}{{\bf c}}
}{q,
\frac{q^{2+\ell(1-m)}d_1\cdots d_v}
{c_1\cdots c_s{\sf Y}}}.\label{tttqCh2b}
\end{eqnarray}
\end{cor}

\begin{proof}
Start with Theorem \ref{thm:4.1} and for the terminating basic hypergeometric series substitute the $q$-Pfaff--Saalsch\"utz sum \eqref{qPS}, i.e., setting $r=2$, $u=2$, $p=0$, ${\bf a}=\{a,b\}$, ${\bf b}=\{c,q^{n-1}ab/c\}$, ${\sf X}=q$. Simplification completes the proof.
\end{proof}

\subsection{Double product nonterminating/terminating representations}\label{sec:4.2}
Now we consider the case of a product where one of the 
van de Bult--Rains basic hypergeometric series is 
terminating and the other is nonterminating.

\begin{thm}
\label{thm:4.1}
Let $n,u,v\in\mathbb N_0$, $r,s\in\mathbb N_0\cup\{-1\}$, 
$p,\ell\in\mathbb Z$, such that $\ell\ge s-v$,
${\bf a}\in\CCast^{r}$, ${\bf b}\in\CCast^{u}$, 
${\bf c}\in\CCast^{s+1}$, ${\bf d}\in\CCast^{v}$, 
$q\in\CCdag$,
${\sf X},{\sf Y}\in\CCast$. Then
\begin{eqnarray}
&&\hspace{-0.5cm}\qphyp{r+1}{u}p{q^{-n},{\bf a}}
{{\bf b}}{q,{\sf X}}
\qphyp{s+1}{v}{\ell}{{\bf c}}{{\bf d}}{q,{\sf Y}}\nonumber\\
&&\hspace{0cm}=\sum_{m=0}^n{\sf X}^m
\frac{(q^{-n},{\bf a};q)_m}{(q,{\bf b};q)_m}
\left((-1)^mq^{\binom{m}{2}}\right)^{u-r+p}
\nonumber\\[-0.1cm]
&&\hspace{3.5cm}\times\qphyp{s+u+2}{r+v+1}{u-r+p+\ell}
{q^{-m},
{\bf c},
\frac{q^{1-m}}{\bf b}
}
{
q^{n-m+1},
{\bf d},
\frac{q^{1-m}}{{\bf a}}
}{q,
\frac{q^{n+1+p(1-m)}b_1\cdots b_u{\sf Y}}
{a_1\cdots a_r{\sf X}}}\nonumber\\[0.1cm]
&&\hspace{0.0cm}+\left(\frac{\sf X}{q}\right)^n
\frac{({\bf a};q)_n}{({\bf b};q)_n}
\left((-1)^n q^{\binom{n}{2}}
\right)^{u-r+p-1}
\sum_{m=1}^{\infty}{\sf Y}^m\frac{({\bf c};q)_m}
{(q,{\bf d};q)_m}\left((-1)^m q^{\binom{m}{2}}\right)^{v-s+\ell}
\nonumber\\
&&\hspace{1.1cm}\times\qphyp{s+u+2}{r+v+1}{u-r+p+\ell}
{q^{-n},
q^m{\bf c},
\frac{q^{1-n}}{\bf b}
}
{
q^{m+1},
q^m{\bf d},
\frac{q^{1-n}}{{\bf a}}
}{q,
\frac{q^{n+1+p(1-n)+m(v-s+\ell)}b_1\cdots b_u{\sf Y}}
{a_1\cdots a_r{\sf X}}}
\label{tttqCh1}\\
&&\hspace{0.0cm}=\sum_{m=0}^\infty{\sf Y}^m
\frac{({\bf c};q)_m}{(q,{\bf d};q)_m}
\left((-1)^mq^{\binom{m}{2}}\right)^{v-s+\ell}
\nonumber\\[-0.1cm]&&\hspace{3.5cm}\times
\qphyp{r+v+2}{s+u+1}{v-s+p+\ell}
{q^{-m},
q^{-n},{\bf a},
\frac{q^{1-m}}{\bf d}
}
{
{\bf b},
\frac{q^{1-m}}{{\bf c}}
}{q,
\frac{q^{1+\ell(1-m)}d_1\cdots d_v{\sf X}}
{c_1\cdots c_s{\sf Y}}}.\label{tttqCh2}
\end{eqnarray}
\end{thm}
\begin{proof}
Let $a_{r+1}\in\CCast$.
Start with the nonterminating representation for 
the left-hand side of \eqref{tqCh1} with
$q^{-n}\mapsto a_{r+1}$
given in Theorem \ref{BT}. 
Now consider the limit in the above sum as 
$a_{r+1}\to q^{-n}$, and redefine ${\bf a}$ as 
in the statement of the theorem.
For \eqref{qCh1}, the straightforward application of 
Theorem \ref{BT} can be written as the sum of two terms, 
the first over $m\in\{0,\ldots,n\}$ and a second over 
$m\in\mathbb N+n$. 
Now in the second term shift the second sum index so 
that it is over $m\in\mathbb N$, namely $m\mapsto m+n$. 
The first sum appears as the first term on the 
right-hand side of \eqref{tqCh1}. Now write the 
terminating basic hypergeometric series in the second 
term as a sum over $j\in\{0,\ldots,m+n\}$. In the limit, 
there is a singular term 
$(q^na_{r+1};q)_m/(q^{1-m-n}/a_{r+1};q)_j$ which is 
only non-vanishing for $j\in\{m,\ldots,m+n\}$. Now shift 
the sum over the terminating basic hypergeometric series 
so that it is over $j\in\{0,\ldots,n\}$, namely $j\mapsto j+m$. 
After simplification, the singular term appears as 
$(q^nb;q)_m/(q^{1-m-n}/b;q)_m$, which in the limit becomes 
$(-1)^mq^{\binom{m}{2}}$. Simplification produces 
\eqref{tttqCh1}. In order to obtain \eqref{tttqCh2}, 
starting with 
\eqref{qCh2} produces a result that is well-behaved for 
all $m\in\mathbb N_0$, so the result is a single term, which 
completes the proof.
\end{proof}

One also has the useful consequence of the above result 
in the special case where $p=\ell=0$.

\begin{cor}
\label{cor42}
Let $n,u,v\in\mathbb N_0$, $r,s\in\mathbb N_0\cup\{-1\}$, 
${\bf a}\in\CCast^{r}$, ${\bf b}\in\CCast^{u}$, 
${\bf c}\in\CCast^{s+1}$, ${\bf d}\in\CCast^{v}$, 
$q\in\CCdag$,
${\sf X},{\sf Y}\in\CCast$. Then
\begin{eqnarray}
&&\hspace{-0.2cm}\qhyp{r+1}{u}{q^{-n},{\bf a}}
{{\bf b}}{q,{\sf X}}
\qhyp{s+1}{v}{{\bf c}}{{\bf d}}{q,{\sf Y}}\nonumber\\
&&\hspace{0.2cm}=\sum_{m=0}^n{\sf X}^m\frac{(q^{-n},
{\bf a};q)_m}{(q,{\bf b};q)_m}
\left((-1)^mq^{\binom{m}{2}}\right)^{u-r}
\qphyp{s+u+2}{r+v+1}{u-r}
{q^{-m},
{\bf c},
\frac{q^{1-m}}{\bf b}
}
{
q^{n-m+1},
{\bf d},
\frac{q^{1-m}}{{\bf a}}
}{q,
\frac{q^{n+1}b_1\cdots b_u{\sf Y}}{a_1\cdots a_r
{\sf X}}}\nonumber\\[0.1cm]
&&\hspace{1.3cm}+\left(\frac{\sf X}{q}\right)^n
\frac{({\bf a};q)_n}{({\bf b};q)_n}
\left((-1)^n q^{\binom{n}{2}}
\right)^{u-r-1}
\sum_{m=1}^{\infty}{\sf Y}^m\frac{({\bf c};q)_m}
{(q,{\bf d};q)_m}\left((-1)^m q^{\binom{m}{2}}\right)^{v-s}
\nonumber\\
&&\hspace{4.5cm}\times\qphyp{s+u+2}{r+v+1}{u-r}
{q^{-n},
q^m{\bf c},
\frac{q^{1-n}}{\bf b}
}
{
q^{m+1},
q^m{\bf d},
\frac{q^{1-n}}{{\bf a}}
}{q,
\frac{q^{n+1+m(v-s)}b_1\cdots b_u{\sf Y}}{a_1\cdots a_r{\sf X}}}
\label{tttttqCh1}\\
&&\hspace{0.2cm}=\sum_{m=0}^\infty{\sf Y}^m
\frac{({\bf c};q)_m}{(q,{\bf d};q)_m}
\left((-1)^mq^{\binom{m}{2}}\right)^{v-s}
\qphyp{r+v+2}{s+u+1}{v-s}
{q^{-m},
q^{-n},{\bf a},
\frac{q^{1-m}}{\bf d}
}
{
{\bf b},
\frac{q^{1-m}}{{\bf c}}
}{q,
\frac{qd_1\cdots d_v{\sf X}}
{c_1\cdots c_s{\sf Y}}}.\label{tttttqCh2}
\end{eqnarray}
\end{cor}
\begin{proof}
Simply replacing $p=\ell=0$ in Theorem \ref{thm:4.1} 
completes the proof.
\end{proof}

\subsection{Double product terminating representations}\label{sec:4.3}
We now present the general theorem for an expansion of 
two arbitrarily shaped terminating basic hypergeometric series.

\begin{thm}
\label{BTt}
Let $n,k\in\mathbb Z$, $r,s\in\mathbb N_0\cup\{-1\}$, $u,v\in\mathbb N_0$, 
${\bf a}\in\CCast^{r}$, ${\bf b}\in\CCast^{u}$, 
${\bf c}\in\CCast^{s}$, ${\bf d}\in\CCast^{v}$, 
$q\in\CCdag$,
${\sf X},{\sf Y}\in\CCast$. Then
\begin{eqnarray}
&&\hspace{-0.5cm}\qphyp{r+1}{u}p{q^{-n},{\bf a}}{{\bf b}}{q,{\sf X}}
\qphyp{s+1}{v}{\ell}{q^{-k},{\bf c}}{{\bf d}}{q,{\sf Y}}\nonumber\\
&&\hspace{0cm}=\sum_{m=0}^n{\sf X}^m\frac{(q^{-n},{\bf a};q)_m}{(q,{\bf b};q)_m}
\left((-1)^mq^{\binom{m}{2}}\right)^{u-r+p}
\nonumber\\[-0.1cm]
&&\hspace{3.5cm}\times\qphyp{s+u+2}{r+v+1}{u-r+p+\ell}
{q^{-m},
q^{-k},{\bf c},
\frac{q^{1-m}}{\bf b}
}
{
q^{n-m+1},
{\bf d},
\frac{q^{1-m}}{{\bf a}}
}{q,
\frac{q^{n+1+p(1-m)}b_1\cdots b_u{\sf Y}}
{a_1\cdots a_r{\sf X}}}\nonumber\\[0.1cm]
&&\hspace{0.0cm}+\left(\frac{\sf X}{q}\right)^n
\frac{({\bf a};q)_n}{({\bf b};q)_n}
\left((-1)^n q^{\binom{n}{2}}\right)^{u-r+p-1}
\sum_{m=1}^{k}{\sf Y}^m\frac{(q^{-k},{\bf c};q)_m}
{(q,{\bf d};q)_m}\left((-1)^m q^{\binom{m}{2}}\right)^{v-s+\ell}
\nonumber\\
&&\hspace{1.1cm}\times\qphyp{s+u+2}{r+v+1}{u-r+p+\ell}
{q^{-n},
q^{m-k},q^m{\bf c},
\frac{q^{1-n}}{\bf b}
}
{
q^{m+1},
q^m{\bf d},
\frac{q^{1-n}}{{\bf a}}
}{q,
\frac{q^{n+1+p(1-n)+m(v-s+\ell)}b_1\cdots b_u{\sf Y}}{a_1\cdots a_r{\sf X}}}
\label{ttqCh1}\\
&&\hspace{0.0cm}=\sum_{m=0}^k{\sf Y}^m
\frac{(q^{-k},{\bf c};q)_m }{(q,{\bf d};q)_m}
\left((-1)^mq^{\binom{m}{2}}\right)^{v-s+\ell}
\nonumber\\[-0.1cm]
&&\hspace{3.5cm}\times\qphyp{r+v+2}{s+u+1}{v-s+p+\ell}
{q^{-m},
q^{-n},{\bf a},
\frac{q^{1-m}}{\bf d}
}
{
q^{k-m+1},
{\bf b},
\frac{q^{1-m}}{{\bf c}}
}{q,
\frac{q^{k+1+\ell(1-m)}d_1\cdots d_v{\sf X}}
{c_1\cdots c_s{\sf Y}}}\nonumber\\
&&+\left(\frac{{\sf Y}}{q}\right)^k\frac{({\bf c};q)_k}{({\bf d};q)_k}\left((-1)^kq^{\binom{k}{2}}\right)^{v-s+\ell-1}\sum_{m=1}^{n}{\sf X}^m
\frac{(q^{-n},{\bf a};q)_m }{(q,{\bf b};q)_m}
\left((-1)^mq^{\binom{m}{2}}\right)^{u-r+p}
\nonumber\\[0.1cm]
&&\hspace{1.1cm}\times\qphyp{r+v+2}{s+u+1}{v-s+p+\ell}
{q^{-k},
q^{m-n},q^m{\bf a},
\frac{q^{1-k}}{\bf d}
}
{
q^{m+1},
q^m{\bf b},
\frac{q^{1-k}}{{\bf c}}
}{q,
\frac{q^{k+1+\ell(1-k)+m(u-r+p)}d_1\cdots d_v{\sf X}}
{c_1\cdots c_s{\sf Y}}}.\label{ttqCh2}
\end{eqnarray}
\end{thm}
\begin{proof}
Let $a_{r+1},c_{s+1}\in\CCast$.
Start with the nonterminating representation for the left-hand side 
of \eqref{tqCh1}
with
$(q^{-n},q^{-k})\mapsto (a_{r+1}, c_{s+1})$
given in Theorem \ref{BT}. 
Now consider the limit in the above sum as $(a_{r+1},c_{s+1})\to (q^{-n},q^{-k})$, and redefine
${\bf a}$, ${\bf c}$ as 
in the statement of the theorem.
For \eqref{qCh1}, the straightforward application of Theorem \ref{BT} can be written as the sum of two terms, the first over $m\in\{0,\ldots,n\}$ and a second over $m\in\{n+1,n+k\}$. 
Now in the second term shift the second sum index so that it is over $m\in\{1,\ldots,k\}$, namely $m\mapsto m+n$. 
The first sum appears as the first term on the right-hand side of \eqref{tqCh1}. Now write the terminating basic hypergeometric series in the second term as a sum over $j\in\{0,\ldots,k\}$. In the limit, there is a singular term $(q^na_{r+1};q)_m/(q^{1-m-n}/a_{r+1};q)_j$ which is only non-vanishing for $j\in\{m,\ldots,m+k\}$. Now shift the sum over the terminating basic hypergeometric series so that it is over $j\in\{0,k-m\}$, namely $j\mapsto j+m$. 
After simplification, the singular term appears as 
$(q^nb;q)_m/(q^{1-m-n}/b;q)_m$, which in the limit 
becomes $(-1)^mq^{\binom{m}{2}}$. Simplification produces 
\eqref{ttqCh1}. In order to obtain \eqref{ttqCh2}, start 
with the right-hand side of \eqref{ttqCh1} and replacing 
\[({\bf a},{\bf b},{\bf c},
{\bf d},n,k,r,u,p,s,v,\ell)\mapsto
({\bf c},{\bf d},{\bf a},{\bf b},k,n,s,v,\ell,r,u,p),
\]
completes the proof.
\end{proof}

Now we present a some results for the product of two 
arbitrary size terminating basic hypergeometric functions 
${}_{r+1}\phi_r(q^{-n},{\bf a};{\bf b};q,{\sf X})
{}_{s+1}\phi_s(q^{-k},{\bf c};{\bf d};q,{\sf Y})$.

\begin{cor}
Let $n,k\in\mathbb N_0$, $q\in\CCdag$, 
${\bf a}:=\{a_1,\ldots,a_r\}$,
${\bf b}:=\{b_1,\ldots,b_r\}$, 
${\bf c}:=\{c_1,\ldots,c_s\}$,
${\bf d}:=\{d_1,\ldots,d_s\}$, 
$a_j,b_j,c_k,d_k,{\sf X},{\sf Y}\in\CCast$ such that 
$|{\sf X}|,|{\sf Y}|<1$, $j=1,\ldots,r$, $k=1,\ldots,s$.
Then
\begin{eqnarray}
&&\hspace{-1.0cm}\qhyp{r+1}{r}{q^{-n},{\bf a}}
{{\bf b}}{q,{\sf X}}
\qhyp{s+1}{s}{q^{-k},{\bf c}}{{\bf d}}{q,{\sf Y}}
\nonumber\\
&&\hspace{-0.5cm}
=\sum_{m=0}^n{\sf X}^m\frac{(q^{-n},{\bf a};q)_m }
{(q,{\bf b};q)_m}
\qhyp{r+s+2}{r+s+1}
{q^{-m},
q^{-k},{\bf c},
\frac{q^{1-m}}{\bf b}
}
{
q^{n-m+1},
{\bf d},
\frac{q^{1-m}}{{\bf a}}
}{q,
\frac{q^{n+1}b_1\cdots b_r{\sf Y}}{a_1\cdots a_{r}
{\sf X}}}\nonumber\\
&&\hspace{0.5cm}+q^{-\binom{n}{2}}
\left(\frac{-\sf X}{q}\right)^n\frac{({\bf a};q)_n}
{({\bf b};q)_n}\sum_{m=1}^k{\sf Y}^m
\frac{(q^{-k},{\bf c};q)_m }{(q,{\bf d};q)_m}\nonumber\\
&&\hspace{4cm}\times\qhyp{r+s+2}{r+s+1}
{q^{-n},
q^{m-k},q^m{\bf c},
\frac{q^{1-n}}{\bf b}
}
{
q^{m+1},
q^m{\bf d},
\frac{q^{1-n}}{{\bf a}}
}{q,
\frac{q^{n+1}b_1\cdots b_r{\sf Y}}{a_1\cdots a_{r}{\sf X}}}
\label{tqCh1}\\
&&\hspace{-0.5cm}=\sum_{m=0}^k{\sf Y}^m
\frac{(q^{-k},{\bf c};q)_m }{(q,{\bf d};q)_m}
\qhyp{r+s+2}{r+s+1}
{q^{-m},q^{-n},
{\bf a},
\frac{q^{1-m}}{\bf d}
}
{
q^{k-m+1},
{\bf b},
\frac{q^{1-m}}{\bf c}
}{q,
\frac{q^{k+1}d_1\cdots d_s{\sf X}}
{c_1\cdots c_{s}{\sf Y}}}\nonumber\\
&&\hspace{0.5cm}+q^{-\binom{k}{2}}
\left(\frac{-\sf Y}{q}\right)^k\frac{({\bf c};q)_k}
{({\bf d};q)_k}
\sum_{m=1}^n{\sf X}^m
\frac{(q^{-n},{\bf a};q)_m }{(q,{\bf b};q)_m}\nonumber\\
&&\hspace{4cm}\times
\qhyp{r+s+2}{r+s+1}
{q^{-k},q^{m-n},
q^m{\bf a},
\frac{q^{1-k}}{\bf d}
}
{
q^{m+1},
q^m{\bf b},
\frac{q^{1-k}}{\bf c}
}{q,
\frac{q^{k+1}d_1\cdots d_s{\sf X}}
{c_1\cdots c_{s}{\sf Y}}}
.\label{tqCh2}
\end{eqnarray}
\end{cor}
\begin{proof}
Let $a_{r+1},c_{s+1}\in\CCast$.
Start with the nonterminating representation for 
the left-hand side 
of \eqref{tqCh1}
with
$(q^{-n},q^{-k})\mapsto (a_{r+1}, c_{s+1})$
given in Theorem \ref{BT} with $p=\ell=0$, $u=r$, $v=s$. 
Now consider the limit in the above sum as 
$a_{r+1}\to q^{-n}$, $c_{s+1}\to q^{-k}$ and redefine
${\bf a}$, ${\bf c}$ as 
in the statement of the theorem.
For each of \eqref{qCh1}, \eqref{qCh2}, the 
straightforward application of Theorem \ref{BT} 
can be written as two sums, the first over 
$m\in\{0,\ldots,n\}$ (resp.~$m\in\{0,\ldots,k\}$) 
and a second over $m\in\{n+1,n+k\}$ 
(resp.~$m\in\{k+1,n+k\}$). 
Now in the second term shift the second sum index so 
that it is over $m\in\{1,\ldots,k\}$ 
(resp.~$m\in\{1,\ldots,n\}$), namely 
$m\mapsto m+n$ (resp.~$m\mapsto m+k$). 
The first sum appears as the first term on the 
right-hand side of \eqref{tqCh1} (resp.~\eqref{tqCh2}). 
Now write the terminating basic hypergeometric series 
in the second term as a sum over $\in\{0,\ldots,k\}$ 
(resp.~$p\in\{0,\ldots,n\}$). In the limit, there is a 
singular term $(q^na_{r+1};q)_m/(q^{1-m-n}/a_{r+1};q)_p$ 
(resp.~$(q^kc_{s+1};q)_m/(q^{1-m-k}/c_{s+1};q)_p$) 
which is only non-vanishing for $p\in\{m,\ldots,m+k\}$ 
(resp.~$p\in\{m,\ldots,m+n\}$). Now shift the sum over the 
terminating basic hypergeometric series so that it is over 
$p\in\{0,k-m\}$ (resp.~$p\in\{0,n-m\}$), namely 
$p\mapsto p+m$. 
After simplification, the singular term appears as 
$(q^nb;q)_m/(q^{1-m-n}/b;q)_m$ 
(resp.~$(q^kc_{s+1};q)_m/(q^{1-m-k}/c_{s+1};q)_m$, which 
after cancelling the vanishing factors in the numerator 
and denominator reduces to $(-1)^mq^{\binom{m}{2}}$. 
Simplifying the result completes the proof.
\end{proof}

Here are some examples of double product terminating basic hypergeometric functions.
Now we present an example for a product of two 
terminating ${}_4\phi_3$'s which will be useful in 
computing a product formula for the Askey--Wilson polynomials.
\begin{cor}
\label{thm42}
Let $n,k\in\mathbb N_0$, $q\in\CCdag$, ${\bf a}:=\{a_1,a_2,a_3\}$,
${\bf b}:=\{b_1,b_2,b_3\}$, 
${\bf c}:=\{c_1,c_2,c_3\}$,
${\bf d}:=\{d_1,d_2,d_3\}$, 
$a_k,b_k,c_k,d_k,{\sf X},{\sf Y}\in\CCast$ such that 
$|{\sf X}|,|{\sf Y}|<1$, $k=1,2,3$. Then
\begin{eqnarray}
&&\hspace{-0.9cm}
\qhyp43{q^{-n},{\bf a}}{{\bf b}}{q,{\sf X}}
\qhyp43{q^{-k},{\bf c}}{{\bf d}}{q,{\sf Y}}
\nonumber\\
&&\hspace{-0.7cm}
=\sum_{m=0}^n
\frac{{\sf X}^m(q^{-n},{\bf a};q)_m}{(q,{\bf b};q)_m}
\qhyp87{q^{-m},q^{-k},{\bf c},\frac{q^{1-m}}{\bf b}}
{q^{n-m+1},{\bf d},\frac{q^{1-m}}{\bf a}}
{q,\frac{q^{n+1}b_1b_2b_3{\sf Y}}{a_1a_2a_3{\sf X}}}
\nonumber\\
&&\hspace{-0.7cm}+q^{-\binom{n}{2}}\left(\frac{-\sf X}
{q}\right)^n\frac{({\bf a};q)_n}{({\bf b};q)_n}
\sum_{m=1}^k
\frac{{\sf Y}^m(q^{-k},{\bf c};q)_m}{(q,{\bf d};q)_m}
\qhyp87{q^{-n},q^{m-k},q^m{\bf c},\frac{1-n}{\bf b}}
{q^{m+1},q^m{\bf d},\frac{q^{1-n}}{\bf a}}
{q,\frac{q^{n+1}b_1b_2b_3{\sf Y}}{a_1a_2a_3{\sf X}}}
\label{prod4phi3s}\\
&&\hspace{-0.7cm}=\sum_{m=0}^k
\frac{{\sf Y}^m(q^{-k},{\bf c};q)_m }
{(q,{\bf d};q)_m}\qhyp{8}{7}
{q^{-m},q^{-n},
{\bf a},
\frac{q^{1-m}}{\bf d}
}
{q^{k-m+1},{\bf b},\frac{q^{1-m}}{\bf c}}
{q,\frac{q^{k+1}d_1d_2d_3{\sf X}}
{c_1c_2c_{3}{\sf Y}}}
\nonumber\\&&\hspace{-0.7cm}
+q^{-\binom{k}{2}}
\left(\frac{-\sf Y}{q}\right)^k\frac{({\bf c};q)_k}
{({\bf d};q)_k} \sum_{m=1}^n
\frac{{\sf X}^m(q^{-n},{\bf a};q)_m }{(q,{\bf b};q)_m}
\qhyp{8}{7}
{q^{-k},q^{m-n},
q^m{\bf a},
\frac{q^{1-k}}{\bf d}
}
{q^{m+1}, q^m{\bf b},
\frac{q^{1-k}}{\bf c}
}{q,
\frac{q^{k+1}d_1d_2d_3{\sf X}}
{c_1c_2c_{3}{\sf Y}}}
\!\!.
\label{prod4phi3t}
\end{eqnarray}
\end{cor}
\begin{proof}
Define ${\bf a}:=\{a_1,a_2,a_3,a_4\}$,
${\bf c}:=\{c_1,c_2,c_4,c_4\}$ and
${\bf b}$, ${\bf d}$ as above. Then one has through 
Theorem \ref{BT}
\begin{eqnarray}
\hspace{-0.4cm}\qhyp43{{\bf a}}{{\bf b}}
{q,{\sf X}}\qhyp43{{\bf c}}{{\bf d}}{q,{\sf Y}}
=\sum_{m=0}^\infty \frac{({\bf a};q)_m{\sf X}^m}
{({\bf b};q)_m}
\qhyp87{q^{-m},{\bf c},\frac{q^{1-m}}{{\bf b}}}
{{\bf d},\frac{q^{1-m}}{\bf a}}
{q,\frac{qb_1b_2b_3{\sf Y}}{a_1a_2a_3a_4{\sf X}}}.
\end{eqnarray}
Now consider the limit in the above sum as 
$a_4\to q^{-n}$, $c_4\to q^{-k}$ and redefine
${\bf a}$, ${\bf c}$ as 
in the statement of the theorem.
The straightforward application of Theorem \ref{BT} 
can be written as two sums, the first over 
$m\in\{0,\ldots,n\}$ and a second over $m\in\{n+1,n+k\}$. 
Now shift the second sum index so that it is over 
$m\in\{1,\ldots,k\}$, namely $m\mapsto m+n$. 
The first sum appears as the first term on the 
right-hand side of \eqref{prod4phi3t}. Now write 
the terminating basic hypergeometric series in the 
second term as a sum over $p\in\{0,\ldots,k\}$. 
In the limit, there is a singular term 
$(q^na_4;q)_m/(q^{1-m-n}/a_4;q)_p$ which is only 
non-vanishing for $p\in\{m,\ldots,m+k\}$. Now shift 
the sum over the terminating basic hypergeometric 
series so that it is over $p\in\{0,k-m\}$, namely 
$p\mapsto p+m$. 
After simplification, the singular term appears as 
$(q^nb;q)_m/(q^{1-m-n}/b;q)_m$, which after cancelling 
the vanishing factors reduces to $(-1)^mq^{\binom{m}{2}}$. 
Simplifying the result completes the proof.
\end{proof}

\begin{lem}
\label{Lem416}
Let $n,k\in\mathbb N_0$, $q\in\CCdag$, 
$x=\frac12(z+z^{-1})\in\CCast$,
$y=\frac12(w+w^{-1})\in\CCast$, $a,\alpha\in\CCast$. 
Then
\begin{eqnarray}
&&\hspace{-1.0cm}\qhyp21{q^{-n},a}{\frac{1}{a}}
{q,\frac{q^n}{az^2}}
\qhyp21{q^{-k},\alpha}{\frac{1}{\alpha}}
{q,\frac{q^k}{\alpha w^2}}\nonumber\\
&&\hspace{0.0cm}=
\sum_{m=0}^n\left(\frac{q^n}{az^2}\right)^m
\frac{(q^{-n},a;q)_m}{(q,\frac{1}{a};q)_m}\qhyp43{q^{-m},
q^{-k},\alpha,q^{1-m}a}{q^{n-m+1},\frac{1}
{\alpha},\frac{q^{1-m}}{a}}{q,\frac{q^{k+1}z^2}
{a\alpha w^2}}\nonumber\\
&&\hspace{1cm}+q^{\binom{n}{2}}\left(\frac{-1}
{az^2}\right)^n\frac{(a;q)_n}{(\frac{1}{a};q)_n}
\nonumber\\ &&\hspace{2cm}\times
\sum_{m=1}^k \left(\frac{q^k}{\alpha w^2}
\right)^m\frac{(q^{-k},\alpha;q)_m}{(q,\frac{1}
{\alpha};q)_m}
\qhyp43{q^{-n},q^{m-k},q^m\alpha,q^{1-n}a}
{q^{m+1},\frac{q^m}{\alpha},\frac{q^{1-n}}{a}}
{q,\frac{q^{k+1}z^2}{a\alpha w^2}}.\label{lemeq416}
\end{eqnarray}
\end{lem}
\begin{proof}
Let $b,\beta\in\CCast$.
Start with Theorem \ref{BT} with $r=u=s=v=1$, 
$p=\ell=0$, ${\bf a}=\{a,b\}$, 
${\bf b}=\{\frac{1}{a}\}$, ${\bf c}=\{\alpha,\beta\}$, 
${\bf d}=\{\frac{1}{\alpha}\}$, ${\sf X}=q^n/(az^2)$, 
${\sf Y}=q^k/(\alpha w^2)$. Consider the limit in 
the sum as $b\to q^{-n}$, $\beta\to q^{-k}$. 
The straightforward application of Theorem \ref{BT} 
can be written as two sums, the first over $m\in
\{0,\ldots,n\}$ and a second over $m\in\{n+1,n+k\}$. 
Now shift the second sum index so that it is over 
$m\in\{1,\ldots,k\}$, namely $m\mapsto m+n$. 
The first sum appears as the first term on the 
right-hand side of \eqref{lemeq416}. Now write the 
terminating basic hypergeometric series in the 
second term as a sum over $p\in\{0,\ldots,k\}$.
In the limit, there is a singular term 
$(q^nb;q)_m/(q^{1-m-n}/b;q)_p$ which is only 
non-vanishing for $p\in\{m,\ldots,k\}$. Now shift 
the sum over the terminating basic hypergeometric 
series so that it is over $p\in\{0,k-m\}$, namely 
$p\mapsto p+m$. 
After simplification, the singular term appears as
$(q^nb;q)_m/(q^{1-m-n}/b;q)_m$, which after cancelling 
the vanishing factors reduces to $(-1)^mq^{\binom{m}{2}}$. 
Simplifying the result completes the proof.
\end{proof}

\subsubsection{Double product basic hypergeometric orthogonal polynomial representations}\label{sec:5.1}
In this subsection, we use the above results to
derive product representations for basic hypergeometric orthogonal polynomials.

Using Theorem \ref{BTt}, we can obtain product 
representations for any basic hypergeometric orthogonal 
polynomial. In this section, we derive these product 
representations for the symmetric
and $q^{-1}$-symmetric subfamilies of the Askey--Wilson 
polynomials.
For instance, for the product of two Askey--Wilson 
polynomials with different arguments and parameters 
one has the following result.

\begin{thm}\label{thm:6.1}
Let $n,k\in\mathbb N_0$, $q\in\CCddag$, $z, w\in\CCast$,
$x=\frac12(z+z^{-1})\in \CCast$,
$y=\frac12(w+w^{-1})\in \CCast$,
$a,b,c,d,\alpha,\beta,\gamma,\delta\in\CCast$. Then
\begin{eqnarray}
&&\hspace{0.2cm}p_n(x;a,b,c,d|q)
p_k(y;\alpha,\beta,\gamma,\delta|q)=
\frac{(ab,ac,ad;q)_n
(\alpha\beta,\alpha\gamma,\alpha\delta;q)_k}
{a^{n}\alpha^{k}}\nonumber\\
&&\hspace{0.6cm}\times\sum_{m=0}^nq^m
\frac{(q^{-n},q^{n-1}abcd,az^\pm;q)_m}{(q,ab,ac,ad;q)_m}
\qhyp87{q^{-m},q^{-k},q^{k-1}
\alpha\beta\gamma\delta,
\frac{q^{1-m}}{ab},\frac{q^{1-m}}{ac},\frac{q^{1-m}}
{ad},\alpha w^\pm}{q^{n-m+1},\alpha\beta,\alpha\gamma,
\alpha\delta,\frac{q^{2-m-n}}{abcd},
\frac{q^{1-m}}{a}z^\pm}{q,q^2}\nonumber\\
&&\hspace{0.5cm}+q^{-\binom{n}{2}}
\frac{(q^{-1}abcd;q)_{2n}
(az^\pm;q)_n(\alpha\beta,\alpha\gamma,\alpha\delta;q)_k}
{(-a)^{n}\alpha^{k}(q^{-1}abcd;q)_n}
\sum_{m=1}^{k}q^m\frac
{(q^{-k},q^{k-1}\alpha\beta\gamma\delta,\alpha w^\pm;q)_m}
{(q,\alpha\beta,\alpha\gamma,\alpha\delta;q)_m}
\nonumber\\
&&\hspace{3.9cm}\times
\qhyp87{q^{-n},q^{m-k},q^{k+m-1}\alpha\beta\gamma\delta,
\frac{q^{1-n}}{ab},\frac{q^{1-n}}{ac},\frac{q^{1-n}}
{ad},q^{m}\alpha w^\pm}
{q^{m+1},q^{m}\alpha\beta,q^{m}
\alpha\gamma,q^{m}\alpha\delta,\frac{q^{2-2n}}
{abcd},\frac{q^{1-n}}{a}z^\pm}{q,q^2}.
\label{prodAW}
\end{eqnarray}
\end{thm}

\begin{proof}
Starting with
Theorem \ref{thm42}
and setting ${\sf X}={\sf Y}=q$, 
\begin{eqnarray}
&&\hspace{-3cm}(a_1,a_2,a_3)\mapsto\left(q^{n-1}
abcd,az^\pm\right),
\quad
(b_1,b_2,b_3)\mapsto(ab,ac,ad)\nonumber\\
&&\hspace{-3cm}(c_1,c_2,c_3)\mapsto
\left(q^{k-1}\alpha\beta\gamma\delta,\alpha w^\pm\right),
\quad
(d_1,d_2,d_3)\mapsto(\alpha\beta,\alpha\gamma,\alpha\delta),
\end{eqnarray}
completes the proof.
\end{proof}

By starting with the product formula for the Askey--Wilson 
polynomials one can obtain some interesting special 
cases. For instance, by taking limits as the parameters 
tend towards zero, we can obtain the symmetric and 
$q^{-1}$-symmetric sub-families of the Askey--Wilson 
polynomials.

\begin{thm}\label{thm:6.2}
Let $n,k\in\mathbb N_0$, $q\in\CCddag$, 
$z, w\in\CCast$,
$x=\frac12(z+z^{-1})\in \CCast$,
$y=\frac12(w+w^{-1})\in \CCast$,
$a,b,c,\alpha,\beta,\gamma\in\CCast$. Then one has 
the following product formula for the continuous dual 
$q$-Hahn polynomials
\begin{eqnarray}
&&\hspace{-0.2cm}p_n(x;a,b,c|q)
p_k(y;\alpha,\beta,\gamma|q)
=\frac{(ab,ac;q)_n(\alpha\beta,\alpha\gamma;q)_k}
{a^{n}\alpha^{k}}
\nonumber\\
&&\hspace{3.2cm}\times
\sum_{m=0}^nq^m\frac{(q^{-n},az^\pm;q)_m}{(q,ab,ac;q)_m}
\qhyp65{q^{-m},q^{-k},\frac{q^{1-m}}{ab},\frac{q^{1-m}}
{ac},\alpha w^\pm}{q^{n-m+1},
\alpha\beta,\alpha\gamma,\frac{q^{1-m}}{a}z^\pm}
{q,q^{n+1}bc}\nonumber\\
&&\hspace{0.2cm}+\frac{q^{-\binom{n}{2}}
(az^\pm;q)_n(\alpha\beta,\alpha\gamma;q)_k}
{(-a)^n\alpha^k}\nonumber\\
&&\hspace{2cm}\times\sum_{m=1}^kq^m\frac{(q^{-k},
\alpha w^\pm;q)_m}
{(q,\alpha\beta,\alpha\gamma;q)_m}
\qhyp65{q^{-n},q^{m-k},\frac{q^{1-n}}{ab},
\frac{q^{1-n}}{ac},q^m \alpha w^\pm}
{q^{m+1},q^m\alpha\beta,q^m\alpha\gamma,\frac{q^{1-n}}
{a}z^\pm}{q,q^{n+1}bc}.
\label{prodcdqh}\end{eqnarray}
\end{thm}
\begin{proof}
This is obtained by starting with product formula 
for the Askey--Wilson polynomials \eqref{prodAW} and 
taking the limit as $d,\delta\to 0$ using 
\eqref{cdqHl}, \eqref{limit3}.
\end{proof}

\begin{thm}
Let $n,k\in\mathbb N_0$, $q\in\CCddag$, 
$z, w\in\CCast$,
$x=\frac12(z+z^{-1})\in \CCast$,
$y=\frac12(w+w^{-1})\in \CCast$,
$a,b,c,\alpha,\beta,\gamma\in\CCast$. 
Then, one has the following product formula 
for the continuous dual $q^{-1}$-Hahn polynomials
\begin{eqnarray}
&&\hspace{-0.9cm}p_n(x;a,b,c|q^{-1})p_k(y;\alpha,
\beta,\gamma|q^{-1})
=q^{-2\binom{n}{2}-2\binom{k}{2}}(abc)^{n}
(\alpha\beta\gamma)^{k}(\tfrac{1}{ab },\tfrac{1}
{ac};q)_n(\tfrac{1}{\alpha\beta},\tfrac{1}
{\alpha\gamma};q)_k
\nonumber\\
&&\hspace{-0.5cm}\times
\sum_{m=0}^n\left(\frac{q^n}{bc}\right)^m
\frac{(q^{-n},\frac{z^\pm}{a};q)_m}
{(q,\frac{1}{ab},\frac{1}{ac};q)_m}
\qhyp65
{q^{-m},q^{-k},q^{1-m}ab,q^{1-m}ac,\frac{w^\pm}{\alpha}}
{q^{n-m+1},\frac{1}{\alpha\beta},\frac{1}
{\alpha\gamma},q^{1-m}az^\pm}
{q,\frac{q^{k+1}}{\beta\gamma}}\nonumber\\
&&\hspace{1.3cm}+q^{-2\binom{n}{2}-2\binom{k}{2}}
(-a)^n(\alpha\beta\gamma)^k(\tfrac{z^\pm}{a};q)_n
(\tfrac{1}{\alpha\beta},\tfrac{1}
{\alpha\gamma};q)_k\nonumber\\
&&\hspace{-0.5cm}\times\sum_{m=1}^k\left(\frac{q^k}
{\beta\gamma}\right)^m\frac{(q^{-k},\frac{w^\pm}{
\alpha};q)_m}
{(q,\frac{1}{\alpha\beta},\frac{1}{\alpha\gamma};q)_m}
\qhyp65{q^{-n},q^{m-k},q^{1-n}ab,q^{1-m}ac,\frac{q^m}
{\alpha}w^\pm}{q^{m+1},\frac{q^m}{\alpha\beta},\frac{q^m}
{\alpha\gamma},q^{1-n}az^\pm}{q,\frac{q^{k+1}}
{\beta\gamma}}.
\end{eqnarray}
\end{thm}
\begin{proof}
Starting with the product formula for the Askey--Wilson 
polynomials \eqref{prodAW}, replacing 
\[
(a,b,c,d)\mapsto(\tfrac1a,\tfrac1b,\tfrac1c,\tfrac1d), 
\quad (\alpha,\beta,\gamma,\delta)
\mapsto(\tfrac1\alpha,\tfrac1\beta,\tfrac1\gamma,
\tfrac1\delta),
\]
multiplying both sides by 
\[
q^{-3\binom{n}{2}-3\binom{k}{2}}
(abcd)^n(\alpha\beta\gamma\delta)^k
\] 
and taking the limit as $d,\delta\to 0$ using 
\eqref{limtcdqiH}, \eqref{limit3} completes the proof.
\end{proof}

\begin{thm}\label{thm:6.4}
Let $n,k\in\mathbb N_0$, $q\in\CCddag$,
$z, w\in\CCast$,
$x=\frac12(z+z^{-1})\in \CCast$,
$y=\frac12(w+w^{-1})\in \CCast$,
$a, b, \alpha, \beta\in\CCast$. Then, one has the following 
product formula Al-Salam--Chihara polynomials:
\begin{eqnarray}
&&\hspace{-0.7cm}Q_n(x;a,b|q)Q_k(y;\alpha,\beta|q)
=\frac{(ab;q)_n(\alpha\beta;q)_k}{a^{n}\alpha^{k}}
\nonumber\\
&&\hspace{3.2cm}\times
\sum_{m=0}^nq^m\frac{(q^{-n},az^\pm;q)_m}{(q,ab;q)_m}
\qphyp54{1}{q^{-m},q^{-k},\frac{q^{1-m}}{ab},
\alpha w^\pm}{q^{n-m+1},\alpha\beta,\frac{q^{1-m}}
{a}z^\pm}{q,\frac{q^{n-m+2}b}{a}}\nonumber\\
&&\hspace{-0.6cm}+\frac{q^{-\binom{n}{2}}
(az^\pm;q)_n(\alpha\beta;q)_k}
{(-a)^n\alpha^k}\sum_{m=1}^kq^m\frac{(q^{-k},
\alpha w^\pm;q)_m}{(q,\alpha\beta;q)_m}
\qphyp54{1}{q^{-n},q^{m-k},\frac{q^{1-n}}{ab},
q^m \alpha w^\pm}{q^{m+1},q^m\alpha\beta,\frac{q^{1-n}}
{a}z^\pm}{q,\frac{q^{2}b}{a}}.\label{prodasc}
\end{eqnarray}
\end{thm}
\begin{proof}
Start with the product formula for the continuous 
dual $q$-Hahn polynomials and take the limit 
as $c,\gamma\to0$ using \eqref{ASCl}, \eqref{limit1}.
\end{proof}

\begin{thm}
Let $n,k\in\mathbb N_0$, $q\in\CCddag$, 
$x=\frac12(z+z^{-1})\in\CCast$,
$y=\frac12(w+w^{-1})\in\CCast$,
$a,b,\alpha,\beta\in\CCast$. 
Then, one has the following product formula for the 
$q^{-1}$-Al-Salam--Chihara polynomials:
\begin{eqnarray}
&&\hspace{-0.9cm}Q_n(x;a,b|q^{-1})
Q_k(y;\alpha,\beta|q^{-1})
=q^{-\binom{n}{2}-\binom{k}{2}}(-b)^{n}(-\beta)^{k}
(\tfrac{1}{ab };q)_n(\tfrac{1}{\alpha\beta};q)_k
\nonumber\\
&&\hspace{-0.5cm}\times
\sum_{m=0}^nq^{-\binom{m}{2}}\left(-\frac{q^na}
{b}\right)^m\frac{(q^{-n},\frac{z^\pm}{a};q)_m}
{(q,\frac{1}{ab};q)_m}
\qphyp54{-1}{q^{-m},q^{-k},q^{1-m}ab,\frac{w^\pm}
{\alpha}}{q^{n-m+1},\frac{1}{\alpha\beta},q^{1-m}az^\pm}
{q,\frac{q^{k+1}\alpha}{\beta}}\nonumber\\
&&\hspace{1.3cm}+q^{-\binom{n}{2}-\binom{k}{2}}
(-a)^n(-\beta)^k(\tfrac{z^\pm}{a};q)_n(\tfrac{1}
{\alpha\beta};q)_k\nonumber\\
&&\hspace{-0.5cm}\times\sum_{m=1}^kq^{-\binom{m}{2}}
\left(-\frac{q^k\alpha}{\beta}\right)^m\frac{(q^{-k},
\frac{w^\pm}{\alpha};q)_m}
{(q,\frac{1}{\alpha\beta};q)_m}
\qphyp54{-1}{q^{-n},q^{m-k},q^{1-n}ab,\frac{q^m}
{\alpha}w^\pm}{q^{m+1},\frac{q^m}{\alpha\beta},q^{1-n}
az^\pm}{q,\frac{q^{k-m+1}\alpha}{\beta}}.
\end{eqnarray}
\end{thm}
\begin{proof}
Start with the product formula for the continuous 
dual $q^{-1}$-Hahn polynomials and take the limit 
as $c,\gamma\to0$ using \eqref{qiASCl}, \eqref{limit2}.
\end{proof}

\begin{thm}\label{thm:6.6}
Let $n,k\in\mathbb N_0$, $q\in\CCddag$, 
$z, w\in\CCast$,
$x=\frac12(z+z^{-1})\in \CCast$,
$y=\frac12(w+w^{-1})\in \CCast$,
$a,\alpha\in\CCast$. Then
\begin{eqnarray}
&&\hspace{-0.0cm}H_n(x;a|q)H_k(y;\alpha|q)
=a^{-n}\alpha^{-k}\sum_{m=0}^nq^m
\frac{(q^{-n},az^\pm;q)_m}{(q;q)_m}
\qphyp43{2}{q^{-m},q^{-k},\alpha w^\pm}{q^{n-m+1},
\frac{q^{1-m}}{a}z^\pm}{q,\frac{q^{n-2m+3}}{a^2}}
\nonumber\\ &&\hspace{1.3cm}+\frac{q^{-\binom{n}{2}}}
{(-a)^n\alpha^k}(az^\pm;q)_n\sum_{m=1}^kq^m
\frac{(q^{-k},\alpha w^\pm;q)_m}{(q;q)_m}
\qphyp43{2}{q^{-n},q^{m-k},q^m \alpha w^\pm}
{q^{m+1},\frac{q^{1-n}}{a}z^\pm}{q,\frac{q^{3-n}}
{a^2}}.\label{prodcqH}
\end{eqnarray}
\end{thm}
\begin{proof}
Start with the product formula for the 
Al-Salam--Chihara polynomials and take the limit 
as $b,\beta\to0$ using \eqref{cbqHl}, \eqref{limit1}.
\end{proof}

\begin{thm}
Let $n,k\in\mathbb N_0$, $q\in\CCddag$, 
$z, w\in\CCast$,
$x=\frac12(z+z^{-1})\in \CCast$,
$y=\frac12(w+w^{-1})\in \CCast$,
$a,\alpha\in\CCast$. Then
\begin{eqnarray}
&&\hspace{-0.5cm}H_n(x;a|q^{-1})H_k(y;\alpha|q^{-1})
=a^{-n}\alpha^{-k}\sum_{m=0}^nq^{-2\binom{m}{2}}
(q^na^2)^m\frac{(q^{-n},\frac{z^\pm}{a};q)_m}
{(q;q)_m}\nonumber\\
&&\hspace{6.2cm}\times\qphyp43{-2}{q^{-m},q^{-k},
\frac{w^\pm}{\alpha}}{q^{n-m+1},q^{1-m}az^\pm}
{q,q^{k+1}\alpha^2}\nonumber\\
&&\hspace{2.3cm}+q^{-\binom{n}{2}}\frac{(-a)^n}{\alpha^k}
(\tfrac{z^\pm}{a};q)_n\sum_{m=1}^kq^{-2\binom{m}{2}}
\left(q^k\alpha^2\right)^m\frac{(q^{-k},\frac{w^\pm}{
\alpha};q)_m}
{(q;q)_m}
\nonumber\\
&&\hspace{6.2cm}\times
\qphyp43{-2}{q^{-n},q^{m-k},\frac{q^m}{\alpha}w^\pm}
{q^{m+1},q^{1-n}az^\pm}{q,q^{k-2m+1}\alpha^2}.
\end{eqnarray}
\end{thm}
\begin{proof}
Start with the product formula for the $q^{-1}$-Al-Salam--Chihara polynomials and take the limit 
as $b,\beta\to0$ using \eqref{cbqiHl}, \eqref{limit2}.
\end{proof}

\begin{thm}\label{thm:6.9}
Let $n,k\in\mathbb N_0$, $q\in\CCddag$, 
$z, w\in\CCast$,
$x=\frac12(z+z^{-1})\in \CCast$,
$y=\frac12(w+w^{-1})\in \CCast$. Then, one has the 
following product formula for the continuous 
$q$-Hermite polynomials:
\begin{eqnarray}
&&\hspace{0.0cm}H_n(x|q)H_k(y|q)
=z^nw^k\sum_{m=0}^nq^{-\binom{m}{2}}\left(-\frac{q^n}
{z^2}\right)^m\frac{(q^{-n};q)_m}{(q;q)_m}
\qphyp21{-2}{q^{-m},q^{-k}}{q^{n-m+1}}
{q,\frac{q^{m+k}z^2}{w^2}}\nonumber\\
&&\hspace{2.2cm}+\frac{w^k}{z^n}\sum_{m=1}^k
q^{-\binom{m}{2}}\left(-\frac{q^k}{w^2}\right)^m
\frac{(q^{-k};q)_m}{(q;q)_m}\qphyp21{-2}{q^{-n},q^{m-k}}
{q^{m+1}}{q,\frac{q^{n-m+k}z^2}{w^2}}.
\label{prodcdqH}
\end{eqnarray}
\end{thm}
\begin{proof}
Start with Lemma \ref{Lem416}, multiply both sides 
by $z^nw^k$ and taking the limit as $a,\alpha\to 0$ 
identifying the continuous $q$-Hermite polynomials 
using \eqref{cqHrep} completes the proof.
\end{proof}

\begin{thm}
Let $n,k\in\mathbb N_0$, $q\in\CCddag$, 
$z, w\in\CCast$,
$x=\frac12(z+z^{-1})\in \CCast$,
$y=\frac12(w+w^{-1})\in \CCast$. Then, one has the 
following product formula for the continuous $q^{-1}$-Hermite polynomials:
\begin{eqnarray}
&&\hspace{0.0cm}H_n(x|q^{-1})H_k(y|q^{-1})
=z^nw^k\sum_{m=0}^nq^{\binom{m}{2}}\left(-\frac{q}
{z^2}\right)^m\frac{(q^{-n};q)_m}{(q;q)_m}
\qphyp212{q^{-m},q^{-k}}{q^{n-m+1}}{q,\frac{q^{n-m+2}z^2}
{w^2}}\nonumber\\
&&\hspace{3.5cm}+\frac{w^k}{z^n}\sum_{m=1}^k
q^{\binom{m}{2}}\left(-\frac{q}{w^2}\right)^m
\frac{(q^{-k};q)_m}{(q;q)_m}\qphyp212{q^{-n},q^{m-k}}
{q^{m+1}}{q,\frac{q^{m+2}z^2}{w^2}}.
\end{eqnarray}
\end{thm}
\begin{proof}
Start with Lemma \ref{Lem416}, replace 
$z^2\mapsto q^{n-1}z^2$,
and multiply both sides by $z^nw^k$. Then take the 
limit as $a,\alpha\to \infty$ identifying the 
continuous $q^{-1}$-Hermite polynomials using 
\eqref{cqiH:def1} completes the proof.
\end{proof}

\subsection{Double product generalized hypergeometric function ($q\to 1^{-}$) representations} 
In this section we 
obtain some $q\in \CCddag$, 
$q\to 1^{-}$ limiting cases for 
the main results of this paper.

Theorem \ref{BT} yields the following result.
\begin{thm}\label{LimBT}
Let $r,s\in\mathbb N_0\cup\{-1\}$, $u,v\in\mathbb N_0$, with 
$u\ge r$ and $v\ge s$, 
${\bf a}\in\CCast^{r+1}$, ${\bf b}\in\CCast^{u}$, 
${\bf c}\in\CCast^{s+1}$, ${\bf d}\in\CCast^{v}$. 
Then
\begin{eqnarray}
\hspace{-7mm}
\hyp{r+1}{u}{{\bf a}}{{\bf b}}{\sf X}
\hspace{-1.3mm}\hyp{s+1}{v}{{\bf c}}{{\bf d}}{{\sf Y}}
&=&\sum_{m=0}^\infty\!\frac{({\bf a})_m}
{({\bf b})_m}\dfrac{{\sf X}^m}{m!}
\hyp{s+u+2}{r+v+1}
{-m, {\bf c}, {1\!-\!m}\!-\!{\bf b}}
{{\bf d},{1\!-\!m}\!-\!{{\bf a}}}{(-1)^{u-r}
\frac{{\sf Y}}{{\sf X}}}
\label{LimqCh1}\\
&=&\sum_{m=0}^\infty
\frac{({\bf c})_m}{({\bf d})_m}
\dfrac{{\sf Y}^m}{m!}
\hyp{r+v+2}{s+u+1}{{-m},{\bf a},{1\!-\!m}\!-\!{\bf d}}
{{\bf b},{1\!-\!m}\!-\!{\bf c}}{(-1)^{v-s}\dfrac{{\sf X}}
{{\sf Y}}},
\label{LimqCh2}
\end{eqnarray}
where ${\sf X}$, ${\sf Y}$ are given 
such that the left-hand side is well-defined.
\end{thm}
\begin{proof}
Replacing $({\bf a},{\bf b},{\bf c},{\bf d})
\mapsto(q^{\bf a},q^{\bf b},q^{\bf c},q^{\bf d})$, 
$({\sf X},{\sf Y})\mapsto ((-1)^p(q-1)^{u-r}{\sf X}, 
(-1)^\ell(q-1)^{v-s} {\sf Y})$, in Theorem \ref{BT}, 
applying Lemma \ref{lem:limbhs2hsp} to 
\eqref{qCh1}, \eqref{qCh2}, and after 
a straightforward calculation the result follows.
\end{proof}
The rest of the results follow applying an 
equivalent process.
\begin{thm}\label{thm:Lim2.4}
Let $r\in\mathbb N_0\cup\{-1\}$, $u\in\mathbb N_0$, with $u\ge r$,
${\bf a}\in\CCast^{r+1}$, ${\bf b}\in\CCast^{u}$, 
$c,{\sf X},{\sf 
Y}\in\CCast$ such that the left-hand side is well-defined.
Then
\begin{eqnarray}
\hspace{-20mm}\hyp{r+1}{u}{{\bf a}}{{\bf b}}{{\sf X}}
&=&(1-{\sf Y})^c\sum_{m=0}^\infty
\frac{({\bf a})_m }{({\bf b})_m}\dfrac{{\sf {\sf X}}^m}{m!}
\hyp{u+2}{r+1}{{-m}, c, {1\!-\!m}\!-\!{\bf b}}
{{1\!-\!m}\!-\!{{\bf a}}}{(-1)^{u-r}\dfrac{{\sf Y}}
{{\sf X}}}\label{LimqCh1a}\\
&=&(1-{\sf Y})^c\sum_{m=0}^\infty
(c)_m \dfrac{{\sf Y}^m}{m!}\hyp{r+2}{u+1}
{{-m},{\bf a}}{{\bf b},{1\!-\!m}\!-\!{c}}{\dfrac{{\sf X}}
{{\sf Y}}}.\label{LimqCh2a}
\end{eqnarray}
\end{thm}
\begin{proof}
Setting ${\sf X}\to (-1)^p(q-1)^{u-r}{\sf X}$, 
${\sf Y}\to (-1)^\ell(q-1)^{v-s} {\sf Y}$, 
$c\to q^c$ in Corollary \ref{thm:2.4}, applying 
Lemma \ref{lem:limbhs2hsp} to \eqref{qCh1a} and 
\eqref{qCh2a}, and 
after a straightforward calculation the result follows.
\end{proof}
\begin{thm}\label{thm:Lim2.5}
Let $r\in\mathbb N_0\cup\{-1\}$, $u\in\mathbb N_0$, with $u\ge r$,
${\bf a}\in\CCast^{r+1}$, ${\bf b}\in\CCast^{u}$, 
 $c, {\sf X}
\in\CCast$ such that the left-hand side is well-defined.
Then
\begin{eqnarray}
\hspace{-22mm}\hyp{r+1}{u}{{\bf a}}{{\bf b}}{{\sf X}}
&=&(1-{\sf X})^c
\sum_{m=0}^\infty
\frac{({\bf a})_m }{({\bf b})_m}
\dfrac{{{\sf X}}^m}{m!}
\hyp{u+2}{r+1}{{-m}, c, {1\!-\!m}\!-\!{\bf b}}
{{1\!-\!m}\!-\!{{\bf a}}}{(-1)^{u-r}}
\label{LimqCh1d}\\
&=&(1-{\sf X})^c\sum_{m=0}^\infty {(c)_m}
\dfrac{{\sf X}^m}{m!}
\hyp{r+2}{u+1}{{-m},{\bf a}}
{{\bf b},{1\!-\!m}\!-\!{c}}{1}.\label{LimqCh2d}
\end{eqnarray} 
\end{thm}
\begin{proof}
Setting ${\sf X}\to (-1)^p(q-1)^{u-r}{\sf X}$, 
${\sf Y}\to (-1)^\ell(q-1)^{v-s} {\sf Y}$, 
$c\to q^c$ in Corollary \ref{thm:2.5}, applying Lemma 
\ref{lem:limbhs2hsp} to \eqref{qCh1d} and 
\eqref{qCh2d}, and 
after a straightforward calculation the result follows.
\end{proof}

\begin{thm}\label{LimBT1}
Let $r\in\mathbb N_0\cup\{-1\}$, $u\in\mathbb N_0$, with $u\ge r$, 
${\bf a}\in\CCast^{r+1}$, ${\bf b}\in\CCast^{u}$, 
$p\in\mathbb Z$, ${\sf X}$, ${\sf Y}\in\CCast$ such that 
the left-hand side is well-defined.
Then
\begin{eqnarray}
\hspace{-32mm}\hyp{r+1}{u}{{\bf a}}{{\bf b}}{{\sf X}}
&=&\exp(-{\sf Y})\sum_{n=0}^\infty\frac{({\bf a})_n}
{({\bf b})_n}\dfrac{{{\sf X}}^n}{n!}
\hyp{u+1}{r+1}{{-n},{1\!-\!n}\!-\!{\bf b}}
{{1\!-\!n}\!-\!{{\bf a}}}{(-1)^{u-r}\dfrac{{\sf Y}}
{{\sf X}}}\label{LimqCh2-a}\\
&=&\exp(-{\sf Y})\sum_{n=0}^\infty
\frac{{\sf Y}^n}{n!}
\hyp{r+2}{u}{{-n},{\bf a}}
{{\bf b}}{-\dfrac{{\sf X}}{{\sf Y}}}.\label{LimqCh3}
\end{eqnarray}
\end{thm}
\begin{proof}
Setting $({\sf X},{\sf Y})\mapsto ((-1)^p(q-1)^{u-r}
{\sf X},(1-q) {\sf Y})$, 
applying Lemma \ref{lem:limbhs2hsp} in Theorem 
\ref{thm:2.5}, and after 
a straightforward calculation the result follows.
\end{proof}

\begin{thm}
\label{Limthm:4.1}
Let $n,u,v\in\mathbb N_0$, $r,s\in\mathbb N_0\cup\{-1\}$, 
$p,\ell\in\mathbb Z$, such that $\ell\ge s-v$,
${\bf a}\in\CCast^{r}$, ${\bf b}\in\CCast^{u}$, 
${\bf c}\in\CCast^{s+1}$, ${\bf d}\in\CCast^{v}$, 
${\sf X},{\sf Y}\in\CCast$. Then
\begin{eqnarray}
\nonumber\hspace{-13mm}
\hyp{r+1}{u}{{-n},{\bf a}}{{\bf b}}{{\sf X}}
\hspace{-1mm}\hyp{s+1}{v}{{\bf c}}{{\bf d}}{{\sf Y}}\\
&&\hspace{-52mm}=\sum_{m=0}^n\frac{({-n},{\bf a})_m}
{({\bf b})_m}
\dfrac{{\sf X}^m}{m!}\hyp{s+u+2}{r+v+1}
{{-m},{\bf c},{1\!-\!m}\!-\!{\bf b}}
{{n-m+1},{\bf d},{1\!-\!m}\!-\!{{\bf a}}}{
(-1)^{u-r}\frac{{\sf Y}}{{\sf X}}}\nonumber\\[0.1cm]
\nonumber\\
&&\hspace{-45mm}
+\frac{({\bf a})_n}{({\bf b})_n}(-{\sf X})^n
\sum_{m=1}^{\infty}\frac{({\bf c})_m}{({\bf d})_m}
\dfrac{{\sf Y}^m}{m!}\hyp{s+u+2}{r+v+1}
{{-n},m+{\bf c},{1\!-\!n}\!-\!{\bf b}}
{{m+1},m+{\bf d},{1\!-\!n}\!-\!{{\bf a}}}
{(-1)^{u-r}\frac{{\sf Y}}{{\sf X}}}\label{LimtttqCh1}\\
&& \hspace{-52mm}=\sum_{m=0}^\infty 
\frac{({\bf c})_m}{({\bf d})_m}
\dfrac{{\sf Y}^m}{m!}
\hyp{r+v+2}{s+u+1}{{-m},{-n},{\bf a},{1\!-\!m}\!-\!{\bf d}}
{{\bf b},{1\!-\!m}\!-\!{{\bf c}}}{(-1)^{v-s-1}
\dfrac{{\sf X}}{\sf Y}}.
\label{LimtttqCh2}
\end{eqnarray}
\end{thm}
\begin{proof}
Replacing $({\bf a},{\bf b},{\bf c},{\bf d})
\mapsto(q^{\bf a},q^{\bf b},q^{\bf c},q^{\bf d})$, 
$({\sf X},{\sf Y})\mapsto ((-1)^p(q-1)^{u-r}{\sf X}, 
(-1)^\ell(q-1)^{v-s} {\sf Y})$, $c\mapsto q^c$ in 
Theorem \ref{thm:4.1}, applying Lemma 
\ref{lem:limbhs2hsp} to 
\eqref{tttqCh1}, \eqref{tttqCh2}, and after 
a straightforward calculation the result follows.
\end{proof}
\begin{thm}
\label{LimBTt}
Let $n,k\in\mathbb Z$, $r,s\in\mathbb N_0\cup\{-1\}$, $u,v\in\mathbb N_0$, 
${\bf a}\in\CCast^{r}$, ${\bf b}\in\CCast^{u}$, 
${\bf c}\in\CCast^{s}$, ${\bf d}\in\CCast^{v}$, 
${\sf X},{\sf Y}\in\CCast$. Then
\begin{eqnarray}
\nonumber\hspace{-1cm} \hyp{r+1}{u}{{-n},{\bf a}}
{{\bf b}}{{\sf X}}
\hyp{s+1}{v}{{-k},{\bf c}}{{\bf d}}{{\sf Y}}&&\\
\nonumber&&\hspace{-65mm}=\sum_{m=0}^n
\frac{({-n},{\bf a})_m}{({\bf b})_m}
\dfrac{{\sf X}^m}{m!}\hyp{s+u+2}{r+v+1}
{{-m},{-k},{\bf c},{1\!-\!m}\!-\!{\bf b}}
{{n-m+1},{\bf d},{1\!-\!m}\!-\!{{\bf a}}}{(-1)^{u-r}
\dfrac{{\sf Y}}{{\sf X}}}\nonumber\\[0.1cm]
&&\hspace{-62mm}+(-{\sf X})^n\frac{({\bf a})_n}{({\bf b})_n}
\sum_{m=1}^{k}\frac{({-k},{\bf c})_m}{({\bf d})_m}
\dfrac{{\sf Y}^m}{m!}\hyp{s+u+2}{r+v+1}
{{-n},{m-k},m+{\bf c},{1\!-\!n}\!-\!{\bf b}}
{{m+1},m+{\bf d},{1\!-\!n}\!-\!{{\bf a}}}
{(-1)^{u-r}\dfrac{{\sf Y}}{{\sf X}}}
\label{LimttqCh1}\\
\nonumber&&\hspace{-65mm}=\sum_{m=0}^k
\frac{({-k},{\bf c})_m }{({\bf d})_m}\dfrac{{\sf Y}^m}{m!}
\hyp{r+v+2}{s+u+1}{{-m},{-n},{\bf a},{1\!-\!m}\!-\!{\bf d}}
{{k-m+1},{\bf b},{1\!-\!m}\!-\!{{\bf c}}}
{(-1)^{v-s}\dfrac{{\sf X}}{{\sf Y}}}\nonumber\\
&&\hspace{-62mm}+(-{\sf Y})^k\frac{({\bf c})_k}{({\bf d})_k}
\sum_{m=1}^{n}\frac{({-n},{\bf a})_m }{({\bf b})_m}
\dfrac{{\sf X}^m}{m!}\hyp{r+v+2}{s+u+1}
{{-k},{m-n},m+{\bf a},{1\!-\!k}\!-\!{\bf d}}
{{m+1},m+{\bf b},{1\!-\!k}\!-\!{{\bf c}}}{
(-1)^{v-s}\dfrac{{\sf X}}{{\sf Y}}}.\label{LimttqCh2}
\end{eqnarray}
\end{thm}
\begin{proof}
Replacing $({\bf a},{\bf b},{\bf c},{\bf d})
\mapsto(q^{\bf a},q^{\bf b},q^{\bf c},q^{\bf d})$, 
$({\sf X},{\sf Y})\mapsto ((-1)^p(q-1)^{u-r}{\sf X}, 
(-1)^\ell(q-1)^{v-s} {\sf Y})$, $c\mapsto q^c$ in 
Theorem \ref{BTt}, applying Lemma \ref{lem:limbhs2hsp} 
to \eqref{ttqCh1}, \eqref{ttqCh2}, and after a 
straightforward calculation the result follows.
\end{proof}

\subsubsection{Double product generalized hypergeometric orthogonal polynomial representations}\label{sec:5.2}

In this subsection, we use the above results to
derive product representations for generalized hypergeometric orthogonal polynomials in the Askey scheme. For the definitions of these polynomials see \S\ref{Askeys}.
We obtain the limit when 
$q\to 1^{-}$ of all the continuous symmetric subfamilies of the 
Askey--Wilson polynomials.

We are going to apply the 
$q\to 1^{-}$ limiting cases 
to all the classical families 
considered in this section.
In the following result we consider product formulas 
for the Wilson and the continuous dual Hahn polynomials.
\begin{thm}
\label{thm:limclcass}
Let $n,k\in\mathbb N_0$,
$x,y,a,b,c,d,\alpha,\beta,\gamma,\delta\in\CCast$. 
Then, one has the following product formula for the Wilson polynomials,
\begin{eqnarray}
&&\hspace{-5mm}W_n(x^2;a,b,c,d)
W_k(y^2;\alpha,\beta,\gamma,\delta)\nonumber\\
&&\hspace{-3mm}=
{(a+b,a+c,a+d)_n(\alpha+\beta,\alpha+\gamma,
\alpha+\delta)_k}\sum_{m=0}^n\frac
{({-n},{n\!-\!1}\!+\!a\!+\!b\!+\!c\!+\!d,a\!\pm\! ix)_m}
{m!(a\!+\!b,a\!+\!c,a\!+\!d)_m}\nonumber\\
&&\hspace{-2mm}\times\hyp87{{-m},{-k},
{k\!-\!1}\!+\!\alpha\!+\!\beta\!+\!\gamma\!+\!\delta,
{{1\!-\!m}}\!-\!a\!-\!b,{{1\!-\!m}}\!-\!a\!-\!c,
{{1\!-\!m}}\!-\!a\!-\!d,\alpha \!\pm\! iy}
{{n\!-\!m\!+\!1},\alpha
\!+\!\beta,\alpha\!+\!\gamma,\alpha\!+\!\delta,
{2\!-\!m\!-\!n}\!-\!a\!-\!b\!-\!c\!-\!d,
{{1\!-\!m}}\!-\!{a}\!\pm\! ix}{1}\nonumber\\
&&\hspace{-4mm}+
\frac{(a\!+\!b\!+\!c\!+\!d-1)_{2n}(a\!\pm\! i x)_n
(\alpha\!+\!\beta,\alpha\!+\!\gamma,\alpha\!+\!\delta;q)_k}
{(-1)^{n}(a\!+\!b\!+\!c\!+\!d-1)_n}
\sum_{m=1}^{k}
\frac{({-k},
{k-1}
\!+\!\alpha\!+\!\beta\!+\!\gamma\!+\!\delta,\alpha\!\pm\! iy)_m}
{m!(\alpha\!+\!\beta,\alpha\!+\!\gamma,\alpha\!+\!\delta)_m}
\nonumber\\
&&\hspace{-4mm}\times\!
\hyp87{\!{-n},{m\!-\!k},
{k\!+\!m\!-\!1}\!+\!\alpha\!+\!\beta\!+\!\gamma\!+\!\delta,
{{1\!-\!n}}\!-\!a\!-\!b,{{1\!-\!n}}\!-\!a\!-\!c,
{{1\!-\!n}}\!-\!a\!-\!d,{m}\!+\!\alpha\!\pm\! i y\!\!}
{{m\!+\!1},{m}\!+\!\alpha\!+\!\beta,
{m}\!+\!\alpha\!+\!\gamma,{m}\!+\!\alpha\!+\!\delta,
{2\!-\!2n}\!-\!a\!-\!b\!-\!c\!-\!d,
{{1\!-\!n}}\!-\!a\!\pm\! i x}{\!1\!}\hspace{-1mm},
\label{LimprodAW}\\
&&\hspace{-5mm}S_n(x^2;a,b,c)S_k(y^2;\alpha,\beta,\gamma)
={(a+b,a+c)_n(\alpha+\beta,\alpha+\gamma)_k}\nonumber\\
&&\hspace{-3mm}\times
\sum_{m=0}^n\frac{({-n},a\pm ix)_m}{m!(a+b,a+c)_m}
\hyp65{{-m},{-k},{{1-m}}-a-b,{{1-m}}-a-c,\alpha\pm iy}
{{n-m+1},
\alpha+\beta,\alpha+\gamma,{{1-m}}-a\pm i x}{1}\nonumber\\
&&\hspace{-3mm}+(-1)^n (a+\pm ix)_n
(\alpha+\beta,\alpha+\gamma)_k\nonumber\\
&&\label{Limprodcdqh}
\hspace{-3mm}\times\sum_{m=1}^k\frac{({-k},\alpha\pm iy)_m}
{m!(\alpha+\beta,\alpha+\gamma)_m}
\hyp65{{-n},{m-k},{{1-n}}-a-b,{{1-n}}-a-c,m+\alpha\pm iy}
{{m+1},m+\alpha+\beta,m+\alpha+\gamma,{1-n}-a\pm ix}{1},\\
\end{eqnarray}
\end{thm}
\begin{proof}
To obtain \eqref{LimprodAW} we must replace in 
Theorem \ref{thm:6.1} the values 
\[({a},{b},{c},{d})
\mapsto(q^{a},q^{b},q^{c},q^{d}),\quad 
({\alpha},{\beta},{\gamma},{\delta})
\mapsto(q^{\alpha},q^{\beta},q^{\gamma},q^{\delta}),\quad 
({z}, {w})\mapsto ({q^{ix}}, 
{q^{iy}}),
\] divide the expression \eqref{prodAW} by 
$(1-q)^{3n+3k}$ (see \cite[(14.1.21)]{Koekoeketal}), 
apply \eqref{eq:limbhs2hs} to 
\eqref{prodAW} and after a straightforward 
calculation the identity holds.
To obtain \eqref{Limprodcdqh} we must replace $({a},{b},{c})
\mapsto(q^{a},q^{b}),^q{c}$, 
$({\alpha},{\beta},{\gamma})
\mapsto(q^{\alpha},q^{\beta},q^{\gamma})$, 
$({z}, {w})\mapsto ({q^{ix}}, 
{q^{iy}})$, in Theorem 
\ref{thm:6.2}, divide the 
expression \eqref{prodcdqh} by 
$(1-q)^{2n+2k}$ (see 
\cite[(14.3.17)]{Koekoeketal}), 
applying
\eqref{eq:limbhs2hs} 
and after a straightforward 
calculation the identity holds.
\end{proof}

By using the limit \eqref{limW2J} and \eqref{LimprodAW} 
the following result for the Jacobi polynomials
\eqref{Jacdef}
follows.
\begin{thm}
\label{Jacprodthm}
Let $n,k\in\mathbb N_0$, $a,b,\alpha,\beta\in \mathbb C$,
$x,y\in\CCast$, $x\ne 1$. Then, one has the following product formula for the Jacobi polynomials:
\begin{eqnarray}
P^{(a,b)}_n(x)P^{(\alpha,\beta)}_k(y)
&=&\dfrac{(a+1)_n}{n!}\dfrac{(\alpha+1)_k}
{k!}\sum_{m=0}^n\frac{({-n},{a+b+n+1})_m}{m!
(a+1)_m}\left(\frac {1-x}2\right)^m\nonumber\\
&&
\times\hyp43{{-m},{-k},{k+1}+\alpha+\beta,-a-m}
{n-m+1,\alpha+1,-m-n-a-b}{\dfrac{1-y}{1-x}}\nonumber\\
&&
+\frac{(a+b+n+1)_{n}}{n!}\dfrac{(\alpha+1)_k}{k!}
\left(\frac{x\!-\!1}2\right)^n
\sum_{m=1}^{k}\frac{({-k},\alpha+\beta+k+1)_m}
{m!(\alpha+1)_m}\left(\frac{1\!-\!y}2\right)^m
\nonumber\\
&&\times \hyp43{{-n},{m-k},{k+m+1+\alpha+\beta},
{-n-a}}{m+1,m+\alpha+1,{-2n}-a-b}{\dfrac{1-y}
{1-x}}.
\label{prodJa}
\end{eqnarray}
\label{thm536}
\end{thm}

\begin{proof}
We start by considering \eqref{LimprodAW}. Then, we set 
the parameters as it is described in \eqref{limW2J}, 
i.e., 
\begin{eqnarray}
&& (a,b,c,t)\mapsto\left(\tfrac12(a+1),\tfrac12(a+1),\tfrac12(b+1)\pm it\right),\\ &&(\alpha,\beta,\gamma, \delta)\mapsto 
\left(\tfrac12(\alpha+1),\tfrac12(\alpha+1),\tfrac12(\beta+1)\pm it\right),\\
&&(x,y)\mapsto t \left(\sqrt
{\tfrac12(1-x)},\sqrt{\tfrac12(1-y)}\right).
\end{eqnarray}
Then, one must divide 
the whole expression by 
$t^{2n+2k}n!k!$ and take $t\to \infty$, since 
\eqref{limW2J} the left hand side converges to 
the desired product of Jacobi polynomials. 
About the right-hand side, taking into 
account that for any $\lambda,\mu\in \mathbb C^*$ and 
$m\in \mathbb N_0$, using \eqref{limPochtmu}
the desired limit between the product 
of two Wilson polynomials and the 
product of two Jacobi polynomials follows after 
some direct calculations.
\end{proof}

\begin{rem}
Observe that one can obtain Theorem \ref{thm536} 
by setting in Theorem \ref{LimBTt} the values 
${\bf a}=\{a+b+n+1\}$, ${\bf b}=\{a+1\}$,
${\bf c}=\{\alpha+\beta+k+1\}$, ${\bf d}=\{\alpha+1\}$,
${\sf X}=\frac12(1-x)$, and ${\sf Y}=\frac12(1-y)$.
\end{rem}

It is well-known that the ultraspherical, Chebyshev 
and Legendre polynomials, are particular cases for 
the Jacobi polynomials, 
therefore we can easily obtain product formulas for these 
polynomial sequences.
\begin{cor}
Let $n,k\in\mathbb N_0$, $\lambda, \mu\in \mathbb C$,
$x,y\in\CCast$, $x\ne 1$. Then, one has the following product formula for the ultraspherical, Chebyshev of the first kind Chebyshev of the second kind and Legendre polynomials:
\begin{eqnarray}
\hspace{-4mm}C_n^{\lambda}(x)C_k^{\mu}(y)
&=&\dfrac{(2\lambda)_{n}}{n!}\dfrac{(2\mu)_k}
{k!}\sum_{m=0}^n\frac{({-n},{2\lambda+n})_m}
{m!(\lambda+\frac12)_m}
\left(\frac {1-x}2\right)^m\nonumber\\ &&
\times\hyp43
{{-m},{-k},{k}+2\mu,-\lambda-m+\frac12}
{n-m+1,\mu+\frac12,-m-n-2\lambda+1}{\dfrac{1-y}
{1-x}}\nonumber\\ &&
+\frac{(2\lambda)_{2n}}{(\lambda+\frac12)_n n!}\dfrac{(2\mu)_k}
{k!}\left(\frac{x\!-\!1}2\right)^n
\sum_{m=1}^{k}\frac{({-k},k+2\mu)_m}
{m!(\mu+\frac12)_m}\left(\frac{1\!-\!y}2\right)^m
\nonumber\\
&&\times \hyp43{{-n},{m-k},{k+m+2\mu},
{-n-\lambda+\frac12}}{m+1,m+\mu+\frac12,{-2n}-2\lambda+1}
{\dfrac{1-y}{1-x}},\\
\nonumber T_n(x)T_k(y)
&=&\sum_{m=0}^n\frac{({-n},{n})_m}
{m!(\frac12)_m}
\left(\frac {1-x}2\right)^m
\hspace{-3mm}\hyp43{{-m},{-k},{k},-m+\frac12}
{n-m+1,\frac12,-m-n+1}{\dfrac{1-y}{1-x}}\nonumber\\
&&\hspace{-18mm}
+\frac{(n)_{n}}{(\frac12)_n}\left(\!\frac{x\!-\!1}2\!\right)^n
\sum_{m=1}^{k}\frac{({-k},k)_m}
{m!(\frac12)_m}\left(\!\frac{1\!-\!y}2\!\right)^m
\hspace{-3mm}\hyp43{{-n},{m-k},{k+m},
{-n+\frac12}}{m+1,m+\frac12,{-2n}+1}{\dfrac{1\!-\!y}
{1\!-\!x}}\!,\\
\nonumber \dfrac{U_n(x)U_k(y)}{(n+1)(k+1)}&=&\sum_{m=0}^n
\frac{({-n},{n+2})_m}{m!(\frac32)_m}
\left(\!\frac{1-x}2\!\right)^m
\hspace{-3mm}\hyp43{{-m},{-k},{k}+2,-m-\frac12}
{n-m+1,\frac32,-m-n-1}{\dfrac{1-y}
{1-x}}\nonumber\\ && \nonumber
+\frac{(2+n)_{n}}{(\frac32)_n}\left(\frac{x\!-\!1}2\right)^n
\sum_{m=1}^{k}\frac{({-k},k+2)_m}
{m!(\frac32)_m}\left(\frac{1\!-\!y}2\right)^m
\\&& \times \hyp43{{-n},{m-k},{k+m+2},
{-n-\frac12}}{m+1,m+\frac32,{-2n}-1}{\dfrac{1-y}
{1-x}},\\
\nonumber 
&& \hspace{-35mm}P_n(x)P_k(y)=\sum_{m=0}^n
\frac{({-n},{n+1})_m}{m!m!}
\left(\frac {1-x}2\right)^m
\hspace{-3mm}\hyp43{{-m},{-k},{k+1},-m}
{n-m+1,1,-m-n}{\dfrac{1-y}{1-x}}
\\ && \hspace{-30mm}
+\frac{(n+1)_{n}}{n!}\left(\frac{x-1}2\right)^n
\sum_{m=1}^{k}\frac{({-k},k+1)_m}
{m!m!}\left(\frac{1-y}2\right)^m
\hspace{-3mm}\hyp43{{-n},{m-k},{k+m+1},
{-n}}{m+1,m+1,{-2n}}{\dfrac{1-y}
{1-x}}\!.
\end{eqnarray}
\end{cor}

\begin{proof}
Start with Theorem \ref{Jacprodthm} and utilizing \cite[(18.7.1), (18.7.3-4), (18.7.9)]{Koekoeketal} completes the proof.
\end{proof}

\noindent In the next result we consider the Hermite polynomial case.
\begin{thm}\label{thm:limHermite}
Let $n,k\in\mathbb N_0$,
$x,y\in\CCast$. Then, one has the following product formula for the Hermite polynomials:
\begin{eqnarray}
\nonumber \dfrac{H_n(x)H_n(y)}{(2x)^n (2y)^k}&=&
\sum_{m=0}^{\lfloor n/2\rfloor} 
\dfrac{(-\lfloor n/2\rfloor,-\lfloor n/2\rfloor+(-1)^n/2)_m}
{m!}\left(\frac{-1}{x^2}\right)^m\\
&&\times\hyp32{-m,-\lfloor k/2\rfloor,-\lfloor k/2\rfloor
+(-1)^k/2}{1-m+\lfloor n/2\rfloor,1-m
+\lfloor n/2\rfloor-(-1)^n/2}{-\dfrac{y^2}{x^2}}
\nonumber\\ \nonumber&&\hspace{-10mm}+
\left(\frac{1}{x^2}\right)^{\lfloor n/2\rfloor}
(-\lfloor n/2\rfloor+(-1)^n/2)_{\lfloor n/2\rfloor}
\sum_{m=1}^{\lfloor k/2\rfloor} 
\dfrac{(-\lfloor k/2\rfloor,-\lfloor k/2\rfloor+(-1)^k/2)_m}
{m!}\left(\frac{-1}{y^2}\right)^m\\
&&\times\hyp32{-\lfloor n/2\rfloor,m-\lfloor k/2\rfloor,
m-\lfloor k/2\rfloor+(-1)^k/2}{m+1,\frac12}
{-\dfrac{y^2}{x^2}}.\label{prodHe}
\end{eqnarray}
\end{thm}
\begin{proof}
Starting with the hypergeometric series representation of
the Hermite polynomials 
\eqref{Hermdef},
one must consider the different settings for $n$, $k$ 
even and odd respectively and apply Theorem \ref{LimBTt}
setting ${\bf a}=\{-\lfloor n/2\rfloor +(-1)^n/2\}$, 
${\bf c}=\{-\lfloor k/2\rfloor +(-1)^k/2\}$ and 
${\bf b}={\bf d}=\emptyset$, and apply directly 
\eqref{LimttqCh1} and the identity follows. 
\end{proof}

\section{Generating functions via double product $q$-Chaundy representations}\label{sec:6}
In this section, we will treat generating functions 
of orthogonal polynomials in the $q$-Askey scheme, 
and also in the $q^{-1}$-Askey scheme using the $q$-Chaundy Theorem \ref{BT}. 
A generating function for a basic hypergeometric 
orthogonal polynomial $p_n(x;{\bf a}|q)$, 
where ${\bf a}$ is a multiset of parameters 
with base $q$, is given by
\begin{equation}
{\sf f}(x,t;{\bf a}|q)=\sum_{n=0}^\infty 
t^n\,h_n({\bf a}|q)\,p_n(x;{\bf a}|q),
\end{equation}
where $h_n$ is a coefficient defined such that the 
infinite series is convergent. 
Unless otherwise stated, we assume $|t|<1$ 
throughout the manuscript. 
Sometimes, other conditions on the parameters might 
be required in order for the expressions to be 
well-defined, and also, in some cases, the generating 
functions might be entire functions of $t$.
The formulas in Theorem \ref{BT} are quite general and due to that generality, they can be 
used to prove some classical generating functions for basic 
hypergeometric orthogonal polynomials in the $q$-Askey scheme. This is accomplished by fitting different terminating basic hypergeometric representations of the polynomials using Theorem \ref{BT}.

\subsection{The Askey--Wilson polynomials}\label{sec:6.1}

\medskip
\noindent As mentioned at the start of \S\ref{sec:3}, one can specifically 
start with Askey--Wilson polynomials and take the sequential limit
$d\to c\to b\to a\to 0$ to obtain the symmetric subfamilies. The Askey--Wilson polynomials \cite[\S14.1]{Koekoeketal} are 
the basic hypergeometric orthogonal polynomials, which are at the 
top of the $q$-Askey scheme and are symmetric in four parameters 
$a,b,c,d\in\CC$. One of the most important generating functions for the Askey--Wilson polynomials is that which is due to 
Ismail--Wilson, namely \cite[(1.9)]{IsmailWilson82}.

\begin{thm}
Let $q\in\CCdag$, $x\in\CC$, $a,b,c,d\in\CC$, 
$x=\frac12(z+z^{-1})\in \CCast$, 
$z\in\CCast$, $t\in\CC$ such that $|tz^\pm|<1$. Then
\begin{eqnarray}
&&\hspace{-3.0cm}\qhyp21{az,bz}{ab}{q,{t}{z}^{-1}}
\qhyp21{c{z}^{-1},{d}{z}^{-1}}{cd}{tz}=\sum_{n=0}^\infty
\frac{p_n(x;{\bf a}|q)\,t^n}{(q,ab,cd;q)_n}.
\label{AWgf}
\end{eqnarray}
\end{thm}
\begin{proof}
Starting with \eqref{qCh1} using
$p=\ell=0$, $r=s=u=v=1$, ${\bf a}=\{az,bz\}$, 
${\bf b}=\{ab\}$,
${\bf c}=\{cz^{-1},dz^{-1}\}$, ${\bf d}=\{cd\}$, 
${\sf X}=tz^{-1}$, ${\sf Y}=tz$, the terminating basic 
hypergeometric series reduces to an Askey--Wilson 
polynomial through \eqref{aw:def3}. After 
simplification, the result follows.
\end{proof}

\noindent 
Another important generating function for the Askey--Wilson polynomials, is that which
generalizes 
\cite[(9.1.15)]{Koekoeketal}. This is Rahman's generating function for the Askey--Wilson polynomials \cite{Rahman96} (see also \cite[(72)]{CohlCostasSantos23}), 
Let $k,p\in\{1,2,3,4\}$,
${\bf a}:=\{a_1,a_2,a_3,a_4\}$,
$z, t,a_k\in\mathbb C^\ast$,
$x=\frac12(z+z^{-1})\in \CCast$,
$q\in\CCdag$,
{$|a_pt|<1$}. Then
\begin{eqnarray}
&&\hspace{-0.5cm}\sum_{n=0}^\infty \frac{t^n\,(q^{-1}a_{1234};q)_n\,
p_n(x;{\bf a}|q)}
{(q,\{a_pa_s\}_{s\ne p};q)_n}\nonumber\\
&&\hspace{-0.0cm}=\frac{(ta_{1234}(qa_p)^{-1};q)_\infty}{(ta_p^{-1};q)_\infty}\,
\qhyp65
{ \pm (q^{-1}a_{1234})^\frac12, \pm (a_{1234})^\frac12,
a_p 
z^\pm 
}
{\{a_pa_s\}_{s\ne p},
ta_{1234}(qa_p)^{-1},qa_pt^{-1}}
{q,q}
\nonumber\\
&&\hspace{0.2cm}+\frac{(
\{ta_s\}_{s\ne p},
q^{-1}a_{1234},
a_p
z^\pm 
;q)_\infty}
{(\{a_pa_s\}_{s\ne p},a_pt^{-1},t
z^\pm
;q)_\infty}
\qhyp65{\pm ta_p^{-1}(q^{-1}a_{1234})^\frac12,\pm
ta_p^{-1}(a_{1234})^\frac12,
t 
z^\pm
}
{\{ta_s\}_{s\ne p},{q^{-1} a_{1234}(t a_p^{-1})^2},
qta_p^{-1}}{q,q}.
\label{genfun2ask}
\end{eqnarray}

\noindent Starting from \eqref{AWgf}, \eqref{genfun2ask}, we can also use sequential limits to obtain 
generating functions for the $q$ and $q^{-1}$-symmetric subfamilies. 
Alternatively, one can use the
rep\-re\-sen\-ta\-tions in Corollary \ref{cor32} 
with Theorem \ref{BT} to verify and derive new 
product generating functions. We now proceed in a systematic way to complete this task.

\subsection{The continuous dual {\it q-}Hahn polynomials}\label{sec:6.2}
Using the $q$-Chaundy result one
can obtain generating functions for the 
continuous dual $q$-Hahn polynomials.

\begin{thm}
\label{thm52}
Let $q\in\CCdag$, $a,b,c\in\CC$, $t\in \mathbb C$, 
$x=\frac12(z+z^{-1})\in\CCast$, $z\in\CCast$,
$|t|<|z^\pm|$. Then, 
one has the following generating function for 
continuous dual $q$-Hahn polynomials, namely
\begin{eqnarray}
&&
\label{sub1}\hspace{-5cm}\sum_{n=0}^\infty 
\frac{p_n(x;a,b,c|q)\,t^n}{(q,ab;q)_n}
=\frac{(ct;q)_\infty}
{(tz;q)_\infty}\qhyp21{az,bz}{ab}{q,tz^{-1}},
\end{eqnarray}
which is invariant under the replacement $z\mapsto z^{-1}$.
\end{thm}
\begin{proof}
The generating function \eqref{sub1} can be derived 
using Theorem \ref{BT} with $r=u=1$, $p=s=v=\ell=0$, 
${\bf a}=\{az,bz\}$, ${\bf b}=\{ab\}$, 
${\bf c}=\{cz^{-1}\}$, ${\bf d}=\emptyset$, 
${\sf X}=tz^{-1}$, ${\sf Y}=tz$ along with the 
representation of the continuous dual $q$-Hahn 
polynomials \eqref{cdqH:def4}.
For the constraint note also Remark \ref{rem:2.7}. 
We will not mention this again. This completes 
the proof.
\end{proof}

\begin{rem}
Note that the generating function \eqref{sub1} 
can also be derived by using the representation 
for the continuous dual $q$-Hahn polynomials 
 \eqref{cdqH:def3} with $s=v=1$, $p=r=u=\ell=0$,
${\bf a}=\{cz^{-1}\}$, ${\bf b}=\emptyset$, 
${\bf c}=\{az,bz\}$, ${\bf d}=\{ab\}$, ${\sf X}=tz$, 
${\sf Y}=tz^{-1}$. This is because the representations 
 \eqref{cdqH:def3}, \eqref{cdqH:def4} are related 
by the inversion transformation. 
There is a similar equivalence for Theorem \ref{thm43} 
using the representation of continuous dual $q$-Hahn 
polynomials \eqref{cdqH:def1}.
\end{rem}

\begin{thm}
\label{thm43}
Let $q\in\CCdag$, $a,b,c\in\CCast$, 
$t\in \mathbb C$, $x=\frac12(z+z^{-1})\in \CCast$, 
$z\in\CCast$, $|t|<|a|$.
Then
\begin{eqnarray}
\label{cdqHgf}
&&\hspace{-3.7cm}\sum_{n=0}^\infty
\frac{t^n\,q^{\binom{n}{2}}\,p_n(x;a,b,c|q)}
{(q,ab,ac;q)_n}=\left(-\tfrac{t}{a};q\right)_\infty
\qphyp22{-1}{az^\pm}{ab,ac}{q,-\frac{t}{a}}.
\end{eqnarray}
\end{thm}
\begin{proof}
This follows by setting $u=2$, $r=1$, $v=l=0$, 
$s=-1$, ${\bf a}=\{az^\pm\}$, ${\bf b}=\{ab,ac\}$, 
${\bf c}={\bf d}=\emptyset$ along with the 
representation
of the continuous dual $q$-Hahn polynomials 
\eqref{cdqH:def2}. Finally replacing 
$t\mapsto -ta^{-1}$ completes the proof.
\end{proof}

Another example can be generated by the 
non-standard generating function due to 
Atakishiyeva and Atakishiyev 
\cite{AtakishiyevaAtakishiyev11}
\begin{equation}
{\sf Q}(x,t;a,b,c|q):=\sum_{n=0}^\infty 
\frac{t^n\,p_n(x;a,b,c|q)}{(q,abct;q)_n}
=\frac{(ta,tb,tc;q)_\infty}{(abct,tz^\pm;q)_\infty}.
\label{ATAT}
\end{equation}
The $q$-Chaundy theorem produces the 
alternative expansions of this 
non-standard generating function.

\begin{thm}
\label{thm55}
Let $q\in\CCdag$, $x=\frac12(z+z^{-1})\in \CCast$, 
$z\in\CCast$, $a, b, c, t\in\CCast$, $|t|<\min\{|c|,|z^{\pm}|\}$. 
Then
\begin{eqnarray}
&&\hspace{-2.6cm}{\sf Q}(x,t;a,b,c|q)=\sum_{n=0}^\infty
\frac{(ac,bc;q)_n\,(t/c)^n}{(q,abct;q)_n}
\qhyp43{q^{-n},zc^\pm,\frac{q^{1-n}}{abct}}
{tz,\frac{q^{1-n}}{ac},\frac{q^{1-n}}{bc}}{q,\frac{qt}{z}}\\
&&\hspace{0.0cm}=\sum_{n=0}^\infty
\frac{(zc^\pm;q)_n\,(t/z)^n}{(q,tz;q)_n}
\qhyp43{q^{-n},ac,bc,\frac{q^{1-n}}{tz}}{abct,
\frac{q^{1-n}}{z}c^\pm}{q,\frac{qt}{c}},
\end{eqnarray}
which is invariant under the replacement $z\mapsto z^{-1}$.
\end{thm}
\begin{proof}
One can use the $q$-Chaundy Theorem \ref{BT} 
with the product generating function \eqref{ATAT} 
and identify
${\bf a}=\{ac,bc\}$, ${\bf b}=\{abct\}$, 
${\bf c}=\{zc^\pm\}$, ${\bf d}=\{tz\}$, 
$r=s=u=v=1$, $\ell=p=0$, ${\sf X}=t/c$, ${\sf Y}=t/z$, 
which upon insertion completes the proof.
\end{proof}

\medskip

\noindent Another two surprising results are new generating functions for the continuous dual $q$-Hahn polynomials which are obtained from the Rahman generating function for the Askey--Wilson polynomials \eqref{genfun2ask} by taking the $a\to 0$ and $b\to 0$ (or equivalently the $c,d\to 0$) limits.

\begin{thm}
Let $q\in\CCdag$, 
$z\in \CCast$, 
$x=\frac12(z+z^{-1})\in\CCast$, 
$a,b,c\in\CCast$, $t\in\CC$. Then
\begin{eqnarray}
&&\hspace{-4cm}\sum_{n=0}^\infty \frac{t^n\,p_n(x;a,b,c|q)}{(q;q)_n}=
\frac{(at,bt,ct;q)_\infty}{(tz^\pm;q)_\infty}
\qhyp23{tz^\pm}{at,bt,ct}{q,abct}.
\label{cdqHgf4}
\end{eqnarray}
\end{thm}

\begin{proof}
Start with the Rahman generating function for the Askey--Wilson polynomials \eqref{genfun2ask} and take the limit as $a\to 0$. The limit of the first term vanishes and the second term can be obtained
using the identity \eqref{qPochq2} and the limit
\eqref{limq2}
twice. Then replacing $d\mapsto a$ completes the proof.
\end{proof}

\begin{thm}
Let $q\in\CCdag$, $x=\frac12(z+z^{-1})\in\CCast$, 
$a,b,c\in\CCast$, $t\in\CC$. Then
\begin{eqnarray}
&&\hspace{-0.4cm}\sum_{n=0}^\infty \frac{t^n\,p_n(x;a,b,c|q)}{(q,ab,ac;q)_n}=
\frac{1}{(\frac{t}{a};q)_\infty}\qhyp43{az^\pm,0,0}{ab,ac,\frac{qa}{t}}{q,q}+\frac{(bt,ct,a^2;q)_\infty}{(ab,ac,tz^\pm,\frac{a}{t};q)_\infty}
\qhyp43{tz^\pm,0,0}{bt,ct,\frac{qt}{a}}{q,q}.
\label{cdqHgf2t}
\end{eqnarray}
\end{thm}

\begin{proof}
Start with the Rahman generating function for the Askey--Wilson polynomials \eqref{genfun2ask} and take the limit as $b\to 0$. The limit of the second term vanishes and the first term can be obtained
using the identity \eqref{qPochq2}
twice. Then replacing $d\mapsto b$ completes the proof.
\end{proof}

\subsection{The continuous dual $q^{-1}$-Hahn polynomials}\label{sec:6.3}
We can also use the $q$-Chaundy result to 
obtain generating functions for the continuous 
dual $q^{-1}$-Hahn polynomials.

\begin{thm}
\label{thm46}
Let $q\in\CCdag$, $x=\frac12(z+z^{-1})\in \CCast$, 
$z\in\CCast$, $a,b,c,t\in\CCast$, $|t|<1$. Then
\begin{eqnarray}
\label{cdqinHgf-1}
&&\hspace{-3.5cm}{\sf G}(t;a,b,c|q):=\sum_{n=0}^\infty 
\frac{t^n\,q^{2\binom{n}{2}}p_n(x;a,b,c|q^{-1})}
{(q,\frac{1}{ab},\frac{1}{ac};q)_n}
=\frac{1}{(abct;q)_\infty}
\qhyp22{\frac{z^\pm}{a}}{\frac{1}{ab},\frac{1}{ac}}{q,at}.
\end{eqnarray}
\label{cfqiT1}
\end{thm}
\begin{proof}
Start with \eqref{cdqiH:1} and identify
${\bf c}=\{\frac{z^\pm}{a}\}$, ${\bf d}=\{\frac{1}{ab},
\frac{1}{ac}\}$, ${\bf a}={\bf b}=\emptyset$, 
$v=2$, $s=1$, $u=\ell=0$, $r=p=-1$, ${\sf X}=bct$, 
${\sf Y}=t$ in Theorem \ref{BT}
with \eqref{qexp}. Finally, replacing $t\mapsto at$, 
completes the proof.
\end{proof}

\begin{rem}
Starting with the representation of the continuous 
dual $q^{-1}$-Hahn polynomials \eqref{cdqiH:2} combined 
with Theorem \ref{BT} produces Theorem \ref{cfqiT1}. 
\end{rem}

\noindent The following generating function 
has been previously discovered by Ismail, Zhang, 
and Zhou in \cite[(5.21)]{IsmailZhang2022}. 
However, we are able to prove it alternatively 
using the $q$-Chaundy Theorem \ref{BT} as follows.
\begin{thm}
\label{thm48}
Let $q\in\CCdag$, $x=\frac12(z+z^{-1})\in \CCast$, 
$z\in\CCast$, $a,b,c,t\in\CCast$, 
$|t|<1$. Then
\begin{eqnarray}
&&\hspace{-3.7cm}\sum_{n=0}^\infty 
\frac{t^n\,q^{2\binom{n}{2}}\,p_n(x;a,b,c|q^{-1})}
{(q,\frac{1}{ab};q)_n}
=\frac{(bt;q)_\infty}{(abct;q)_\infty}
\qhyp22{\frac{z^\pm}{a}}{\frac{1}{ab},bt}{q,at}
\label{cdqiHgf2b}\\
&&\hspace{1.1cm}=\frac{(\frac{abt}{z};q)_\infty}
{(abct;q)_\infty}
\qhyp21{\frac{z}{a},\frac{z}{b}}{\frac{1}{ab}}
{q,\frac{abt}{z}},
\label{cdqiHgf2}
\end{eqnarray}
where the second expression is invariant under the replacement $z\mapsto z^{-1}$ and one must require $|abt|<|z^\pm|$. 
\label{cfqiT2}
\end{thm}
\begin{proof}
Start with \eqref{cdqiH:4} and identify
${\bf a}=\{\frac{z}{a},\frac{z}{b}\}$, 
${\bf b}=\{\frac{1}{ab}\}$, ${\bf c}=\{\frac{1}{cz}\}$, 
$d=\emptyset$, $u=r=1$, $s=v=p=\ell=0$, ${\sf X}=t$, 
${\sf Y}=czt$ in Theorem \ref{BT}
with \eqref{qbinom}. Finally replacing $t\mapsto 
abt/z$ completes the proof.
\end{proof}

\begin{rem}
Starting with the representation of the continuous dual 
$q^{-1}$-Hahn polynomials \eqref{cdqiH:3} combined with 
Theorem \ref{BT} produces a generating function which is 
equivalent to Theorem \ref{cfqiT2}.
\end{rem}

Now we present the following ${}_3\phi_3$ product generating
function for the continuous dual $q^{-1}$-Hahn polynomials.

\begin{thm}
\label{thm:4.9}
Let $q\in\CCdag$, $x=\frac12(z+z^{-1})\in \CCast$, 
$z\in\CCast$, $\gamma,a,b,c,t\in\CCast$, $|t|<1$. Then
\begin{eqnarray}
&&\hspace{-4.1cm}{\sf G}_\gamma(x,t;a,b,c|q):=
\sum_{n=0}^\infty 
\frac{t^n\,(\gamma;q)_n\,q^{2\binom{n}{2}}\,
p_n(x;a,b,c|q^{-1})}
{(q,\frac{1}{ab},\frac{1}{ac};q)_n}\nonumber\\
\label{eq:thm49-gf}&&\hspace{-1.2cm}
=\frac{(\gamma abct;q)_\infty}{(abct;q)_\infty}
\qhyp33{\gamma,\frac{z^\pm}{a}}{\frac{1}{ab},
\frac{1}{ac},\gamma abct}{q,at}.
\end{eqnarray}
\label{gf33}
\end{thm}
\begin{proof}
Start with the definition of ${\sf G}_\gamma$ in 
\eqref{eq:thm49-gf} and 
insert the representation of the continuous dual $q^{-1}$-Hahn 
polynomials \eqref{cdqiH:1}, which then is a double sum over $n,k$.
Then reverse the order of summation and shift the $n$ index 
$n\mapsto n+k$. This converts the outer sum to the form of a 
$q$-binomial and the result follows.
\end{proof}

Using the $q$-Chaundy product representations 
we can obtain the following double sum 
representations of the ${}_3\phi_3$ in 
${\sf G}_\gamma(x,t;a,b,c|q)$.

\begin{cor}
\label{cor410}
Let $q\in\CCdag$, $x=\frac12(z+z^{-1})\in \CCast$, 
$z\in\CCast$, $\gamma,a,b,c,t\in\CCast$, $|t|<1$. Then
\begin{eqnarray}
&&\hspace{-1.05cm}{\sf G}_\gamma(x,t;a,b,c|q)
=\sum_{n=0}^\infty\frac{(\gamma;q)_n\,(abct)^n}{(q;q)_n}
\qhyp44{q^{-n},\gamma,\frac{z^\pm}{a}}
{\frac{1}{ab},\frac{1}{ac},\gamma abct,\frac{q^{1-n}}
{\gamma}}{q,\frac{q}{\gamma bc}}\\
&&\hspace{17mm}=\sum_{n=0}^\infty 
\frac{(\gamma,\frac{z^\pm}{a};q)_n\,(-at)^n\,q^{\binom{n}{2}}}
{(q,\frac{1}{ab},\frac{1}{ac},\gamma abct;q)_n}
\qphyp53{1}{q^{-n},\gamma,q^{1-n}ab,q^{1-n}ac,
\frac{q^{1-n}}{\gamma abct}}
{\frac{q^{1-n}}{\gamma},q^{1-n}az^\pm}{q,qabct}.
\end{eqnarray}
\end{cor}
\begin{proof}
Applying the $q$-Chaundy product representation
Theorem \ref{BT} and in particular \eqref{qCh1} 
and \eqref{qCh2} respectively produces the double
sum representations of the ${}_3\phi_3$ in 
Theorem \ref{thm:4.9}.
\end{proof}

\begin{rem}
If one takes the limit $\gamma\to 0$ then the 
representations of the generating function 
${\sf G}_\gamma$ produces the Theorem \ref{thm46}.
Taking the limit as $\gamma\to 0$ in 
Corollary \ref{cor410} produces representations 
of ${\sf G}_0$ using \eqref{cdqH:def1}, \eqref{cdqH:def2}
respectively.
Replacing $\gamma=\frac{1}{ac}$ in 
 \eqref{eq:thm49-gf} and using \eqref{rel2122} 
produces Theorem \ref{thm48}.
\end{rem}

\begin{rem}
If one sets $\gamma=q$ in Theorem \ref{thm:4.9} then one obtains the following 
interesting generating function
\begin{eqnarray}
&&\hspace{-2.4cm}\sum_{n=0}^\infty 
\frac{t^n\,q^{2\binom{n}{2}}\,p_n(x;a,b,c|q^{-1})}
{(\frac{1}{ab},\frac{1}{ac};q)_n}=\frac{(q abct;q)_\infty}
{(abct;q)_\infty}\qhyp33{q,\frac{z^\pm}{a}}
{\frac{1}{ab},\frac{1}{ac},q abct}{q,at}.
\label{cdqiHgf3}
\end{eqnarray}
\label{gf33x}
\end{rem}

\noindent 
By starting with the Rahman's generating function for the Askey--Wilson polynomials \eqref{genfun2ask}, 
we can use 
\eqref{limtcdqiH2} to obtain a new generating function for the continuous dual $q^{-1}$-Hahn polynomials.

\noindent 
\begin{thm}
Let $q\in\CCdag$, $x=\frac12(z+z^{-1})\in\CCast$, $a,b,c\in\CCast$, $t\in\CC$. Then
\begin{eqnarray}
&&\hspace{-2cm}\sum_{n=0}^\infty \frac{t^nq^{3\binom{n}{2}}p_n(x;a,b,c|q^{-1})}{(q,\tfrac{1}{ab},\tfrac{1}{ac};q)_n} 
=(-abct;q)_\infty\qhyp23{\frac{z^\pm}{a}}{\frac{1}{ab},\frac{1}{ac},-abct}{q,-at}.
\label{newgfcdqiH}
\end{eqnarray}
\end{thm}

\begin{proof}
Start with Rahman's generating function for the Askey--Wilson polynomials \eqref{genfun2ask} replace 
\begin{equation}
(a,b,c,d)\mapsto(\tfrac{1}{a},\tfrac{1}{b},\tfrac{1}{c},\tfrac{1}{d}),
\end{equation} 
multiply the top and the bottom of the summand by $q^{-3\binom{n}{2}}(-abcd)$ and take the limit as $c\to 0$ using \eqref{limtcdqiH2}. 
In the limit, the second term vanishes. Then replace $t\mapsto -qab^2d^2t$ throughout the generating function. The limit of the first term can be obtained by writing the nonterminating basic hypergeometric series as a sum, using the identity \eqref{qPochq2}
and the limit \eqref{limq2}
twice, then replace $d\mapsto c$, which completes the proof.
\end{proof}

\begin{rem}
If you start with \eqref{newgfcdqiH} and replace $t\mapsto t/a^2$ and take the limit as $a\to 0$ then one arrives at the generating function \eqref{qiASCgfY} given below. Similarly if you start with \eqref{newgfcdqiH}, replace $t\mapsto t/c$ and take the limit as $c\to 0$ then one arrives at the generating function \eqref{cfqiT2-b}, which is also given below.
\end{rem}

\subsection{The Al-Salam--Chihara polynomials} \label{sec:6.4}
The Al-Salam--Chihara polynomials
have three standard (well-known) generating functions \cite[(14.8.13-15)]{Koekoeketal} which all follow easily using the $q$-Chaundy
Theorem \ref{BT}. 

\begin{thm}
\label{ASCgfthm}
Let $q\in\CCdag$, $x=\frac12(z+z^{-1})\in \CCast$, 
$z\in\CCast$, $a,b,t\in\CCast$, $|t|<1$. Then
\begin{eqnarray}
\label{ASCgf1}
&&\hspace{-4.6cm}\sum_{n=0}^\infty\frac{t^n\,Q_n(x;a,b|q)}{(q;q)_n}=\frac{(at,bt;q)_\infty}{(tz^\pm;q)_\infty},\\
\label{ASCgf2}
&&\hspace{-4.6cm}\sum_{n=0}^\infty\frac{t^n\,Q_n(x;a,b|q)}{(q,ab;q)_n}=\frac{1}{(tz;q)_\infty}\qhyp21{az,bz}{ab}{q,\frac{t}{z}},\\
\label{ASCgf3}
&&\hspace{-4.6cm}\sum_{n=0}^\infty\frac{t^n\,q^{\binom{n}{2}}\,Q_n(x;a,b|q)}{(q,ab;q)_n}=(-\tfrac{t}{a};q)_\infty\qhyp21{az^\pm}{ab}{q,-\frac{t}{a}},
\end{eqnarray}
where $|t|<|z^{\pm}|$, $|t|<|a|$, in the second and third generating functions respectfully, so that the nonterminating Gauss basic hypergeometric series are convergent. Clearly these generating functions are  invariant under the replacement $z\mapsto z^{-1}$.
\end{thm}
\begin{proof}
The generating function \eqref{ASCgf1} follows from the representation 
 \eqref{ASC:def4} using the $q$-Chaundy Theorem \ref{BT} 
with $r=s=u=v=p=\ell=0$, ${\bf a}=\{az^{-1}\}$, ${\bf c}=\{bz\}$, ${\bf b}={\bf d}=\emptyset$, ${\sf X}=zt$, ${\sf Y}=tz^{-1}$, and the $q$-binomial theorem twice.
The generating function \eqref{ASCgf2} follows from the representation \eqref{ASC:def5} (or \eqref{ASC:def3}) with $r=p=-1$, $u=\ell=0$, $s=v=1$, ${\bf c}=\{az,bz\}$, ${\bf d}=\{ab\}$, ${\bf a}={\bf b}=\emptyset$, ${\sf X}=zt$, ${\sf Y}=tz^{-1}$, and the application of Euler's Theorem \ref{Euler} once.
The generating function \eqref{ASCgf3} follows from the representation \eqref{ASC:def1} (or \eqref{ASC:def2}) with $r=-1$, $u=p=\ell=0$, $s=v=1$, ${\bf c}=\{az^\pm\}$, ${\bf d}=\{ab\}$, ${\bf a}={\bf b}=\emptyset$, ${\sf X}={\sf Y}=t$, and the application of Euler's Theorem \ref{Euler} once.
\end{proof}

\noindent There's another generating function for Al-Salam--Chihara polynomials \cite[(14.8.16)]{Koekoeketal}
\begin{eqnarray}
\label{ASCgf4}
&&\hspace{-1.1cm}{\sf L}_\gamma(x,t;a,b|q):=\sum_{n=0}^\infty\frac{t^n\,(\gamma;q)_n\,Q_n(x;a,b|q)}{(q,ab;q)_n}=\frac{(\gamma tz;q)_\infty}{(tz;q)_\infty}\qhyp32{\gamma,az,bz}{ab,\gamma tz}{q,\frac{t}{z}},
\end{eqnarray}
where $|t|<|z^\pm|$ and if convergent ${\sf L}_\gamma$ is invariant under the replacement $z\mapsto z^{-1}$.

\begin{thm}
\label{thm:4.16}
Let $q\in\CCdag$, $x=\frac12(z+z^{-1})\in \CCast$, 
$z\in\CCast$, $a,b,t\in\CCast$, $|t|<1$. Then
\begin{eqnarray}
&&\hspace{-2.0cm}{\sf L}_\gamma(x,t;a,b|q)=\sum_{n=0}^\infty
\frac{(\gamma;q)_n\,(tz)^n}{(q;q)_n}\qhyp43{q^{-n},\gamma,az,bz}{ab,\gamma tz,\frac{q^{1-n}}{\gamma}}{q,\frac{q}{\gamma z^2}}\\
&&\hspace{0.05cm}=\sum_{n=0}^\infty
\frac{(\gamma,az,bz;q)_n\,(t/z)^n}{(q,ab,\gamma tz;q)_n}\qhyp43{q^{-n},\gamma,\frac{q^{1-n}}{ab},\frac{q^{1-n}}{\gamma tz}}{\frac{q^{1-n}}{\gamma},\frac{q^{1-n}}{az},\frac{q^{1-n}}{bz}}{q,qtz}.
\end{eqnarray}
\end{thm}
\begin{proof}
Starting with the generating function \eqref{ASCgf4}, and applying both expansions of the $q$-Chaundy Theorem \ref{BT} using $r=u=p=\ell=0$, $s=v=2$, ${\sf X}=tz$, ${\sf Y}=t/z$, ${\bf a}=\{\gamma\}$, ${\bf c}=\{\gamma,az,bz\}$, ${\bf d}=\{ab,\gamma tz\}$, ${\bf b}=\emptyset$, completes the proof.
\end{proof}
\subsection{The $q^{-1}$-Al-Salam--Chihara polynomials}\label{sec:6.5}
One has the following generating function for the $q^{-1}$-Al-Salam--Chihara polynomials which come from the representations \eqref{qiASC:1}-\eqref{qiASC:4}.

\begin{thm}
Let $q\in\CCdag$, $x=\frac12(z+z^{-1})\in \CCast$, 
$z\in\CCast$, $a,b,t\in\CCast$, $|t|<|z^\pm/(ab)|$. 
Then
\begin{eqnarray}
\label{qiASCgf1}
&&\hspace{-3.8cm}
\sum_{n=0}^\infty 
\frac{t^n\,q^{2\binom{n}{2}}Q_n(x;a,b|q^{-1})}
{(q,\frac{1}{ab};q)_n}
=\left(\frac{abt}{z};q\right)_\infty
\qhyp21{\frac{z}{a},\frac{z}{b}}{\frac{1}{ab}}{q,\frac{abt}{z}},
\end{eqnarray}
which is invariant under the replacement $z\mapsto z^{-1}$.
\label{cfqiT2-b}
\end{thm}
\begin{proof}
This generating function can be obtained by 
starting with the $q$-Chaundy Theorem \ref{BT} with representation \eqref{qiASC:4} (or \eqref{qiASC:5}), $s=-1$, $v=p=\ell=0$, $u=r=1$, ${\bf a}=\{\frac{1}{az},\frac{1}{bz}\}$, ${\bf b}=\{\frac{1}{ab}\}$, ${\bf c}={\bf d}=\emptyset$, ${\sf X}={\sf Y}=t$, replacing $t\mapsto abtz$ and then $z\mapsto z^{-1}$. Similarly, one can take \eqref{qiASCgf1} and take the limit as $c\to 0$. This completes the proof. 
\end{proof}

\noindent Similarly, from \eqref{qiASC:3}, we obtain the following infinite product generating function which was originally obtained in \cite{AskeyIsmail84} (see also \cite[(5.29)]{IsmailMasson1994}, \cite[(3.1)]{ChristansenIsmail2006}).

\begin{thm}
\label{thm420}
Let $q\in\CCdag$, $x=\frac12(z+z^{-1})\in \CCast$, 
$z\in\CCast$, $a,b,t\in\CCast$, $|t|<1$. Then 
\begin{eqnarray}
&&\hspace{-6.7cm}\sum_{n=0}^\infty \frac{t^n\,q^{\binom{n}{2}}\,Q_n(x;a,b|q^{-1})}{(q;q)_n}
=\frac{(-tz^\pm;q)_\infty}{(-ta,-tb;q)_\infty}.
\label{qiASCgfY}
\end{eqnarray}
\end{thm}
\begin{proof}
This generating function can be obtained by 
starting with the $q$-Chaundy Theorem 
\ref{BT} with representation \eqref{qiASC:3} 
and $s=-1$, $r=u=s=v=p=\ell=0$, 
${\bf a}=\{\frac{1}{az}\}$, 
${\bf c}=\{\frac{z}{b}\}$, 
${\bf b}={\bf d}=\emptyset$, ${\sf X}=at$, 
${\sf Y}=bt$, and replacing $t\mapsto -t$. 
This completes the proof. 
\end{proof}

\noindent By starting with \eqref{qiASC:2} (or 
\eqref{qiASC:1}) and the $q$-Chaundy Theorem 
\ref{BT}, we can obtain another generating 
function.

\begin{thm}
Let $q\in\CCdag$, $x=\frac12(z+z^{-1})\in \CCast$, 
$z\in\CCast$, $a,b,t\in\CCast$, $|ta|<1$. 
Then
\begin{eqnarray}
&&\hspace{-4.2cm}\sum_{n=0}^\infty 
\frac{t^n\,q^{\binom{n}{2}}\,Q_n(x;a,b|q^{-1})}
{(q,\frac{1}{ab};q)_n}=
\frac{1}{(-bt;q)_\infty}
\qhyp21{\frac{z^\pm}{a}}{\frac{1}{ab}}{q,-ta}.
\label{qiASCgfX}
\end{eqnarray}
\end{thm}
\begin{proof}
This generating function can be obtained by 
starting with the $q$-Chaundy Theorem 
\ref{BT} with representation \eqref{qiASC:2} 
(or \eqref{qiASC:1}) and $s=\ell=-1$, 
$u=v=p=0$, $r=u=1$, 
${\bf a}=\{\frac{z^\pm}{a}\}$, 
${\bf b}=\{\frac{1}{ab}\}$, 
${\bf c}={\bf d}=\emptyset$, ${\sf X}=at$, 
${\sf Y}=bt$, and replacing $t\mapsto -t$. 
This completes the proof. 
\end{proof}

One also has the following interesting 
generating function for the $q^{-1}$-Al-Salam--Chihara polynomials with 
arbitrary parameter $\gamma$.

\begin{thm}
\label{thm:4.17}
Let $q\in\CCdag$, $x=\frac12(z+z^{-1})\in \CCast$, 
$z\in\CCast$, $\gamma, a,b,t\in\CCast$, 
$|at|<1$. 
Then
\begin{eqnarray}
\hspace{-6mm}&&{\sf H}_\gamma(x,t;a,b|q)
:=\sum_{n=0}^\infty\frac{t^n(\gamma;q)_n
q^{\binom{n}{2}}Q_n(x;a,b|q^{-1})}
{(q,\frac{1}{ab};q)_n}=\frac{(-\gamma bt;q)_\infty}
{(-bt;q)_\infty}\qhyp32{\gamma,\frac{z^\pm}{a}}
{\frac{1}{ab},-\gamma bt}{q,-at}.
\label{gf32h}
\end{eqnarray}
\end{thm}
\begin{proof}
Start with the definition of ${\sf H}_\gamma$ in 
 \eqref{gf32h} and insert the representation of 
the $q^{-1}$-Al-Salam--Chihara polynomials 
 \eqref{qiASC:1}, which then is a double sum 
over $n,k$, then reverse the order of summation 
and shift the $n$ index $n\mapsto n+k$. 
This converts the outer sum to the form of a 
$q$-binomial and the result follows.
\end{proof}

\noindent Using the $q$-Chaundy product representations 
we can obtain the following double sum representations 
of the ${}_3\phi_2$ in 
${\sf H}_\gamma(x,t;a,b|q)$.

\begin{cor}
\label{cor520}
Let $q\in\CCdag$, $x=\frac12(z+z^{-1})\in \CCast$, 
$z\in\CCast$, $\gamma, a,b,t\in\CCast$, $|t|<1$. Then
\begin{eqnarray}
&&\hspace{-0.7cm}{\sf H}_\gamma(x,t;a,b|q)
=\sum_{n=0}^\infty\frac{(\gamma;q)_n\,(-bt)^n}
{(q;q)_n}\qhyp43{q^{-n},\gamma,\frac{z^\pm}{a}}
{\frac{1}{ab},-\gamma bt,\frac{q^{1-n}}{\gamma}}
{q,\frac{qa}{\gamma b}}\\
&&\hspace{1.72cm}=\sum_{n=0}^\infty 
\frac{(\gamma,\frac{z^\pm}{a};q)_n\,(-at)^n}
{(q,\frac{1}{ab},-\gamma bt;q)_n}
\qhyp43{q^{-n},\gamma,q^{1-n}ab,-\frac{q^{1-n}}
{\gamma bt}}
{\frac{q^{1-n}}{\gamma},q^{1-n}az^\pm}{q,-qbt}.
\end{eqnarray}
\end{cor}
\begin{proof}
Applying the $q$-Chaundy product representation
Theorem \ref{BT} and, in particular, \eqref{qCh1} and 
 \eqref{qCh2} respectively produces the double
sum representations of the ${}_3\phi_2$ in 
Theorem \ref{thm:4.17}.
\end{proof}

\begin{rem}
Inserting $\gamma=\frac{1}{ab}$ in Theorem 
 \ref{thm:4.17} produces Theorem \ref{thm420}.
\end{rem}

\noindent 
There is another generating function with one free parameter for the $q^{-1}$-Al-Salam--Chihara polynomials that Ismail discovered in \cite[Theorem 8.2]{Ismail2020},
namely
\begin{eqnarray}
\hspace{-2cm}\sum_{n=0}^\infty
\frac{q^{\binom{n}{2}}t^n(\gamma;q)_n}{(q,q;q)_n}Q_n(x;a,b|q^{-1})=
\frac{(\gamma,-tz^\pm;q)_\infty}{(q,-at,-bt;q)_\infty}\qhyp32{\frac{q}{\gamma},-at,-bt}{-tz^\pm}{q,\gamma}.
\label{qiASCgfIsm}
\end{eqnarray}
This generating function can be obtained from \cite[Theorem 8.2]{Ismail2020} after identifying Ismail's $q^{-1}$-Al-Salam--Chihara polynomial $Q_n(x;a,b)$ as
\begin{equation}
\hspace{0.5cm}Q_n(x;a,b)=\frac{q^{\binom{n}{2}}i^n}{(q;q)_n}Q_n[iz;ia,ib|q^{-1}],
\end{equation}
and replacing $(t,z,a,b)\mapsto-i(t,z,a,b)$. 
This is an extremely nice generating function which gives 
straightforward limits for corresponding generating functions for the continuous big $q^{-1}$-Hermite polynomials and the continuous $q^{-1}$-Hermite polynomials. Using Ismail's series rearrangement method, we were able to find a generalization of Ismail's generating function \eqref{qiASCgfIsm} with another arbitrary free parameter $\delta$. We present this now.

\begin{thm}Let $q\in\CCdag$, $x=\frac12(z+z^{-1})$, $a,b,t,\delta\in\CC$, $z,\gamma\in\CCast$ such that $|\gamma|<1$. Then
\begin{eqnarray}
\sum_{n=0}^\infty
\frac{q^{\binom{n}{2}}t^n(\gamma;q)_n}{(q,\delta;q)_n}Q_n(x;a,b|q^{-1})=
\frac{(\gamma,-tz^\pm;q)_\infty}{(\delta,-at,-bt;q)_\infty}\qhyp32{\frac{\delta}{\gamma},-at,-bt}{-tz^\pm}{q,\gamma}.
\label{gfIsmg}
\end{eqnarray}
\end{thm}

\begin{proof}
As in the proof of \cite[Theorem 8.2]{Ismail2020}, start with the left-hand side of \eqref{gfIsmg} and insert the representation of the $q^{-1}$-Al-Salam--Chihara polynomials \eqref{qiASC:3} with $z\mapsto z^{-1}$. The write the terminating ${}_2\phi_1$ as a finite sum over $k$, reverse the order of the sum, shift $n\mapsto n+k$ and apply the Heine transformation \cite[(III.1)]{GaspRah} so that the argument of the resulting ${}_2\phi_1$ is $q^k\gamma$, then expressing this as an infinite series, and reversing the sum produces a sum which can be evaluated using the $q$-binomial theorem \eqref{qbinom}, which completes the proof.
\end{proof}

\noindent 
The above generating function \eqref{qiASCgfIsm} is extremely interesting for several reasons. For instance, it has limits to the continuous big $q^{-1}$-Hermite polynomials and the continuous $q^{-1}$-Hermite polynomials without losing the free $\gamma$ and $\delta$ parameters! It also has some interesting limits. If one takes $\gamma=q$ then one obtains the following result.

\begin{cor}
Let $q\in\CCdag$, $a,b,t,\delta\in\CC$. Then
\begin{eqnarray}
\sum_{n=0}^\infty
\frac{q^{\binom{n}{2}}t^n}{(\delta;q)_n}Q_n(x;a,b|q^{-1})=
\frac{(q,-tz^\pm;q)_\infty}{(\delta,-at,-bt;q)_\infty}\qhyp32{\frac{\delta}{q},-at,-bt}{-tz^\pm}{q,q}.
\label{gfIsmg2}
\end{eqnarray}
\end{cor}
\begin{proof}
Taking the limit $\gamma\to q$ in \eqref{gfIsmg} completes the proof.
\end{proof}
\noindent If one then sets either $\gamma=\delta$ in \eqref{gfIsmg} or $\delta=q$ in \eqref{gfIsmg2}, then one can see that they both generalize \eqref{qiASCgfY}.

\subsection{The continuous big {\it q-}Hermite polynomials}\label{sec:6.6}

The continuous big $q$-Hermite polynomials
have three standard (well-known) generating 
functions \cite[(14.18.13-14)]{Koekoeketal} which all 
follow easily using the $q$-Chaundy Theorem \ref{BT}. 

\begin{thm}
Let $q\in\CCdag$, $x=\frac12(z+z^{-1})\in \CCast$, 
$z\in\CCast$, $a,t\in\CCast$, $|t|<1$. Then
\begin{eqnarray}
\label{cbqHgf1}
&&\hspace{-5cm}\sum_{n=0}^\infty\frac{t^n\,H_n(x;a|q)}{(q;q)_n}=\frac{(at;q)_\infty}{(tz^\pm;q)_\infty},\\
&&\hspace{-5cm}\sum_{n=0}^\infty\frac{t^n\,q^{\binom{n}{2}}\,H_n(x;a|q)}{(q;q)_n}
=(-tz;q)_\infty\qhyp11{az}{-{t}{z}}{q,-\frac{t}{z}},
\label{gf2cbqH}
\end{eqnarray}
which are invariant under the replacement $z\mapsto z^{-1}$.
\end{thm}
\begin{proof}
One can use the $q$-Chaundy Theorem \ref{BT} with the representations \eqref{cbqH:def1}-- \eqref{cbqH:def4}. 
For instance, the generating function \eqref{cbqHgf1} follows
with \eqref{qbinom}, \eqref{qexp}, $p=v=0$, 
$r=u=1$, $s=\ell=-1$, 
${\bf a}=\{az\}$, 
${\bf b}={\bf c}={\bf d}=\emptyset$, ${\sf X}=tz^{-1}$, 
${\sf Y}=tz$ along with the representation
of the continuous big $q$-Hermite polynomials 
 \eqref{cbqH:def3}.
However, it is easier to just take the limit
as $b\to 0$ in Theorem 
 \ref{ASCgfthm}. Note that the limit as $b\to 0$ in both \eqref{ASCgf1}, \eqref{ASCgf2} produce \eqref{cbqHgf1}. 
The generating function \eqref{gf2cbqH} follows in a similar way except one may also use
\cite[(1.13.12)]{Koekoeketal}. This completes the proof.
\end{proof}

\noindent 
Using the compact connection relation for the continuous big $q$-Hermite polynomials we can obtain the following generalized generating function.
\begin{thm}
Let $q\in\CCdag$, $x=\frac12(z+z^{-1})\in\CCast$, $t,a,b\in\CCast$, $|t|<|a|$. Then
\begin{eqnarray}
(-\tfrac{t}{a};q)_\infty\qphyp201{az^\pm}{-}{q,-\frac{t}{a}}=\sum_{n=0}^\infty H_n(x;b|q)\frac{q^{\binom{n}{2}}t^n}{(q;q)_n}\qhyp11{\frac{a}{b}}{0}{q,-q^nbt}.
\end{eqnarray}
\end{thm}
\begin{proof}
Starting with the generating function
\eqref{gf2cbqH} and inserting the connection relation, rearranging and simplifying completes the proof.
\end{proof}

\noindent Another product generating 
function for the continuous $q$-Hermite polynomials 
is \cite[(14.18.15)]{Koekoeketal}
\begin{eqnarray}
\label{cbqHgf4}
&&\hspace{-1.6cm}{\sf M}_\gamma(x,t;a|q):=
\sum_{n=0}^\infty
\frac{t^n\,(\gamma;q)_n\,H_n(x;a|q)}
{(q;q)_n}=\frac{(\gamma tz;q)_\infty}
{(tz;q)_\infty}\qhyp21{\gamma,az}{\gamma tz}
{q,\frac{t}{z}},
\end{eqnarray}
where $|t|>|z^{\pm}|$ and is invariant under the replacement $z\mapsto z^{-1}$.
One may use the $q$-Chaundy Theorem \ref{BT} 
to produce alternative expansions of this 
generating function which we reproduce in 
the following theorem.
\begin{thm}
Let $q\in\CCdag$, $x=\frac12(z+z^{-1})\in \CCast$, 
$z\in\CCast$, $\gamma, a,t\in\CCast$, $|t|<1$. 
Then
\begin{eqnarray}
&&\hspace{-3.0cm}{\sf M}_\gamma(x,t;a|q)=\sum_{n=0}^\infty
\frac{(\gamma;q)_n\,(tz)^n}{(q;q)_n}\qhyp32{q^{-n},\gamma,az}{\gamma tz,\frac{q^{1-n}}{\gamma}}{q,\frac{q}{\gamma z^2}},\\
&&\hspace{-0.85cm}=\sum_{n=0}^\infty
\frac{(\gamma,az;q)_n\,(t/z)^n}{(q,\gamma tz;q)_n}\qhyp32{q^{-n},\gamma,\frac{q^{1-n}}{\gamma tz}}{\frac{q^{1-n}}{\gamma},\frac{q^{1-n}}{az}}{q,\frac{qtz^2}{a}}.
\end{eqnarray}
\end{thm}
\begin{proof}
One can use the $q$-Chaundy Theorem \ref{BT} with the product 
generating function \eqref{cbqHgf4}. However, it is easier 
to take the limit as $b\to 0$ in Theorem \ref{thm:4.16}. 
This completes the proof. 
\end{proof}

\noindent 
Using the compact connection relation for the continuous big $q$-Hermite polynomials we can obtain another generalized generating function for the continuous big $q$-Hermite polynomials.
\begin{thm}
Let $q\in\CCdag$, $x=\frac12(z+z^{-1})\in\CCast$, $t,a,b\in\CCast$, $|t|<|a|$. Then
\begin{eqnarray}
\frac{(\gamma tz;q)_\infty}
{(tz;q)_\infty}\qhyp21{\gamma,az}{\gamma tz}
{q,\frac{t}{z}}=\sum_{n=0}^\infty H_n(x;b|q)t^n\frac{(\gamma;q)_n}{(q;q)_n}\qhyp21{\frac{a}{b},q^k\gamma}{0}{q,bt},
\end{eqnarray}
which is invariant under the replacement $z\mapsto z^{-1}$.
\end{thm}
\begin{proof}
Starting with the generating function
${\sf M}_\gamma$ and inserting the connection relation, rearranging and simplifying completes the proof.
\end{proof}

\subsection{The continuous big $q^{-1}$-Hermite polynomials}\label{sec:6.7}

From \eqref{qiASCgf1}, we can obtain the following generating function for the continuous big $q^{-1}$-Hermite polynomials.

\begin{thm}
Let $q\in\CCdag$, $x=\frac12(z+z^{-1})\in \CCast$, 
$z\in\CCast$, $a,t\in\CCast$, $|t|<1$. Then
\begin{eqnarray}
\label{cbqinHegf-2}
&&\hspace{-7.6cm}
{\sf I}(t;a|q):=\sum_{n=0}^\infty 
\frac{q^{\binom{n}{2}}t^nH_n(x;a|q^{-1})}
{(q;q)_n}
=\frac{(-tz^\pm;q)_\infty}{(-ta;q)_\infty}.
\end{eqnarray}
\label{gfqbqiH1}
\end{thm}
\begin{proof}
If you replace $t\mapsto t/a$ and take the limit as $a\to 0$ in \eqref{qiASCgf1} and then replace $b\mapsto a$ you obtain the following generating function.
\end{proof}

\begin{rem}
If one starts with Theorem \ref{thm420} and takes the limit $b\to 0$, one arrives at Theorem \ref{gfqbqiH1}. Also, if one uses the $q$-Chaundy Theorem \ref{BT} with representations \eqref{cbqiH:3} or \eqref{cbqiH:4}, one arrives at the same generating function.
Unfortunately, the application of the $q$-Chaundy Theorem \ref{BT} to the representations \eqref{cbqiH:1}, \eqref{cbqiH:2} leads to a product of two nonterminating basic hypergeometric series with either one or both being divergent. 
\end{rem}

\noindent 
There is a generating function with two free parameters which comes from the generalization of Ismail's generating function \eqref{gfIsmg} and taking the limit as $b\to 0$.

\begin{thm}Let $q\in\CCdag$, $a,t,\delta\in\CC$, $\gamma\in\CCast$ such that $|\gamma|<1$. Then
\begin{eqnarray}
\sum_{n=0}^\infty
\frac{q^{\binom{n}{2}}t^n(\gamma;q)_n}{(q,\delta;q)_n}H_n(x;a|q^{-1})=
\frac{(\gamma,-tz^\pm;q)_\infty}{(\delta,-at;q)_\infty}\qhyp32{\frac{\delta}{\gamma},-at,0}{-tz^\pm}{q,\gamma}.
\label{gfIsmgH}
\end{eqnarray}
\end{thm}

\begin{proof}
Starting with \eqref{gfIsmg} and taking the limit as $b\to 0$, completes the proof.
\end{proof}
\noindent This generating function has some interesting limits. The limit $\gamma\to q$ gives the following result.

\begin{cor}
Let $q\in\CCdag$, $a,t,\delta\in\CC$, $\gamma\in\CCast$ such that $|\gamma|<1$. Then
\begin{eqnarray}
\sum_{n=0}^\infty
\frac{q^{\binom{n}{2}}t^n}{(\delta;q)_n}H_n(x;a|q^{-1})=
\frac{(q,-tz^\pm;q)_\infty}{(\delta,-at;q)_\infty}\qhyp32{\frac{\delta}{q},-at,0}{-tz^\pm}{q,\gamma}.
\label{gfIsmgH2}
\end{eqnarray}
\end{cor}
\begin{proof}
Taking the limit $\gamma\to q$ in 
\eqref{gfIsmgH} completes the proof.
\end{proof}
\noindent If one sets $\gamma=\delta$ in \eqref{gfIsmgH} or $\delta=q$ in \eqref{gfIsmgH2} then one can see that they both are generalizations of \eqref{cbqinHegf-2}.

\subsection{The continuous {\it q-}Hermite polynomials}\label{sec:6.8}

The standard generating function for the continuous 
$q$-Hermite polynomials \cite[(14.26.11)]{Koekoeketal} 
can be easily obtained using the $q$-Chaundy 
Theorem \ref{BT}.

\begin{thm}
Let $q\in\CCdag$, $x=\frac12(z+z^{-1})\in \CCast$, $z\in\CCast$, $t\in\CC$, $|t|<1$. Then
\begin{equation}
\sum_{n=0}^\infty\frac{t^n\,H_n(x|q)}{(q;q)_n}=\frac{1}{(tz^\pm;q)_\infty}. 
\label{cqHgf}
\end{equation}
\end{thm}
\begin{proof}
The generating function \eqref{cqHgf} follows easily from the representation 
 \eqref{cqHrep} using the $q$-Chaundy Theorem \ref{BT} 
with $r=s=-1$, $u=v=0$, $p=\ell=-1$, ${\bf a}={\bf b}={\bf c}={\bf d}=\emptyset$, ${\sf X}=zt$, ${\sf Y}=tz^{-1}$, and Euler's Theorem \ref{Euler} twice.
\end{proof}

\noindent Another generating function for 
continuous $q$-Hermite polynomials is given by
 \cite[(14.26.12)]{Koekoeketal}
\begin{equation}
{\sf J}(x,t|q):=\sum_{n=0}^\infty\frac{q^{\binom{n}{2}}\,t^n\,H_n(x|q)}{(q;q)_n}=(-tz;q)_\infty\qphyp{0}{1}{-1}{-}{-tz}{q,-\frac{t}{z}},
\label{cqHgf2}
\end{equation}
where $|t|<|z^{\pm}|$ and is invariant under the replacement $z\mapsto z^{-1}$.
By applying the $q$-Chaundy Theorem \ref{BT}, we can obtain the following results.
\begin{thm}
Let $q\in\CCdag$, $x=\frac12(z+z^{-1})\in \CCast$, $z\in\CCast$, $t\in\CCast$, $|t|<1$. Then
\begin{eqnarray}
&&\hspace{-4.4cm}{\sf J}(x,t|q)=\sum_{n=0}^\infty
\frac{q^{\binom{n}{2}}(tz)^n}{(q;q)_n}\qhyp11{q^{-n}}{-tz}{q,\frac{q}{z^2}},\\
&&\hspace{-3.0cm}=\sum_{n=0}^\infty
\frac{q^{\binom{n}{2}}(t/z)^n}{(q,-tz;q)_n}\qphyp201{q^{-n},-\frac{q^{1-n}}{tz}}{-}{q,-q^ntz^3}.
\end{eqnarray}
\end{thm}
\begin{proof}
Starting with the generating function \eqref{cqHgf2}, and applying both expansions of the $q$-Chaundy Theorem \ref{BT} using $r=s=\ell=-1$, $u=p=0$, $v=1$, ${\sf X}=-tz$, ${\sf Y}=-t/z$, ${\bf d}=\{-tz\}$, ${\bf a}={\bf b}={\bf c}=\emptyset$ completes the proof.
\end{proof}

\noindent 
We can produce a non-standard generating function for the continuous big $q$-Hermite polynomials by applying the connection relation, Corollary \ref{cor327}, to ${\sf J}(x,t|q)$.

\begin{thm}Let $q\in\CCdag$, $a,t\in\CCast$. Then
\begin{equation}
\sum_{n=0}^\infty H_n(x;a|q)
\frac{q^{\binom{n}{2}}t^n}{(q,-at;q)_n}
=\frac{(-tz;q)_\infty}{(-at;q)_\infty}
\qhyp11{0}{-tz}{q,-\frac{t}{z}},
\label{gcbqHgf3}
\end{equation}
which is invariant under the replacement $z\mapsto z^{-1}$.
\end{thm}

\begin{proof}
Start with ${\sf J}(x,t|q)$, insert the connection relation, Corollary \ref{cor327}, 
and reversing the order of the sum leaves a ${}_0\phi_0$ which can be summed using \eqref{qexp2}, which completes the proof.
\end{proof}

\noindent A third generating function for the continuous $q$-Hermite polynomials is given by \cite[(14.26.13)]{Koekoeketal}
\begin{equation}
{\sf K}(x,t|q):=\sum_{n=0}^\infty
\frac{t^n\,(\gamma;q)_n\,H_n(x|q)}{(q;q)_n}=\frac{(\gamma tz;q)_\infty}{(tz;q)_\infty}\qphyp{1}{1}{-1}{\gamma}{\gamma tz}{q,\frac{t}{z}},
\label{cqHgf3}
\end{equation}
which is invariant under the replacement $z\mapsto z^{-1}$.
\begin{thm}
Let $q\in\CCdag$, $x=\frac12(z+z^{-1})\in \CCast$, $z\in\CCast$, $t\in\CCast$, $|t|<1$. Then
\begin{eqnarray}
&&\hspace{-4.0cm}{\sf K}(x,t|q)=\sum_{n=0}^\infty
\frac{(\gamma;q)_n\,(tz)^n}{(q;q)_n}\qphyp22{-1}{q^{-n},\gamma}{\gamma tz,\frac{q^{1-n}}{\gamma}}{q,\frac{q}{\gamma z^2}},\\
&&\hspace{-2.5cm}=\sum_{n=0}^\infty
\frac{(\gamma;q)_n\,(t/z)^n}{(q,\gamma tz;q)_n}\qhyp31{q^{-n},\gamma,-\frac{q^{1-n}}{\gamma tz}}{\frac{q^{1-n}}{\gamma}}{q,q^ntz^3}.
\end{eqnarray}
\end{thm}
\begin{proof}
Starting with the generating function \eqref{cqHgf3}, and applying both expansions of the $q$-Chaundy Theorem \ref{BT} using $\ell=-1$, $r=s=u=p=0$, $v=1$, ${\sf X}=tz$, ${\sf Y}=t/z$, ${\bf a}={\bf c}=\{\gamma\}$, ${\bf d}=\{\gamma tz\}$, ${\bf b}=\emptyset$, completes the proof.
\end{proof}

\noindent 
We can also produce a non-standard generating function for the continuous big $q$-Hermite polynomials by applying the connection relation, Corollary \ref{cor327}, to ${\sf K}(x,t|q)$.

\begin{thm}Let $q\in\CCdag$, $x=\frac12(z+z^{-1})\in\CCast$, $a,t\in\CCast$, $|t|<|z|$. Then
\begin{equation}
\sum_{n=0}^\infty H_n(x;a|q)
\frac{t^n(\gamma;q)_n}{(q,\gamma at;q)_n}
=\frac{(at,\gamma tz;q)_\infty}{(tz,\gamma at;q)_\infty}
\qhyp21{\gamma,0}{\gamma tz}{q,\frac{t}{z}},
\label{gcbqHgf5}
\end{equation}
which is invariant under the replacement $z\mapsto z^{-1}$.
\end{thm}

\begin{proof}
Start with ${\sf K}(x,t|q)$, insert the connection relation, Corollary \ref{cor327}, 
and reversing the order of the sum leaves a ${}_0\phi_0$ which can be summed using \eqref{qexp2}, which completes the proof.
\end{proof}

One also has the following interesting generating function due to Ismail for the continuous $q$-Hermite polynomials cf.~ \cite[Theorem 14.2.1]{Ismail:2009:CQO}.

\begin{thm}
Let $q\in\CCdag$, $x=\frac12(z+z^{-1})\in \CCast$, 
$z\in\CCast$, $t\in\CCast$, 
$|t|<1$. Then
\begin{eqnarray}
&&\hspace{-2.8cm}{\sf O}(x,t|q):=\sum_{n=0}^\infty\frac{q^{\frac12\binom{n}{2}}(q^\frac14 t)^n H_n(x|q)}{(q;q)_n}=(-t;q^\frac12)_\infty\qhyp21{q^\frac14z^\pm}{-q^\frac12}{q^\frac12,-t}.
\label{cqHgf4}
\end{eqnarray}
\label{thm634}
\end{thm}
\begin{proof}
One should see the proof of \cite[Theorem 14.2.1]{Ismail:2009:CQO}.
\end{proof}

\noindent 
One may use the connection relation, Theorem \ref{cor327}, in conjunction with the Ismail--Masson exponential generating function ${\sf O}(x,t|q)$ to obtain a generalization of this generating function with continuous big $q$-Hermite polynomials.

\begin{cor}Let $q\in\CCdag$,
$a,t\in\CC$. Then
\begin{equation}
{\sf O}(x,t|q)=
\sum_{n=0}^\infty
H_n(x;a|q)
\frac{q^{\frac12\binom{n}{2}}\left(q^\frac14t\right)^n}{(q;q)_n}
\qhyp11{0}{-q^\frac12}{q^\frac12,-q^{\frac14+\frac12n}at}.
\label{cqHgf4coreq}
\end{equation} 
\label{cqHgf4cor}
\end{cor}

\begin{proof}
Start with the Ismail--Masson exponential generating function for the continuous $q^{-1}$-Hermite polynomials, apply the connection relation, Theorem \ref{cor334}, and reversing the order of the two sums and simplifying completes the proof.
\end{proof}

\noindent Using the $q$-Chaundy Theorem \ref{BT}, 
one is able to derive alternate expressions for 
the generating function for the continuous $q$-Hermite 
polynomials ${\sf O}(x,t|q)$.

\begin{thm}
\label{thm430}
Let $q\in\CCdag$, $x=\frac12(z+z^{-1})\in \CCast$, $z\in\CCast$, $t\in\CCast$, $|t|<1$. Then
\begin{eqnarray}
&&\hspace{-3.8cm}{\sf O}(x,t|q)=\sum_{n=0}^\infty
\frac{q^{\frac12\binom{n}{2}}\,t^n}{(q^\frac12;q^\frac12)_n}\qphyp311{q^{-\frac12n},q^\frac14z^\pm}{-q^\frac12}{q^\frac12,q^{\frac12}},\\
&&\hspace{-2.3cm}=\sum_{n=0}^\infty
\frac{(q^\frac14z^\pm;q^\frac12)_n\,(-t)^n}{(\pm q^\frac12;q^\frac12)_n}\qhyp22{\pm q^{-\frac12n}}{q^{\frac14-\frac12n}z^\pm}{q^\frac12,-q^\frac12}.
\end{eqnarray}
\end{thm}
\begin{proof}
Starting with the generating function \eqref{cqHgf4}, 
and replacing $q\mapsto q^2$ converts the right-hand 
side to a form where the $q$-Chaundy Theorem \ref{BT} 
can be used.
Using $r=-1$, $u=p=\ell=0$, $s=v=1$, ${\sf X}={\sf Y}=-t$, 
${\bf c}=\{q^\frac12z^\pm\}$, ${\bf d}=\{-q\}$, 
${\bf a}={\bf b}=\emptyset$, and then replacing 
$q\mapsto q^\frac12$ completes the proof.
\end{proof}

\begin{rem} 
\label{rem5554} 
One should observe the surprising fact that the alternate expressions 
for ${\sf O}(x,t|q)$ in Theorem \ref{thm430}, have the property that the 
terminating basic hypergeometric series are only a 
function of $x$, $q$ and $n$. Comparing these expressions 
with the original generating function, it can 
be seen that these terminating basic hypergeometric 
series must represent alternative basic 
hypergeometric representations for the continuous 
$q$-Hermite polynomials!
\end{rem}

\noindent Remark \ref{rem5554} leads us 
to Theorem \ref{thm542} which we presented above, namely two new quadratic terminating representations for the continuous $q$-Hermite polynomials. That result 
then leads us to the following quadratic transformations for terminating basic 
hypergeometric series.

\begin{thm}
\label{cqHtran}
Let $n\in\mathbb N_0$, $q\in\CCddag$, $z\in\CCast$. Then, one has 
the following terminating quadratic transformation:
\begin{eqnarray}
&&\hspace{-1.5cm}z^n\qphyp10{-1}{q^{-2n}}{-}{q^2,\frac{q^{2n}}{z^2}}=q^{-\frac12n}(-q;q)_n\qphyp311{q^{-n},q^\frac12z^\pm}{-q}{q,q}\\
&&\hspace{2.9cm}=(-1)^n
q^{-\frac12n^2}(q^\frac12z^\pm;q)_n\qhyp22{\pm q^{-n}}{q^{\frac12-n}z^\pm}{q,-q},
\end{eqnarray}
which is invariant under the transformation $z\mapsto z^{-1}$.
\end{thm}
\begin{proof}
Comparing \eqref{acqH1}, \eqref{acqH2}, 
 \eqref{cqHrep} and making the replacement 
$q\mapsto q^2$ completes the proof.
\end{proof}

\noindent The above terminating quadratic 
transformation formula leads to an interesting 
summation formula.
\begin{cor}\label{cor:6.35}
Let $n\in\mathbb N_0$, $q\in\CCddag$. Then, one 
has the following summation formula
\begin{equation}
\qphyp10{-1}{q^{-2n}}{-}{q^2,q^{2n\mp 1}}=q^{-n(\frac12\pm\frac12)}(-q;q)_n,
\label{cor535}
\end{equation}
where unlike Definition \ref{def:1.1}, the $\pm$ and $\mp$ factors in the exponents represent two separate cases.
\end{cor}
\begin{proof}
Setting $z=q^{\pm\frac12}$ in Theorem \ref{cqHtran} completes the proof.
\end{proof}
\subsection{The continuous $q^{-1}$-Hermite polynomials}\label{sec:6.9}
The following result can be found in 
\cite[Theorem 2.22]{CohlIsmail20}. The result 
can be found using the $q$-Chaundy Theorem 
\ref{BT}, but we provide a slightly 
different proof.
This infinite product generating function 
was originally found in 
\cite[(2.4)]{IsmailMasson1994}.

\begin{thm}\label{thm:6.36}
Let $q\in\CCdag$, $x=\frac12(z+z^{-1})\in \CCast$, 
$z\in\CCast$, $|t|<1$. Then
\begin{eqnarray}
&&\hspace{-8.0cm}\sum_{n=0}^\infty\frac{t^nq^{\binom{n}2}\,H_n(x|q^{-1})}
{(q;q)_n}=(-tz^\pm;q)_\infty.
\label{prodgfcqiH}
\end{eqnarray}
\end{thm}

\begin{proof}
First start with the left-hand side of
\eqref{prodgfcqiH} and use the terminating representation of the
continuous $q^{-1}$-Hermite polynomials \eqref{cqiH:def1}. Then reversing the order of the summation followed by evaluating the outer sum using 
Euler's Theorem \ref{Euler}, the inner sum can be evaluated using the $q$-binomial theorem.
This completes the proof.
\end{proof}

\begin{rem}
If one starts with the representation \eqref{cbqiH:1} (or \eqref{cbqiH:2}) and utilize the $q$-Chaundy Theorem \ref{BT}, one arrives at a nonterminating product representation of the corresponding generating function, which happens to be divergent (it is proportional to a ${}_2\phi_0$).
\end{rem}

One also has the following interesting generating function for the continuous $q^{-1}$-Hermite polynomials cf.~ \cite[Corollary 14.2.2]{Ismail:2009:CQO}.

\begin{thm}\label{thm:6.38}
Let $q\in\CCdag$, $x=\frac12(z+z^{-1})\in \CCast$, $z\in\CCast$, $|t|<1$. 
Then
\begin{eqnarray}
&&\hspace{-2.5cm}{\sf N}(x,t|q):=
\sum_{n=0}^\infty 
\frac{q^{\frac12\binom{n}{2}} (q^\frac14 t)^n H_n(x|q^{-1})}{(q;q)_n}
=\frac{1}{(t;q^\frac12)_\infty}
\qhyp21{q^\frac14z^\pm}{-q^\frac12}{q^\frac12,-t}.
\label{cqiHgf4}
\end{eqnarray}
\end{thm}
\begin{proof}
One should see the proof of \cite[Corollary 14.2.2]{Ismail:2009:CQO}.
\end{proof}

\noindent One may use the connection relation Corollary \ref{cor334} in conjunction with the generating function ${\sf N}(x,t|q)$ to obtain a generalization of this generating function with continuous big $q^{-1}$-Hermite polynomials.

\begin{cor}Let $q\in\CCdag$, $x\in\CC$, 
$a,t\in\CC$. Then
\begin{equation}
{\sf N}(x,t|q)=
\sum_{n=0}^\infty
H_n(x;a|q^{-1})
\frac{q^{\frac12\binom{n}{2}}\left(q^\frac14t\right)^n}{(q;q)_n}
\qhyp11{0}{-q^\frac12}{q^\frac12,-q^{\frac14-\frac12n}at}.
\label{cqiHgf4b}
\end{equation} 
\label{cqiHgf4gcor}
\end{cor}

\begin{proof}
Start with the Ismail--Masson exponential generating function for the continuous $q^{-1}$-Hermite polynomials, apply the connection relation Corollary \ref{cor334}, and reversing the order of the two sums and simplifying completes the proof.
\end{proof}

\noindent Using the $q$-Chaundy Theorem \ref{BT}, 
one is able to derive alternate expressions for 
the generating function for the continuous $q^{-1}$-Hermite polynomials ${\sf N}(x,t|q)$.

\begin{thm}
\label{thm:6.39}
Let $q\in\CCdag$, $x=\frac12(z+z^{-1})\in \CCast$, $z\in\CCast$, $t\in\CCast$, $|t|<1$. Then
\begin{eqnarray}
&&\hspace{-3.9cm}{\sf N}(x,t|q)=\sum_{n=0}^\infty
\frac{t^n}{(q^\frac12;q^\frac12)_n}
\qhyp31{q^{-\frac12n},q^\frac14z^\pm}{-q^\frac12}
{q^\frac12,-q^{\frac12n}}\\
&&\hspace{-2.4cm}=\sum_{n=0}^\infty
\frac{(q^\frac14z^\pm;q^\frac12)_n\,(-t)^n}
{(\pm q^\frac12;q^\frac12)_n}
\qphyp22{-1}{\pm q^{-\frac12n}}
{q^{\frac14-\frac12n}z^\pm}{q^\frac12,q^\frac12}.
\end{eqnarray}
\end{thm}
\begin{proof}
Starting with the generating function \eqref{cqiHgf4}, 
and replacing $q\mapsto q^2$ converts the right-hand 
side to a form where the $q$-Chaundy Theorem \ref{BT} 
can be used.
Using $r=p=-1$, $u=\ell=0$, $s=v=1$, ${\sf X}=t$, ${\sf Y}=-t$, 
${\bf c}=\{q^\frac12z^\pm\}$, ${\bf d}=\{-q\}$, 
${\bf a}={\bf b}=\emptyset$, and then replacing 
$q\mapsto q^\frac12$ completes the proof.
\end{proof}

\begin{rem} 
\label{rmk:6.40}
Observe surprisingly that the alternate expressions for 
${\sf N}(x,t|q)$ have the property that the 
terminating basic hypergeometric series are only 
functions of $x$, $q$ and $n$. Comparing these expressions 
with the original generating function, we realize 
that in these terminating basic hypergeometric 
series must represent alternative basic hypergeometric 
representations for the continuous $q^{-1}$-Hermite 
polynomials! 
\end{rem}

\noindent Remark \ref{rmk:6.40} leads us to the next significant result which are terminating quadratic representations of the continuous $q^{-1}$-Hermite polynomials which have been given previously in 
Theorem \ref{thm550}. This then leads us to the following quadratic transformations for terminating 
basic hypergeometric series.

\begin{thm}
\label{thm:6.42}
Let $n\in\mathbb N_0$, $q\in\CCddag$, 
$z\in\CCast$. Then, one 
has the following terminating quadratic transformation: 
\begin{eqnarray}
&&\hspace{-2.0cm}z^n\qphyp101{q^{-2n}}{-}
{q^2,\frac{q^{2}}{z^2}}=q^{-\frac12n^2}(-q;q)_n\qhyp31{q^{-n},q^\frac12z^\pm}{-q}{q,-q^n}\\
&&\hspace{2.0cm}=(-1)^nq^{-\frac12n^2}(q^\frac12z^\pm;q)_n
\qphyp22{-1}{\pm q^{-n}}{q^{\frac12-n}z^\pm}{q,q},
\end{eqnarray}
which is invariant under the replacement $z\mapsto z^{-1}$.
\end{thm}
\begin{proof}
Comparing \eqref{acqiH1}, \eqref{acqiH2}, \eqref{cqiH:def1},
 \eqref{cqiH:def2}
and making the replacement $q\mapsto q^2$ completes the proof.
\end{proof}

\begin{rem}\label{rmk:6.43}
Note that another equality which is 
nonterminating, can be 
inserted in Theorem \ref{thm:6.42}, since
\begin{equation}
\qphyp101{q^{-2n}}{-}
{q^2,\frac{q^{2}}{z^2}}
=(\pm\tfrac{q}{z};q)_\infty
\qhyp01{-}{\frac{q^2}{z^2}}{q^2,\frac{q^{2-2n}}{z^2}},
\label{thm542b}
\end{equation}
which is an easy special case of \cite[(III.4)]{GaspRah}.
Notice that the second expression on the right-hand side of \eqref{thm542b} is nonterminating and entire in its argument. Moreover, 
due the invariance under the replacement $z\mapsto z^{-1}$ one has
\[\begin{split}
\qphyp101{q^{-2n}}{-}
{q^2,z^2}=&(\pm z;q)_\infty
\qhyp01{-}{z^2}{q^2,q^{-2n}z^2},
\end{split}\]
which is invariant under the replacement $z\mapsto -z$.
\end{rem}

\noindent 
The above terminating quadratic transformation formula 
leads to an interesting summation formula.
\begin{cor}\label{cor:6.44}
Let $n\in\mathbb N_0$, $q\in\CCdag$. Then, one has the 
following summation formula:
\begin{eqnarray}
&&\hspace{-4.2cm}\qphyp101{q^{-2n}}{-}{q^2,q^{2\mp1}}=
(\pm q^{1\mp\frac12};q)_\infty\qhyp01{-}{q^{2\mp1}}
{q^2,q^{-2n+2\mp1}}\nonumber\\
\label{cor:4.39-gf}&&\hspace{-0.5cm}=
q^{-\frac12n^2\mp\frac12 n} (-q;q)_n,
\end{eqnarray}
where unlike Definition \ref{def:1.1}, the $\pm$ and $\mp$ factors in the exponents represent two separate cases.
\end{cor}
\begin{proof}
Setting $z=q^{\pm\frac12}$ in Theorem \ref{thm:6.42} 
completes the proof.
\end{proof}

\begin{rem}\label{rem:6.45}
Regarding the sum in Corollary \ref{cor:6.44}, as pointed out by Ole Warnaar \cite{Warnaarpriv2023}, this is a special case of \cite[(3.30)]{BerkovichWarnaar2005}
\begin{equation}
\qhyp21{q^{-2n},a}{\frac{q^{-2n}}{a}}
{q^2,\frac{q}{a}}=\frac{(-q,qa;q)_n}{(q^2a;q^2)_n}.
\label{Warnaar}
\end{equation} 
which itself is a special instance of a quadratic transformation formula between a $_5\phi_4$ series and a $_{3}\phi_2$ series \cite{BerkovichWarnaar2005}
\begin{equation}
\qhyp54{a,q^2a,b^2,c,q^2c}{qab,q^3ab,qc,q^3c}{q^4,q^4}
=\frac{(q,\frac{qc}{a},\frac{qc}{b},\frac{q}{ab};q^2)_\infty}{(qc,\frac{q}{a},\frac{q}{b},\frac{qc}{ab};q^2)_\infty}
\qhyp32{a,b,c}{qa,\frac{qc}{b}}{q^2,-\frac{q^2}{b}}.
\end{equation}
Furthermore, the sum \eqref{cor535}, can be obtained by letting $a\to 0$ in \eqref{Warnaar} or 
\cite[(3.34)]{BerkovichWarnaar2005}
\begin{equation}
\qhyp21{q^{-2n},qa}{\frac{q^{3-2n}}{a}}{q^2,\frac{q^2}{a}}=q^{-n}\frac{(-q,a;q)_n}{(\frac{a}{q};q^2)_n}.
\label{Warnaar2}
\end{equation}
As pointed out in \cite{BerkovichWarnaar2005}, 
\eqref{Warnaar} also corresponds to the terminating analogue of the Andrews--Askey sum \cite[\href{http://dlmf.nist.gov/17.6.E4}{(17.6.4)}]{NIST:DLMF}
\begin{equation}
\qhyp21{a^2,\frac{a^2}{b}}{b}{q^2,\frac{qb}{a^2}}=\frac{(q,a^2;q^2)_\infty}{2(b,\frac{qb}{a^2};q)_\infty}\left(\frac{(\frac{b}{a};q)_\infty}{(a;q)_\infty}+\frac{(-\frac{b}{a};q)_\infty}{(-a;q)_\infty}\right).
\end{equation}
\end{rem}

\begin{rem}\label{rem:6.46}
Note that for the continuous $q$-Hermite 
polynomials there exists a generating function 
with arbitrary numerator dependence given by 
$(\gamma;q)_n$, i.e., \eqref{cqHgf3}. We have 
not, as of yet, been able to derive an analogous 
generating function for the continuous 
$q^{-1}$-Hermite polynomials.
\end{rem}

\noindent 
There is a generating function with two free parameters which comes from the generalization of Ismail's generating function \eqref{gfIsmg} and taking the limit as $a,b\to 0$.

\begin{thm}Let $q\in\CCdag$, $t,\delta\in\CC$, $\gamma\in\CCast$ such that $|\gamma|<1$. Then
\begin{eqnarray}
&&\hspace{-3cm}\sum_{n=0}^\infty
\frac{q^{\binom{n}{2}}t^n(\gamma;q)_n}{(q,\delta;q)_n}H_n(x|q^{-1})=
\frac{(\gamma,-tz^\pm;q)_\infty}{(\delta;q)_\infty}\qhyp32{\frac{\delta}{\gamma},0,0}{-tz^\pm}{q,\gamma}.
\label{gfIsmgHa}
\end{eqnarray}
\end{thm}

\begin{proof}
Starting with \eqref{gfIsmg} and taking the limit as $b\to 0$, completes the proof.
\end{proof}

\noindent 
This generating function has some interesting limits. For instance, the limit of $\gamma\to 0$ produces 
the following generating function.

\begin{cor}
Let $q\in\CCdag$, $t,\delta\in\CC$. Then
\begin{eqnarray}
&&\hspace{-4.0cm}\sum_{n=0}^\infty \frac{q^{\binom{n}{2}}t^n}{(\delta;q)_n}H_n(x|q^{-1})=
\frac{(q,-tz^\pm;q)_\infty}{(\delta;q)_\infty}\qhyp32{\frac{\delta}{q},0,0}{-tz^\pm}{q,q}.
\label{gfIsmagHaa}
\end{eqnarray}
\end{cor}

\begin{proof}
Take the limit as $\gamma\to q$ in \eqref{gfIsmgHa} completes the proof.
\end{proof}
Note that in the case that $\delta=q$, then \eqref{gfIsmagHaa} clearly becomes \eqref{prodgfcqiH}.

\subsection{The big and little $q$-Jacobi polynomials and the $q$-Bessel polynomials}

For the big and little $q$-Jacobi polynomials and the $q$-Bessel polynomials, there are known generating functions. These are given as follows.
For the big $q$-Jacobi polynomials, one has the following known generating functions.
\begin{thm}Let $q\in\CCdag$, $a,b,c,t,x\in\CCast$, $|t|<1/|x|$. Then
\begin{eqnarray}
&&\hspace{-2cm}\sum_{n=0}^\infty 
\frac{(qc;q)_n\,t^n}{(q,qb;q)_n}P_n(x;a,b,c;q)
=\qhyp11{\frac{bx}{c}}{qb}{q,qct}
\qphyp11{-1}{\frac{qa}{x}}{qa}{q,xt},\\
&&\hspace{-2cm}\sum_{n=0}^\infty 
\frac{(qa;q)_n\,t^n}{(q,\frac{qab}{c};q)_n}P_n(x;a,b,c;q)
=\qhyp11{\frac{bx}{c}}{\frac{qab}{c}}{q,qat}
\qphyp11{-1}{\frac{qc}{x}}{qc}{q,xt}.
\end{eqnarray}
\end{thm}
\begin{proof}
See proof of \cite[(14.5.11--12)]{Koekoeketal}.
\end{proof}

\noindent For little $q$-Jacobi polynomials there is the following generating function.
\begin{thm}Let $q\in\CCdag$, $a,b,t,x\in\CCast$, $|t|<1/|x|$. Then
\begin{eqnarray}
&&\hspace{-3cm}\sum_{n=0}^\infty \frac{q^{\binom{n}{2}}\,t^n}{(q,qb;q)_n}p_n(x;a,b;q)=\qhyp01{-}{qa}{q,-qaxt}\qphyp11{-1}{\frac{1}{x}}{qb}{q,-xt}.
\end{eqnarray}
\end{thm}
\begin{proof}
See proof of \cite[(14.12.11)]{Koekoeketal}.
\end{proof}

\noindent For $q$-Bessel polynomials there is the following generating function.
\begin{thm}Let $q\in\CCdag$, $a,t,x\in\CCast$, $|t|<1/|x|$. Then
\begin{eqnarray}
&&\hspace{-3cm}\sum_{n=0}^\infty \frac{q^{\binom{n}{2}}\,t^n}{(q;q)_n}y_n(x;a;q)=\frac{(-t;q)_\infty}{(-xt;q)_\infty}\qphyp112{-xt}{-t}{q,qaxt}.
\label{qBgf}
\end{eqnarray}
\end{thm}
\begin{proof}
See proof of \cite[(14.22.12)]{Koekoeketal}.
\end{proof}

\noindent By starting with the above generating functions for big $q$-Jacobi, little $q$-Jacobi and $q$-Bessel polynomials and adopting the duality relations given in \S\ref{sec:3.8} and by employing the orthogonality relations given in \S\ref{sec:4O}, one can obtain a host of alternative infinite series expressions for these polynomials and for the $q$ and $q^{-1}$-symmetric polynomials. We give one example applied to the $q$-Bessel generating function \eqref{qBgf} by applying the duality to continuous big $q^{-1}$-Hermite polynomials \eqref{dHnymu2}.
\begin{thm}Let $m\in\N_0$, $q\in\CCdag$, $a\in\CCast$. Then
\begin{eqnarray}
&&\hspace{-3cm}\sum_{n=0}^\infty
\frac{q^{3\binom{n}{2}}\,t^n}
{(q;q)_n}H_m[q^{-n}a;a|q^{-1}]
=a^{-m}(-\tfrac{a^2t}{q};q)_m\qphyp112{-q^{m-1}a^2t}{-\tfrac{a^2t}{q}}{q,-q^mt}.
\end{eqnarray}
\end{thm}
\begin{proof}
Start with the $q$-Bessel generating function \eqref{qBgf}, apply the duality relation to continuous big $q^{-1}$-Hermite polynomials \eqref{dHnymu2} and making necessary replacements completes the proof.
\end{proof}

\noindent However, we will leave these computations to a later publication.

\section{Orthogonal integral generating relations}\label{integralsgf}

\subsection{The Askey--Wilson polynomials}\label{sec:5A}

\begin{thm}
Let $n\in\N_0$, $q\in\CCdag$, $x=\cos\psi$, $z=\expe^{i\psi}$, $t,a,b,c,d\in\CCast$, $|t|<1$. Then
\begin{eqnarray}
&&\hspace{-1.4cm}\int_0^\pi p_n(x;a,b,c,d|q)\qhyp21{az,bz}{ab}{q,\frac{t}{z}}
\qhyp21{\frac{c}{z},\frac{d}{z}}{cd}{q,tz}\frac{(z^{\pm 2};q)_\infty}{(az^\pm,bz^\pm,cz^\pm,dz^\pm;q)_\infty}\,\dd \psi\nonumber\\
&&\hspace{3.5cm}=
\frac{2\pi\, t^n(abcd;q)_\infty(\pm\sqrt{\frac{abcd}{q}},ac,ad,bc,bd;q)_n}
{(q,ab,ac,ad,bc,bd,cd;q)_\infty(\frac{abcd}{q},\pm\sqrt{qabcd};q)_n}.
\end{eqnarray}
\end{thm}
\begin{proof}
Start with the generating function for the Askey--Wilson polynomials \eqref{AWgf}, multiply both sides by the Askey--Wilson polynomial multiplied by $w_q(z;{\bf a})$ \eqref{AWw} with $z=\expe^{i\psi}$ and integrating over $\psi\in(0,\pi)$ utilizing the orthogonality of Askey--Wilson polynomials \eqref{AWO}, completes the proof.
\end{proof}

\begin{rem}
One may apply orthogonality of the Askey--Wilson polynomials to the generating function \eqref{genfun2ask} and obtain an integral relation, but the integral is a sum of two terms and is therefore not as interesting as, for instance, the above result. 
\end{rem}

\subsection{The continuous dual {\it q-}Hahn polynomials}\label{sec:5a}

\begin{thm}
Let $n\in\N_0$, $q\in\CCdag$, $x=\cos\psi$, $z=\expe^{i\psi}$, $t,a,b,c\in\CCast$, $|t|<|1|$. Then
\begin{eqnarray}
&&\hspace{-1.4cm}\int_0^\pi p_n(x;a,b,c|q)\qhyp21{az,bz}{ab}{q,\frac{t}{z}}
\frac{(z^{\pm 2};q)_\infty}{(az^\pm,bz^\pm,cz^\pm,dz^\pm,tz;q)_\infty}\,\dd \psi
=
\frac{2\pi\, t^n(ac,bc;q)_n}
{(q,ab,ac,bc,ct;q)_\infty}.
\end{eqnarray}
\end{thm}
\begin{proof}
Start with the generating function for the continuous dual $q$-Hahn polynomials \eqref{sub1}, multiply both sides by the continuous dual $q$-Hahn polynomial multiplied by $w_q(z;{\bf a})$ \eqref{cdqHw} with $z=\expe^{i\psi}$ and integrating over $\psi\in(0,\pi)$ utilizing the orthogonality of the continuous dual $q$-Hahn polynomials \eqref{cdqHO}, completes the proof.
\end{proof}

\begin{thm}
Let $n\in\N_0$, $q\in\CCdag$, $x=\cos\psi$, $z=\expe^{i\psi}$, $t,a,b,c\in\CCast$, $|t|<|a|$. Then
\begin{eqnarray}
&&\hspace{-1.4cm}\int_0^\pi p_n(x;a,b,c|q)\qhyp21{az^\pm,0}{ab,ac}{q,-\frac{t}{a}}
\frac{(z^{\pm 2};q)_\infty}{(az^\pm,bz^\pm,cz^\pm;q)_\infty}\,\dd \psi
=
\frac{q^{\binom{n}{2}}2\pi\, t^n(bc;q)_n}
{(q,ab,ac,bc,-\frac{t}{a};q)_\infty}.
\end{eqnarray}
\end{thm}
\begin{proof}
Start with the generating function for the continuous dual $q$-Hahn polynomials \eqref{cdqHgf}, multiply both sides by the continuous dual $q$-Hahn polynomial multiplied by $w_q(z;{\bf a})$ \eqref{cdqHw} with $z=\expe^{i\psi}$ and integrating over $\psi\in(0,\pi)$ utilizing the orthogonality of continuous dual $q$-Hahn polynomials \eqref{cdqHO}, completes the proof.
\end{proof}

\begin{thm}
Let $n\in\N_0$, $q\in\CCdag$, $x=\cos\psi$, $z=\expe^{i\psi}$, $t,a,b,c\in\CCast$, $|t|<1$. Then
\begin{eqnarray}
&&\hspace{-1.4cm}\int_0^\pi p_n(x;a,b,c|q)
\frac{(z^{\pm 2};q)_\infty}{(az^\pm,bz^\pm,cz^\pm,tz^\pm;q)_\infty}\,\dd \psi
=
\frac{2\pi\, t^n(ab,ac,bc;q)_n}
{(q,ab,ac,bc,at,bt,ct;q)_\infty}.
\end{eqnarray}
\end{thm}
\begin{proof}
Start with the generating function for the continuous dual $q$-Hahn polynomials \eqref{ATAT}, multiply both sides by the continuous dual $q$-Hahn polynomial multiplied by $w_q(z;{\bf a})$ \eqref{cdqHw} with $z=\expe^{i\psi}$ and integrating over $\psi\in(0,\pi)$ utilizing the orthogonality of continuous dual $q$-Hahn polynomials \eqref{cdqHO}, completes the proof.
\end{proof}

\begin{thm}
Let $n\in\N_0$, $q\in\CCdag$, $x=\cos\psi$, $z=\expe^{i\psi}$, $t,a,b,c\in\CCast$. Then
\begin{eqnarray}
&&\hspace{-1.2cm}\int_0^\pi p_n(x;a,b,c|q)\qhyp23{tz^\pm}{at,bt,ct}{q,abct}
\frac{(z^{\pm 2};q)_\infty}{(az^\pm,bz^\pm,cz^\pm,tz^\pm;q)_\infty}\,\dd \psi
=
\frac{2\pi\, t^n(ab,ac,bc;q)_n}
{(q,ab,ac,bc,at,bt,ct;q)_\infty}.
\end{eqnarray}
\end{thm}
\begin{proof}
Start with the generating function for the continuous dual $q$-Hahn polynomials \eqref{cdqHgf4}, multiply both sides by the continuous dual $q$-Hahn polynomial multiplied by $w_q(z;{\bf a})$ \eqref{cdqHw} with $z=\expe^{i\psi}$ and integrating over $\psi\in(0,\pi)$ utilizing the orthogonality of continuous dual $q$-Hahn polynomials \eqref{cdqHO}, completes the proof.
\end{proof}

\begin{rem}
One may apply orthogonality of the continuous dual $q$-Hahn polynomials to the generating function \eqref{cdqHgf2t} and obtain an integral relation, but the integral is a sum of two terms and is therefore not as interesting as, for instance, the above results. 
\end{rem}

\subsection{The continuous dual $q^{-1}$-Hahn polynomials}\label{sec:5b}

\begin{thm}
Let $n\in\N_0$, $q\in\CCdag$, $x=\frac12(z+z^{-1})\in \CCast$, $t,a,b,c\in\CCast$, $|t|<1$. Then
\begin{eqnarray}
&&\hspace{-1.4cm}\int_0^{i\infty} p_n(x;a,b,c|q^{-1})\qhyp22{\frac{z^\pm}{a}}{\frac{1}{ac},\frac{1}{ac}}{q,at}
\frac{(z-z^{-1})(qaz^\pm,qbz^\pm,qcz^\pm;q)_\infty}{\vartheta(z^2;q)}\,\dd z\nonumber\\
&&\hspace{3.5cm}=
q^{-2\binom{n}{2}}\log q\left(\frac{a^2b^2c^2t}{q}\right)^n(q,qab,qac,qbc,abct;q)_\infty(\tfrac{1}{bc};q)_n.\end{eqnarray}
\end{thm}
\begin{proof}
Start with the generating function for the continuous dual $q^{-1}$-Hahn polynomials \eqref{cdqinHgf-1}, multiply both sides by the continuous dual $q^{-1}$-Hahn polynomial multiplied by $w_q(z;{\bf a})$ \eqref{cdqiHw} and integrating over $z\in(0,i\infty)$ utilizing the orthogonality of continuous dual $q^{-1}$-Hahn polynomials \eqref{cdqiHO}, completes the proof.
\end{proof}

\begin{thm}
Let $n\in\N_0$, $q\in\CCdag$, $x=\frac12(z+z^{-1})\in \CCast$, $t,a,b,c\in\CCast$, $|t|<|z^\pm/(ab)|$. Then
\begin{eqnarray}
&&\hspace{-1.4cm}\int_0^{i\infty}p_n(x;a,b,c|q^{-1})\qhyp21{\frac{z}{a},\frac{z}{b}}{\frac{1}{ab}}{q,\frac{abt}{z}}
\frac{(z-z^{-1})(qaz^\pm,qbz^\pm,qcz^\pm,\frac{abt}{z};q)_\infty}{\vartheta(z^2;q)}\,\dd z\nonumber\\
&&\hspace{3.5cm}=
q^{-2\binom{n}{2}}\log q\left(\frac{a^2b^2c^2t}{q}\right)^n(q,qab,qac,qbc,abct;q)_\infty(\tfrac{1}{ac},\tfrac{1}{bc};q)_n.\end{eqnarray}
\end{thm}
\begin{proof}
Start with the generating function for the continuous dual $q^{-1}$-Hahn polynomials \eqref{cdqiHgf2}, multiply both sides by the continuous dual $q^{-1}$-Hahn polynomial multiplied by $w_q(z;{\bf a})$ \eqref{cdqiHw} integrating over $z\in(0,i\infty)$ utilizing the orthogonality of continuous dual $q^{-1}$-Hahn polynomials \eqref{cdqiHO}, completes the proof.
\end{proof}

\begin{thm}
Let $n\in\N_0$, $q\in\CCdag$, $x=\frac12(z+z^{-1})\in \CCast$, $\gamma,t,a,b,c\in\CCast$, $|t|<1$. Then
\begin{eqnarray}
&&\hspace{-1.0cm}\int_0^{i\infty}p_n(x;a,b,c|q^{-1})\qhyp33{\gamma,\frac{z^\pm}{a}}{\frac{1}{ab},\frac{1}{ac},\gamma abct}{q,at}
\frac{(z-z^{-1})(qaz^\pm,qbz^\pm,qcz^\pm;q)_\infty}{\vartheta(z^2;q)}\,\dd z\nonumber\\
&&\hspace{3.5cm}=
q^{-2\binom{n}{2}}\log q\left(\frac{a^2b^2c^2t}{q}\right)^n\frac{(q,qab,qac,qbc,abct;q)_\infty}{(\gamma abct;q)_\infty}\left(\gamma,\frac{1}{bc};q\right)_n.\end{eqnarray}
\end{thm}

\begin{proof}
Start with the generating function for the continuous dual $q^{-1}$-Hahn polynomials \eqref{eq:thm49-gf}, multiply both sides by the continuous dual $q^{-1}$-Hahn polynomial multiplied by $w_q(z;{\bf a})$ \eqref{cdqiHw} and integrating over $z\in(0,i\infty)$ utilizing the orthogonality of continuous dual $q^{-1}$-Hahn polynomials \eqref{cdqiHO}, completes the proof.
\end{proof}

\begin{thm}
Let $n\in\N_0$, $q\in\CCdag$, $x=\frac12(z+z^{-1})\in \CCast$, $t,a,b,c\in\CCast$, $|t|<1$. Then
\begin{eqnarray}
&&\hspace{-1.0cm}\int_0^{i\infty}p_n(x;a,b,c|q^{-1})\qhyp23{\frac{z^\pm}{a}}{\frac{1}{ab},\frac{1}{ac},-abct}{q,-at}
\frac{(z-z^{-1})(qaz^\pm,qbz^\pm,qcz^\pm;q)_\infty}{\vartheta(z^2;q)}\,\dd z\nonumber\\
&&\hspace{3.5cm}=
q^{-\binom{n}{2}}\log q\left(\frac{a^2b^2c^2t}{q}\right)^n\frac{(q,qab,qac,qbc;q)_\infty}
{(-abct;q)_\infty}\left(\tfrac{1}{bc};q\right)_n.\end{eqnarray}
\end{thm}
\begin{proof}
Start with the generating function for the continuous dual $q^{-1}$-Hahn polynomials \eqref{newgfcdqiH}, multiply both sides by the continuous dual $q^{-1}$-Hahn polynomial multiplied by $w_q(z;{\bf a})$ \eqref{cdqiHw} and integrating over $z\in(0,i\infty)$ utilizing the orthogonality of continuous dual $q^{-1}$-Hahn polynomials \eqref{cdqiHO}, completes the proof.
\end{proof}

\subsection{The Al-Salam--Chihara polynomials}
\label{sec:5c}

\begin{thm}
Let $n\in\N_0$, $q\in\CCdag$, $x=\cos\psi$, $z=\expe^{i\psi}$, $t,a,b\in\CCast$, $|t|<1$. Then
\begin{eqnarray}
&&\hspace{-1.4cm}\int_0^\pi Q_n(x;a,b|q)
\frac{(z^{\pm 2};q)_\infty}{(az^\pm,bz^\pm,tz^\pm;q)_\infty}\,\dd \psi=
\frac{2\pi\,t^n(ab;q)_n}
{(q,ab,at,bt;q)_\infty}.
\end{eqnarray}
\end{thm}
\begin{proof}
Start with the generating function for the Al-Salam--Chihara polynomials \eqref{ASCgf1}, multiply both sides by the Al-Salam--Chihara polynomial multiplied by $w_q(z;{\bf a})$ \eqref{ASCw} with $z=\expe^{i\psi}$ and integrating over $\psi\in(0,\pi)$ utilizing the orthogonality of Al-Salam--Chihara polynomials \eqref{ASCO}, completes the proof.
\end{proof}

\begin{thm}
Let $n\in\N_0$, $q\in\CCdag$, $x=\cos\psi$, $z=\expe^{i\psi}$, $t,a,b\in\CCast$, $|t|<1$. Then
\begin{eqnarray}
&&\hspace{-1.4cm}\int_0^\pi Q_n(x;a,b|q)\qhyp21{az,bz}{ab}{q,\frac{t}{z}}
\frac{(z^{\pm 2};q)_\infty}{(az^\pm,bz^\pm,tz;q)_\infty}\,\dd \psi=
\frac{2\pi\,t^n}
{(q,ab;q)_\infty}.
\end{eqnarray}
\end{thm}
\begin{proof}
Start with the generating function for the Al-Salam--Chihara polynomials \eqref{ASCgf2}, multiply both sides by the Al-Salam--Chihara polynomial multiplied by $w_q(z;{\bf a})$ \eqref{ASCw} with $z=\expe^{i\psi}$ and integrating over $\psi\in(0,\pi)$ utilizing the orthogonality of Al-Salam--Chihara polynomials \eqref{ASCO}, completes the proof.
\end{proof}

\begin{thm}
Let $n\in\N_0$, $q\in\CCdag$, $x=\cos\psi$, $z=\expe^{i\psi}$, $t,a,b\in\CCast$, $|t|<|a|$. Then
\begin{eqnarray}
&&\hspace{-1.4cm}\int_0^\pi Q_n(x;a,b|q)\qhyp21{az^\pm}{ab}{q,-\frac{t}{a}}
\frac{(z^{\pm 2};q)_\infty}{(az^\pm,bz^\pm;q)_\infty}\,\dd \psi=
\frac{q^{\binom{n}{2}}2\pi\,t^n}
{(q,ab,-\frac{t}{a};q)_\infty}.
\end{eqnarray}
\end{thm}
\begin{proof}
Start with the generating function for the Al-Salam--Chihara polynomials \eqref{ASCgf3}, multiply both sides by the Al-Salam--Chihara polynomial multiplied by $w_q(z;{\bf a})$ \eqref{ASCw} with $z=\expe^{i\psi}$ and integrating over $\psi\in(0,\pi)$ utilizing the orthogonality of Al-Salam--Chihara polynomials \eqref{ASCO}, completes the proof.
\end{proof}

\begin{thm}
Let $n\in\N_0$, $q\in\CCdag$, $x=\cos\psi$, $z=\expe^{i\psi}$, $t,a,b\in\CCast$, $|t|<1$. Then
\begin{eqnarray}
&&\hspace{-1.4cm}\int_0^\pi Q_n(x;a,b|q)\qhyp32{\gamma,az,bz}{ab,\gamma tz}{q,\frac{t}{z}}
\frac{(z^{\pm 2},\gamma tz;q)_\infty}{(az^\pm,bz^\pm,tz;q)_\infty}\,\dd \psi=
\frac{2\pi\,t^n(\gamma;q)_n}
{(q,ab;q)_\infty}.
\end{eqnarray}
\end{thm}
\begin{proof}
Start with the generating function for the Al-Salam--Chihara polynomials \eqref{ASCgf4}, multiply both sides by the Al-Salam--Chihara polynomial multiplied by $w_q(z;{\bf a})$ \eqref{ASCw} with $z=\expe^{i\psi}$ and integrating over $\psi\in(0,\pi)$ utilizing the orthogonality of Al-Salam--Chihara polynomials \eqref{ASCO}, completes the proof.
\end{proof}

\subsection{The $q^{-1}$-Al-Salam--Chihara polynomials}\label{sec:5d}

\begin{thm}
Let $n\in\N_0$, $q\in\CCdag$, $x=\frac12(z+z^{-1})\in \CCast$, $t,a,b\in\CCast$, $|t|<|z^{\pm}/(ab)|$. Then
\begin{eqnarray}
&&\hspace{-2.0cm}\int_0^{i\infty}Q_n(x;a,b|q^{-1})\qhyp21{\frac{z}{a},\frac{z}{b}}{\frac{1}{ab}}{q,\frac{abt}{z}}
\frac{(z-z^{-1})(qaz^\pm,qbz^\pm,\frac{abt}{z};q)_\infty}{\vartheta(z^2;q)}\,\dd z
\nonumber\\
&&\hspace{5.5cm}
=
\log q\left(\frac{abt}{q}\right)^n(q,qab;q)_\infty.
\end{eqnarray}
\end{thm}
\begin{proof}
Start with the generating function for the $q^{-1}$-Al-Salam--Chihara polynomials \eqref{qiASCgf1}, multiply both sides by the $q^{-1}$-Al-Salam--Chihara polynomial multiplied by $w_q(z;{\bf a})$ \eqref{qiASCw} and integrating over $z\in(0,i\infty)$ utilizing the orthogonality of $q^{-1}$-Al-Salam--Chihara polynomials \eqref{qiASCO}, completes the proof.
\end{proof}

\begin{thm}
Let $n\in\N_0$, $q\in\CCdag$, $x=\frac12(z+z^{-1})\in \CCast$, $t,a,b\in\CCast$, $|t|<1$. Then
\begin{eqnarray}
&&\hspace{-0.2cm}\int_0^{i\infty}Q_n(x;a,b|q^{-1})
\frac{(z-z^{-1})(qaz^\pm,qbz^\pm,-tz^\pm;q)_\infty}{\vartheta(z^2;q)}\,\dd z
\nonumber\\
&&\hspace{5.5cm}
=q^{-\binom{n}{2}}
\log q\left(\frac{abt}{q}\right)^n(q,qab,-at,-bt;q)_\infty(\tfrac{1}{ab};q)_n.
\end{eqnarray}
\end{thm}
\begin{proof}
Start with the generating function for the $q^{-1}$-Al-Salam--Chihara polynomials \eqref{qiASCgfY}, multiply both sides by the $q^{-1}$-Al-Salam--Chihara polynomial multiplied by $w_q(z;{\bf a})$ \eqref{qiASCw} and integrating over $z\in(0,i\infty)$ utilizing the orthogonality of $q^{-1}$-Al-Salam--Chihara polynomials \eqref{qiASCO}, completes the proof.
\end{proof}

\begin{thm}
Let $n\in\N_0$, $q\in\CCdag$, $x=\frac12(z+z^{-1})\in \CCast$, $t,a,b\in\CCast$, $|t|<1/|a|$. Then
\begin{eqnarray}
&&\hspace{-2.0cm}\int_0^{i\infty}Q_n(x;a,b|q^{-1})
\qhyp21{\frac{z^\pm}{a}}{\frac{1}{ab}}{q,-at}
\frac{(z-z^{-1})(qaz^\pm,qbz^\pm;q)_\infty}{\vartheta(z^2;q)}\,\dd z
\nonumber\\
&&\hspace{5.5cm}
=q^{-\binom{n}{2}}
\log q\left(\frac{abt}{q}\right)^n(q,qab,-bt;q)_\infty.
\end{eqnarray}
\end{thm}
\begin{proof}
Start with the generating function for the $q^{-1}$-Al-Salam--Chihara polynomials \eqref{qiASCgfX}, multiply both sides by the $q^{-1}$-Al-Salam--Chihara polynomial multiplied by $w_q(z;{\bf a})$ \eqref{qiASCw} and integrating over $z\in(0,i\infty)$ utilizing the orthogonality of $q^{-1}$-Al-Salam--Chihara polynomials \eqref{qiASCO}, completes the proof.
\end{proof}

\begin{thm}
Let $n\in\N_0$, $q\in\CCdag$, $x=\frac12(z+z^{-1})\in \CCast$, $t,a,b\in\CCast$, $|t|<1/|a|$. Then
\begin{eqnarray}
&&\hspace{-1.0cm}\int_0^{i\infty}Q_n(x;a,b|q^{-1})\qhyp32{\gamma,\frac{z^\pm}{a}}{\frac{1}{ab},-\gamma bt}{q,-at}
\frac{(z-z^{-1})(qaz^\pm,qbz^\pm;q)_\infty}{\vartheta(z^2;q)}\,\dd z
\nonumber\\
&&\hspace{5.5cm}
=
q^{-\binom{n}{2}}\log q\left(\frac{abt}{q}\right)^n\frac{(q,qab,-bt;q)_\infty}{(-\gamma bt;q)_\infty}(\gamma;q)_n.
\end{eqnarray}
\end{thm}
\begin{proof}
Start with the generating function for the $q^{-1}$-Al-Salam--Chihara polynomials \eqref{gf32h}, multiply both sides by the $q^{-1}$-Al-Salam--Chihara polynomial multiplied by $w_q(z;{\bf a})$ \eqref{qiASCw} and integrating over $z\in(0,i\infty)$ utilizing the orthogonality of $q^{-1}$-Al-Salam--Chihara polynomials \eqref{qiASCO}, completes the proof.
\end{proof}

\subsection{The continuous big $q$-Hermite polynomials}\label{sec:5e}

\begin{thm}
Let $n\in\N_0$, $q\in\CCdag$, $x=\cos\psi$, $z=\expe^{i\psi}$, $t,a\in\CCast$, $|t|<1$. Then
\begin{eqnarray}
&&\hspace{-1.4cm}\int_0^\pi H_n(x;a|q)
\frac{(z^{\pm 2};q)_\infty}{(az^\pm,tz^\pm;q)_\infty}\,\dd \psi=
\frac{2\pi\,t^n}
{(q,at;q)_\infty}.
\end{eqnarray}
\end{thm}
\begin{proof}
Start with the generating function for the continuous big $q$-Hermite polynomials \eqref{cbqHgf1}, multiply both sides by the continuous big $q$-Hermite polynomial multiplied by $w_q(z;{\bf a})$ \eqref{cbqHw} with $z=\expe^{i\psi}$ and integrating over $\psi\in(0,\pi)$ utilizing the orthogonality of continuous big $q$-Hermite polynomials \eqref{cbqHO}, completes the proof.
\end{proof}

\begin{thm}
Let $n\in\N_0$, $q\in\CCdag$, $x=\cos\psi$, $z=\expe^{i\psi}$, $t,a\in\CCast$, $|t|<|a|$. Then
\begin{eqnarray}
&&\hspace{-1.4cm}\int_0^\pi H_n(x;a|q)\qhyp21{az^\pm}{0}{q,-\frac{t}{a}}
\frac{(z^{\pm 2};q)_\infty}{(az^\pm;q)_\infty}\,\dd \psi=
\frac{q^{\binom{n}{2}}2\pi\,t^n}
{(q,-\frac{t}{a};q)_\infty}.
\end{eqnarray}
\end{thm}
\begin{proof}
Start with the generating function for the continuous big $q$-Hermite polynomials \eqref{gf2cbqH}, multiply both sides by the continuous big $q$-Hermite polynomial multiplied by $w_q(z;{\bf a})$ \eqref{cbqHw} with $z=\expe^{i\psi}$ and integrating over $\psi\in(0,\pi)$ utilizing the orthogonality of continuous big $q$-Hermite polynomials \eqref{cbqHO}, completes the proof.
\end{proof}

\begin{thm}
Let $n\in\N_0$, $q\in\CCdag$, $x=\cos\psi$, $z=\expe^{i\psi}$, $t,a\in\CCast$, $|t|<1$. Then
\begin{eqnarray}
&&\hspace{-1.4cm}\int_0^\pi H_n(x;a|q)\qhyp21{\gamma,az}{\gamma tz}{q,\frac{t}{z}}
\frac{(z^{\pm 2},\gamma tz;q)_\infty}{(az^\pm,tz;q)_\infty}\,\dd \psi=
\frac{2\pi\,t^n(\gamma;q)_n}
{(q;q)_\infty}.
\end{eqnarray}
\end{thm}
\begin{proof}
Start with the generating function for the continuous big $q$-Hermite polynomials \eqref{cbqHgf4}, multiply both sides by the continuous big $q$-Hermite polynomial multiplied by $w_q(z;{\bf a})$ \eqref{cbqHw} with $z=\expe^{i\psi}$ and integrating over $\psi\in(0,\pi)$ utilizing the orthogonality of continuous big $q$-Hermite polynomials \eqref{cbqHO}, completes the proof.
\end{proof}

\begin{thm}
Let $n\in\N_0$, $q\in\CCdag$, $x=\cos\psi$, $z=\expe^{i\psi}$, $t,a\in\CCast$, $|t|<1$. Then
\begin{eqnarray}
&&\hspace{-1.4cm}\int_0^\pi H_n(x;a|q)\qhyp11{0}{-tz}{q,-\frac{t}{z}}
\frac{(z^{\pm 2},-tz;q)_\infty}{(az^\pm;q)_\infty}\,\dd \psi=
\frac{q^{\binom{n}{2}}2\pi\,t^n(-at;q)_\infty}
{(q;q)_\infty(-at;q)_n}.
\end{eqnarray}
\end{thm}
\begin{proof}
Start with the generating function for the continuous big $q$-Hermite polynomials \eqref{gcbqHgf3}, multiply both sides by the continuous big $q$-Hermite polynomial multiplied by $w_q(z;{\bf a})$ \eqref{cbqHw} with $z=\expe^{i\psi}$ and integrating over $\psi\in(0,\pi)$ utilizing the orthogonality of continuous big $q$-Hermite polynomials \eqref{cbqHO}, completes the proof.
\end{proof}

\begin{thm}
Let $n\in\N_0$, $q\in\CCdag$, $x=\cos\psi$, $z=\expe^{i\psi}$, $t,a\in\CCast$, $|t|<1$. Then
\begin{eqnarray}
&&\hspace{-1.4cm}\int_0^\pi H_n(x;a|q)\qhyp21{\gamma,0}{\gamma tz}{q,\frac{t}{z}}
\frac{(z^{\pm 2},\gamma tz;q)_\infty}{(az^\pm,tz;q)_\infty}\,\dd \psi=
\frac{2\pi\,t^n(\gamma at;q)_\infty(\gamma;q)_n}
{(q,at;q)_\infty(\gamma at;q)_n}.
\end{eqnarray}
\end{thm}
\begin{proof}
Start with the generating function for the continuous big $q$-Hermite polynomials \eqref{gcbqHgf5}, multiply both sides by the continuous big $q$-Hermite polynomial multiplied by $w_q(z;{\bf a})$ \eqref{cbqHw} with $z=\expe^{i\psi}$ and integrating over $\psi\in(0,\pi)$ utilizing the orthogonality of continuous big $q$-Hermite polynomials \eqref{cbqHO}, completes the proof.
\end{proof}

\begin{thm}
Let $n\in\N_0$, $q\in\CCdag$, $x=\cos\psi$, $z=\expe^{i\psi}$, $t,a\in\CCast$, $|t|<1$. Then
\begin{eqnarray}
&&\hspace{-1.0cm}\int_0^\pi H_n(x;a|q)\qhyp21{q^\frac14 z^\pm}{-q^\frac12}{q^\frac12,-t}
\frac{(z^{\pm 2};q)_\infty}{(az^\pm;q)_\infty}\,\dd \psi=
\frac{q^{\frac12\binom{n}{2}}2\pi\left(q^\frac14 t\right)^n}
{(q,-t,-q^{\frac12}t;q)_\infty
}
\qhyp11{0}{-q^\frac12}{q^\frac12,-q^{\frac14+\frac12 n}at}.
\end{eqnarray}
\end{thm}
\begin{proof}
Start with the generating function for the continuous big $q$-Hermite polynomials \eqref{cqHgf4coreq}, multiply both sides by the continuous big $q$-Hermite polynomial multiplied by $w_q(z;{\bf a})$ \eqref{cbqHw} with $z=\expe^{i\psi}$ and integrating over $\psi\in(0,\pi)$ utilizing the orthogonality of continuous big $q$-Hermite polynomials \eqref{cbqHO}, completes the proof.
\end{proof}

\subsection{The continuous big $q^{-1}$-Hermite polynomials}\label{sec:5f}

\begin{thm}
Let $n\in\N_0$, $q\in\CCdag$, $x=\frac12(z+z^{-1})\in \CCast$, $t,a\in\CCast$, $|t|<1$. Then
\begin{eqnarray}
&&\hspace{-1.0cm}\int_0^{i\infty}H_n(x;a|q^{-1})
\frac{(z-z^{-1})(qaz^\pm,-tz^\pm;q)_\infty}{\vartheta(z^2;q)}\,\dd z
=
\log q\left(-\frac{t}{q}\right)^n(q,-at;q)_\infty.
\end{eqnarray}
\end{thm}
\begin{proof}
Start with the generating function for the continuous big $q^{-1}$-Hermite polynomials \eqref{cbqinHegf-2}, multiply both sides by the continuous big $q^{-1}$-Hermite polynomial multiplied by $w_q(z;{\bf a})$ \eqref{cbqiHw} and integrating over $z\in(0,i\infty)$ utilizing the orthogonality of continuous big $q^{-1}$-Hermite polynomials \eqref{cbqiHO}, completes the proof.
\end{proof}

\begin{thm}
Let $n\in\N_0$, $q\in\CCdag$, $x=\frac12(z+z^{-1})\in \CCast$, $t,a\in\CCast$, $|t|<1$. Then
\begin{eqnarray}
&&\hspace{-0.0cm}\int_0^{i\infty}H_n(x;a|q^{-1})\qhyp21{q^{\frac14}z^\pm}{-q^\frac12}{q^\frac12,-t}
\frac{(z-z^{-1})(qaz^\pm;q)_\infty}{\vartheta(z^2;q)}\,\dd z
\nonumber\\
&&\hspace{3.5cm}
=
q^{-\frac12\binom{n}{2}}\log q\left(-q^{-\frac34}t\right)^n(q,t,q^\frac12 t;q)_\infty
\qhyp11{0}{-q^\frac12}{q^\frac12,-q^{\frac14-\frac12 n}at}.
\end{eqnarray}
\end{thm}
\begin{proof}
Start with the generating function for the continuous big $q^{-1}$-Hermite polynomials \eqref{cqiHgf4b}, multiply both sides by the continuous big $q^{-1}$-Hermite polynomial multiplied by $w_q(z;{\bf a})$ \eqref{cbqiHw} and integrating over $z\in(0,i\infty)$ utilizing the orthogonality of continuous big $q^{-1}$-Hermite polynomials \eqref{cbqiHO}, completes the proof.
\end{proof}

\subsection{The continuous $q$-Hermite polynomials}\label{sec:5g}

\begin{thm}
Let $n\in\N_0$, $q\in\CCdag$, $x=\cos\psi$, $z=\expe^{i\psi}$, $t\in\CCast$, $|t|<1$. Then
\begin{eqnarray}
&&\hspace{-1.4cm}\int_0^\pi H_n(x|q)
\frac{(z^{\pm 2};q)_\infty}{(tz^\pm;q)_\infty}\,\dd \psi=
\frac{2\pi}
{(q;q)_\infty}.
\end{eqnarray}
\end{thm}
\begin{proof}
Start with the generating function for the continuous $q$-Hermite polynomials \eqref{cqHgf}, multiply both sides by the continuous $q$-Hermite polynomial multiplied by $w_q(z;{\bf a})$ \eqref{cqHw} with $z=\expe^{i\psi}$ and integrating over $\psi\in(0,\pi)$ utilizing the orthogonality of continuous $q$-Hermite polynomials \eqref{cqHO}, completes the proof.
\end{proof}

\begin{thm}
Let $n\in\N_0$, $q\in\CCdag$, $x=\cos\psi$, $z=\expe^{i\psi}$, $t\in\CCast$, $|t|<1$. Then
\begin{eqnarray}
&&\hspace{-1.4cm}\int_0^\pi H_n(x|q)\qhyp11{0}{-tz}{q,-\frac{t}{z}}
(z^{\pm 2},-tz;q)_\infty\,\dd \psi=
\frac{q^{\binom{n}{2}}2\pi\,t^n}
{(q;q)_\infty}.
\end{eqnarray}
\end{thm}
\begin{proof}
Start with the generating function for the continuous $q$-Hermite polynomials \eqref{cqHgf2}, multiply both sides by the continuous $q$-Hermite polynomial multiplied by $w_q(z;{\bf a})$ \eqref{cqHw} with $z=\expe^{i\psi}$ and integrating over $\psi\in(0,\pi)$ utilizing the orthogonality of continuous $q$-Hermite polynomials \eqref{cqHO}, completes the proof.
\end{proof}

\begin{thm}
Let $n\in\N_0$, $q\in\CCdag$, $x=\cos\psi$, $z=\expe^{i\psi}$, $t\in\CCast$, $|t|<1$. Then
\begin{eqnarray}
&&\hspace{-1.4cm}\int_0^\pi H_n(x|q)\qhyp21{\gamma,0}{\gamma tz}{q,\frac{t}{z}}
\frac{(z^{\pm 2},\gamma tz;q)_\infty}{(tz;q)_\infty}\,\dd \psi=
\frac{2\pi\,t^n(\gamma;q)_n}
{(q;q)_\infty}.
\end{eqnarray}
\end{thm}
\begin{proof}
Start with the generating function for the continuous $q$-Hermite polynomials \eqref{cqHgf3}, multiply both sides by the continuous $q$-Hermite polynomial multiplied by $w_q(z;{\bf a})$ \eqref{cqHw} with $z=\expe^{i\psi}$ and integrating over $\psi\in(0,\pi)$ utilizing the orthogonality of continuous $q$-Hermite polynomials \eqref{cqHO}, completes the proof.
\end{proof}

\begin{thm}
Let $n\in\N_0$, $q\in\CCdag$, $x=\cos\psi$, $z=\expe^{i\psi}$, $t\in\CCast$, $|t|<1$. Then
\begin{eqnarray}
&&\hspace{-1.4cm}\int_0^\pi H_n(x|q)\qhyp21{q^\frac14 z^\pm}{-q^\frac12}{q^\frac12,-t}
(z^{\pm 2};q)_\infty\,\dd \psi=
\frac{q^{\frac12\binom{n}{2}}2\pi\,(q^\frac14 t)^n}
{(q,-t,-q^\frac12 t;q)_\infty
}.
\end{eqnarray}
\end{thm}
\begin{proof}
Start with the generating function for the continuous $q$-Hermite polynomials \eqref{cqHgf4}, multiply both sides by the continuous $q$-Hermite polynomial multiplied by $w_q(z;{\bf a})$ \eqref{cqHw} with $z=\expe^{i\psi}$ and integrating over $\psi\in(0,\pi)$ utilizing the orthogonality of continuous $q$-Hermite polynomials \eqref{cqHO}, completes the proof.
\end{proof}

\subsection{The continuous $q^{-1}$-Hermite polynomials}\label{sec:5h}

\begin{thm}
Let $n\in\N_0$, $q\in\CCdag$, $x=\frac12(z+z^{-1})\in \CCast$, $t\in\CCast$, $|t|<1$. Then
\begin{eqnarray}
&&\hspace{-1.0cm}\int_0^{i\infty}H_n(x|q^{-1})
\frac{(z-z^{-1})(-tz^\pm;q)_\infty}{\vartheta(z^2;q)}\,\dd z
=
\log q\left(-\frac{t}{q}\right)^n(q;q)_\infty.
\end{eqnarray}
\end{thm}
\begin{proof}
Start with the generating function for the continuous $q^{-1}$-Hermite polynomials \eqref{prodgfcqiH}, multiply both sides by the continuous $q^{-1}$-Hermite polynomial multiplied by $w_q(z;{\bf a})$ \eqref{cqiHw} and integrating over $z\in(0,i\infty)$ utilizing the orthogonality of continuous $q^{-1}$-Hermite polynomials \eqref{cqiHO}, completes the proof.
\end{proof}

\begin{thm}
Let $n\in\N_0$, $q\in\CCdag$, $x=\frac12(z+z^{-1})\in \CCast$, $t\in\CCast$, $|t|<1$. Then
\begin{eqnarray}
&&\hspace{-1.0cm}\int_0^{i\infty}H_n(x|q^{-1})
\qhyp21{q^\frac14 z^\pm}{-q^\frac12}{q^\frac12,-t}
\frac{z-z^{-1}}{\vartheta(z^2;q)}\,\dd z
=
q^{-\frac12\binom{n}{2}}\log q\left(-q^{-\frac34}t\right)^n(q,t,q^\frac12 t;q)_\infty
.
\end{eqnarray}
\end{thm}
\begin{proof}
Start with the generating function for the continuous $q^{-1}$-Hermite polynomials \eqref{cqiHgf4}, multiply both sides by the continuous $q^{-1}$-Hermite polynomial multiplied by $w_q(z;{\bf a})$ \eqref{cqiHw} and integrating over $z\in(0,i\infty)$ utilizing the orthogonality of continuous $q^{-1}$-Hermite polynomials \eqref{cqiHO}, completes the proof.
\end{proof}

\section{Dual and orthogonal generating relations }\label{sec7}

In this section we present dual generating relations for the big and little $q$-Jacobi polynomials, as well as for the $q^{-1}$-Bessel polynomials.
Typically when we speak of generating functions for the $q$-orthogonal polynomials, we refer to power series expansions over the degree of an orthogonal polynomial, multiplied by coefficients which are usually ratios of $q$-shifted factorials. In this section, we will not be summing over the degree of the big and little $q$-Jacobi polynomial, but instead over 
an infinite set of non-negative integers which sample discrete points on the real line. Some of these points are in the domain of orthogonality for these polynomials. 

\medskip
\noindent We will refer to these particular series relations as generating {\it relations}.
For instance, for the little $q$-Jacobi polynomials, some of our generating relations will be a convergent sequence of functions over $m\in\mathbb N_0$, which are given by a infinite 
series of the form 
\begin{eqnarray}
{\sf f}_{m}(t;a,b;q):=
\sum_{n=0}^\infty t^n\alpha_n\,p_m(q^n;a,b;q),
\end{eqnarray}
where $\alpha_n\in\CC$ are some coefficients.
It is interesting to note that the orthogonality relation for little $q$-Jacobi polynomials is exactly of this form \cite[(14.12.2)]{Koekoeketal}
\begin{equation}
\sum_{n=0}^\infty \frac{(qb;q)_n}{(q;q)_n}
(qa)^n p_m(q^n;a,b;q)\,p_k(q^n;a,b;q)
=\frac{(q^2ab;q)_\infty}{(qa;q)_\infty}
\frac{(1-qab)(qa)^m}{(1-q^{2m+1}ab)}
\frac{(q,qb;q)_m}{(qa,qab;q)_m}
\delta_{m,k},
\end{equation}
where $m,k\in\mathbb N_0$ and $t=qa$. 
For the little $q$-Jacobi polynomials we will also consider generating relations of the form
\begin{eqnarray}
{\sf g}_m(t;a,b;q):=\sum_{n=0}^\infty t^n\,\alpha_n\,p_m\left(\frac{q^{-n-1}}{a};a,b;q\right).
\end{eqnarray}
For the big $q$-Jacobi polynomials
we will consider generating relations 
of the form
\begin{eqnarray}
{\sf h}_m(t;a,b,c;q):=\sum_{n=0}^\infty t^n\alpha_nP_m(q^{-n};a,b,c;q),
\end{eqnarray}
and
\begin{eqnarray}
{\sf i}_m(t;a,b,c;q):=\sum_{n=0}^\infty t^n\alpha_nP_m(q^{n+1}a;a,b,c;q).
\end{eqnarray}
Note that $q^{n+1}a$ is also in the domain
of orthogonality for the big $q$-Jacobi polynomials (see \cite[(14.5.2)]{Koekoeketal}).
We will now use the properties of duality that the $q$ and $q^{-1}$-symmetric families satisfy with the dual families 
to derive new generating relations for these polynomials.
\subsection{Generating relations for the big {\it q-}Jacobi polynomials and functions}\label{sec:7.1}

\noindent Utilizing the duality of the continuous $q^{-1}$-Hahn polynomials to the big $q$-Jacobi polynomials, Theorem \ref{thm311}, we can derive using generating functions for these polynomials, 
new regular and orthogonal generating relations for the big $q$-Jacobi polynomials.

\subsubsection{Generating relations from the continuous dual $q^{-1}$-Hahn polynomials}
In this subsection, using the duality of the continuous 
dual $q^{-1}$-Hahn polynomials \eqref{dcdqiHbqJab} 
with generating functions for these polynomials, we 
derive generating relations for the big $q$-Jacobi 
polynomials. First, consider Theorem \ref{thm46}.

\begin{thm}Let $q\in\CCdag$, $m\in\mathbb N_0$, 
$a,b,c,t\in\CCast$, $|t|<1$. Then
\begin{eqnarray}
&&\hspace{-1.5cm}\sum_{n=0}^\infty 
\frac{t^n}{(q;q)_n}P_m(q^{n+1}a;a,b,c;q)
=\frac{q^{\binom{m}{2}}}{(t;q)_\infty}(-qc)^m 
\frac{(\frac{qab}{c};q)_m}{(qc;q)_m}
\qhyp22{q^{-m},q^{m+1}ab}{qa,\frac{qab}{c}}{q,\frac{qat}{c}}.
\label{gfbqiJgfA}
\end{eqnarray}
\end{thm}
\begin{proof}
Start with the generating function for the continuous 
dual $q^{-1}$-Hahn polynomials, Theorem \ref{thm46} 
and replace $z\mapsto q^{-m}a$. Then inserting the 
duality relation \eqref{dcdqiHbqJab} provides a 
generating relation in terms of big $q$-Jacobi 
polynomials. Finally replacing $t\mapsto 
t/(abc)$ followed by $(a,b,c)\mapsto(\frac{1}
{\sqrt{qab}},\frac{\sqrt{b}}{\sqrt{qa}},\frac{c}
{\sqrt{qab}})$ completes the proof.
\end{proof}

\noindent Next consider Theorem \ref{thm48}.

\begin{thm}Let $q\in\CCdag$, $m\in\mathbb N_0$, $a,b,c,t\in\CCast$, $|t|<1$. Then
\begin{eqnarray}
&&\hspace{-1.5cm}\sum_{n=0}^\infty t^n\frac
{(\frac{qab}{c};q)_n}{(q;q)_n}P_m(q^{n+1}a;a,b,c;q)
\nonumber\\
&&\hspace{1cm}
=q^{\binom{m}{2}}(-qc)^m 
\frac{(\frac{qabt}{c};q)_\infty}{(t;q)_\infty}
\frac{(\frac{qab}{c};q)_m}{(qc,\frac{qabt}{c};q)_m}
\qhyp21{q^{-m},\frac{q^{-m}}{b}}{qa}{q,
\frac{q^{m+1}abt}{c}}.
\label{gfbqiJgfB}
\end{eqnarray}
\end{thm}
\begin{proof}
Start with the generating function for the continuous 
dual $q^{-1}$-Hahn polynomials, Theorem 
\ref{thm48} and replace
$z\mapsto q^{-m}a$. Then inserting the duality 
relation \eqref{dcdqiHbqJab} provides a generating 
relation in terms of big $q$-Jacobi polynomials. 
Finally replacing $t\mapsto t/(abc)$ followed 
by $(a,b,c)\mapsto(\frac{1}{\sqrt{qab}},
\frac{\sqrt{b}}{\sqrt{qa}},\frac{c}{\sqrt{qab}})$ 
completes the proof.
\end{proof}

\noindent Now we consider \eqref{eq:thm49-gf}.

\begin{thm}\label{thm:7.9}
Let $q\in\CCdag$, $m\in\mathbb N_0$, $a,b,c,\gamma,t\in\CCast$, 
$|t|<1$. Then
\begin{eqnarray}
&&\hspace{-1.5cm}\sum_{n=0}^\infty t^n
\frac{(\gamma;q)_n}{(q;q)_n}P_m(q^{n+1}a;a,b,c;q)
\nonumber\\
&&\hspace{1cm}
=q^{\binom{m}{2}}(-qc)^m 
\frac{(\gamma t;q)_\infty}{(t;q)_\infty}
\frac{(\frac{qab}{c};q)_m}{(qc;q)_m}
\qhyp33{q^{-m},\gamma,q^{m+1}ab}{qa,\frac{qab}{c},
\gamma t}{q,\frac{qat}{c}}.
\label{gfbqiJgfC}
\end{eqnarray}
\end{thm}
\begin{proof}
Start with the generating function for the continuous 
dual $q^{-1}$-Hahn polynomials \eqref{eq:thm49-gf} 
and replace $z\mapsto q^{-m}a$. Then inserting the 
duality relation \eqref{dcdqiHbqJab} provides a 
generating relation in terms of big $q$-Jacobi 
polynomials. Finally replacing $t\mapsto t/(abc)$ 
followed by $(a,b,c)\mapsto(\frac{1}{\sqrt{qab}},
\frac{\sqrt{b}}{\sqrt{qa}},\frac{c}{\sqrt{qab}})$ 
completes the proof.
\end{proof}

\begin{rem}\label{rem:7.10}
Applying the above procedure to Theorem \ref{cor410} provides two different representations of the ${}_3\phi_3$ generating function for big $q$-Jacobi polynomials given by \eqref{gfbqiJgfC}. However, we will leave these as an exercise to the reader.
\end{rem}

\noindent Setting $\gamma=q$ in Theorem \ref{thm:7.9} provides the following interesting corollary.

\begin{cor}\label{cor:7.11}
Let $q\in\CCdag$, $m\in\mathbb N_0$, $a,b,c,t\in\CCast$, $|t|<1$. Then
\begin{eqnarray}
&&\hspace{-1.5cm}\sum_{n=0}^\infty t^nP_m(q^{n+1}a;a,b,c;q)
\nonumber\\
&&\hspace{1cm}
=q^{\binom{m}{2}}(-qc)^m 
\frac{(q t;q)_\infty}{(t;q)_\infty}
\frac{(\frac{qab}{c};q)_m}{(qc;q)_m}
\qhyp33{q,q^{-m},q^{m+1}ab}{qa,\frac{qab}{c},q t}{q,\frac{qat}{c}}.
\label{gfbqiJgfDA}
\end{eqnarray}
\end{cor}
\begin{proof}
Setting $\gamma=q$ in Theorem \ref{thm:7.9} completes the proof.
\end{proof}

\noindent Setting $\gamma=qa$ in the above result provides the following interesting result.

\begin{thm}\label{thm:7.12}
Let $q\in\CCdag$, $m\in\mathbb N_0$, $a,b,c,t\in\CCast$, 
$|t|<1$. Then
\begin{eqnarray}
&&\hspace{-1.5cm}\sum_{n=0}^\infty t^n\frac{(q a;q)_n}{(q;q)_n}P_m(q^{n+1}a;a,b,c;q)
\nonumber\\
&&\hspace{1cm}
=q^{\binom{m}{2}}(-qc)^m 
\frac{(q a t;q)_\infty}{(t;q)_\infty}
\frac{(\frac{qab}{c};q)_m}{(qc;q)_m}
\qhyp22{q^{-m},q^{m+1}ab}{\frac{qab}{c},q a t}{q,\frac{qat}{c}}\\
&&\hspace{1cm}=
q^{\binom{m}{2}}(-qc)^m 
\frac{(q a t;q)_\infty}{(t;q)_\infty}
\frac{(\frac{qab}{c};q)_m}{(qc,qat;q)_m}
\qhyp21{q^{-m},\frac{q^{-m}}{c}}{\frac{qab}{c}}{q,q^{m+1}at}.
\label{gfbqiJgfDB}
\end{eqnarray}
\end{thm}
\begin{proof}
Start with the generating function for the continuous dual $q^{-1}$-Hahn polynomials, \eqref{eq:thm49-gf} and
and replace
$z\mapsto q^{-m}a$. Then inserting the duality relation
\eqref{dcdqiHbqJab} provides an generating relation in terms of big $q$-Jacobi polynomials. Finally replacing $t\mapsto t/(abc)$ followed by $(a,b,c)\mapsto(\frac{1}{\sqrt{qab}},\frac{\sqrt{b}}{\sqrt{qa}},\frac{c}{\sqrt{qab}})$ completes the proof.
\end{proof}

\subsubsection{Orthogonal discrete generating relations for the big $q$-Jacobi polynomials}

\begin{thm}
\label{thm:7.30}
Let $m\in\mathbb N_0$, $q\in\CCdag$, $t,a,b,c\in\CCast$.
Then
\begin{eqnarray}
&&\hspace{0cm}\sum_{n=0}^\infty \frac{1}{(qa)^n}\frac{(qa,qab,\pm\sqrt{q^3ab};q)_n}{(q,qb,\pm\sqrt{qab};q)_n}
\qhyp22{q^{-n},q^{n+1}ab}{qa,\frac{qab}{c}}{q,\frac{qat}{c}}P_n(q^{m+1}a;a,b,c;q)\nonumber\\
&&\hspace{5cm}=
\left(\frac{t}{q}\right)^m
\frac{(q^2ab,\frac{c}{a},t;q)_\infty(\frac{qa}{c};q)_m}{(qb,qc;q)_\infty(qa,\frac{qab}{c};q)_m}.
\end{eqnarray}
\end{thm}
\begin{proof}
Start with \eqref{gfbqiJgfA} and multiply both sides 
by 
\begin{equation}
\frac{q^{-\binom{m}{2}}}{(-q^2ac)^m}
\frac{(qa,qc,qab,\pm\sqrt{q^3ab};q)_m}{(q,qb,\frac{qab}{c},\pm\sqrt{qab};q)_m}
\,P_m\left(q^{n'+1}a;a,b,c;q\right)
\label{bqJfac}
\end{equation}
and then sum both sides over $m\in\mathbb N_0$ using
the discrete orthogonality given by 
Theorem \ref{thm316} selects the $n=n'$ term 
in the sum over $n$. Replacing $n \leftrightarrow n'$,
completes the proof.

\end{proof}

\begin{thm}
\label{thm:7.31}
Let $m\in\mathbb N_0$, $q\in\CCdag$, $t,a,b,c\in\CCast$.
Then
\begin{eqnarray}
&&\hspace{0cm}\sum_{n=0}^\infty \frac{1}{(qa)^n}\frac{(qa,qab,\pm\sqrt{q^3ab};q)_n}{(q,qb,\frac{qabt}{c},\pm\sqrt{qab};q)_n}
\qhyp21{q^{-n},\frac{q^{-n}}{b}}{qa}{q,\frac{q^{n+1}abt}{c}}P_n(q^{m+1}a;a,b,c;q)\nonumber\\
&&\hspace{6cm}=
\left(\frac{t}{q}\right)^m
\frac{(q^2ab,\frac{c}{a},t;q)_\infty}{(qb,qc,\frac{qabt}{c};q)_\infty(qa;q)_m}
\left(\dfrac{qa}{c};q\right)_m.
\end{eqnarray}
\end{thm}
\begin{proof}
Start with \eqref{gfbqiJgfB} and multiply both sides 
by \eqref{bqJfac}
and then sum both sides over $m\in\mathbb N_0$ using
the discrete orthogonality given by 
Theorem \ref{thm316} selects the $n=n'$ term 
in the sum over $n$. Replacing $n \leftrightarrow n'$,
completes the proof.
\end{proof}

\begin{thm}
\label{thm:7.32}
Let $m\in\mathbb N_0$, $q\in\CCdag$, 
$t,a,b,c,\gamma\in\CCast$. Then
\begin{eqnarray}
&&\hspace{0cm}\sum_{n=0}^\infty \frac{1}{(qa)^n}
\frac{(qa,qab,\pm\sqrt{q^3ab};q)_n}{(q,qb,\pm\sqrt{qab};q)_n}
\qhyp33{q^{-n},q^{n+1}ab,\gamma}{qa,\frac{qab}{c},\gamma t}{q,\frac{qat}{c}}P_n(q^{m+1}a;a,b,c;q)\nonumber\\
&&\hspace{5cm}=
\left(\frac{t}{q}\right)^m
\frac{(q^2ab,\frac{c}{a},t;q)_\infty(\gamma,\frac{qa}{c};q)_m}{(qb,qc,\gamma t;q)_\infty(qa,\frac{qab}{c};q)_m}.
\end{eqnarray}
\end{thm}
\begin{proof}
Start with \eqref{gfbqiJgfC} and multiply both sides 
by \eqref{bqJfac}
and then sum both sides over $m\in\mathbb N_0$ using
the discrete orthogonality given by 
Theorem \ref{thm316} selects the $n=n'$ term 
in the sum over $n$. Replacing $n \leftrightarrow n'$,
completes the proof.
\end{proof}

\begin{rem}\label{rmk:7.33}
If $\gamma=qab/c$ then the limit
series formula is equivalent to Theorem \ref{thm:7.31} using \cite[(III.4)]{GaspRah}.
If one takes the limit $\gamma\to 0$ in Theorem \ref{thm:7.32} then this infinite
series expression becomes Theorem \ref{thm:7.30}. 
If $\gamma=qa$ then the limit
series formula is
\begin{eqnarray}
&&\hspace{0cm}\sum_{n=0}^\infty \frac{1}{(qa)^n}\frac{(qa,qab,\pm\sqrt{q^3ab};q)_n}{(q,qb,qat,\pm\sqrt{qab};q)_n}
\qhyp21{q^{-n},\frac{q^{-n}}{c}}{\frac{qab}{c}}{q,q^{n+1}at}P_n(q^{m+1}a;a,b,c;q)\nonumber\\
&&\hspace{8.0cm}=\label{eq:rmk:7.33}
\left(\frac{t}{q}\right)^m
\frac{(q^2ab,\frac{c}{a},t;q)_\infty(\frac{qa}{c};q)_m}{(qb,qc,qat;q)_\infty(\frac{qab}{c};q)_m}.
\end{eqnarray}
\end{rem}


\subsubsection{Generating relations from continuous dual $q$-Hahn polynomials}

By using the duality relation between continuous dual $q$-Hahn polynomials and 
the big $q$-Jacobi function, we are able to derive generating relations for the big $q$-Jacobi function. From these generating relations, by using the continuous orthogonality of the big $q$-Jacobi functions we are able to compute integral generating relations. In the rest of this subsection we will derive these relations.

In this subsection, using the duality of the continuous dual $q$-Hahn polynomials \eqref{dcdqHbqJab} with generating functions for these polynomials we derive generating relations for the big $q$-Jacobi polynomials. We start by transforming Theorem \ref{thm52}.
\begin{thm}
\label{thm:7.1}
Let $q\in\CCdag$, $\mu\in\CC$, $a,b,c,t\in\CCast$, $|t|<|q^{-\mu-1}/(ab)|$. Then
\begin{eqnarray}
&&\hspace{-1.0cm}\sum_{n=0}^\infty t^n\frac{(qc;q)_n}{(q;q)_n}P_\mu(q^{-n};a,b,c;q)
=\frac{(qct;q)_\infty}{(q^{-\mu}t;q)_\infty}\qhyp21{q^{-\mu},\frac{q^{-\mu}}{b}}{qa}{q,q^{\mu+1}abt}.
\end{eqnarray}
\end{thm}
\begin{proof}
Start with the generating function for the continuous dual $q$-Hahn polynomials, Theorem \ref{thm52}, and replace
$z\mapsto \frac{q^{-m}}{a}$. Then inserting the duality relation
\eqref{dcdqHbqJab} provides an generating relation in terms of the big $q$-Jacobi function. Replace $t\mapsto at$ and then take $(a,b,c)\mapsto(\sqrt{qab},\sqrt{\frac{qa}{b}},\frac{\sqrt{q}c}{\sqrt{ab}})$.
Finally, letting
$m\mapsto \mu\in\CC$ by analytic continuation completes the proof.
\end{proof}

\noindent Now we consider
Theorem \ref{thm43}.

\begin{thm}
\label{thm:7.2}
Let $q\in\CCdag$, $m\in\mathbb N_0$, $\mu\in\CC$, $a,b,c,t\in\CCast$, $|t|<1$. Then
\begin{eqnarray}
&&\hspace{-1.5cm}\sum_{n=0}^\infty t^n\frac{q^{\binom{n}{2}}}{(q;q)_n}P_\mu(q^{-n};a,b,c;q)
=(-t;q)_\infty\qphyp22{-1}{q^{-\mu},q^{\mu+1}ab}{qa,qc}{q,-t}.
\end{eqnarray}
\end{thm}
\begin{proof}
Start with the generating function for the continuous dual $q$-Hahn polynomials, Theorem \ref{thm43}, and replace
$z\mapsto \frac{q^{-m}}{a}$. Then inserting the duality relation
\eqref{dcdqHbqJab} provides an generating relation in terms of big $q$-Jacobi polynomials. Then replace $t\mapsto at$ followed by $(a,b,c)\mapsto(\sqrt{qab},\sqrt{\frac{qa}{b}},\frac{\sqrt{q}c}{\sqrt{ab}})$.
Finally, letting
$m\mapsto \mu\in\CC$ by analytic continuation completes the proof.
\end{proof}

\noindent Now we derive from the non-standard generating function \eqref{ATAT},
a non-standard generating relation for big $q$-Jacobi polynomials.
Recall that a non-standard generating function is one in which the parameter $t$ appears in the series coefficient as well as appearing as the power series parameter.

\begin{thm}
\label{thm:7.3}
Let $q\in\CCdag$, $\mu\in\CC$, $a,b,c,t\in\CCast$, $|t|<1$. Then
\begin{eqnarray}
&&\hspace{-1.5cm}\sum_{n=0}^\infty t^n\frac{(qa,qc;q)_n}{(q,q^2act;q)_n}P_\mu(q^{-n};a,b,c;q)
=\frac{(qat,qct,qabt;q)_\infty}{(q^2act,q^{-\mu}t,q^{\mu+1}abt;q)_\infty}.
\label{nsbqJgf}
\end{eqnarray}
\end{thm}
\begin{proof}
Start with the non-standard generating function for the continuous dual $q$-Hahn polynomials \eqref{ATAT}, and replace
$z\mapsto \frac{q^{-m}}{a}$. Then inserting the duality relation
\eqref{dcdqHbqJab} provides an generating relation in terms of big $q$-Jacobi polynomials. Then replace $t\mapsto at$ followed by $(a,b,c)\mapsto(\sqrt{qab},\sqrt{\frac{qa}{b}},\frac{\sqrt{q}c}{\sqrt{ab}})$.
Finally, letting
$m\mapsto \mu\in\CC$ by analytic continuation
converts the big $q$-Jacobi polynomial to a big $q$-Jacobi function. This completes the proof.
\end{proof}

\begin{rem}
Applying the above procedure to Theorem \ref{thm55} provides two different representations of the non-standard generating relation given by \eqref{nsbqJgf}. However, we will leave these as an exercise to the reader.
\end{rem}

\noindent Now we consider
the generating function for the continuous dual $q$-Hahn polynomials
with a free parameter $\gamma\in\CC$ which appears in \cite[Theorem 3.8]{CohlCostasSantos23}. 

\begin{thm}
\label{thm:7.5}
Let $q\in\CCdag$, $\mu\in\CC$, $\gamma,a,b,c,t\in\CCast$, $|t|<1$. Then
\begin{eqnarray}
&&\hspace{-0.2cm}\sum_{n=0}^\infty t^n\frac{(\gamma,qc;q)_n}{(q,\frac{qc}{b};q)_n}P_\mu(q^{-n};a,b,c;q)\nonumber\\
&&\hspace{0.5cm}
=\frac{(\gamma b t;q)_\infty}{(bt;q)_\infty}
\qphyp33{-1}{\gamma,\frac{q^{-\mu}}{b},q^{\mu+1}a}
{qa,\frac{qc}{b},\frac{q}{bt}}{q,q}\nonumber\\
&&\hspace{1.5cm}+
\frac{(\gamma,\frac{1}{a},\frac{1}{b},qb,qct,\frac{q^{-\mu}}{abt},\frac{q^{\mu+1}}{t};q)_\infty}
{(t,\frac{q}{t},\frac{qc}{b},\frac{1}{bt},\frac{1}{abt},\frac{q^{-\mu}}{a},q^{\mu+1}b;q)_\infty}
\qphyp33{-1}{\gamma bt,q^{-\mu}t,q^{\mu+1}abt}
{qbt,qct,qabt}{q,q}.
\label{nsbqJgfB}
\end{eqnarray}
\end{thm}
\begin{proof}
Start with the generating function for the continuous 
dual $q$-Hahn polynomials \cite[Theorem 3.8]{CohlCostasSantos23}, 
replace $a\leftrightarrow b$ due to the symmetry 
of continuous dual $q$-Hahn polynomials and then replace
$z\mapsto \frac{q^{-m}}{a}$. Then inserting the duality relation
\eqref{dcdqHbqJab} provides an generating relation in terms of big $q$-Jacobi polynomials. Then replace $t\mapsto at$ followed by $(a,b,c)\mapsto(\sqrt{qab},\sqrt{\frac{qa}{b}},\frac{\sqrt{q}c}{\sqrt{ab}})$.
Finally, letting
$m\mapsto \mu\in\CC$ by analytic continuation
converts the big $q$-Jacobi polynomial to a big $q$-Jacobi function. This completes the proof.
\end{proof}

\noindent If one sets $\gamma=q^{-m}$ in Theorem \ref{thm:7.5}, then the second term on the right-hand side of the above theorem vanishes and we are led to the following result.

\begin{cor}
\label{cor:7.6}
Let $m\in\mathbb N_0$, $q\in\CCdag$, $\mu\in\CC$, $a,b,c,t\in\CCast$, $|t|<1$. Then
\begin{eqnarray}
&&\hspace{-0.8cm}\sum_{n=0}^\infty t^n\frac{(q^{-m},qc;q)_n}{(q,\frac{qc}{b};q)_n}P_\mu(q^{-n};a,b,c;q)
\!=\!q^{-\binom{m}{2}}\!\!\left(\frac{bt}{q}\right)^m\!\!\left(\frac{q}{bt};q\right)_m\!
\!\!\qphyp33{-1}{q^{-m},\frac{q^{-\mu}}{b},q^{\mu+1}a}
{qa,\frac{qc}{b},\frac{q}{bt}}{q,q}\!.
\label{nsbqJgfC}
\end{eqnarray}
\end{cor}
\begin{proof}
Setting $\gamma=q^{-m}$ in the above theorem
completes the proof.
\end{proof}

\noindent Another important special case of Theorem \ref{thm:7.5} is $\gamma=\frac{qc}{b}$ and the generating function reduces to 
Theorem \ref{thm:7.1}. In this case the 
right-hand side of the generating function becomes a sum of two ${}_2\phi_2^{-1}$'s which reduces to a ${}_2\phi_1$ using \cite[\href{http://dlmf.nist.gov/17.9.E3}{(17.9.3)}]{NIST:DLMF}.

\subsubsection{Orthogonal integral relations for big $q$-Jacobi functions}
\noindent Recall the definition of ${\sf A}={\sf A}(a,b;q)$ \eqref{Adef}.
From Theorem \ref{thm:7.1} and using the continuous orthogonality of big $q$-Jacobi functions in Theorem \ref{cobqJ}, we can obtain the following result.

\begin{thm}\label{thm:7.38}
Let $n\in\mathbb N_0$, $q\in\CCdag$, $a,b,c,t\in\CCast$, $|t|<1/\sqrt{|qab|}$. Then
\begin{eqnarray}
&&\hspace{-0cm}\int_{0}^{\frac{\pi}{\log q^{-1}}}
\frac{(q^{\pm 2ix};q)_\infty\,P_{{\sf A}+ix}\left(q^{-n};a,b,c;q\right)}{(q^{\frac12\pm ix}\sqrt{ab},q^{\frac12\pm ix}\sqrt{\frac{a}{b}},q^{\frac12\pm ix}\frac{c}{\sqrt{ab}},q^{\frac12-ix}\sqrt{ab}t;q)_\infty}
\qhyp21{q^{\frac12-ix}\sqrt{ab^\pm}}{qa}{q,q^{\frac12+ix}\sqrt{ab}t}
\,\dd x
\nonumber\\&&\hspace{6.cm}
=\frac{2\pi(\frac{qc}{b};q)_n(qabt)^n}{(q,qa,qc,\frac{qc}{b},qct;q)_\infty
\log q^{-1}(qa;q)_n}.
\end{eqnarray}
\end{thm}
\begin{proof}
Start with Theorem \ref{thm:7.1}, and multiply 
both sides by the the appropriate big $q$-Jacobi function and weight function. Then integrating taking advantage of the orthogonality 
relation in Theorem \ref{cobqJ} completes the proof.
\end{proof}

\noindent From Theorem \ref{thm:7.2} and using the continuous orthogonality of big $q$-Jacobi functions we can obtain the following result.
\begin{thm}\label{thm:7.39}
Let $n\in\mathbb N_0$, $q\in\CCdag$, $a,b,c,t\in\CCast$, $|t|<1$. Then
\begin{eqnarray}
&&\hspace{-0cm}\int_{0}^{\frac{\pi}{\log q^{-1}}}
\frac{(q^{\pm 2ix};q)_\infty\,P_{{\sf A}+ix}\left(q^{-n};a,b,c;q\right)}{(q^{\frac12\pm ix}\sqrt{ab},q^{\frac12\pm ix}\sqrt{\frac{a}{b}},q^{\frac12\pm ix}\frac{c}{\sqrt{ab}};q)_\infty}
\qphyp22{-1}{q^{\frac12\pm ix}\sqrt{ab}}
{qa,qc}{q,-t}
\,\dd x
\nonumber\\&&\hspace{6cm}
=\frac{q^{\binom{n}{2}}2\pi(\frac{qc}{b};q)_n(qabt)^n}{(q,qa,qc,\frac{qc}{b},-t;q)_\infty\log q^{-1}(qa,qc;q)_n}.
\end{eqnarray}
\end{thm}
\begin{proof}
Start with Theorem \ref{thm:7.2}, and multiply 
both sides by the the appropriate big $q$-Jacobi function and weight function. Then integrating taking advantage of the orthogonality 
relation in Theorem \ref{cobqJ} completes the proof.
\end{proof}

\noindent From Theorem \ref{thm:7.3} and using the continuous orthogonality of big $q$-Jacobi functions we can obtain the following result.
\begin{thm}\label{thm:7.40}
Let $n\in\mathbb N_0$, $q\in\CCdag$, $a,b,c,t\in\CCast$, $|t|<1$. Then
\begin{eqnarray}
&&\hspace{-0cm}\int_{0}^{\frac{\pi}{\log q^{-1}}}
\frac{(q^{\pm 2ix};q)_\infty\,P_{{\sf A}+ix}\left(q^{-n};a,b,c;q\right)}{(  q^{\frac12\pm ix}\sqrt{ab},q^{\frac12\pm ix}\sqrt{\frac{a}{b}},q^{\frac12\pm ix}\frac{c}{\sqrt{ab}},q^{\frac12\pm ix}\sqrt{ab}t;q)_\infty}
\,\dd x
\nonumber\\&&\hspace{4.4cm}\label{eq:thm:7.40}
=\frac{2\pi (q^2act;q)_\infty(\frac{qc}{b};q)_n(qabt)^n}{(q,qa,qc,\frac{qc}{b},qat,qct,qabt;q)_\infty\log q^{-1}(q^2act;q)_n}.
\end{eqnarray}
\end{thm}
\begin{proof}
Start with Theorem \ref{thm:7.3}, and multiply 
both sides by the the appropriate big $q$-Jacobi function and weight function. Then integrating taking advantage of the orthogonality 
relation in Theorem \ref{cobqJ} completes the proof.
\end{proof}

\noindent From Theorem \ref{thm:7.5} and using the continuous orthogonality of big $q$-Jacobi functions we can obtain the following result.

\begin{thm}\label{thm:7.41}
Let $n\in\mathbb N_0$, $q\in\CCdag$, $a,b,c,\gamma,t\in\CCast$, and define
\begin{eqnarray}
&&\hspace{-1cm}{\sf F}_x(a,b,c,\gamma,t;q):=\frac{(\gamma b t;q)_\infty}{(bt;q)_\infty}
\qphyp33{-1}{\gamma,q^{\frac12\pm ix}\sqrt{\frac{a}{b}}}
{qa,\frac{qc}{b},\frac{q}{bt}}{q,q}\nonumber\\
&&\hspace{2.2cm}+
\frac{(\gamma,\frac{1}{a},\frac{1}{b},qb,qct,\frac{q^{\frac12\pm ix}}{\sqrt{ab}t};q)_\infty}
{(t,\frac{q}{t},\frac{qc}{b},\frac{1}{bt},\frac{1}{abt},q^{\frac12\pm ix}\sqrt{\frac{a}{b}};q)_\infty}
\qphyp33{-1}{\gamma bt,q^{\frac12\pm ix}\sqrt{ab}t}
{qbt,qct,qabt}{q,q}.
\label{nsbqJgfB2}
\end{eqnarray}
Then
\begin{eqnarray}
&&\hspace{-1.3cm}\int_{0}^{\frac{\pi}{\log q^{-1}}}
\frac{(q^{\pm 2ix};q)_\infty\,P_{{\sf A}+ix}\left(q^{-n};a,b,c;q\right)}{(q^{\frac12\pm ix}\sqrt{ab},q^{\frac12\pm ix}\sqrt{\frac{a}{b}},q^{\frac12\pm ix}\frac{c}{\sqrt{ab}};q)_\infty}
{\sf F}_x(a,b,c,\gamma,t;q)
\,\dd x
\nonumber\\&&\hspace{5.6cm}
=\frac{2\pi(qabt)^n(\gamma;q)_n}{(q,qa,qc,\frac{qc}{b};q)_\infty\log q^{-1}(qa;q)_n
}.
\end{eqnarray}
\end{thm}
\begin{proof}
Start with Theorem \ref{thm:7.5}, and multiply 
both sides by the the appropriate big $q$-Jacobi function and weight function. Then integrating taking advantage of the orthogonality 
relation in Theorem \ref{cobqJ} completes the proof.
\end{proof}

\noindent From Corollary \ref{cor:7.6} and using the continuous orthogonality of big $q$-Jacobi functions we can obtain the following result.
\begin{thm}\label{thm:7.42}
Let $m,n\in\mathbb N_0$, $q\in\CCdag$, $a,b,c,t\in\CCast$. Then
\begin{eqnarray}
&&\hspace{-0cm}\int_{0}^{\frac{\pi}{\log q^{-1}}}
\frac{(q^{\pm 2ix};q)_\infty\,P_{{\sf A}+ix}\left(q^{-n};a,b,c;q\right)}{(q^{\frac12\pm ix}\sqrt{ab},q^{\frac12\pm ix}\sqrt{\frac{a}{b}},q^{\frac12\pm ix}\frac{c}{\sqrt{ab}};q)_\infty}
\qphyp33{-1}{q^{-m},q^{\frac12\pm ix}\sqrt{\frac{a}{b}}}{qa,\frac{qc}{b},\frac{q}{bt}}{q,q}
\,\dd x
\nonumber\\&&\hspace{5.6cm}
=\frac{q^{\binom{m}{2}}2\pi(q^{-m};q)_n(-\frac{q}{bt})^m(qabt)^n}{(q,qa,qc,\frac{qc}{b};q)_\infty\log q^{-1}(\frac{q}{bt};q)_m(qa;q)_n},\label{eq:thm:7.42}
\end{eqnarray}
which vanishes of $n\ge m+1$.
\end{thm}
\begin{proof}
Start with Corollary \ref{cor:7.6}, and multiply 
both sides by the the appropriate big $q$-Jacobi function and weight function. Then integrating taking advantage of the orthogonality 
relation in Theorem \ref{cobqJ} completes the proof.
\end{proof}

\subsection{Generating relations for the little {\it q-}Jacobi polynomials and functions}\label{sec:7.2}

By using the duality relation between $q^{-1}$-Al-Salam--Chihara polynomials and 
the little $q$-Jacobi polynomials, we are able to derive generating relations for the little $q$-Jacobi polynomials. From these generating relations, by using the infinite discrete orthogonality of the little $q$-Jacobi polynomials we are able to derive discrete infinite series generating relations. In the rest of this subsection we will derive these relations.

\subsubsection{Generating relations from the $q^{-1}$-Al-Salam--Chihara polynomials}
In this subsection, using the duality of the $q^{-1}$-Al-Salam--Chihara polynomials with generating functions for 
these polynomials we derive generating relations for the little 
$q$-Jacobi polynomials. We first consider \eqref{qiASCgf1}.

\begin{thm}\label{thm:7.18}
Let $q\in\CCdag$, $m\in\mathbb N_0$, $a,b,t\in\CCast$, $|t|<1$. Then
\begin{eqnarray}
&&\hspace{-0.5cm}
\sum_{n=0}^\infty t^n\frac{q^{\binom{n}{2}}}{(q;q)_n}
\,p_m\left(q^{n};a,b;q\right)
=
q^{\binom{m}{2}}(-qa)^m
\frac{(-t;q)_\infty(qb;q)_m}{(qa,-t;q)_m}
\qhyp21{q^{-m},\frac{q^{-m}}{a}}{qb}{q,-q^{m}t}.
\label{lqJgr1}
\end{eqnarray}
\end{thm}
\begin{proof}
Start with the generating function for the $q^{-1}$-Al-Salam--Chihara polynomials \eqref{qiASCgf1}
and replace
$z\mapsto q^{-m}a$. Then inserting the duality relation
\eqref{dqiASClqJa} provides an generating relation in terms of little $q$-Jacobi polynomials. Finally replacing $t\mapsto -t/b$ followed by $(a,b)\mapsto(1/\sqrt{qab},\sqrt{\frac{a}{qb}})$ completes the proof.
\end{proof}

\noindent Next we consider \eqref{qiASCgfY}.

\begin{thm}
\label{thm:7.19}
Let $q\in\CCdag$, $m\in\mathbb N_0$, $a,b,t\in\CCast$, $|t|<1$. Then
\begin{eqnarray}
&&\hspace{-0.5cm}
\sum_{n=0}^\infty t^n \frac{(qb;q)_n}{(q;q)_n}
\,p_m\left(q^{n};a,b;q\right)
=
t^m
\frac{(qbt;q)_\infty}{(t;q)_\infty}
\frac{(qb;q)_m}{(qa,qbt;q)_m}
\left(\dfrac{qa}{t};q\right)_m.
\label{lqJgr2}
\end{eqnarray}
\end{thm}
\begin{proof}
Start with the generating function for the $q^{-1}$-Al-Salam--Chihara polynomials \eqref{qiASCgfY}
and replace
$z\mapsto q^{-m}a$. Then, inserting the duality relation
\eqref{dqiASClqJa} provides a generating relation in terms of little $q$-Jacobi polynomials. Finally replacing $t\mapsto -t/b$ followed by $(a,b)\mapsto(1/\sqrt{qab},\sqrt{\frac{a}{qb}})$ completes the proof.
\end{proof}

\noindent Now we consider \eqref{qiASCgfX}.

\begin{thm}
\label{thm:7.20}
Let $q\in\CCdag$, $m\in\mathbb N_0$, $a,b,t\in\CCast$, $|t|<1$. Then
\begin{eqnarray}
&&\hspace{-0.8cm}
\sum_{n=0}^\infty \frac{t^n}{(q;q)_n}
\,p_m\left(q^{n};a,b;q\right)
=
q^{\binom{m}{2}}(-qa)^m
\frac{(qb;q)_m}{(t;q)_\infty(qa;q)_m}
\qhyp21{q^{-m},q^{m+1}ab}{qb}{q,\frac{t}{a}}.
\label{lqJgr3}
\end{eqnarray}
\end{thm}
\begin{proof}
Start with the generating function for the $q^{-1}$-Al-Salam--Chihara polynomials \eqref{qiASCgfX}
and replace
$z\mapsto q^{-m}a$. Then, inserting the duality relation
\eqref{dqiASClqJa} provides a generating relation in terms of little $q$-Jacobi polynomials. Finally replacing $t\mapsto -t/b$ followed by $(a,b)\mapsto(1/\sqrt{qab},\sqrt{\frac{a}{qb}})$ completes the proof.
\end{proof}

\noindent We now consider \eqref{gf32h}.

\begin{thm}
\label{thm:7.21}
Let $q\in\CCdag$, $m\in\mathbb N_0$, $a,b,\gamma,t\in\CCast$, $|t|<1$. Then
\begin{eqnarray}
&&\hspace{-0.8cm}
\sum_{n=0}^\infty t^n\frac{(\gamma;q)_n}{(q;q)_n}
\,p_m\left(q^{n};a,b;q\right)
=
q^{\binom{m}{2}}
\frac{(\gamma t;q)_\infty}{(t;q)_\infty}(-qa)^m
\frac{(qb;q)_m}{(qa;q)_m}
\qhyp32{q^{-m},\gamma,q^{m+1}ab}{qb,\gamma t}{q,\frac{t}{a}}.
\label{gammagflqJ}
\end{eqnarray}
\end{thm}
\begin{proof}
Start with the generating function for the $q^{-1}$-Al-Salam--Chihara polynomials \eqref{gf32h}
and replace
$z\mapsto q^{-m}a$. Then, inserting the duality relation
\eqref{dqiASClqJa} provides a generating relation in terms of little $q$-Jacobi polynomials. Finally replacing $t\mapsto -t/b$ followed by $(a,b)\mapsto(1/\sqrt{qab},\sqrt{\frac{a}{qb}})$ completes the proof.
\end{proof}

By setting $t=qa$ in \eqref{gammagflqJ}, one can obtain a non-standard connection formula for big $q$-Jacobi polynomials in terms of little $q$-Jacobi polynomials. However, this particular big $q$-Jacobi polynomial is summable using the $q$-Pfaff-Saalsch\"utz sum \cite[\href{http://dlmf.nist.gov/17.7.E5}{(17.7.5)}]{NIST:DLMF}, namely
\begin{equation}
P_m(\gamma;b,a,a\gamma;q)=\qhyp32{q^{-m},q^{m+1}ab,\gamma}{qb,qa\gamma}{q,q}=
\gamma^m\frac{(qa;q)_m}{(qb,qa\gamma;q)_m}
\left(\dfrac{qb}{\gamma};q\right)_m.
\label{sumthree2}
\end{equation}
By making the replacement $t=qa$ in the above
result and using this identity, we obtain
the following result.
\begin{cor}\label{cor:7.22}
Let $m\in\mathbb N_0$, $q\in\CCdag$, $a,b,\gamma\in\CCast$.
Then
\begin{eqnarray}
&&\sum_{n=0}^\infty
(qa)^n
\frac{(\gamma;q)_n}{(q;q)_n}p_m(q^n;a,b;q)
=q^{\binom{m}{2}}(-qa\gamma)^m
\frac{(qa\gamma;q)_\infty}{(qa;q)_\infty(qa\gamma;q)_m}
\left(\dfrac{qb}{\gamma};q\right)_m.
\end{eqnarray}
\end{cor}
\begin{proof}
Start with \eqref{gammagflqJ}
and set $t=qa$, then simplifying,
and using \eqref{sumthree2} completes the
proof.
\end{proof}

\begin{rem}\label{rmk:7.23}
Applying the above procedure to Corollary \ref{cor520} 
provides two different representations of the 
${}_3\phi_2$ generating function for little $q$-Jacobi 
polynomials given by \eqref{gammagflqJ}. However, 
we will leave these as an exercise to the reader.
\end{rem}

\noindent An interesting special case of Theorem 
\ref{thm:7.21} is if we set $\gamma=q$.

\begin{cor}\label{cor:7.24}
Let $q\in\CCdag$, $m\in\mathbb N_0$, 
$a, b, t\in\CCast$, $|t|<1$. Then
\begin{eqnarray}
&&\hspace{-0.5cm}
\sum_{n=0}^\infty t^n
\,p_m\left(q^{n};a,b;q\right)
=
q^{\binom{m}{2}}
\frac{(q t;q)_\infty}{(t;q)_\infty}(-qa)^m
\frac{(qb;q)_m}{(qa;q)_m}
\qhyp32{q^{-m},q,q^{m+1}ab}{qb,q t}{q,\frac{t}{a}}.
\end{eqnarray}
\end{cor}
\begin{proof}
Setting $\gamma=q$ in the above theorem completes 
the proof.\end{proof}

\noindent
If we take the limit as $\gamma\to 0$ 
in Theorem \ref{thm:7.21}, then
using \eqref{critlim}, it becomes
Theorem \ref{thm:7.20}.
Setting $\gamma=qb$ in Theorem \ref{thm:7.21} then using 
\eqref{critlim} it becomes Theorem \ref{thm:7.19} 
using the $q$-Gauss sum 
\cite[\href{http://dlmf.nist.gov/17.6.E1}{(17.6.1)}]{NIST:DLMF}.

\subsubsection{Orthogonal discrete generating relations for the little $q$-Jacobi polynomials}

Utilizing the duality of the Al-Salam--Chihara and $q^{-1}$-Al-Salam--Chihara polynomials to the little $q$-Jacobi polynomials, Theorem \ref{thm312}, we can derive using generating functions for these polynomials, 
new generating relations for the little $q$-Jacobi polynomials.

\begin{thm}\label{thm:7.25}
Let $m\in\mathbb N_0$, $q\in\CCdag$, 
$a,b,t\in\CCast$. Then
\begin{eqnarray}
&&\hspace{-1.5cm}\sum_{n=0}^\infty
q^{\binom{n}{2}}(-1)^n
\frac{(qab,\pm\sqrt{q^3ab};q)_n}
{(q,-t,\pm\sqrt{qab};q)_n}
\qhyp21{q^{-n},\frac{q^{-n}}{a}}{qb}
{q,-q^nt}p_n(q^m;a,b;q)\nonumber\\
&&\hspace{7cm}=
\frac{(q^2ab;q)_\infty}{(qa,-t;q)_\infty}
\frac{q^{\binom{m}{2}}
}{(qb;q)_m}\left(\frac{t}{qa}\right)^m.
\end{eqnarray}
\end{thm}
\begin{proof}
Start with the dual generating relation 
for little $q$-Jacobi polynomials given 
by \eqref{lqJgr1} and then multiply both 
sides by 
\begin{equation}
\left(\frac{1}{qa}\right)^m
\frac{(qa,qab,\pm\sqrt{q^3ab};q)_m}
{(q,qb,\pm\sqrt{qab};q)_m}
\,p_m\left(q^n;a,b;q\right),
\label{prodterm}
\end{equation}
and sum the resulting expression over 
$m\in\mathbb N_0$.
Then using the dual orthogonality of little 
$q$-Jacobi polynomials selects one term of
the generating relation corresponding to $n=n'$.
Finally relabeling $n'\mapsto n$ and reversing
the definitions $n\leftrightarrow m$ 
completes the proof.
\end{proof}

Now by starting with \eqref{lqJgr2} and 
applying dual orthogonality of the little 
$q$-Jacobi polynomials we can produce a 
non-standard generating function for 
these polynomials.

\begin{thm}\label{thm:7.26}
Let $m\in\mathbb N_0$, $q\in\CCdag$, 
$a,b,t\in\CCast$. Then
\begin{eqnarray}
&&\sum_{n=0}^\infty t^n \frac{(qab,\pm\sqrt{q^3ab},t^{-1};q)_n}{(q,\pm\sqrt{qab},q^2abt;q)_n}\,
p_n(q^m;a,b;q)=\frac{(q^2ab,qat;q)_\infty}{(qa,q^2abt;q)_\infty}t^m
\end{eqnarray}
\end{thm}
\begin{proof}
Start with the dual generating relation for little 
$q$-Jacobi polynomials given by \eqref{lqJgr2} and 
then multiply both sides by \eqref{prodterm}
and sum the resulting expression over $m\in\mathbb N_0$.
Then using the dual orthogonality of little 
$q$-Jacobi polynomials select one term of
the generating relation corresponding to $n=n'$.
Finally relabeling $n'\mapsto n$ and reversing
the definitions $n\leftrightarrow m$ followed by 
the map $t\mapsto qat$ completes the proof.
\end{proof}

\begin{thm}\label{thm:7.27}
Let $m\in\mathbb N_0$, $q\in\CCdag$, 
$a,b,t\in\CCast$. Then
\begin{eqnarray}
&&\hspace{-1.0cm}\sum_{n=0}^\infty (-1)^n 
q^{\binom{n}{2}} \frac{(qab,\pm\sqrt{q^3ab};q)_n}
{(q,\pm\sqrt{qab};q)_n}\qhyp21{q^{-n},q^{n+1}ab}{qb}
{q,\frac{t}{a}}p_n(q^m;a,b;q)\nonumber\\
&&\hspace{7cm}=\frac{(q^2ab,t;q)_\infty}
{(qa;q)_\infty(qb;q)_m}\left(\frac{t}{qa}\right)^m
\end{eqnarray}
\end{thm}
\begin{proof}
Start with the dual generating relation 
for little $q$-Jacobi polynomials given 
by \eqref{lqJgr3} and then multiply 
both sides by \eqref{prodterm}
and sum the resulting expression over 
$m\in\mathbb N_0$.
Then, using the dual orthogonality of 
little $q$-Jacobi polynomials 
select one term of the generating relation 
corresponding to $n=n'$.
Finally relabeling $n'\mapsto n$ and reversing
the definitions $n\leftrightarrow m$ followed 
by the map $t\mapsto qat$ completes the proof.
\end{proof}

\begin{thm}\label{thm:7.28}
Let $m\in\mathbb N_0$, $q\in\CCdag$, $a,b,\gamma,t\in\CCast$. Then
\begin{eqnarray}
&&\hspace{-1.0cm}\sum_{n=0}^\infty (-1)^n q^{\binom{n}{2}} \frac{(qab,\pm\sqrt{q^3ab};q)_n}{(q,\pm\sqrt{qab};q)_n}\qhyp32{q^{-n},q^{n+1}ab,\gamma}{qb,\gamma t}{q,\frac{t}{a}}p_n(q^m;a,b;q)\nonumber\\
&&\hspace{7cm}=\frac{(q^2ab,t;q)_\infty(\gamma;q)_m}{(qa,\gamma t;q)_\infty(qb;q)_m}\left(\frac{t}{qa}\right)^m.
\label{lqJgamdual}
\end{eqnarray}
\end{thm}
\begin{proof}
Start with the dual generating relation for little $q$-Jacobi polynomials given by \eqref{gammagflqJ} and then multiply both sides by \eqref{prodterm}
and sum the resulting expression over $m\in\mathbb N_0$.
Then using the dual orthogonality of little $q$-Jacobi polynomials selects one term of
the generating relation corresponding to $n=n'$.
Finally relabeling $n'\mapsto n$ and reversing
the definitions $n\leftrightarrow m$ followed by the map $t\mapsto qat$ completes the proof.
\end{proof}

By specializing the above sum with a choice of $t=qa$
then one may use the nonterminating $q$-Pfaff-Saalsch\"utz sum \cite[\href{http://dlmf.nist.gov/17.7.E5}{(17.7.5)}]{NIST:DLMF} to express the ${}_3\phi_2$ as a product. This produces an interesting non-standard generating function for little $q$-Jacobi polynomials.

\begin{cor}\label{cor:7.29}
Let $m\in\mathbb N_0$, $q\in\CCdag$, $a,b,t\in\CCast$. Then
\begin{eqnarray}
&&\hspace{-1.0cm}\sum_{n=0}^\infty (-t)^n q^{\binom{n}{2}} \frac{(qab,qa,\pm\sqrt{q^3ab},\frac{qb}{t};q)_n}{(q,qb,\pm\sqrt{qab},qat;q)_n}p_n(q^m;a,b;q)=\frac{(q^2ab,qa;q)_\infty(t;q)_m}{(qa,qat;q)_\infty(qb;q)_m}.
\end{eqnarray}
\end{cor}
\begin{proof}
Start with \eqref{lqJgamdual} and set $t=qa$.
This converts the ${}_3\phi_2$ into a form in
which the $q$-Pfaff-Saalsch\"utz sum \cite[\href{http://dlmf.nist.gov/17.7.E5}{(17.7.5)}]{NIST:DLMF} can be applied. Adopting these values and inserting \eqref{sumthree2} into \eqref{lqJgamdual} with $t=qa$, simplifying and
then setting $\gamma\mapsto t$ completes the proof.
\end{proof}


\subsubsection{Generating relations from the Al-Salam--Chihara polynomials}

\noindent By using the duality relation between Al-Salam--Chihara polynomials and 
the little $q$-Jacobi functions, we are able to derive generating relations for the little $q$-Jacobi functions. From these generating relations, by using the continuous orthogonality of the little $q$-Jacobi functions we are able to derive integral generating relations. In the rest of this subsection we will derive these relations.

In this subsection, using the duality of the Al-Salam--Chihara polynomials \eqref{dASClqJa} with generating functions for these polynomials we derive generating relations for the little $q$-Jacobi polynomials. First consider \eqref{ASCgf1}.

\begin{thm}
\label{thm:7.13}
Let $q\in\CCdag$, $m\in\mathbb N_0$,
$\mu\in\CC$, 
$a,b,t\in\CCast$, $|t|<1$. Then
\begin{eqnarray}
&&\sum_{n=0}^\infty t^n \frac{(qb;q)_n}{(q;q)_n}
\,p_\mu\left(\frac{q^{-1-n}}{b};a,b;q\right)=
\frac{(qbt,qabt,\frac{q^{-\mu}}{b},q^{\mu+1}a;q)_\infty}
{(qa,\frac{1}{b},q^{-\mu}t,q^{\mu+1}abt;q)_\infty}.
\end{eqnarray}
\end{thm}
\begin{proof}
Start with the generating function for Al-Salam--Chihara polynomials \eqref{ASCgf1}
and replace
$z\mapsto \frac{q^{-m}}{a}$. Then inserting the duality relation
\eqref{dASClqJa} provides an generating relation in terms of little $q$-Jacobi polynomials. Then replace $t\mapsto at$ followed by $(a,b)\mapsto(\sqrt{qab},\sqrt{\frac{qb}{a}})$. Finally, letting
$m\mapsto \mu\in\CC$ by analytic continuation completes the proof.
\end{proof}

\noindent Next consider \eqref{ASCgf2}.

\begin{thm}
\label{thm:7.14}
Let $q\in\CCdag$, $\mu\in\CC$, $a,b,t\in\CCast$, $|t|<|q^{-\mu-1}/(ab)|$. Then
\begin{eqnarray}
&&\hspace{-0.8cm}\sum_{n=0}^\infty \frac{t^n}{(q;q)_n}
\,p_\mu\left(\frac{q^{-1-n}}{b};a,b;q\right)
=
\frac{(\frac{q^{-\mu}}{b},q^{\mu+1}a;q)_\infty}{(qa,\frac{1}{b},q^{-\mu}t;q)_\infty}
\qhyp21{q^{-\mu},\frac{q^{-\mu}}{a}}{qb}{q,q^{\mu+1}abt}.
\end{eqnarray}
\end{thm}
\begin{proof}
Start with the generating function for Al-Salam--Chihara polynomials \eqref{ASCgf2}
and replace
$z\mapsto \frac{q^{-m}}{a}$. Then inserting the duality relation
\eqref{dASClqJa} provides an generating relation in terms of little $q$-Jacobi polynomials. Then replace $t\mapsto at$ followed by $(a,b)\mapsto(\sqrt{qab},\sqrt{\frac{qb}{a}})$. Finally
letting $m\mapsto \mu\in\CC$ by analytic continuation completes the proof.
\end{proof}

\noindent Next consider \eqref{ASCgf3}.
\begin{thm}
\label{thm:7.15}
Let $q\in\CCdag$, $\mu\in\CC$, $a,b,t\in\CCast$, $|t|<1$. Then
\begin{eqnarray}
&&\hspace{-0.8cm}\sum_{n=0}^\infty t^n\frac{q^{\binom{n}{2}}}{(q;q)_n}
\,p_\mu\left(\frac{q^{-1-n}}{b};a,b;q\right)\!=\!
\frac{(-t,\frac{q^{-\mu}}{b},q^{\mu+1}a;q)_\infty}{(qa,\frac{1}{b};q)_\infty}
\qhyp21{q^{-\mu},q^{\mu+1}ab}{qb}{q,-t}\!\!.
\end{eqnarray}
\end{thm}
\begin{proof}
Start with the generating function for Al-Salam--Chihara polynomials \eqref{ASCgf3}
and replace
$z\mapsto \frac{q^{-m}}{a}$. Then inserting the duality relation
\eqref{dASClqJa} provides a generating relation in terms of little $q$-Jacobi polynomials. Then replace $t\mapsto at$ followed by $(a,b)\mapsto(\sqrt{qab},\sqrt{\frac{qb}{a}})$.
Finally
letting $m\mapsto \mu\in\CC$ by analytic continuation completes the proof.
\end{proof}

\noindent Now we consider the generating function with a free parameter $\gamma\in\CC$ given by
\eqref{ASCgf4}.

\begin{thm}
\label{thm:7.16}
Let $q\in\CCdag$, $\mu\in\CC$, $\gamma,a,b,t\in\CCast$, $|t|<|q^{-\mu-1}/(ab)|$. Then
\begin{eqnarray}
&&\hspace{-0.8cm}
\sum_{n=0}^\infty t^n\frac{(\gamma;q)_n}{(q;q)_n}
\,p_\mu\left(\frac{q^{-1-n}}{b};a,b;q\right)
=
\frac{(\frac{q^{-\mu}}{b},q^{-\mu}\gamma t,q^{\mu+1}a;q)_\infty}{(qa,\frac{1}{b},q^{-\mu}t;q)_\infty}
\qhyp32{q^{-\mu},\gamma,\frac{q^{-\mu}}{a}}{qb,q^{-\mu}\gamma t}{q,q^{\mu+1}abt}.
\end{eqnarray}
\end{thm}
\begin{proof}
Start with the generating function for Al-Salam--Chihara polynomials \eqref{ASCgf4}
and replace
$z\mapsto \frac{q^{-m}}{a}$. Then inserting the duality relation
\eqref{dASClqJa} provides an generating relation in terms of little $q$-Jacobi polynomials, replace $t\mapsto at$ and take $(a,b)\mapsto(\sqrt{qab},\sqrt{\frac{qb}{a}})$.
Finally
letting $m\mapsto \mu\in\CC$ by analytic continuation completes the proof.
\end{proof}

\noindent One interesting special case of Theorem \ref{thm:7.16} is when $\gamma=q$.

\begin{cor}\label{cor:7.17}
Let $q\in\CCdag$, $\mu\in\CC$, $a,b,t\in\CCast$, $|t|<|q^{-\mu-1}/(ab)|$. Then
\begin{eqnarray}
\sum_{n=0}^\infty t^n
\,p_\mu\left(\frac{q^{-1-n}}{b};a,b;q\right)
=\frac{(\frac{q^{-\mu}}{b},q^{1-\mu} t,q^{\mu+1}a;q)_\infty}{(qa,\frac{1}{b},q^{-\mu}t;q)_\infty}
\qhyp32{q,q^{-\mu},\frac{q^{-\mu}}{a}}{qb,q^{1-\mu}t}{q,q^{\mu+1}abt}.
\end{eqnarray}
\end{cor}
\begin{proof}
Replacing $\gamma=q$ in Theorem \ref{thm:7.16} completes the proof.
\end{proof}

\noindent In the limit as $\gamma\to 0$, Theorem \ref{thm:7.16} 
becomes Theorem \ref{thm:7.14} using \eqref{critlim}. Also if 
$\gamma=qb$ then Theorem \ref{thm:7.16} becomes Theorem 
\ref{thm:7.13} using the $q$-Gauss sum 
\cite[\href{http://dlmf.nist.gov/17.6.E1}{(17.6.1)}]{NIST:DLMF}.

\subsubsection{Orthogonal integral relations for little $q$-Jacobi functions}
\noindent 
\noindent Recall the definition of ${\sf A}={\sf A}(a,b;q)$ \eqref{Adef}.
From Theorem \ref{thm:7.13} and using the continuous orthogonality relations, we can obtain the following results.
\begin{thm}\label{thm:7.34}
Let $n\in\mathbb N_0$, $q\in\CCdag$, $a,b,t\in\CCast$. Then
\begin{eqnarray}
&&\hspace{-.7cm}\int_{0}^{\frac{\pi}{\log q^{-1}}}
\frac{(q^{\pm 2ix};q)_\infty\,p_{{\sf A}+ix}\left(\frac{q^{-1-n}}{b};a,b;q\right)}{(q^{\frac12\pm ix}\sqrt{ab},q^{\frac12\pm ix}\sqrt{\frac{a}{b}},q^{\frac12\pm ix}\sqrt{\frac{b}{a}},q^{\frac12\pm ix}\sqrt{ab}t;q)_\infty}
\,\dd x
\label{eq:thm:7.34}
=\frac{2\pi(qabt)^n}{(q,qa,qb,\frac{1}{b},qbt,qabt;q)_\infty\log q^{-1}}.
\end{eqnarray}
\end{thm}
\begin{proof}
Start with Theorem \ref{thm:7.13}, and multiply 
both sides by the the appropriate little $q$-Jacobi function and weight function. Then integrating taking advantage of the orthogonality 
relation in Theorem \ref{colqJ} completes the proof.
\end{proof}

\noindent From Theorem \ref{thm:7.14} and using the continuous orthogonality of little $q$-Jacobi functions we can obtain the following result.
\begin{thm}\label{thm:7.35}
Let $n\in\mathbb N_0$, $q\in\CCdag$, $a,b,t\in\CCast$, $|t|<1/\sqrt{|qab|}$. Then
\begin{eqnarray}
&&\hspace{-0cm}\int_{0}^{\frac{\pi}{\log q^{-1}}}
\frac{(q^{\pm 2ix};q)_\infty\,p_{{\sf A}+ix}\left(\frac{q^{-1-n}}{b};a,b;q\right)}{(q^{\frac12\pm ix}\sqrt{ab},q^{\frac12\pm ix}\sqrt{\frac{a}{b}},q^{\frac12\pm ix}\sqrt{\frac{b}{a}},q^{\frac12-ix}\sqrt{ab}t;q)_\infty}
\qhyp21{q^{\frac12-ix}\sqrt{ba^\pm}}{qb}{q,q^{\frac12+ix}\sqrt{ab}t}
\,\dd x\nonumber\\&&\hspace{8.5cm}\label{eq:thm:7.35}
=\frac{2\pi (qabt)^n}{(q,qa,qb,\frac{1}{b};q)_\infty\log q
^{-1}(qb;q)_n}.
\end{eqnarray}
\end{thm}
\begin{proof}
Start with Theorem \ref{thm:7.14}, and multiply 
both sides by the the appropriate little $q$-Jacobi function and weight function. Then integrating taking advantage of the orthogonality 
relation in Theorem \ref{colqJ} completes the proof.
\end{proof}

\noindent From Theorem \ref{thm:7.15} and using the continuous orthogonality of little $q$-Jacobi functions we can obtain the following result.
\begin{thm}\label{thm:7.36}
Let $n\in\mathbb N_0$, $q, t\in\CCdag$, $a,b\in\CCast$, $|t|<1$. Then
\begin{eqnarray}
&&\hspace{-2.7cm}\int_{0}^{\frac{\pi}{\log q^{-1}}}
\frac{(q^{\pm 2ix};q)_\infty\,p_{{\sf A}+ix}\left(\frac{q^{-1-n}}{b};a,b;q\right)}{(q^{\frac12\pm ix}\sqrt{ab},q^{\frac12\pm ix}\sqrt{\frac{a}{b}},q^{\frac12\pm ix}\sqrt{\frac{b}{a}};q)_\infty}
\qhyp21{q^{\frac12\pm ix}\sqrt{ab}}{qb}{q,-t}
\,\dd x
\nonumber\\&&\hspace{3.5cm}\label{eq:thm:7.36}
=\frac{q^{\binom{n}{2}}2\pi(qabt)^n}{(q,qa,qb,\frac{1}{b},-t;q)_\infty\log q^{-1}
(qb;q)_n}.
\end{eqnarray}
\end{thm}
\begin{proof}
Start with Theorem \ref{thm:7.15}, and multiply 
both sides by the the appropriate little $q$-Jacobi function and weight function. Then integrating taking advantage of the orthogonality 
relation in Theorem \ref{colqJ} completes the proof.
\end{proof}

\noindent From Theorem \ref{thm:7.16} and using the continuous orthogonality of little $q$-Jacobi functions we can obtain the following result.
\begin{thm}\label{thm:7.37}
Let $n\in\mathbb N_0$, $q\in\CCdag$, $a,b,\gamma,t\in\CCast$, $|t|<1/\sqrt{|qab|}$. Then
\begin{eqnarray}
&&\hspace{-0cm}\int_{0}^{\frac{\pi}{\log q^{-1}}}
\frac{(q^{\pm 2ix},q^{\frac12-ix}\sqrt{ab}\gamma t;q)_\infty\,p_{{\sf A}+ix}\left(\frac{q^{-1-n}}{b};a,b;q\right)}{(q^{\frac12\pm ix}\sqrt{ab},q^{\frac12\pm ix}\sqrt{\frac{a}{b}},q^{\frac12\pm ix}\sqrt{\frac{b}{a}},q^{\frac12-ix}\sqrt{ab}t;q)_\infty}
\qhyp32{\gamma,q^{\frac12-ix}\sqrt{ba^{\pm}}}{qb,q^{\frac12-ix}\sqrt{ab}\gamma t}{q,q^{\frac12+ix}\sqrt{ab}t}
\,\dd x
\nonumber\\&&\hspace{7.4cm}\label{eq:thm:7.37}
=\frac{2\pi(qabt)^n(\gamma;q)_n}{(q,qa,qb,\frac{1}{b};q)_\infty\log q
^{-1}(qb;q)_n}.
\end{eqnarray}
\end{thm}
\begin{proof}
Start with Theorem \ref{thm:7.16}, and multiply 
both sides by the the appropriate little $q$-Jacobi function and weight function. Then integrating taking advantage of the orthogonality 
relation in Theorem \ref{colqJ} completes the proof.
\end{proof}

\noindent From Corollary \ref{cor:7.17} and using the continuous orthogonality of little $q$-Jacobi functions we can obtain the following result.
\begin{thm}\label{thm:8.38}
Let $n\in\mathbb N_0$, $q\in\CCdag$, $a,b,t\in\CCast$, $|t|<1/\sqrt{|qab|}$.  Then
\begin{eqnarray}
&&\hspace{-0cm}\int_{0}^{\frac{\pi}{\log q^{-1}}}
\frac{(q^{\pm 2ix},q^{\frac32-ix}\sqrt{ab} t;q)_\infty\,p_{{\sf A}+ix}\left(\frac{q^{-1-n}}{b};a,b;q\right)}{(q^{\frac12\pm ix}\sqrt{ab},q^{\frac12\pm ix}\sqrt{\frac{a}{b}},q^{\frac12\pm ix}\sqrt{\frac{b}{a}},q^{\frac12-ix}\sqrt{ab}t;q)_\infty}
\qhyp32{q,q^{\frac12-ix}\sqrt{ba^{\pm}}}{qb,q^{\frac32-ix}\sqrt{ab}t}{q,q^{\frac12+ix}\sqrt{ab}t}
\,\dd x
\nonumber\\&&\hspace{7.4cm}\label{eq:thm:7.37}
=\frac{2\pi(qabt)^n(q;q)_n}{(q,qa,qb,\frac{1}{b};q)_\infty\log q
^{-1}(qb;q)_n}.
\end{eqnarray}
\end{thm}
\begin{proof}
Start with Corollary \ref{cor:7.17}, and multiply 
both sides by the the appropriate little $q$-Jacobi function and weight function. Then integrating taking advantage of the orthogonality 
relation in Theorem \ref{colqJ} completes the proof.
\end{proof}

\subsection{Generating relations for the $q$-Bessel polynomials}

By using the duality relation between the continuous big $q^{-1}$-Hermite polynomials and 
the $q$-Bessel polynomials, we are able to derive generating relations for the $q$-Bessel polynomials. From these generating relations, by using the infinite discrete orthogonality of the $q$-Bessel polynomials, we are able to compute infinite discrete generating relations. In the rest of this subsection we will derive these relations.

\subsubsection{Generating relations from the continuous big $q^{-1}$-Hermite polynomials}

\noindent Starting from generating functions for the continuous big $q^{-1}$-Hermite polynomials, we can obtain dual generating relations for the $q$-Bessel polynomials. Here's an example which comes from the generating function \eqref{cbqinHegf-2}.

\begin{thm}
Let $q\in\CCdag$, $m\in\N_0$, $a, t\in\CCast$, $|t|<1$. Then
\begin{eqnarray}
&&\sum_{n=0}^\infty
\frac{t^nq^{\binom{n}{2}}}{(q;q)_n}y_m(q^n;a;q)
=q^{\binom{m}{2}}t^m\frac{(-t;q)_\infty}{(-t;q)_m}\left(\dfrac{qa}{t};q\right)_m.
\label{grqB1}
\end{eqnarray}
\end{thm}

\begin{proof}
Start from the generating function for the continuous big $q^{-1}$-Hermite polynomials \eqref{cbqinHegf-2}, utilizing its duality relation with the $q$-Bessel polynomials 
\eqref{dHqinym} and then replacing $t\mapsto ta$, followed by $-1/a^2\mapsto a$ completes the proof.
\end{proof}

\noindent As noted in \S\ref{sec:6.7}, there is only one known generating function known for the continuous big $q^{-1}$-Hermite polynomials. However, as demonstrated in 
Corollary \ref{cqiHgf4gcor}, there does exist a generalized generating function for the continuous big $q^{-1}$-Hermite polynomials and
we can use that in order to obtain a generalized generating relation for the $q$-Bessel polynomials.

\begin{cor}
Let $q\in\CCdag$, $m\in\N_0$, $a,t\in\CCast$, $|t|<\sqrt{|a|}$. Then
\begin{eqnarray}
&&\hspace{-2.5cm}\sum_{n=0}^\infty
\frac{q^{\frac12\binom{n}{2}}(q^\frac14t)^n}{(q;q)_n}
\qhyp11{0}{-q^\frac12}{q^\frac12,q^{\frac14-\frac12n}\frac{t}{a}}
y_m(q^n;a;q)\nonumber\\
&&\hspace{1cm}
=\frac{q^{2\binom{m}{2}}(-qa)^m}{
(\frac{t}{\sqrt{-a}},
\frac{q^\frac12 t}{\sqrt{-a}}
;q)_\infty
}
\qhyp21{
q^{m+\frac14}\sqrt{-a},\frac{q^{\frac14-m}}{\sqrt{-a}}}{-q^\frac12}{q^{\frac12},-\frac{t}{\sqrt{-a}}}.
\label{grqB2}
\end{eqnarray}
\end{cor}

\begin{proof}
Start from the generalized generating function for the continuous big $q^{-1}$-Hermite polynomials given in Corollary \ref{cqiHgf4gcor}, utilizing its duality relation with the $q$-Bessel polynomials 
\eqref{dHqinym} and then replacing $t\mapsto ta$ followed by $-1/a^2\mapsto a$ completes the proof.
\end{proof}

\subsubsection{Orthogonal discrete generating relations for the $q$-Bessel polynomials}

\begin{thm}
Let $q\in\CCdag$, $m\in\N_0$, $a,t\in\CCast$. Then
\begin{equation}
\sum_{n=0}^\infty \frac{t^n(-qa;q)_{2n}(-a,\frac{1}{t};q)_n}{(-a;q)_{2n}(q,-qat;q)_n}y_n(q^m;a;q)=\frac{(-qa;q)_\infty}{(-qat;q)_\infty}t^m.
\end{equation}
\end{thm}

\begin{proof}
Start with the generating relation for the $q$-Bessel polynomials \eqref{grqB1}, multiply
both sides by the appropriate weight function and $q$-Bessel polynomial so that one may use the orthogonality relation \eqref{orthogqB1} with application and replacing $n\leftrightarrow m$ followed by $-1/a^2\mapsto a$ completes the proof.
\end{proof}

\noindent Note that in the limit case $a\to\infty$ the sum on the left-hand side reduces to a terminating ${}_2\phi_0$ which can be summed using
\cite[(1.11.7)]{Koekoeketal}.

\begin{thm}
Let $m\in\N_0$, $q\in\CCdag$, $a,t\in\CCast$, $|t|<\sqrt{|a|}$. Then
\begin{eqnarray}
&&\hspace{-0.9cm}\sum_{n=0}^\infty (-1)^n q^{\binom{n}{2}}\frac{(-qa;q)_{2n}(-a;q)_n}{(-a;q)_{2n}(q;q)_n}
\qhyp21
{
q^{n+\frac14}\sqrt{-a},
\frac{q^{\frac14-n}}{\sqrt{-a}}
}
{-q^\frac12}{q^\frac12,-\frac{t}{\sqrt{-a}}} y_n(q^m;a;q)\nonumber\\
&&\hspace{1cm}=
q^{-\frac12\binom{m}{2}}\left(-q^{-\frac34}\frac{t}{a}\right)^m(-qa;q)_\infty
\left(\frac{t}{\sqrt{-a}};q^\frac12\right)_\infty
\qhyp11{0}{-q^{\frac12}}{q^\frac12,q^{\frac14-\frac{1}{2}m}\frac{t}{a}}.
\end{eqnarray}
\end{thm}

\begin{proof}
Start with the generating relation for the $q$-Bessel polynomials \eqref{grqB2}, multiply
both sides by the appropriate weight function and $q$-Bessel polynomial so that one may use the orthogonality relation \eqref{orthogqB1} with application and replacing $n\leftrightarrow m$ followed by $-1/a^2\mapsto a$ completes the proof.
\end{proof}

\subsection{Generating relations for the $q^{-1}$-Bessel functions}

By using the duality relation between the continuous big $q$-Hermite 
polynomials and the $q^{-1}$-Bessel functions, we are able to derive 
generating relations for the $q^{-1}$-Bessel functions. 
From these generating relations, by using the continuous orthogonality 
of the $q^{-1}$-Bessel functions, we are able to compute integral 
generating relations. In the rest of this subsection we will derive 
these relations.

\subsubsection{Generating relations from continuous 
big $q$-Hermite polynomials}

\noindent Utilizing the duality of the continuous big $q^{-1}$-Hermite 
polynomials with the $q^{-1}$-Bessel function \eqref{dHnymu} and the 
duality of the continuous big $q^{-1}$-Hermite polynomials with the 
$q$-Bessel polynomials \eqref{dHqinym}, we can derive using generating 
functions for these polynomials, 
new generating relations for these polynomials and functions.

\medskip

\noindent 
Starting from \eqref{cbqHgf1}, and utilizing the duality relation 
between continuous big $q$-Hermite polynomials and $q^{-1}$-Bessel 
functions we obtain the following generating relation.

\begin{thm}
\label{thm848}
Let $q\in\CCdag$, $\mu\in\CC$, $t,a\in\CCast$. 
Then
\begin{eqnarray}
&& \sum_{n=0}^\infty \frac{t^n}{(q;q)_n}y_\mu(q^{-n};a;q^{-1})
=\frac{q^{-2\binom{\mu}{2}-\mu}
(-a)^\mu(-\frac{t}{a};q)_\infty}{(q^{-\mu}t,-q^\mu \frac{t}{a};q)_\infty}.
\label{grqibf1}
\end{eqnarray} 
\end{thm}

\begin{proof}
Starting with the generating function for the continuous big $q$-Hermite 
polynomials \eqref{cbqHgf1} and inserting its duality relation with 
the $q^{-1}$-Bessel functions \eqref{dHnymu} and replacing 
$t\mapsto at$ followed by $-1/a^2\mapsto a$ completes the proof.
\end{proof}

\begin{thm}
\label{thm849}
Let $q\in\CCdag$, $\mu\in\CC$, $t,a\in\CCast$, $|t|<1$. Then
\begin{eqnarray}
&& \sum_{n=0}^\infty \frac{q^{\binom{n}{2}}t^n}{(q;q)_n}y_\mu
(q^{-n};a;q^{-1})=q^{-2\binom{\mu}{2}-\mu}
(-a)^{\mu}(-t;q)_\infty\qhyp21{q^{-\mu},-\frac{q^\mu}{a}}{0}{q,-t}.
\label{grqibf2}
\end{eqnarray} 
\end{thm}

\begin{proof}
Starting with the generating function for the continuous big $q$-Hermite 
polynomials \eqref{gf2cbqH}, inserting its duality relation with 
the $q^{-1}$-Bessel functions \eqref{dHnymu}, replacing $t\mapsto at$ 
followed by $-1/a^2\mapsto a$ completes the proof.
\end{proof}

\begin{thm}
\label{thm850}
Let $q\in\CCdag$, $\mu\in\CC$, $\gamma,t,a\in\CCast$, 
 $|t|<|q^{-\mu}a|$. Then
\begin{eqnarray}
&& \sum_{n=0}^\infty \frac{(\gamma;q)_nt^n}{(q;q)_n}y_\mu(q^{-n};a;q^{-1})=q^{-2\binom{\mu}{2}-\mu}(-a)^\mu
\frac{(q^{-\mu}\gamma t;q)_\infty}{(q^{-\mu}t;q)_\infty}
\qhyp21{q^{-\mu},\gamma}{q^{-\mu}\gamma t}{q,-\frac{q^\mu t}{a}}.
\label{grqibf3}
\end{eqnarray} 
\end{thm}

\begin{proof}
Starting with the generating function for the continuous big $q$-Hermite 
polynomials \eqref{cbqHgf4}, inserting its duality relation with 
the $q^{-1}$-Bessel functions \eqref{dHnymu} and replacing 
$t\mapsto at$ followed by $-1/a^2\mapsto a$ completes the proof.
\end{proof}

One also can compute a generalized generating relation for 
$q^{-1}$-Bessel functions by starting with an extension to 
continuous big $q$-Hermite polynomials for its $q$-exponential 
generating function \eqref{cqHgf4cor}.

\begin{cor}
\label{cor851}
Let $q\in\CCdag$, $\mu\in\CC$, $a,t\in\CCast$, $|t|<\sqrt{|a|}$. Then
\begin{eqnarray}
&&\hspace{-0.3cm}\sum_{n=0}^\infty
\frac{q^{\frac12\binom{n}{2}}(q^\frac14t)^n}{(q;q)_n}
\qhyp11{0}{-q^\frac12}{q^\frac12,q^{\frac14+\frac12n}\frac{t}{a}}y_\mu(q^{-n};a;q^{-1})\nonumber\\
&&\hspace{2cm}=q^{-2\binom{\mu}{2}-\mu}\expe^{i\pi\mu}
a^{\mu}\left(-\frac{t}{\sqrt{-a}};q^\frac12\right)_\infty
\qhyp21{\frac{q^{\mu+\frac14}}{\sqrt{-a}},q^{\frac14-\mu}\sqrt{-a}}{-q^\frac12}{q^\frac12,-\frac{t}{\sqrt{-a}}}.
\label{grqibf4}
\end{eqnarray}
\end{cor}

\begin{proof}
Start with the generalized generating function for
continuous big $q$-Hermite polynomials which generalizes 
is $q$-exponential generating function Corollary 
\ref{cqHgf4cor}. Then inserting the duality
relation with the $q^{-1}$-Bessel functions \eqref{dHnymu} 
and replacing $t\mapsto at$ followed by $-1/a^2\mapsto a$ 
completes the proof.
\end{proof}

\subsubsection{Orthogonal integral relations for 
the $q^{-1}$-Bessel functions}

Using the continuous orthogonality of the $q^{-1}$-Bessel function which arises from duality with the continuous big $q$-Hermite polynomials we can obtain integral generating relations.  Recall the definition of ${\sf B}={\sf B}(a;q)$ \eqref{Bdef}.
For instance, from Theorem \ref{thm848} we obtain the following integral evaluation.

{\begin{thm}
Let $q\in\CCdag$, $n\in\N_0$,
$a\in\CC\setminus(1,\infty)$, such that $|a|>1$,
$t\in\CC$, $|t|<\sqrt{|a|}$. Then
\begin{eqnarray}
&&\hspace{-4cm}\int_{0}^{\frac{\pi}{\log q^{-1}}}
\frac{y_{{\sf B}+ix}(q^{-n};a;q^{-1})q^{-x^2}(q^{\pm 2ix};q)_\infty}
{(\frac{q^{\pm ix}}{\sqrt{-a}},\frac{q^{\pm ix}t}{\sqrt{-a}};q)_\infty}
\,\dd x
=\frac{2\pi \expe^{-\frac{\log(-a)^2}{4\log q^{-1}}}\left(-\frac{t}{a}\right)^n}{(q,-\frac{t}{a};q)_\infty\log q^{-1}}.
\end{eqnarray}
\end{thm}}

\begin{proof}
Start with the generating relation \eqref{grqibf1} and applying the continuous orthogonality relation \eqref{COqiBf} completes the proof.
\end{proof}

\noindent From Theorem \ref{thm849} we obtain the following integral evaluation.

{
\begin{thm}
Let $q\in\CCdag$, $n\in\N_0$, $t\in\CCast$, $|t|<1$, $a\in\CC\setminus(1,\infty)$, such that $|a|>1$. Then
\begin{eqnarray}
&&\hspace{-1cm}\int_{0}^{\frac{\pi}{\log q^{-1}}}
\frac{y_{{\sf B}+ix}(q^{-n};a;q^{-1})q^{-x^2}(q^{\pm 2ix};q)_\infty}
{(\frac{q^{\pm ix}}{\sqrt{-a}};q)_\infty}\qhyp21{\frac{q^{\pm ix}}{\sqrt{-a}}}{0}{q,-t}
\dd x
=\frac{q^{\binom{n}{2}}2\pi \expe^{-\frac{\log(-a)^2}{4\log q^{-1}}}\left(-\frac{t}{a}\right)^n}{(q,-t;q)_\infty\log q^{-1}}.
\end{eqnarray}
\end{thm}}

\begin{proof}
Start with the generating relation \eqref{grqibf2} and applying the continuous orthogonality relation \eqref{COqiBf} completes the proof.
\end{proof}

\noindent From Theorem \ref{thm850} we obtain the following integral evaluation.

{
\begin{thm}
Let $q\in\CCdag$, $n\in\N_0$, $\gamma\in \C$, $t\in\CCast$, $a\in\CC\setminus(1,\infty)$, such that $|a|>1$, $|t|<\sqrt{|a|}$. Then
\begin{eqnarray}
&&\hspace{-3.5cm}\int_{0}^{\frac{\pi}{\log q^{-1}}}
\frac{y_{{\sf B}+ix}(q^{-n};a;q^{-1})q^{-x^2}(q^{\pm 2ix},\frac{q^{-ix}\gamma t}{\sqrt{-a}};q)_\infty}
{(\frac{q^{\pm ix}}{\sqrt{-a}},\frac{q^{-ix}t}{\sqrt{-a}};q)_\infty}\qhyp21{\gamma,\frac{q^{-ix}}{\sqrt{-a}}}{\frac{q^{-ix}\gamma t}{\sqrt{-a}}}{q,\frac{q^{ix}t}{\sqrt{-a}}}
\dd x
\nonumber\\
&&\hspace{4cm}
=\frac{2\pi \expe^{-\frac{\log(-a)^2}{4\log q^{-1}}}(\gamma;q)_n\left(-\frac{t}{a}\right)^n}{(q;q)_\infty\log q^{-1}}.
\end{eqnarray}
\end{thm}}

\begin{proof}
Start with the generating relation \eqref{grqibf3} and applying the continuous orthogonality relation \eqref{COqiBf} completes the proof.
\end{proof}

\noindent From Corollary \ref{cor851} we obtain the following integral evaluation.
\begin{thm}
Let $q\in\CCdag$, $n\in\N_0$, $t,a\in\CCast$, $a\in\CC\setminus(1,\infty)$, such that $|a|>1$, $|t|<\sqrt{|a|}$. Then
\begin{eqnarray}
&&\hspace{-1.8cm}\int_{0}^{\frac{\pi}{\log q^{-1}}}
\frac{y_{{\sf B}+ix}(q^{-n};a;q^{-1})q^{-x^2}(q^{\pm 2ix};q)_\infty}
{(\frac{q^{\pm ix}}{\sqrt{-a}};q)_\infty}\qhyp21{q^{\frac14\pm ix}}{-q^\frac12}{q^\frac12,\frac{-t}{\sqrt{-a}}}
\,\dd x
\nonumber\\
&&\hspace{3cm}
=\frac{2\pi \expe^{-\frac{\log(-a)^2}{4\log q^{-1}}}q^{\frac12\binom{n}{2}}\left(-q^\frac14\frac{t}{a}\right)^n}{(q,\frac{-t}{\sqrt{-a}},\frac{-\sqrt{q}t}{\sqrt{-a}};q)_\infty\log q^{-1}}\qhyp11{0}{-q^\frac12}{q^\frac12,q^{\frac{n}{2}+\frac14}\frac{t}{a}}.
\end{eqnarray}
\end{thm}

\begin{proof}
Start with the generating relation \eqref{grqibf4} and applying the continuous orthogonality relation \eqref{COqiBf} completes the proof.
\end{proof}

\subsection*{Acknowledgements}
We would like to thank Wolter Groenevelt, Mourad Ismail, 
Tom Koornwinder, Hjalmar Rosengren, Michael Schlosser, 
Hans Volkmer and Ole Warnaar for valuable discussions. 
We would like to point out that the duality relations 
for the continuous big $q$ and big $q^{-1}$-Hermite polynomials 
to the $q$ and $q^{-1}$-Bessel polynomials were originally 
pointed out by Wolter Groenevelt in a personal 
discussion with the first author. Thanks also to Xiang-Sheng Wang for assistance with some crucial asymptotic computations using Darboux's method.

\def\cprime{$'$} \def\dbar{\leavevmode\hbox to 0pt{\hskip.2ex \accent"16\hss}d}

\end{document}